\pdfoutput=1
\documentclass[11pt]{amsart}
\usepackage{pinlabel}
\usepackage{amsmath} 
\usepackage{cjhebrew}
\usepackage{graphicx} 
\usepackage{subcaption}
\usepackage{tikz-cd}
\usepackage[all,cmtip]{xy}
\setkeys{Gin}{keepaspectratio}
\usepackage{hyperref}
\usepackage{url}
\usepackage{bm}
\usepackage{mathabx}
\usepackage[left=1.25in,top=1in,right=1.25in,bottom=1in,head=.1in]{geometry}
\usepackage{xcolor}
\usepackage{tikz}
\usepackage{adjustbox}
\usepackage[normalem]{ulem} % 
\usepackage{enumerate}
\usepackage{hyperref}

\usepackage{verbatim}
\newcommand{\items}{\begin{itemize}[leftmargin=25pt,rightmargin=15pt]
  \setlength\itemsep{2pt}}

\usepackage{hyperref}

\newcommand{\stopitems}{\end{itemize}}

\usepackage{fancyhdr}
\newtheorem{thm}{Theorem} % lettered theorems (A,B,C,D)

\newtheorem{cor}{Corollary}

\newtheorem{theorem}{Theorem}[section] % numbered theorems, lemmas, etc
\newtheorem*{theorem*}{Theorem}
\newtheorem{lemma}[theorem]{Lemma}
\newtheorem{conjecture}[theorem]{Conjecture}

\newtheorem*{conjecture*}{Conjecture}
\newtheorem*{question*}{Question}
\newtheorem*{lemma*}{Lemma}
\newtheorem{proposition}[theorem]{Proposition}
\newtheorem{corollary}[theorem]{Corollary}
\newtheorem*{corollary*}{Corollary}

\theoremstyle{definition}
\newtheorem{definition}[theorem]{Definition}

\newtheorem{remark}[theorem]{Remark}

\newtheorem{example}[theorem]{Example}

\newtheorem*{example*}{Example}
\newtheorem*{remark*}{Remark}
\newtheorem*{remarks*}{Remarks}
\newtheorem*{addenda*}{Addenda}
\newtheorem*{construction*}{Construction}

\renewcommand{\phi}{\varphi}

\newcommand{\unred}[1]{ \ignorespaces}  %% to omit that red text
\usepackage{stmaryrd}

\usepackage{mathrsfs}

\title[Counting $SL(2,\mathbb{C})$ connections on Seifert-fibered spaces]{Counting $SL(2,\mathbb{C})$ connections on Seifert-fibered spaces}

\author{Juan Muñoz-Echániz}
\address{Simons Center for Geometry and Physics, State University of New York, Stony Brook, 11794, USA}
\email{jmunozechaniz@scgp.stonybrook.edu}

\begin{document}
\maketitle
\setlength{\headheight}{12.0pt}

\begin{abstract}
We study the $SL(2, \mathbb{C})$ character variety of a Seifert-fibered homology $3$-sphere from the point of view of gauge theory. Namely, we introduce a class of perturbations of the $SL(2,\mathbb{C})$ Chern--Simons functional and prove a localisation result: the perturbed critical points either approach a compact subset of the $SL(2, \mathbb{C})$ character variety or else `escape to infinity'. Furthermore, the Euler characteristic and Poincaré polynomial of the stable locus of the character variety are obtained by suitably counting the localising critical points.

As an application, we obtain formulae for the Euler characteristic and Poincaré polynomial of the stable locus of the $SL(2, \mathbb{C})$ character variety of a Seifert-fibered homology $3$-sphere. In particular, we prove that the Euler characteristic equals the Milnor number (divided by four) of any weighted-homogeneous isolated complete intersection singularity whose link is the given $3$-manifold.

\end{abstract}

\section{Introduction}

In the recent years, new invariants of manifolds of dimension $3$ and $4$ have been introduced, motivated by $SL(2,\mathbb{C})$ generalisations of the anti-self-duality equation \cite{VW,KW,fivebranes,haydys}. Abouzaid--Manolescu \cite{abouzaid-manolescu,cote-manolescu} have defined a sheaf-theoretic model $HP^\ast (Y)$ for $SL(2, \mathbb{C})$ Floer cohomology of a closed oriented $3$-manifold $Y$ based on the theory of perverse sheaves and derived critical loci \cite{bussi,joyce,shifted}; and Tanaka--Thomas \cite{tanaka-thomas,tanaka-thomasii} have defined Vafa--Witten invariants of complex projective surfaces using virtual localisation. %There are physicial predictions relating these invariants with the Jones polynomial and Khovanov homology \cite{fivebranes,witten-lectures}, and to 
However, it remains an open problem to \textit{re-interpret these invariants in differential-geometric terms}. The main obstacle here is the non-compactness of the relevant gauge-theoretic moduli spaces, due to the phenomenon of energy concentration (`bubbling') along codimension $2$ singular sets studied by Taubes and other authors \cite{taubes-flat,taubes-kw,taubes-vw,HW,endshiggs,WZ}.

In this article, we consider the $3$-dimensional case. For a $3$-manifold $Y$ (always assumed closed, oriented and connected) its $SL(2, \mathbb{C})$ \textit{character scheme} (informally, the character `variety') is the affine scheme $\mathscr{M} = \mathscr{M}(Y)$ obtained from the scheme of representations $\rho : \pi_1 (Y) \to SL(2, \mathbb{C})$ as the affine GIT quotient by the conjugation $SL(2, \mathbb{C})$ action:
\[
\mathscr{M} = \mathrm{Hom}( \pi_1 (Y) , SL(2, \mathbb{C}) )\sslash SL(2, \mathbb{C} ).
\]
A representation $\rho$ is \textit{polystable} (or completely reducible) if every $\rho$--invariant line $L \subset \mathbb{C}^2$ admits a $\rho$--invariant complement, and is \textit{stable} (or irreducible) if it admits no $\rho$--invariant line at all. The (closed) points of $\mathscr{M}$ are in correspondence with the orbits of polystable representations, and we denote by $\mathscr{M}^\ast \subset \mathscr{M}$ the open subset of stable representations. %(From now on, we shall only consider the analytic topology on these spaces.)

%Typically $\mathscr{M}$ and $\mathscr{M}^\ast$ are non-compact (in the analytic topology), which makes it difficult to extract a meaningful numerical invariant of $Y$ by `counting' points on them.%See \cite{sikora} for details. 

From the point of view of gauge theory, $\mathscr{M}^\ast$ arises as the critical locus of the $SL(2, \mathbb{C})$ Chern--Simons functional on a suitable space of $SL(2, \mathbb{C})$ connections. %, but its non-compactness makes it difficult to extract an invariant of $Y$ which `counts' the points of $\mathscr{M}^\ast$. 
In this article, we consider the case when $Y$ is a \textit{Seifert-fibered homology sphere} and interpret the Euler characteristic and Poincaré polynomials of $\mathscr{M}^\ast $ as a \textit{gauge-theoretic count of critical points} (Theorems \ref{theorem:localisationseifert}-\ref{theorem:localisationseifert_exact}). As an application, we derive formulae for these (Corollary \ref{corollary:poincare}) and prove that the Euler characteristic of $\mathscr{M}^\ast$ equals one-fourth the \textit{Milnor number} of any weighted-homogeneous isolated complete intersection singularity whose link is $Y$ (Corollary \ref{corollary:milnornumber}). 

%However, typically $\mathscr{M}$ and $\mathscr{M}^\ast$ are non-compact (in the analytic topology), which makes it difficult to extract a numerical invariant of $Y$ which `counts' the points in $\mathscr{M}$ or $\mathscr{M}^\ast$ in a meaningful way. 
%In this article %we put forward a proposal in this direction %using differential-geometric techniques, 
%and we test it on the class of Seifert-fibered $3$-manifolds. As an application, we derive formulae for the Poincaré polynomial and Euler characteristic of $\mathscr{M}^\ast$ for a Seifert-fibered homology sphere $Y$, and prove that the Euler characteristic recovers one-fourth the Milnor number of any weighted-homogeneous isolated complete intersection singularity whose link is $Y$.

\subsection{A perturbation of the Chern--Simons functional}

Let $Y$ be a $3$-manifold (always assumed closed, oriented and connected), equipped with a fixed $SL(2, \mathbb{C})$ bundle $E \to Y$ (necessarily trivial, since $\mathrm{dim}Y =3$). The $SL(2, \mathbb{C})$ \textit{Chern--Simons functional} is a natural function on the space of $SL(2, \mathbb{C})$ connections on $E$, whose value on a connection $\mathbb{A}$ is denoted $CS ( \mathbb{A} ) \in \mathbb{C}$ (cf. Definition \ref{definition:complexCS}). The critical points of $CS$ are the \textit{flat} connections: $F_{\mathbb{A}} = 0$. %, which correspond with conjugacy classes of representations $\pi_1(Y) \to SL(2, \mathbb{C})$.
We would like to extract a numerical invariant of $Y$ by `counting' flat $SL(2, \mathbb{C})$ connections--much as in the $SU(2)$ case \cite{taubes-casson,floer,donaldson-floer}. However, the lack of a good compactness theory makes this more difficult. 
%a (conjectural) proposal to extract a numerical invariant of $Y$ by introducing a 
Instead, we propose to count (certain) critical points of a \textit{perturbation} of the $SL(2, \mathbb{C})$ Chern--Simons functional, which we now describe. We shall later use this perturbation to obtain well-defined counts of $SL(2, \mathbb{C})$ connections on Seifert-fibered homology spheres (Theorem \ref{theorem:localisationseifert} and Theorem \ref{theorem:localisationseifert_exact} below).

%, and counting those perturbed critical points which remain nearby the unperturbed critical points.

Our starting point is to fix a Riemannian metric $g$ on $Y$ and a reduction of the structure group of $E$ to $SU(2) \subset SL(2,\mathbb{C})$. An $SL(2, \mathbb{C})$ connection $\mathbb{A}$ is then uniquely expressed as $\mathbb{A} = A + i \Phi$ where $A$ is an $SU(2)$ connection and $\Phi$ is a $1$-form valued in the bundle $\mathfrak{su}(E)$ of skew-adjoint endomorphisms of $E$. Fixing an angle $\theta \in \mathbb{R}/2\pi\mathbb{Z}$ we can consider the function $S_\theta (\mathbb{A} )= \mathrm{Re}(e^{-i \theta} CS (\mathbb{A} )) \in \mathbb{R}$, whose critical points are also the flat connections.

\begin{definition}\label{perturbation:intro} For a parameter $\varepsilon \in \mathbb{R}$ we consider the \textit{perturbation} of $S_\theta$ whose value at an $SL(2, \mathbb{C})$ connection $\mathbb{A} = A + i \Phi$ is\\
\begin{align*}
S^{\varepsilon}_\theta  (\mathbb{A}) := S_\theta (\mathbb{A}) + \varepsilon \| \Phi \|_{L^{2}}^2 + \varepsilon^2  f ( \underline{\sigma}(\mathbb{A} ) ) 
\end{align*}
where:
\begin{enumerate}
\item $\| \Phi \|_{L^2}^2 = - \int_Y \mathrm{Tr} (\Phi \wedge \ast \Phi )$ is the \textit{$L^2$ norm squared} of $\Phi$. 
\item For some integer $N>0$, $f: \mathbb{C}^N \to \mathbb{R}$ is a smooth function and $\underline{\sigma}(\mathbb{A}) =  ( \sigma_1 (\mathbb{A} ) , \ldots , \sigma_N (\mathbb{A}))$ is a function where each component $\sigma_i ( \mathbb{A} )$ is an $SL(2, \mathbb{C})$ \textit{holonomy perturbation} (or `regularised' Wilson loop operator) associated with a smooth solid torus embedding $q_i : S^1 \times D^2 \hookrightarrow Y $ and a non-negative $2$-form $\mu_i$ on $D^2$ with integral one which vanishes near the boundary:
\[
\sigma_i ( \mathbb{A} ) = \int_{D^2} \mathrm{Tr} ( \mathrm{Hol}_{\mathbb{A}} (q(\cdot , z ) )) \cdot \mu_i (z)  .
\]
\end{enumerate}
\end{definition}

To properly make sense of this perturbation we now describe the appropriate space of connections on which $S_{\theta}^{\varepsilon}$ acts. In our context--with a fixed Riemannian metric and a reduction to $SU(2)$--we will say that $\mathbb{A}$ is \textit{polystable} (or completely reducible) if $d_{A}^\ast \Phi = 0$, where $d_{A}^\ast$ is the $L^2$--adjoint of the exterior covariant derivative, and that $\mathbb{A}$ is \textit{stable} (or irreducible) if in addition its stabilizer $\mathscr{G}_{\mathbb{A}}$ under the action of the $SU(2)$ gauge group $\mathscr{G}$ is finite (necessarily, it must then be $\{ \pm I \}$, the center of $SU(2)$). It will be convenient to work with connections in the Sobolev class $L^{2}_k$ and gauge transformations in $L^{2}_{k+1}$ (where $k \geq 2$ shall suffice for our purposes). The space $\mathscr{B}^\ast$ of stable $\mathscr{G}$--orbits of the action can then be regarded as an `infinite-dimensional' Kähler quotient (of the space of all $SL(2, \mathbb{C})$ connections by the $\mathscr{G}$--action) and is naturally a Hilbert Kähler manifold (cf. Corollary \ref{corollary:Bhilbert}):
\[
\mathscr{B}^\ast = \Big\{ \mathbb{A} = A + i \Phi \, | \, d_{A}^\ast \Phi = 0 \text{ and } \mathscr{G}_{\mathbb{A}} = \{ \pm I \}  \Big\}/\mathscr{G}.
\]

The function $S_{\theta}^\varepsilon$ has periods under the action of gauge-transformation (except when $\theta =  \pm \pi/2$ modulo $2 \pi$). Nevertheless, its derivative $d S_{\theta}^\varepsilon$ makes sense as a smooth $1$-form on $\mathscr{B}^\ast$. The critical locus of $S_\theta = S_{\theta}^0$ on $\mathscr{B}^\ast$ %The critical locus of $S_\theta  = S_{\theta}^0$ 
is the moduli space of $\mathscr{G}$--orbits of \textit{stable flat connections}
\begin{align*}
F_{A} - \Phi \wedge \Phi = 0 \quad , \quad d_{A}\Phi = 0 \quad , \quad d_{A}^\ast\Phi  = 0 %\label{stableflat}
\end{align*}
%The $SL(2, \mathbb{C})$ Chern--Simons functional (cf. Definition \ref{definition:complexCS}) can now be regarded as a closed holomorphic $1$-form $dCS$ on the Hilbert Kähler manifold $\mathscr{B}^\ast$, and its critical points (i.e. $(dCS)_{[\mathbb{A}]} = 0$) are given by the $\mathscr{G}$--orbits of \textit{stable flat connections}:
%\[
%F_{A} - \Phi \wedge \Phi = 0 \quad , \quad d_{A}\Phi = 0 \quad , d_{A}^\ast = 0.
%\]
which by a fundamental result of Donaldson and Corlette \cite{donaldson-harmonic,corlette-harmonic} is identified homeomorphically with the stable locus $\mathscr{M}^\ast$ of the character variety (with the analytic topology). Subsequently, we shall use the notation $\mathscr{M}^\ast$ to refer both to the critical locus of $S_\theta$ on $\mathscr{B}^\ast$ and to the stable locus of the $SL(2, \mathbb{C})$ the character variety of $Y$, and the notation $\mathscr{M}_{\theta}^{\ast , \varepsilon }$ for the critical locus of the perturbation $S_{\theta}^\varepsilon$.

When it comes to `counting', a crucial property is that the \textit{Hessian} of $S^{\varepsilon}_\theta$ at an orbit $[\mathbb{A}]$ defines a self-adjoint Fredholm operator on the ($L^2$ completion of the) tangent space to $\mathscr{B}^\ast$ at $[\mathbb{A}]$.
%\[
%\mathscr{D}_{[\mathbb{A}] , \theta}^\varepsilon : T_{[\mathbb{A}]} \mathscr{B}_{k }^\ast \to T_{[\mathbb{A}]} \mathscr{B}_{k-1}^\ast . 
%\]
In addition, $\mathscr{D}_{[\mathbb{A}],\theta}^\varepsilon$ has a complete $L^2$--orthonormal system of eigenvectors%(contained in $L^{2}_{k+1})$
, and its eigenvalues form a discrete set which is unbounded both from above and below (see Theorem \ref{theorem:fredholm}).
%a formally $L^2$--self-adjoint Fredholm operator from the $L^{2}_k$ to the $L^{2}_{k-1}$ completion of the tangent space to $\mathscr{B}^\ast$ at $[\mathbb{A}]$
%\[
%\mathscr{D}_{[\mathbb{A}] , \theta}^\varepsilon : T_{[\mathbb{A}]} \mathscr{B}_{k }^\ast \to T_{[\mathbb{A}]} \mathscr{B}_{k-1}^\ast . 
%\]
%In addition, $\mathscr{D}_{[\mathbb{A}],\theta}^\varepsilon$ has a complete $L^2$--orthonormal system of eigenvectors (contained in $L^{2}_{k+1})$, and its eigenvalues form a discrete set which is unbounded both from above and below (see Theorem \ref{theorem:fredholm}). 
We say $[\mathbb{A}] \in \mathscr{M}_{\theta}^{\ast, \varepsilon}$ is \textit{non-degenerate} if $\mathscr{D}_{[\mathbb{A}] , \theta}^\varepsilon$ is invertible; thus $[\mathbb{A}]$ is an \textit{isolated} point of $\mathscr{M}_{\theta}^{\ast , \varepsilon }$ by standard results in Fredholm theory, and we can try to count $[\mathbb{A}]$ with a sign defined using the spectral flow construction.

 The perturbed moduli space $\mathscr{M}_{\pm \pi/2}^{\ast , \varepsilon}$ has been studied by Taubes \cite{taubes-vw,taubes24a,taubes24b} (without the $\varepsilon^2$ perturbation in Definition \ref{perturbation:intro}, i.e. set $f = 0$). Its compactness theory follows the same principle as in the unperturbed case (`compactness modulo bubbling in codimension 2'). Nevertheless, it is reasonable to expect that the compactness theory of $\mathscr{M}_{\theta}^{\ast, \varepsilon }$ \textit{should simplify for small $\varepsilon \neq 0$}. The intuition here %(some of which we make precise below) 
 comes from the fact that the perturbed moduli space should be closely-related to the `critical
  locus' of $\| \Phi \|_{L^2}^2 : \mathscr{M}^\ast \to \mathbb{R}$, and the latter is a \textit{proper} function \cite{taubes-flat}. For example, if $\mathscr{M}^\ast$ was non-singular and $\| \Phi \|_{L^2}^2$ was an algebraic function on $\mathscr{M}^\ast$, then its critical locus would be compact. (Some of this intuition will be made precise below).

%To give some intuition,
%we recall that the function $\| \Phi \|_{L^2}$ gives a proper function on the moduli space $\mathscr{M}$ \cite{taubes-flat}. %(in particular, so does the stronger $L^\infty$ norm $\| \Phi \|_{L^\infty}$, which we shall use below). 

%Thus if--for the sake of simplifying--the moduli space $\mathscr{M}^\ast$ happened to be non-singular and $\| \Phi \|_{L^2}^2$ had finitely-many critical values (for example, if $\| \Phi \|_{L^2}$ was algebraic) then the critical loci $\mathscr{M}_{\theta}^{\ast , \varepsilon}$ would localise for small $\varepsilon \neq 0$ around the \textit{compact} set $\mathscr{Z}$ of critical points of $\| \Phi \|_{L^2}^2$ on $\mathscr{M}^\ast$.

\subsection{Localisation results for Seifert-fibered spaces}

We would like to understand the behavior of the moduli spaces $\mathscr{M}_{\theta}^{\ast , \varepsilon }$ for a small non-zero real parameter $\varepsilon$. 

In this article, we focus on the case when $Y$ is a \textit{Seifert-fibered} $3$-manifold: there is an orbifold principal $S^1$--bundle bundle $S^1 \curvearrowright Y  \xrightarrow{\pi} C$ over a closed oriented $2$-orbifold $C$.
%\[
%S^1 \curvearrowright Y^3  \xrightarrow{\pi} C^2 .
%\]
By a \textit{Seifert} metric $g$ on the Seifert-fibered $3$-manifold $(Y, \pi )$ we shall mean a Riemannian metric on $Y$ obtained as follows: choose an orbifold Riemannian metric $g_C$ on $C$ and a connection $1$-form $i\eta \in \Omega^1 (Y, i \mathbb{R} )$ on the principal $S^1$--bundle $\pi$ whose curvature is a constant multiple of the Riemannian volume form $\omega_C$ of $(C, g_C )$ (i.e. $d\eta = \mathrm{const}\cdot  \pi^\ast \omega_C$); then $g = \eta \otimes \eta + \pi^\ast g_C .$

%\subsubsection{The Euler characteristic of $\mathscr{M}^\ast$ from gauge theory}

\begin{thm}\label{theorem:localisationseifert}
Let $(Y,\pi , g)$ be a Seifert-fibered homology $3$-sphere equipped with a Seifert metric $g$. Then there exists a compact subset $\mathscr{Z} = \mathscr{Z}(g) \subset \mathscr{M}^\ast$, a collection $q_1 , \ldots , q_N : S^1 \times D^2 \hookrightarrow Y$ of disjoint smooth embeddings, and a smooth function $f : \mathbb{C}^N \rightarrow \mathbb{R}$ such that the following holds. Given any open set $\mathscr{U}\subset \mathscr{B}^\ast$ and a constant $C > 0$ such that
\begin{itemize}
\item $\mathscr{Z} \subset \mathscr{U} \subset \{  \| \Phi \|_{L^\infty} < C\}$
\item the closure of $\mathscr{U}$ does not contain the trivial connection,
\end{itemize}
there exists a constant $\varepsilon_0 = \varepsilon_0 ( g , q_i , f, \mathscr{U} , C ) > 0$ such that for $0 < |\varepsilon | \leq \varepsilon_0$ and $\theta \in \mathbb{R}/2\pi \mathbb{Z}$ we have:
\begin{enumerate}
\item If $[\mathbb{A}] \in \mathscr{M}_{\theta}^{\ast ,\varepsilon}$ then either $[(A, \Phi )] \in \mathscr{U}$, or else $\| \Phi \|_{L^\infty} > C$. Furthermore, points in $\mathscr{M}_{\theta}^{\ast , \varepsilon} \cap \mathscr{U}$ are non-degenerate critical points of $S_{\theta}^\varepsilon$ and comprise a finite set.

\item The count of points of $\mathscr{M}_{\theta}^{\ast , \varepsilon}$ in $\mathscr{U}$ with signs given by the mod $2$ spectral flow $\mathbf{sf}_2 (\mathscr{D}_{[\mathbb{A}], \theta}^t ) \in \mathbb{Z}/2$ of the path of Hessians $\mathscr{D}_{[\mathbb{A}],\theta}^t$ with $t$ going from $t = \varepsilon$ to $t = 0$ recovers the Euler characteristic of the moduli space of stable flat $SL(2, \mathbb{C})$ connections on $Y$:
\[
\sum_{[\mathbb{A}] \in \mathscr{M}_{\theta }^{\ast , \varepsilon} \cap \mathscr{U}} (-1)^{\mathbf{sf}_{2}(\mathscr{D}_{[\mathbb{A}],\theta}^t )} = \chi ( \mathscr{M}^\ast ) .
\]
\end{enumerate}
\end{thm}

Thus, as $\varepsilon \neq 0$ becomes small, the points on $\mathscr{M}_{\theta }^{\ast , \varepsilon}$ either tend to \textit{localise} around the compact set $\mathscr{Z}$ or else lie `far away' as witnessed by the norm $\| \Phi \|_{L^\infty}$. Perturbed solutions of the second kind are known to exist and have been recently studied by Taubes \cite{taubes24a,taubes24b} (without the $\varepsilon^2$ perturbation in Definition \ref{perturbation:intro}, i.e. set $f = 0$).

Theorem \ref{theorem:localisationseifert} should be compared with Taubes' Theorem \cite{taubes-casson} describing the $SU(2)$ \textit{Casson invariant} of a homology sphere as a count of critical points of a holonomy perturbation of the $SU(2)$ Chern--Simons functional. %There are two main differences, though. First, Taubes' description involves all critical points, whereas we only count the localising critical points. Second, Taubes used the spectral flow to the trivial connection (which gives a well-defined index in $\mathbb{Z}/8$), 
Taubes's Theorem uses instead the spectral flow to the trivial connection. In contrast, a special feature of our situation is that the mod $2$ spectral flow of the `small' path $\mathscr{D}_{[\mathbb{A}] , \theta}^t$ is well-defined: because at $t = \varepsilon$ the operator is invertible, and at $t = 0$ the operator has \textit{complex-linear} kernel since $S_\theta$ is the real part of the holomorphic functional $CS$.

%we show that configurations $A+i \Phi $ in $\mathscr{Z}$ are of two kinds: irreducible flat $SU(2)$ $\Phi =0$ or reduce to the abelian vortex equations on $C$.
%The fact that enables Theorem \ref{theorem:localisationseifert} is the fact that as a critical manifold of $S_\theta : \mathscr{B}^\ast \to \mathbb{R}$ $\mathscr{M}^\ast$ a non-degenerate critical 
\subsection*{The exact case} When $\theta =  \pm\pi/2 $ modulo $2\pi$ then $S_\theta = \pm \mathrm{Im}(CS)$ is an \textit{exact} functional: $S_\theta$ has no periods under the action of gauge transformations. Then $S_\theta$ defines an honest real-valued smooth function on $\mathscr{B}^\ast$ (as opposed to just a closed $1$-form), and admits a simpler expression:
\[
S_{ \pi/2} (A + i \Phi ) = (\mathrm{Im} CS)(A + i \Phi ) = \int_Y \mathrm{Tr} \Big( \Phi \wedge ( F_A - \frac{1}{3} \Phi \wedge \Phi  )\Big) .
\]
In the exact setting, the family of Hessians $\mathscr{D}_{\theta}^\varepsilon$ has vanishing spectral flow along loops in $\mathscr{B}^\ast$, and we can assign an \textit{integral} `Morse index' to critical points to obtain a refined count. Below, we use $\mathcal{P}_T (M)$ (resp. $\mathcal{P}^{c}_T(M)$) to denote the Poincaré polynomial (resp. compactly-supported Poincaré polynomial) of a given smooth manifold $M$. 

\begin{thm}\label{theorem:localisationseifert_exact}
Suppose that $\theta =  \pm\pi/2$ modulo $2\pi$. Let $(Y,\pi , g)$ be a Seifert-fibered homology $3$-sphere equipped with a Seifert metric $g$. Then there exists a compact subset $\mathscr{Z} = \mathscr{Z}(g) \subset \mathscr{M}^\ast$, a collection $q_1 , \ldots , q_N : S^1 \times D^2 \hookrightarrow Y$ of disjoint smooth embeddings, and a smooth function $f : \mathbb{C}^N \rightarrow \mathbb{R}$ such that the following holds.
Let $\mathscr{Z} = \cup_i \mathscr{Z}_i$ denote the decomposition into connected components. Given an open subset $\mathscr{U} \subset \mathscr{B}^\ast$ which is a disjoint union of open sets $\mathscr{U} = \sqcup_{i} \mathscr{U}_i$ and a constant $C>0$, such that for every $i$ we have that
\begin{itemize}
    \item $\mathscr{Z}_i \subset \mathscr{U}_i \subset \{  \| \Phi \|_{L^\infty}  < C \}  $
    \item the closure of $\mathscr{U}_i$ does not contain the trivial connection,
\end{itemize}
there exists a constant $\varepsilon_0 = \varepsilon_0 ( g , q_i , f, \mathscr{U} , C ) > 0$ such that for $0 < |\varepsilon | \leq \varepsilon_0$ we have:
\begin{enumerate}
\item If $[\mathbb{A}] \in \mathscr{M}_{\theta}^{\ast ,\varepsilon}$ then either $[(A, \Phi )] \in \mathscr{U}$, or else $\| \Phi \|_{L^\infty} > C$. Furthermore, points in $\mathscr{M}_{\theta}^{\ast , \varepsilon} \cap \mathscr{U}$ are non-degenerate critical points of $S_{\theta}^\varepsilon$ and comprise a finite set.

\item For each $[\mathbb{A}] \in \mathscr{M}_{\theta}^{\ast , \varepsilon} \cap \mathscr{U}_i$ there exists a path $[\mathbb{A}_t]$ with $t$ going from $t = \varepsilon$ to $t = 0$, such that $[\mathbb{A}_\varepsilon] = [\mathbb{A}]$, $[\mathbb{A}_0 ] \in \mathscr{Z}_i$, and $[\mathbb{A}_t]\in \mathscr{U}_i$ for all $t$; fix one such path for each $[\mathbb{A}] \in \mathscr{M}_{\theta}^{\ast , \varepsilon} \cap \mathscr{U}$. Let $\delta = \delta ( \varepsilon )>0$ be a constant small enough so that the Hessian $\mathscr{D}_{\theta}$ at every point of $\mathscr{Z}$ has spectral gap $> \delta$, and so does the perturbed Hessian $\mathscr{D}_{\theta}^\varepsilon$ at every point in $ \mathscr{M}_{\theta}^{\ast , \varepsilon} \cap \mathscr{U}$. For each $[\mathbb{A}]\in \mathscr{M}_{\theta}^{\ast , \varepsilon} \cap \mathscr{U}$ let $\mathbf{sf}(\mathscr{D}_{[\mathbb{A}_t], \theta}^t ) \in \mathbb{Z}$ denote the spectral flow of the path of operators $\mathscr{D}_{[\mathbb{A}_t]}^t + \delta$ with $t$ going from $t = \varepsilon$ to $t = 0$. %and with $\delta>0$ sufficiently small (namely, so that the Hessian $\mathscr{D}$ at every point of $\mathscr{Z}$ has spectral gap $> \delta$, and so does the perturbed Hessian $\mathscr{D}_{\mathbb{A}}^\varepsilon$ at every $[\mathbb{A}] \in \mathrm{Crit}(S_\varepsilon )\cap \mathscr{N}$). 
Then 
\[
\sum_{[\mathbb{A}] \in \mathrm{Crit}(S_\varepsilon ) \cap \mathscr{U}} T^{ \mathbf{sf}(\mathscr{D}_{[\mathbb{A}_t] , \theta}^t + \delta )} =  \begin{cases}
       \mathcal{P}_T ( \mathscr{M}^\ast (Y) ) &\quad\text{if } \varepsilon >0 \\
       \mathcal{P}^{c}_T ( \mathscr{M}^\ast (Y) ) &\quad\text{if } \varepsilon < 0 .\\
     \end{cases}
\]
\end{enumerate}
\end{thm}

The sheaf-theoretic $SL(2, \mathbb{C})$ Floer cohomology $HP^\ast (Y)$ of Abouzaid--Manolescu \cite{abouzaid-manolescu} in an invariant of a $3$-manifold taking the form of a finitely-generated abelian group. The \textit{sheaf-theoretic} $SL(2, \mathbb{C})$ \textit{Casson invariant} is defined as its Euler characteristic 
\[
\lambda^P ( Y ) = \sum_{n \in \mathbb{Z}} (-1)^n \mathrm{rank}_{\mathbb{Z}}HP^n (Y) \in \mathbb{Z}.
\]
For a Seifert-fibered homology sphere $Y$ one has $\lambda^P (Y) = \chi ( \mathscr{M}^\ast (Y) )$ \cite[Formula 35]{abouzaid-manolescu}, and thus Theorem \ref{theorem:localisationseifert} provides a gauge-theoretic interpretation of $\lambda^P (Y)$ in this case. By a variant of Theorem \ref{theorem:localisationseifert_exact}, we can also interpret the Poincaré polynomial of $HP^n (Y)$ using gauge theory (see Corollary \ref{corollary:sheaf}).

The set $\mathscr{Z}$ appearing in Theorems \ref{theorem:localisationseifert}-\ref{theorem:localisationseifert_exact} has the following geometric significance. For a Seifert-fibered $3$-manifold, the moduli space $\mathscr{M}^\ast (Y) = \mathrm{Crit}(S_\theta )$ is a smooth manifold: more precisely, it is Morse--Bott non-degenerate (Theorem \ref{theorem:dimensionalreduction}). Then $\mathscr{Z}$ is given by the critical locus of the restricted function $\| \Phi \|_{L^2}^2: \mathscr{M}^\ast (Y) \to \mathbb{R}$, and it is also Morse--Bott non-degenerate. A crucial fact that enables the above localisation phenomenon is that $\mathscr{Z}$ is compact (Proposition \ref{proposition:Zmorsebott}). This is verified by a direct calculation, based on an argument due to Hitchin \cite[Proposition 7.1]{hitchin}. 

In order to achieve non-degeneracy in Theorem \ref{theorem:localisationseifert}(1) we will arrange our choices as follows. First, we show that the $q_i$ and $\mu_i$ can be chosen so that the function $\underline{\sigma} : \mathscr{B}^\ast \to \mathbb{C}^N$ induces a smooth \textit{embedding} of $\mathscr{Z}$ into $\mathbb{C}^N$ (Proposition \ref{proposition:embedding}). The function $f : \mathbb{C}^N \to \mathbb{R}$ is then chosen so as to restrict `generically' onto $\mathscr{Z} \subset \mathbb{C}^N$ (see Theorem \ref{theorem:perturbation2general}). In Theorem \ref{theorem:localisationseifert_exact}(2) we must choose perturbations with additional care in order to achieve equality (in general, only the inequality $\geq$ holds). For this, we shall make use of the fact that $\| \Phi \|_{L^2}^2: \mathscr{M}^\ast \to \mathbb{R}$ is a \textit{perfect} Morse--Bott function and that $\mathscr{Z}$ admits a \textit{perfect} Morse function.
%We note that making different choices of perturbation (most crucially, of the function $f$ from Definition \ref{perturbation:intro}) leads to inequality ($\geq$) rather than equality in Theorem \ref{theorem:localisationseifert_exact}(2). Thus, here we must choose perturbations with additional care. A crucial ingredient for achieving equality will be that $\| \Phi \|_{L^2}^2: \mathscr{M}^\ast \to \mathbb{R}$ is a \textit{perfect} (and bounded below) Morse--Bott function, whose (compact) critical locus $\mathscr{Z}$ admits a \textit{perfect} Morse function.

%The reason for this is the fact that $\| \Phi \|_{L^2}^2 : \mathscr{M}^\ast (Y) \to \mathbb{R}$ can be regarded as the moment map for a 

%Let us give some additional information about the set $\mathscr{Z} \subset \mathscr{M}^\ast$ from Theorem \ref{theorem:localisationseifert}. First, we shall see that $\mathscr{M}^\ast = \mathrm{Crit}S_\theta$ is a \textit{Morse--Bott} critical locus (i.e. $\mathscr{M}^\ast$ is a smooth finite-dimensional manifold whose tangent space at a point is the kernel of $\mathscr{D}_\theta$). The locus $\mathscr{Z}$ is then given by the critical locus of the restriction $\| \Phi \|_{L^2}^2 : \mathscr{M}^\ast \to \mathbb{R}$. Regarding assertion (1) a crucial ingredient is the \textit{compactness} of $\mathscr{Z}$, which we prove using an argument due to Hitchin \cite[Proposition 7.1]{hitchin}: namely, elements of $\mathscr{Z}$ are either irreducible flat $SU(2)$ connections, or else reduce to certain abelian vortex equation on $C$ (Proposition \ref{proposition:Zmorsebott}).

\subsection{Dimensional reduction results}

Our next result describes $\mathscr{M}^\ast (Y)$ for a Seifert-fibered $3$-manifold $Y \to C$ in terms of algebro-geometric data on the base $2$-orbifold $C$.

\begin{thm}\label{theorem:reductionHiggs}
Let $(Y,\pi , g)$ be a Seifert-fibered $3$-manifold equipped with a Seifert metric $g = \eta\otimes \eta + \pi^\ast g_C$ (in particular, the $2$-orbifold $C$ and the orbifold line bundle $N := Y\times_{S^1} \mathbb{C}$ are equipped with holomorphic structures induced by $g_C$ and $i\eta$, respectively). Then the moduli space of stable flat $SL(2, \mathbb{C})$ connections $\mathscr{M}^\ast (Y)$ %= \mathrm{Crit}( S_{\theta} )$ is Morse--Bott, 
is a smooth manifold, and there is a natural diffeomorphism
\[
\mathscr{M}^\ast (Y ) \cong \Big( \bigsqcup_{E_0 \, , \,  \mathrm{det}E_0 = \mathscr{O}} \mathscr{M}_{E_0}^{\ast , \mathrm{Higgs}}(C) \Big) \sqcup \Big( \bigsqcup_{E_0 \, , \,  \mathrm{det}E_0 = N} \mathscr{M}_{E_0}^{\ast , \mathrm{Higgs}}(C) \Big)\]
where $\mathscr{M}_{E_0}^{\ast , \mathrm{Higgs}} (C)$ denotes the moduli space of stable Higgs bundles on the orbifold Riemann surface $C$ in the (topological) rank two orbifold vector bundle $E_0\to C$ with fixed determinant (given by either the trivial line bundle $\mathscr{O}$ or by $N$).
\end{thm}

For a more precise version of this result see Theorem \ref{theorem:dimensionalreduction}. The analogue of Theorem \ref{theorem:reductionHiggs} for the moduli space of irreducible $SU(2)$ flat connections on $Y$ was proved by Furuta--Steer \cite[Theorem 4.1]{furutasteer}. The moduli spaces $\mathscr{M}_{E_0}^{\ast , \mathrm{Higgs}}(C)$ have been studied in \cite{hitchin,nasatyr-steer,boden-yokogawa,BGM}. %, and is defined as the moduli space of pairs $(\mathscr{E} , \Psi )$, where $\mathscr{E}$ is an orbifold holomorphic structure on $E_0 $ with fixed determinant and $\Psi$ is a holomorphic $1$--form on $C$ valued in the bundle of trace-free endomorphism of $\mathscr{E}$, subject to a `stability' condition: there exists no $\Psi$--invariant holomorphic orbifold line subbundle $\mathscr{L} \subset \mathscr{E}$. 
(In the literature, the more common term for a Higgs bundle over an orbifold Riemann surface $C$ is that of a \textit{parabolic} Higgs bundle). When $Y$ is an integral homology sphere, the moduli spaces $\mathscr{M}_{E_0}^{\ast , \mathrm{Higgs}}(C)$ appearing in Theorem \ref{theorem:reductionHiggs} are connected, simply-connected, complete hyperKähler manifolds (of varying dimensions). %\cite{hitchin,nasatyr-steer,boden-yokogawa}. 

%the compactness of the set $\mathscr{Z}$ appearing in Theorems \ref{theorem:localisationseifert}-\ref{theorem:localisationseifert_exact}.

%\begin{corollary}\label{corollary:computationEuler}
%Let $Y \to C$ be a Seifert-fibered homology $3$-sphere. Then
%\[
%\mathscr{Z}(g) \cong \mathscr{M}_{SU(2)}^\ast (Y) \sqcup \mathcal{D}_{< - \chi ( C)} (C).
%\]
%In particular,
%\[
%\chi (\mathscr{M}_{SL(2,\mathbb{C})}^\ast (Y) ) = \chi (\mathscr{M}_{SU(2)}^\ast (Y) ) + \chi ( \mathcal{D}_{< - \chi (C)}(C) )
%\]
%where $\mathcal{D}_{< - \chi (C) }(C)$ is the moduli space of effective orbifold divisors on $C$ (with respect to any given orbifold holomorphic structure on $C$) of orbifold degree less than $-\chi (C)$.
%\end{corollary}

\subsection*{Formulae for Poincaré polynomial and Euler characteristic}

%\textcolor{red}{The moduli spaces $\mathscr{M}_{E_0}^{\ast , \mathrm{Higgs}}(C)$ carry a natural $\mathbb{C}^\ast$--action given by $(\mathscr{E} , \Psi ) \mapsto (\mathscr{E} , \lambda \Psi )$, and hence so does $\mathscr{M}^\ast (Y)$. In fact, we shall see that the set $\mathscr{Z}$ appearing in Theorems \ref{theorem:localisationseifert}-\ref{theorem:localisationseifert_exact} is the locus $\mathbb{C}^\ast$--fixed points in $\mathscr{M}^\ast$. 
%The locus $\mathscr{M}^{\ast}(Y)^{\mathbb{C}^\ast }$ of $\mathbb{C}^\ast$--fixed points is \textit{compact}, and this is what will ultimately enable the finiteness assertion in Theorem \ref{theorem:localisationseifert}(1). In fact, we have $\mathscr{Z} =\mathscr{M}^{\ast}(Y)^{\mathbb{C}^\ast }$ where $\mathscr{Z}$ is the compact subset appearing in Theorem \ref{theorem:localisationseifert}.
%We then use a method due to Frankel and Atiyah--Bott \cite{frankel,atiyah-bott} %the Poincaré polynomial of $\mathscr{M}^\ast (Y)$ can be related to that of $\mathscr{Z}$, and we use this 
%to derive formulae for the Poincaré polynomial and Euler characteristic of the stable $SL(2, \mathbb{C})$ character variety $\mathscr{M}^\ast (Y)$ of a Seifert-fibered homology sphere $Y$. }
For a Seifert-fibered homology sphere $Y$, the description of the moduli space $\mathscr{M}^\ast (Y)$ provided by Theorem \ref{theorem:reductionHiggs} plays a key role in establishing that $\| \Phi \|_{L^2}^2 : \mathscr{M}^\ast (Y) \to \mathbb{R}$ is a perfect Morse--Bott function  with compact critical locus $\mathscr{Z}$, and in the explicit description of $\mathscr{Z}$ (see Proposition \ref{proposition:Zmorsebott}).
%On the other hand, the fact that $\| \Phi \|_{L^2}^2 : \mathscr{M}^\ast (Y) \to \mathbb{R}$ is a perfect Morse--Bott function with compact critical locus $\mathscr{Z}$ gives 
In turn, this yields a relationship between the Poincaré polynomials of $\mathscr{M}^\ast (Y)$ and $\mathscr{Z}$. This is most simply expressed at the level of Euler characteristics: $\chi ( \mathscr{M}^\ast (Y) ) = \chi (\mathscr{Z} )$, because the Morse--Bott indices happen to be even integers. Thus, we can use this to derive new formulae for the Euler characteristic and Poincaré polynomial of $\mathscr{M}^\ast (Y)$. %for a Seifert-fibered homology sphere.

%Thus $\mathscr{Z}$ (Proposition \ref{proposition:Zmorsebott}), from which we obtain formulae for the Euler characteristic of the moduli space $\mathscr{M}^\ast (Y)$ for a Seifert-fibered homology sphere.

To describe these formulae, fix a Seifert-fibration $Y \to C$ where $C$ is a $2$-sphere with $n$ orbifold points with isotropy orders $\alpha_1 , \ldots , \alpha_n$ (where $\alpha_i \in \mathbb{Z}_{\geq 2}$ for each $i$). We introduce the following notation: given an integer vector $\underline{e} = ( e , \beta_1 , \ldots , \beta_n )$ with $0 \leq \beta_i < \alpha_i$ define its \textit{orbifold degree} by
\[
\mathrm{deg}(\underline{e}) = e + \sum_{i =1}^n \frac{\beta_i}{\alpha_i} \in \mathbb{Q} .
\]
In particular, the \textit{orbifold Euler characteristic} $\chi(C) \in \mathbb{Q}$ is given by minus the orbifold degree of the `canonical' vector $K = ( -2 , \alpha_1 - 1 , \ldots , \alpha_n -1 )$, i.e. $\chi (C) = - \mathrm{deg}(K)$.
%\[
%\chi(C) = 2-n + \sum_{i=1}^n \frac{1}{\alpha_i}.
%\]
%ntity $\chi (C) \in \mathbb{Q}$ is the orbifold Euler characteristic of $C$. If $C$ is a sphere with $n$ orbifold points of isotropy orders $\alpha_1, \ldots , \alpha_n$, then we have
%\[
%- \chi (C) = n -2 - \sum_{i = 1}^n \frac{1}{\alpha_i }.
%\]
%\begin{theorem}
%\[
%P_T ( \mathscr{M}^\ast (Y) ) = P_T (\mathscr{N}^\ast (Y) )  + \sum_{\underline{e}} (T^2)^{n -3 -e - \sum_{i=1}^n \lfloor \frac{\beta_i +1}{\alpha_i}\rfloor  } \cdot P_{T} ( \mathbb{C}P^e ) . 
%\]
%where $\underline{e} = ( e , \beta_1 , \ldots , \beta_n )$ ranges over all $(n+1)$--tuples of non-negative integers such that 
%\[
%0 \leq \beta_i < \alpha_i \quad \text{and} \quad e+ \sum_{i = 1}^n \frac{\beta_i +1}{\alpha_i } < n-2 .
%\]
%In particular, 
%\[
%\chi ( \mathscr{M}^\ast (Y) ) = \chi (\mathscr{N}^\ast (Y) )  + \sum_{\underline{e}}  \chi( \mathbb{C}P^e ) \]
%\end{theorem}
%\[
%\mathrm{deg}(\underline{e}) = e + \sum_{i =1}^n \frac{\beta_i}{\alpha_i}.
%\]
%\[
%\chi(C) = 2-n + \sum_{i=1}^n \frac{1}{\alpha_i}.
%\]

\begin{cor}\label{corollary:poincare} Let $Y$ be a Seifert-fibered integral homology $3$-sphere over the orbifold $C = S^2 ( \alpha_1 , \ldots , \alpha_n )$. Then:
\begin{enumerate}
\item The Poincaré polynomials of the moduli space of stable flat $SL(2, \mathbb{C})$ connections $\mathscr{M}^\ast (Y)$ and the moduli space of irreducible $SU(2)$ flat connections $\mathscr{N}^\ast (Y)$ are related by the following formula (where $\{ x\} \in [0,1)$ denotes the fractional part of a non-negative number $x$):
\[
\mathcal{P}_T ( \mathscr{M}^\ast (Y) ) = \mathcal{P}_T (\mathscr{N}^\ast (Y) )  + \sum_{\underline{e}} (T^2)^{ -\chi(C) - \mathrm{deg} (\underline{e}) -1  +\sum_{i=1}^n   \big\{ \frac{\beta_i +1}{\alpha_i }\big\}}  \cdot \mathcal{P}_{T} ( \mathbb{C}P^e ) . 
\]
where $\underline{e} = ( e , \beta_1 , \ldots , \beta_n )$ ranges over all vectors of non-negative integers such that 
\[
\beta_i < \alpha_i \quad \text{and} \quad \mathrm{deg}(\underline{e}) < -\chi (C) .
\]
\item In particular (noting that the exponents in $T^2$ are non-negative integers), the Euler characteristics are related by:
\[
\chi ( \mathscr{M}^\ast (Y) ) = \chi (\mathscr{N}^\ast (Y) )  + \sum_{\underline{e}}  \chi( \mathbb{C}P^e ) \]
\end{enumerate}
\end{cor}

For a Seifert-fibered integral homology $3$-sphere, the Poincaré polynomial of the moduli space of irreducible $SU(2)$ flat connections $\mathscr{N}^\ast (Y)$ on a Seifert-fibered homology $3$-sphere was calculated by Bauer \cite{bauer} and Furuta--Steer \cite{furutasteer}. %\cite[Theorem 4.9]{furutasteer}. 
Orienting $Y$ as the boundary of an isolated surface singularity then Fintushel--Stern proved that the $SU(2)$ \textit{Casson invariant} $\lambda(Y)$ (with its standard normalization: $\lambda (\Sigma(2,3,5) ) = -1$) is given by $ - (1/2) \chi ( \mathscr{N}^\ast (Y) )$. %agrees with the $SU(2)$ \textit{Casson invariant} of $Y$ \cite{fintushel-stern}, denoted $\lambda(Y)$, . 
Several different formulas for $\lambda(Y)$ of are given in \cite{fintushel-stern,neumann-wahl,fukuhara,kirk-klassen,saveliev-book}. 

We also note that using the results in \cite{boden-yokogawa,boden-curtis} one can obtain other formulae for $\chi ( \mathscr{M}^\ast (Y) )$ which are different from the ones that we obtain in Corollary \ref{corollary:poincare}.

\subsection{The Milnor number and $SL(2, \mathbb{C})$ representations}\label{subsection:milnor}

The formula from Corollary \ref{corollary:poincare}(2) can be used to give a singularity-theoretic interpretation of the Euler characteristic $\chi (\mathscr{M}^\ast (Y))$ of the stable locus of the $SL(2, \mathbb{C})$ character variety of a Seifert-fibered integral homology $3$-sphere. %(Equivalently, one can replace of the sheaf-theoretic Casson invariant $\lambda^P (Y)$, or of the gauge-theoretic count of $SL(2, \mathbb{C})$ connections from Corollary \ref{corollary:sheaf}(1) because all three agree.

Let $(X, 0 ) \subset \mathbb{C}^{2+k}$ be an \textit{isolated complete intersection singularity} (ICIS) of complex dimension $2$, where $0 \in X$ is the singular point. The singularity $(X,0)$ is given by the vanishing locus of a polynomial map $f = ( f_1 , \ldots , f_k ): (\mathbb{C}^{2+k}, 0 ) \to ( \mathbb{C}^k , 0 )$. The \textit{Milnor fiber} of $(X, 0 )$ \cite{milnor,hamm,le} is the compact $4$-manifold-with-boundary given as
\[
M = f^{-1}(z) \cap B_\varepsilon (\mathbb{C}^{2+k}, 0 )
\]
where $0 < \varepsilon \ll 1$, and $z$ is a `generic' point in $B_\delta (\mathbb{C}^k , 0 )$ with $0 < \delta \ll\varepsilon  $. The boundary of $M$ is the \textit{link} of the singularity $(X, 0 )$. A fundamental analytic invariant of the singularity $(X, 0 )$ is its \textit{Milnor number} $\mu (X, 0 )$, which is defined as the second Betti number of the Milnor fiber: $\mu (X, 0 ) = b_2 (M)$. %(Equivalently, $M$ is homotopy-equivalent to a wedge of $n$ $2$-spheres, and $n$ is the Milnor number)

An important class of singularities are the \textit{weighted-homogeneous} ICIS, where each $f_i$ is a weighted-homogeneous polynomial \cite{milnor} (for example, these include the Brieskorn--Pham complete intersections). For such singularities, it is easily seen that the link $Y$ is a Seifert-fibered $3$-manifold \cite{neumann-raymond}. By our next result, the analytic invariant $\mu (X, 0 )$ can be interpreted as a `count' of $SL(2, \mathbb{C} )$ representations of $\pi_1 (Y)$:

\begin{cor}\label{corollary:milnornumber} Let $(X,0) $ be a $2$-dimensional weighted-homogeneous ICIS, whose link $Y$ is an integral homology $3$-sphere. Then the Euler characteristic of the stable locus $\mathscr{M}^\ast (Y)$ in the $SL(2, \mathbb{C})$ character variety of $Y$ is one-fourth the Milnor number of $(X, 0 )$:
\[\chi (\mathscr{M}^\ast (Y)) = \frac{\mu (X, 0 )}{4}.\]
%\quad \big( = \frac{b_2 (M)}{4}\big) .  \]
\end{cor}
For the Brieskorn--Pham isolated hypersurface singularity $X = \{ x^p + y^q + z^r = 0 \} $, where $p,q,r$ are pairwise coprime positive integers, the above identity was established by Boden--Curtis \cite{boden-curtis}. For this singularity, the link $Y$ is the Brieskorn homology sphere $Y = \Sigma ( p,q,r)$, and it is shown in \cite{boden-curtis} that the character variety $\mathscr{M}^\ast (Y)$ is finite and comprises $(p-1)(q-1)(r-1)/4$ elements, which gives the Milnor number of $(X, 0 )$.

We obtain Corollary \ref{corollary:milnornumber} as follows. First, we recall the well-known \textit{Casson Invariant Conjecture}:

%(Alternatively, $\mu (X, 0 )$ can be defined inductively, see \cite[\S 1]{ebeling-complete}). The \textit{link}

\begin{conjecture*}[Neumann--Wahl \cite{neumann-wahl}]
Let $(X,0) $ be a $2$-dimensional ICIS, whose link $Y$ is an integral homology $3$-sphere. Then the $SU(2)$ Casson invariant $\lambda (Y)$ is one-eight the signature of the Milnor fiber:
\[
\lambda (Y) = \frac{\sigma (M)}{8}.
\]
\end{conjecture*}
%The \textit{Neumann--Wahl Conjecture} \cite{neumann-wahl} states that if $Y$ is an integral homology sphere, then its $SU(2)$--Casson invariant $\lambda (Y)$ is a multiple of the signature of $M$:
%\[
%\lambda (Y) = \frac{\sigma (M)}{8}.
%\]
The Casson Invariant Conjecture is known to hold for weighted-homogeneous ICIS \cite{fintushel-stern,neumann-wahl} and singularities of splice-type \cite{nemethi-okuma}, but remains open in general.  %We recall that an ICIS $(X, 0 ) \subset \mathbb{C}^{2+k}, 0 )$ is said to be \textit{weighted-homogeneous} if each $f_1, \ldots , f_k$ is a weighted-homogeneous polynomial. %(i.e. for each $j =1 , \ldots , k$ there exists positive integer weights $w_{1,j}, \ldots , w_{2+k , j} , d_j$ such that $f_j ( t^{w_{1,j}}x_1, \ldots , \lambda^{w_{2+k , j}}x_{2+k} ) = t^{d_j}f_j (x_1 , \ldots , x_{2+k} )$ for all $t \in \mathbb{C}$).
 %In this case, we relate the $SL(2, \mathbb{C})$ Casson invariant $\lambda^P (Y)$ in terms of the 
 
Corollary \ref{corollary:milnornumber} can be deduced from Corollary \ref{corollary:poincare}(2). Namely, the first term $\chi ( \mathscr{N}^\ast (Y) )$ can be rewritten as $\chi ( \mathscr{N}^\ast (Y) ) = -2 \lambda (Y) = -\sigma (M) /4$ \cite{fintushel-stern,neumann-wahl}. The key point (which we explain later) is to identify the second term $\sum_{\underline{e}} \chi ( \mathbb{C}P^e )$ as the \textit{geometric genus} $p_g (X, 0 )$ of the singularity $(X, 0 )$, which is the content of a formula due to Pinkham \cite{pinkham} and Dolgachev \cite{dolgachev} (also attributed to Demazure). In turn, there is the identity $b^{+}(M)  = 2 p_g (X, 0 ) $ due to Durfee and Steenbrink \cite{durfee,steenbrink}. Thus, Corollary \ref{corollary:poincare}(2) yields the required identity:\\
\begin{align*}
\chi (\mathscr{M}^\ast (Y) ) = -\frac{\sigma (M)}{4} + \frac{b^{+}(M)}{2} = \frac{\mu (X, 0 )}{4}.\\
\end{align*}

The Euler characteristic $\chi (\mathscr{M}^\ast (Y) )$ agrees with the sheaf-theoretic $SL(2, \mathbb{C})$ Casson invariant $\lambda^P (Y)$ for a Seifert-fibered homology sphere \cite[Formula 35]{abouzaid-manolescu}. Thus, Corollary \ref{corollary:milnornumber} provides evidence towards the following $SL(2, \mathbb{C})$ analogue of the Casson Invariant Conjecture:
\begin{conjecture}\label{conjecture:milnornumber}
Let $(X,0) $ be a $2$-dimensional ICIS, whose link $Y$ is an integral homology $3$-sphere. Then the sheaf-theoretic $SL(2, \mathbb{C})$ Casson invariant of $Y$ is a multiple of the Milnor number of $(X, 0 )$:
\[\lambda^P (Y) = \frac{\mu (X, 0 )}{4}.\]
\end{conjecture}

Remarkably, Conjecture \ref{conjecture:milnornumber} implies that the Milnor number $\mu (X,0)$ of a $2$-dimensional ICIS (which is an \textit{analytic} invariant of the singularity) should \textit{only depend on its link} $Y$ (or equivalently, the \textit{homeomorphism} type of the germ $(X, 0 )$). This is consistent with the prediction from the Casson Invariant Conjecture that $\sigma (M)$ only depends on $Y$, thanks to formulae of Laufer \cite{laufer-mu} and Durfee--Steenbrink \cite{durfee,steenbrink} which combined show that the quantity $2\mu(X,0) + 3\sigma (M) $ only depends on $Y$. %(In fact, from these formulae it would also follow that $p_g (X,0)$ and $\chi (M)$ only depend on $Y$).

\subsection{Outline and Comments}

\S \ref{section:connections} introduces the Hilbert manifold $\mathscr{B}^\ast$ of orbits of stable $SL(2, \mathbb{C})$ connections and discusses the deformations of flat stable $SL(2, \mathbb{C})$ connections on a (closed, oriented and connected) $3$-manifold based on this model of the configuration space. The main result we establish is the Fredholm property of the Hessians (Theorem \ref{theorem:fredholm}).

\S \ref{section:dimensionalreduction} establishes the Dimensional Reduction Theorem \ref{theorem:dimensionalreduction} (which implies Theorem \ref{theorem:reductionHiggs}). The crucial results here are the vanishing results from Corollary \ref{corollary:vanishing} and Proposition \ref{prop:vanishing_inf}, which are obtained by establishing suitable Bochner--Weitzenböck type formulae. In particular, from Theorem \ref{theorem:dimensionalreduction} we also deduce that $\mathscr{M}^\ast (Y) = \mathrm{Crit}S_\theta \subset \mathscr{B}^\ast $ is a Morse--Bott critical locus.

\S \ref{section:perturbations} studies the perturbed critical loci $\mathscr{M}_{\theta}^{\ast, \varepsilon}(Y)$ for a Seifert-fibered $3$-manifold and proves Theorems \ref{theorem:localisationseifert}-\ref{theorem:localisationseifert_exact}. 

The proof is modeled on a finite-dimensional perturbation argument for Morse--Bott functions (Theorem \ref{theorem:perturbation2}) which should be of independent interest. In order to apply this argument in our `infinite-dimensional' context several additional ingredients are needed: the existence of finite-dimensional `Kuranishi' models for $S_\theta$ (Corollary \ref{corollary:kuranishi}); the Morse--Bott non-degeneracy of the critical manifolds of $S_\theta$ and of $\| \Phi \|_{L^2}^2 : \mathscr{M}^\ast (Y) \to \mathbb{R}$ (Theorem \ref{theorem:dimensionalreduction}, Proposition \ref{proposition:Zmorsebott}); a method for achieving non-degeneracy of the perturbed critical points, for which we employ the $SL(2, \mathbb{C}$ holonomy perturbations (the crucial transversality property is Proposition \ref{proposition:embedding}); and finally, we need to ensure that basic linear analysis and compactness properties of the unperturbed moduli spaces go through in the perturbed setting, for which we need to establish various bounds on the $SL(2, \mathbb{C})$ holonomy perturbations and their derivatives (Proposition \ref{proposition:boundsholonomy}).
%\begin{enumerate}
%\item A finite-dimensional model for the functional $S_\theta $ acting $\mathscr{B}^\ast$; this a consequence of the Fredholm property of the Hessians (Corollary \ref{corollary:kuranishi}).
%\item The Morse--Bott property of $S_\theta$, provided by Theorem \ref{theorem:dimensionalreduction}.
%\item A method for achieving non-degeneracy of the perturbed critical points; for this we study the class of $SL(2, \mathbb{C})$ holonomy perturbations (the crucial transversality result is Proposition \ref{proposition:embedding}).
%\item In order to guarantee that the basic compactness and linear analysis results hold in the perturbed setting, we establish bounds on the $SL(2, \mathbb{C})$ holonomy perturbations and their derivatives (Proposition \ref{proposition:boundsholonomy}). 
%\end{enumerate}

With these ingredients in place, we establish Theorems \ref{theorem:localisationseifert}--\ref{theorem:localisationseifert_exact} by an extension of the finite-dimensional argument with several modifications. %To establish the equality in Theorem \ref{theorem:localisationseifert_exact}(2), we will make use of the fact that $\| \Phi \|_{L^2}^2 : \mathscr{M}^\ast (Y) \to \mathbb{R}$ is a perfect Morse--Bott function, and that its critical locus also admits a perfect Morse function.

Finally, \S \ref{section:calculations} contains the proofs of the formulae in Corollaries \ref{corollary:poincare}--\ref{corollary:milnornumber}, and the gauge-theoretic interpretation of the sheaf-theoretic $SL(2, \mathbb{C})$ Floer cohomology $HP^\ast (Y)$ for Seifert-fibered homology spheres.

\subsection*{Acknowledgements} The author thanks Simon Donaldson and Aleksander Doan for their encouragement and many discussions on the topic of this article. This work has also benefited from conversations with Gora Bera, Óscar García-Prada, Francesco Lin and Cliff Taubes. The author was partially supported by NSF grant DMS-2203498.

\section{Stable flat connections on $3$-manifolds}\label{section:connections}

This section sets up a framework for studying the deformations of flat $SL(2,\mathbb{C})$ connections over a closed oriented $3$-manifold. The flat connections are the critical points of the Chern--Simons functional on the space of all connections and for our purposes--counting flat connections with gauge-theoretic techniques--we must make sense of the Hessian as a formally self-adjoint Fredholm operator. (This is well-understood for compact structure groups, such as $SU(2)$, see e.g. \cite{taubes-casson,floer,donaldson-floer}). In order to achieve this we work with a suitable space of connections obtained as an infinite-dimensional Kähler (symplectic) quotient of the space of connections. %This viewpoint will be essential throughout the article.

%suitable Hilbert manifold of configurations with a built-in \textit{stability condition}. %, by requiring that configurations satisfy the `zero momentum' equation $d_{A}^\ast \Phi = 0$. 

%\textcolor{red}{Compare with other approaches, Doan--Walpuski, etc}

%The structure group is set to $SL(2,\mathbb{C})$ for the sake of concreteness; there are no essential differences when using instead any other complex semi-simple Lie group.

\subsection{Setup}\label{subsection:setup}
We describe the basic setup that we use throughout article. We fix the following data:
\begin{enumerate}
    \item a closed, oriented, Riemannian $3$-manifold $(Y,g)$. (In fact, the `$3$-dimensional' hypothesis won't play a role until \S \ref{subsection:CS}).
    \item a $U(2)$ bundle $E \rightarrow Y$, equipped with a $U(1)$ connection $\lambda$ on its determinant line bundle $\Lambda = \Lambda^2 E$. (Informally, we shall refer to $E$ as a `bundle with fixed determinant').
\end{enumerate}

We will often incur in abuse of notation, by sometimes treating $E$ as a hermitian vector bundle of rank two, and other times as a principal bundle with structure group $U(2)$. It should be clear from context which viewpoint we are using.

\subsubsection{Spaces of connections and Sobolev completions}

We let $\mathscr{A}_E$ denote the space of $U(2)$ connections $A$ on $E$ which induce the fixed connection $\lambda$ on its determinant $\Lambda$. Informally, we refer to $A$ as a `unitary' or $U(2)$ connection with `fixed determinant'. 

What we are really interested in is the space $\mathscr{A}_{E}^c$ of $GL(2,\mathbb{C})$ connections $\mathbb{A}$ on $E$ inducing the fixed connection $\lambda$ on its determinant $\Lambda$. Using the unitary structure on $E$, such a connection $\mathbb{A}$ can be expressed decomposed as 
\[
\mathbb{A} = A + i \Phi\]
where $A \in \mathscr{A}_E$ is a unitary connection with fixed determinant and $\Phi \in \Omega^1 (Y , \mathfrak{g}_E )$ is a $\mathfrak{g}_E$-valued $1$-form on $Y$. Here $\mathfrak{g}_E = \mathfrak{su}(E)$ is the \textit{adjoint bundle} consisting of trace-less, skew-adjoint endomorphisms of $E$. Thus, the space of $GL(2,\mathbb{C})$ connections with fixed determinant is $\mathscr{A}_{E}^c = \mathscr{A}_E \times \Omega^1 (Y,\mathfrak{g}_E )$, and we will typically refer to the connection $\mathbb{A}$ by its corresponding pair $(A, \Phi )$. Informally, we shall refer to $\mathbb{A}$ as a `complex' or $GL(2,\mathbb{C})$ connection on $E$ with `fixed determinant'.

Let $\mathscr{G}_E$ be the group of $U(2)$ bundle automorphisms of $E$ which induce the identity on its determinant $\Lambda$, and $\mathscr{G}_{E}^c$ be the group of $GL(2,\mathbb{C})$ bundle automorphisms of $E$ which induce the identity on $\Lambda$ (again, we refer to these as `unitary' or `complex' gauge transformations with `fixed determinant'). The groups $\mathscr{G}_{E}$ and $\mathscr{G}_{E}^c$ act on $\mathscr{A}_E$ and $\mathscr{A}_{E}^c$, respectively, in a natural way (our convention is that they act on the left): for example, a unitary gauge transformation $u$ with fixed determinant acts on the complex connection $(A, \Phi )$ by
\begin{align*}
u\cdot (A, \Phi ) = (A - d_{A}u \cdot u^{-1} , u \Phi u^{-1} ) .
\end{align*}

For later purposes, it will be convenient to make $\mathscr{A}_{E},\mathscr{A}_{E}^c$ and $\mathscr{G}_E , \mathscr{G}_{E}^c$ into Hilbert manifolds and Hilbert Lie groups, respectively, by considering their completions with respect to suitable Sobolev norms. Namely, from now on we will only consider unitary or complex connections with regularity in $L^{2}_{k}$, and unitary or complex gauge transformations with regularity in $L^{2}_{k+1}$. For our purposes, choosing any $k \geq 2$ will suffice. 
%
%\[
%k > \frac{\mathrm{dim}Y}{2}
%\]
%will suffice (we 
%For all our purposes, fixing any $k \geq 1$ will suffice. (This ensures that $L^{2}_{k+1} \subset C^0$, and that there are continuous multiplications $L^{2}_{k+1}\times L^{2}_{k+1} \rightarrow L^{2}_{k+1}$, $L^{2}_{k+1}\times L^{2}_k \rightarrow L^{2}_k$ and $L^{2}_k \times L^{2}_k \rightarrow L^{2}_{k-1}$, which is all we shall need). 
We will denote Sobolev-completed spaces simply as $\mathscr{A}_E , \mathscr{A}_{E}^c , \mathscr{G}_E , \mathscr{G}_{E}^c$ unless we need to make reference to the Sobolev structure, in which case we shall use the notation $\mathscr{A}_{E,k}, \mathscr{A}_{E,k}^c$, $\mathscr{G}_{E,k+1}, \mathscr{G}_{E,k+1}^c$.

When the determinant $(\Lambda, \lambda )$ is trivial as a bundle with connection, then $E$ has structure group $SU(2)$ and $\mathscr{A}_E, \mathscr{A}_{E}^c$ consist of the $SU(2)$ or $SL(2,\mathbb{C})$ connections. We allow for non-trivial $(\Lambda, \lambda )$ in order to also study also the $PSU(2) = SO(3)$ connections and $PSL(2,\mathbb{C})$ connections, via the following standard construction.

\subsubsection{The ` $U(2)$--model ' for $SO(3)$ and $PSL(2,\mathbb{C})$ connections}\label{U2model}
The $U(2)$ connections on $E$ can be regarded as $PSU(2) = SO(3)$ connections by the following procedure (see \cite[Part II]{donaldson-braam} and \cite[\S 5.6]{donaldson-floer} for further details). 

Associated to a $U(2)$ bundle $E$ there is a $PSU(2) = SO(3)$ bundle $P = E \times_{U(2)} SO(3)$, whose corresponding rank $3$ real vector bundle over $Y$ is the adjoint bundle $\mathfrak{g}_E = \mathfrak{su}(E)$. Conversely, any $SO(3)$ bundle $P$ lifts to a $U(2)$ bundle $E$ which is unique up to replacing $E$ by its tensor product $E \otimes L$ with a $U(1)$ bundle $L$. (Indeed, on a $3$-manifold the $U(2)$ bundles are classified by $c_1 (E) = c_1 (\Lambda ) \in H^2 (Y  , \mathbb{Z} )$ and the $SO(3)$ bundles by $w_2 (P) \in H^2 (Y , \mathbb{Z}/2)$; a choice of lift of $P$ corresponds to an integral lift of $w_2 (P)$, and this is unique modulo even classes in $H^2(Y, \mathbb{Z})$). Of course, a similar relationship holds if we replace $U(2)$ and $SO(3)$ by their complexifications $GL(2,\mathbb{C})$ and $PSL(2,\mathbb{C})$. 

A $U(2)$ connection with fixed determinant on $E$ induces an $SO(3)$ connection on $P$, and conversely an $SO(3)$ connection on $P$ lifts to a $U(2)$ connection on $E$ with fixed determinant which is unique up to tensoring $(E,A)$ by a flat real line bundle. This gives a bijection between $\mathscr{A}_P / \mathscr{G}_P$ and the orbits of the action of $H^1(Y, \mathbb{Z}/2)$ on $\mathscr{A}_E /\mathscr{G}_E$. The same relationship holds between the $GL(2,\mathbb{C})$ connections with fixed determinant and the $PSL(2,\mathbb{C})$ connections.

This gives a fairly concrete model for studying the spaces of $SO(3)$ (or $PSL(2,\mathbb{C})$) connections; effectively translating the complexity in the $SO(3)$ stabilizers to that of the $SU(2)$ stabilizers and the stabilizers of the action of $H^1 (Y, \mathbb{Z}/2 )$ on $\mathscr{A}_E / \mathscr{G}_E$.

\subsection{Orbit spaces}\label{subsection:orbitspaces}

We now discuss the Kähler quotient of the space $\mathscr{A}_{E}^c$ of complex connections.

\subsubsection{The Kähler structure of $\mathscr{A}_{E}^c$}

The space $\mathscr{A}_{E}^c$ carries a canonical complex structure 
\begin{align}
J (\dot{A}, \dot{\Phi} ) = (-\dot{\Phi } ,\dot{A} )  \label{J} 
\end{align}
%On the other hand, we can also regard $\mathscr{A}_{E}^c$ as the cotangent bundle of $\mathscr{A}_E$ by identifying $\Omega^1 (Y, \mathfrak{g}_E ) $ with its dual using the Riemannian metric on $Y$. 
and a canonical symplectic form $\omega_J$ 
\begin{align}
\omega_J ((\dot{A}_1 , \dot{\Phi}_1) , (\dot{A}_2, \dot{\Phi}_2) ) = - \int_Y \mathrm{Tr}\big(  \dot{A}_1 \wedge \ast \dot{\Phi}_2  -\dot{\Phi_1}\wedge \ast \dot{A}_2\big). \label{omegaJ}
\end{align}
The pair $(\omega_J , J )$ is compatible, making $\mathscr{A}_{E}^c$ into a Hilbert \textit{Kähler manifold}. The complex gauge group $\mathscr{G}_{E}^c$ is a Hilbert \textit{complex Lie group} which acts \textit{holomorphically} on $(\mathscr{A}_{E}^c , J )$. In turn, the unitary gauge group $\mathscr{G}_E$ preserves the whole Kähler structure. 

\subsubsection{Moment maps and stability}\label{subsubsection:moment}

Furthermore, the unitary gauge group $\mathscr{G}_E$ acts on $(\mathscr{A}_{E}^c , \omega_J )$ in a \textit{Hamiltonian} fashion, with moment map given by 
\[
\mathscr{A}_{E}^c \rightarrow \mathrm{Lie}\mathscr{G}_E = \Omega^0 (Y, \mathfrak{g}_E ) \quad , \quad (A, \Phi ) \mapsto d_{A}^\ast \Phi .
\]
Indeed, the vector field $X_\xi$ given by the infinitesimal action of $\xi \in \Omega^0 (Y, \mathfrak{g}_E)$ is $X_\xi = (-d_{A}\xi , [\xi, \Phi])$ and we have
\begin{align*}
\omega_J (   X_\xi   , (\dot{A}, \dot{\Phi})) %= \int_{Y} \mathrm{Tr}\Big( \dot{\Phi} \wedge (-d_{A}\xi ) - \dot{A}\wedge \ast [\xi , \Phi]\Big) 
= -\langle \dot{\Phi}, d_{A}\xi \rangle_{L^2} - \langle \dot{A},[\xi,\Phi]\rangle_{L^2} = \langle \ast [ \dot{A} , \ast \Phi ] - d_{A}^\ast \dot{\Phi} , \xi \rangle_{L^2}
\end{align*}
where $   \ast [ \dot{A} , \ast \Phi ] - d_{A}^\ast \dot{\Phi}$ is the linearisation of $-d_{A}^\ast \Phi$ at $(A, \Phi)$. We may thus attempt to carry out the \textit{Kähler quotient} construction, %(see e.g. \cite[Chapter 5]{mcduff-salamon-intro})
albeit in infinite dimensions, of the Hilbert Kähler manifold $(\mathscr{A}_E , \omega_J , J )$ by the $\mathscr{G}_E$--action.

The zeros of the moment map (i.e. $d_{A}^\ast \Phi = 0$) have the following interpretation:
\begin{proposition}\label{proposition:invariant}
Suppose $\mathbb{A} = A + i \Phi \in \mathscr{A}_{E}^c$ satisfies $d_{A}^\ast \Phi = 0$, then:
\begin{enumerate}
\item Every $\mathbb{A}$--invariant complex line subbundle $L \subset E$ admits an $\mathbb{A}$--invariant complement.
\item If the $\mathscr{G}_{E}$--stabilizer of $\mathbb{A}$ is $\{ \pm I\} $, then $E$ admits no $\mathbb{A}$--invariant complex line subbundles.
\end{enumerate}
\end{proposition}
\begin{proof}
(1) is proved in \cite[Proposition 3.2]{corlette-harmonic}. Also in \cite[\S 2.3]{wu-zhang} where the following is established: if $L\subset E$ is an $\mathbb{A}$--invariant line subbundle and $L^\perp$ is denotes its orthogonal complement with respect to the unitary structure on $E$, then the vanishing of $d_{A}^\ast \Phi = 0 $ implies that $L^\perp$ is also $\mathbb{A}$--invariant. 

In this case, the $U(2)$ bundle-with-connection $(E, \mathbb{A})$ splits as a sum $(L,a) \oplus (L^\perp , a^\perp )$ of $U(1)$ bundles-with-connection (such that $(\Lambda,\lambda ) = (L , a ) \otimes (L^\perp , a^\perp )$), and $\Phi$ acts diagonally with respect to this decomposition. Hence the $\mathscr{G}_E$--stabilizer of $\mathbb{A} = A + i \Phi$ contains a $U(1)$ subgroup. Assertion (2) follows.
 \end{proof}
A complex connection $\mathbb{A} $ satisfying the conclusion of Proposition \ref{proposition:invariant}(1), resp. (2), is said to be \textit{polystable} (or completely-reducible, or reductive), resp. \textit{stable} (or irreducible). By a recent result by Wu--Zhang \cite{wu-zhang}, the converse to Proposition \ref{proposition:invariant} is known (but we shall not need this result): $\mathbb{A}$ is polystable if and only if there exists a complex gauge transformation $u \in \mathscr{G}_{E}^c$ such that $u \cdot \mathbb{A} = A_u + i \Phi_u$ solves $d_{A_u}^\ast \Phi_u = 0$; and $\mathbb{A}$ and is stable if and only if, in addition, the $\mathscr{G}_E$--stabilizer of $u \cdot \mathbb{A}$ is $\{\pm I\}$. (In the case when $\mathbb{A}$ is \textit{flat}, this is a fundamental result by Donaldson \cite{donaldson-harmonic} and Corlette \cite{corlette-harmonic}). In light of this, and abusing terminology, we make the following
\begin{definition}\label{definition:stable}
The space of \textit{polystable} configurations $\mathscr{C}_E \subset \mathscr{A}_{E}^c = \mathscr{A}_E \times \Omega^1 (Y , \mathfrak{g}_E )$ is the subspace of pairs $(A, \Phi )$ satisfying $d_{A}^\ast \Phi = 0$. The unitary gauge group $\mathscr{G}_E$ acts on $\mathscr{C}_E$, and the space of \textit{stable} configurations is defined as the open subset $\mathscr{C}_{E}^\ast\subset \mathscr{C}_{E}$ consisting of configurations with smallest possible stabilizer under the action of $\mathscr{G}_E$, namely $\{\pm I\} = Z (SU(2))$. When reference to the Sobolev structure is needed, we shall use the notation $\mathscr{C}_{E,k}^\ast, \mathscr{C}_{E,k}$.
\end{definition}

%\begin{remark}
%The terminology from Definition \ref{definition:stable} is motivated by the Kempf--Ness Theorem \cite{kempf-ness}, on the equivalence of symplectic and GIT quotients (see e.g. \cite{thomas-git}). In our context, the analogue of this result was recently proved by Wu--Zhang \cite{wu-zhang}, which we quickly summarize (however, we won't use this result other than for motivation). Say a vector bundle with connection $(E, \mathbb{A})$ is stable if it contains no non-trivial $\mathbb{A}$-invariant proper subbundle, and say it is polystable if it is a direct sum of stable ones. Then $(E, \mathbb{A})$ is polystable iff there exists a complex gauge transformation $u \in \mathscr{G}_{E}^c$ such that $u \cdot \mathbb{A} \in \mathscr{C}_E$, and it is stable iff in addition, the $\mathscr{G}_E$--stabilizer of $u \cdot \mathbb{A}$ is smallest possible. In the case of \textit{flat} connections, this is a well-known result by Donaldson \cite{donaldson-harmonic} and Corlette \cite{corlette-harmonic}.
%\end{remark}

The first step towards carrying out the infinite-dimensional Kähler quotient construction is the following

\begin{proposition}\label{proposition:stableconf}
The space of stable configurations $\mathscr{C}_{E,k}^\ast \subset \mathscr{A}_{E,k}^c$ is a Hilbert submanifold. Its tangent space is 
\[
T_{(A, \Phi )}\mathscr{C}_{E, k}^\ast = \Big\{ (\dot{A}, \dot{\Phi} ) \in (\Omega^1 \oplus \Omega^1)(Y, \mathfrak{g}_E )_k \, | \, \ast [ \dot{A} , \ast \Phi ] - d_{A}^\ast \dot{\Phi} = 0 \Big\} .
\]
\end{proposition}
(Above, and likewise subsequently, the subscript $k$ stands for the Sobolev $L^{2}_k$ completion of the space of sections shown). 

Before proving this result we make some general observations. Consider, for each $i = 1, 2$ and $0 \leq j \leq k$ the pair of bounded linear operators 
\begin{align*}
 \delta_{(A, \Phi)}^i : & \Omega^0 (Y, \mathfrak{g}_E )_{j+1} \rightarrow (\Omega^1 \oplus \Omega^1 )(Y, \mathfrak{g}_E )_{j}\\
 (\delta_{(A, \Phi)}^i)^\ast : & (\Omega^1 \oplus \Omega^1 )(Y, \mathfrak{g}_E )_{j} \rightarrow \Omega^0 (Y, \mathfrak{g}_E )_{j-1}
\end{align*}
given by
\begin{align*}
 \delta_{(A, \Phi)}^1 (\dot{\xi}) & = (- d_A \dot{\xi} ,   [\dot{\xi} , \Phi ] ) \\
 (\delta_{(A, \Phi)}^1)^\ast (\dot{A} , \dot{\Phi}) & = - d_{A}^\ast \dot{A} - \ast [ \dot{\Phi} , \ast \Phi ]\\
  \delta_{(A, \Phi)}^2 (\dot{\xi}) & = ( -[\dot{\xi} , \Phi ] , - d_A \dot{\xi}) \\
  (\delta_{(A, \Phi)}^2)^\ast (\dot{A} , \dot{\Phi}) &=  \ast [ \dot{A} , \ast \Phi ] - d_{A}^\ast \dot{\Phi}.
\end{align*}
One easily sees that $(\delta_{(A, \Phi )}^i)^\ast$ is the formal $L^2$ adjoint of $\delta_{(A, \Phi )}^i$. Let us also introduce the following notation: 
\begin{align}
\mathscr{T}_{(A, \Phi), j}^i & := \mathrm{Im} \delta_{(A, \Phi )}^i \subset (\Omega^1 \oplus \Omega^1) (\mathfrak{g}_E )_j \label{T}\\
\mathscr{K}_{(A, \Phi ) , j}^i &:= \mathrm{Ker} (\delta_{(A, \Phi )}^i )^\ast \subset (\Omega^1 \oplus \Omega^1) (\mathfrak{g}_E )_j \label{K}
\end{align}

The significance of these operators is the following. Identifying $(\Omega^0 \oplus \Omega^0 )(Y, \mathfrak{g}_E)_{k+1} \cong \mathrm{Lie} \mathscr{G}_{E,k+1}^c$, the map sending 
\begin{align}
(\dot{\xi}_1 , \dot{\xi}_2 ) \in (\Omega^0 \oplus \Omega^0 )(Y, \mathfrak{g}_E)_{k+1} \mapsto \delta_{(A, \Phi )}^1 (\dot{\xi}_1) +  \delta_{(A, \Phi)}^2 (\dot{\xi}_2) = \delta_{(A, \Phi )}^1 (\dot{\xi}_1) +  J\delta_{(A, \Phi)}^1 (\dot{\xi}_2) \label{delta}
\end{align}

\noindent can be identified with the infinitesimal action of the complex gauge group $\mathscr{G}_{E,k+1}^c$ on the space of complex connections. 
%Thus,
%\begin{align*}
%\mathrm{Lie}\mathscr{G}_{E,k}^c  = \mathrm{Im}\delta_{(A, \Phi )}^1  +  \mathrm{Im}\delta_{(A, \Phi)}^2 \quad , \quad \mathrm{Lie}\mathscr{G}_{E,k}  =\mathrm{Im} \delta_{(A, \Phi )}^1.
%\end{align*}

The bounded linear operator $\delta_{(A, \Phi )}^i$ is associated to a first order differential operator with injective symbol, and as a matter of Hodge Theory we thus have that for $0 \leq j \leq k$ the subspace $\mathrm{Ker}\delta_{(A, \Phi ) }^i \subset  \Omega^0 (Y, \mathfrak{g}_E )_{j+1}$ is finite dimensional (and independent of $j$), and there are $L^2$--orthogonal decompositions 
\begin{align}
\Omega^0 (Y, \mathfrak{g}_E )_{j+1} & = \mathrm{Ker}\delta_{(A, \Phi)}^i \oplus \mathrm{Im}(\delta_{(A, \Phi)}^i )^\ast  \label{hodge0}\\
(\Omega^1 \oplus \Omega^1 )(Y, \mathfrak{g}_E )_j & =   \mathrm{Im}\delta_{(A, \Phi)}^i  \oplus \mathrm{Ker}(\delta_{(A, \Phi)}^i)^\ast = \mathscr{T}_{(A, \Phi ) , j }^i \oplus \mathscr{K}_{(A, \Phi ) , j }^i.\label{hodge1}
\end{align}
In particular, Proposition \ref{proposition:stableconf} asserts that, at a stable configuration
\[
T_{(A, \Phi )}\mathscr{C}_{E , k}^\ast = \mathscr{K}_{(A, \Phi ) , k}^2 .
\]

Observe also that the composition $(\delta_{(A, \Phi )}^1)^\ast \circ \delta_{(A,\Phi)}^2 $ is trivial at a polystable configuration $(A, \Phi )$: 
\[
(\delta_{(A, \Phi )}^1)^\ast \circ \delta_{(A,\Phi)}^2 (\dot{\xi})  = - \ast d_A [\dot{\xi}, \ast \Phi] + \ast [d_{A}\dot{\xi}, \Phi] = [\dot{\xi}, d_{A}^\ast \Phi ] = 0 
\]
where the last equality used $d_{A}^\ast \Phi = 0$. In particular, it follows from this that the decomposition (\ref{hodge1}) restricts to $L^2$--orthogonal decompositions for $0\leq j \leq k$
\begin{align}
\mathscr{K}_{(A, \Phi ) , j}^2 = \mathscr{T}_{(A, \Phi ) , j}^1 \oplus ( \mathscr{K}_{(A, \Phi ) , j}^1 \cap \mathscr{K}_{(A, \Phi ) , j}^2 ) \label{K2} \\
\mathscr{K}_{(A, \Phi ) , j}^1 = \mathscr{T}_{(A, \Phi ) , j}^2 \oplus ( \mathscr{K}_{(A, \Phi ) , j}^1 \cap \mathscr{K}_{(A, \Phi ) , j}^2 ) \label{K1}
\end{align}

%\begin{align}
%\mathrm{Ker}(\delta_{(A,\Phi )}^2)^\ast = ( \mathrm{Ker}(\delta_{(A,\Phi )}^1)^\ast \cap \mathrm{Ker}(\mathrm{\delta}_{(A,\Phi )}^2)^\ast ) \oplus \mathrm{Im}\delta_{(A,\Phi)}^1 \quad \subset \quad (\Omega^1 \oplus \Omega^1)(Y, \mathfrak{g}_E )_{k}. \label{delta1delta2}
%\end{align}

\begin{proof}[Proof of Proposition \ref{proposition:stableconf}]
As noted above, the linearization of the equation $d_{A}^\ast \Phi = 0$ is $-\delta_{2}^\ast (\dot{A}, \dot{\Phi} ) = 0$. By the Implicit Function Theorem for Hilbert spaces, it suffices to establish the surjectivity of $\delta_{2}^\ast : (\Omega^1 \oplus \Omega^1 )(Y, \mathfrak{g}_E )_{k} \rightarrow \Omega^0 (Y, \mathfrak{g}_E )_{k-1}$ whenever $(A, \Phi) \in \mathscr{C}_{E,k}^\ast$. By the decomposition (\ref{hodge0}) with $i = 2$ on the $L^{2}_{k-1}$ spaces, this is equivalent to $\mathrm{Ker} \delta_{(A, \Phi )}^2 = 0$. Indeed $\mathrm{Ker} \delta_{(A, \Phi )}^2$ is identified as the (rotation by $J$) of the tangent space to the $\mathscr{G}_{E, k}$--stabilizer of $(A, \Phi )$, which vanishes when $(A, \Phi )$ is stable. 

The latter assertion can be verified directly as follows. Suppose for a contradiction that $[\dot{\xi} , \Phi] $ and $d_{A} \dot{\xi}$ vanish identically, where $0 \neq \dot{\xi} \in \Omega^0 (Y, \mathfrak{g}_E )_{k-1}$. Elliptic regularity ensures that $\dot{\xi}$ has regularity $L^{2}_{k+1}$ (and in particular, $\dot{\xi}$ is $C^0$). Since $\dot{\xi}$ is $A$-parallel (i.e. $d_A \dot{\xi} = 0$) and non-trivial, then the eigenspaces of $\dot{\xi}$ induce a bundle decomposition $E = L \oplus L^{-1}\Lambda$ which is preserved by the connection $A$, for some $U(1)$ bundle $L$. The equation $[\dot{\xi} , \Phi] = 0$ says that $\Phi$ is diagonal with respect to this decomposition. Altogether, it follows that the $\mathscr{G}_{E, k+1}$--stabilizer of $(A, \Phi )$ contains a $U(1)$ subgroup, which is a contradiction.
%Suppose for a contradiction this were false. Since the range of $\delta_{2}^\ast$ is $L^2$-closed (because $\delta_2$ has injective symbol), we can find a non-zero $L^{2}_{k-1}$ section $\dot{\xi} \in \Omega^0 (Y, \mathfrak{g}_E )$ which is $L^2$-orthogonal to the range, and thus for any two $L^{2}_k$ sections $\dot{A} , \dot{\Phi}$ we have
%\begin{align*}
%& 0 = \langle \dot{\xi} , - \ast [\dot{A} , \ast \Phi ] \rangle = \langle [\dot{\xi} , \Phi] , \dot{A} \rangle \\
%& 0 = \langle \dot{\xi} , d_{A}^\ast \dot{\Phi} \rangle = \langle d_{A} \dot{\xi} , \dot{\Phi} \rangle %.
%\end{align*}
%Thus, $[\dot{\xi} , \Phi] $ and $d_{A} \dot{\xi}$ vanish identically. Elliptic regularity ensures that $\dot{\xi}$ has regularity $L^{2}_{k+1}$ (and in particular, $\dot{\xi}$ is $C^0$). Since $\dot{\xi}$ is $A$-parallel (i.e. $d_A \dot{\xi} = 0$) and non-trivial, then the eigenspaces of $\dot{\xi}$ induce a bundle decomposition $E = L \oplus L^{-1}\Lambda$ which is preserved by the connection $A$ (for some $U(1)$ bundle $L$). The equation $[\dot{\xi} , \Phi] = 0$ says that $\Phi$ is diagonal with respect to this decomposition. Altogether, it follows that the stabilizer of $(A, \Phi )$ under $\mathscr{G}_E$ contains a $U(1)$ subgroup, giving a contradiction.
\end{proof}

\subsubsection{Local structure of the stable orbit space}

Formally, the Kähler quotient construction yields the space of \textit{stable orbits} $\mathscr{B}_{E,k}^\ast = \mathscr{C}_{E,k}^\ast / \mathscr{G}_{E,k+1}$. For the moment, we regard it as a topological space with the quotient topology. We now describe how to equip $\mathscr{B}_{E,k}^\ast$ with a Hilbert manifold structure by constructing slices for the action of $\mathscr{G}_{E,k+1}$ on $\mathscr{C}_{E,k}^\ast$.

%For a stable configuration $(A,\Phi )\in \mathscr{C}_{E}^\ast$ consider the linear subspace of $(\Omega^1 \oplus \Omega^1)(Y, \mathfrak{g}_E )$ given by
%\begin{align*}
%\mathscr{K}_{(A,\Phi)} & = \Big\{ (a,\phi) \, |\, d_{A}^\ast a + \ast [\phi , \ast \Phi ] = 0 \Big\}.
%&= \Big\{ (\dot{A}, \dot{\Phi} ) \in (\Omega^1 \oplus \Omega^1 ) (Y , \mathfrak{g}_E )  \, |\,  d_{A}^\ast \dot{A} + \ast [\dot{\Phi} , \ast \Phi ] = 0 \text{ and } -\ast [ \dot{A} , \ast \Phi ] + d_{A}^\ast \Phi = 0  \Big\} .
%\end{align*}

\begin{definition} Let $(A, \Phi ) \in \mathscr{C}_{E,k}^\ast$ be a stable configuration. The \textit{stable Coulomb slice at} $(A, \Phi )$ is the subspace $\mathcal{S}_{(A, \Phi ),k} \subset \mathscr{C}^{\ast}_{E,k}$ given by
\[
\mathcal{S}_{(A, \Phi ),k}  = \Big( (A, \Phi ) + \mathscr{K}_{(A, \Phi ) , k}^1 \Big) \cap \mathscr{C}^{\ast}_{E,k} .
\]
\end{definition}
\begin{proposition}[Local Slices] \label{proposition:slices}
Let $(A, \Phi ) \in \mathscr{C}_{E,k}^\ast$ be a stable configuration. Then:
\begin{enumerate}
    \item $\mathcal{S}_{(A, \Phi),k}$ is a Hilbert submanifold of $\mathscr{C}_{E,k}^\ast$ in a neighborhood of $(A, \Phi )$,
    \item $\mathcal{S}_{(A, \Phi ),k}$ meets the $\mathscr{G}_{E,k+1}$--orbit of $(A, \Phi )$ transversely and $L^2$--orthogonally at $(A, \Phi )$,
    \item There exists a neighborhood $\mathscr{U} \subset \mathscr{C}_{E,k}^\ast$ of $(A, \Phi )$ such for any $(A^\prime,\Phi^\prime) \in \mathscr{U}$ there exists a unitary gauge transformation $u \in \mathscr{G}_{E,k+1}$ such that $u\cdot (A^\prime , \Phi^\prime ) \in \mathcal{S}_{(A, \Phi ), k}$. (The gauge transformation $u$ is necessarily unique up to multiplication by $\pm I$, by stability).
\end{enumerate} 
\end{proposition}
\begin{proof}
%Recall the operator $\delta_2 $ and its formal $L^2$ adjoint $\delta_{2}^\ast$ from (\ref{delta1}), and consider the operator $\delta_{1}^\ast : (\Omega^1 \oplus \Omega^1 )(Y, \mathfrak{g}_E ) \rightarrow \Omega^0 (Y, \mathfrak{g}_E )$ and its formal $L^2$ adjoint $\delta_{1} : \Omega^0 (Y, \mathfrak{g}_E ) \rightarrow (\Omega^1 \oplus \Omega^1 )(Y, \mathfrak{g}_E )$ given by 
%\begin{align}
%\delta_{1}^\ast (\dot{A} , \dot{\Phi}) = - d_{A}^\ast \dot{A} - \ast [ \dot{\Phi} , \ast \Phi ] \quad , \quad \delta_{1} (\dot{\xi}) = (- d_A \dot{\xi} ,   [\dot{\xi} , \Phi ] ). \label{delta1}
%\end{align}
%We regard the operators $\delta_{i}^\ast ,\delta_i$ as bounded linear operators between the Sobolev spaces $L^{2}_{k}\rightarrow L^{2}_{k-1}$. Since $\delta_i$ has injective symbol, by Hodge Theory there are $L^2$ orthogonal decompositions
%\begin{align}
%\Omega^0 (Y, \mathfrak{g}_E ) = \mathrm{Ker}\delta_{i} \oplus \mathrm{Im}\delta_{i}^\ast \text{  on  } L^{2}_{k-1} \label{hodge0}\\
%(\Omega^1 \oplus \Omega^1 )(Y, \mathfrak{g}_E ) = \mathrm{Ker}\delta_{i}^\ast \oplus \mathrm{Im}\delta_i \text{  on  } L^2_{k}\label{hodge1}
%\end{align}

To prove (1) it suffices, by the Implicit Function Theorem for Hilbert spaces, to show that the operator
\begin{align*}
T_{(A, \Phi )}\mathscr{C}_{E,k}^\ast = \mathscr{K}_{(A, \Phi ) , k }^2 \xrightarrow{(\delta^{1}_{(A, \Phi)})^\ast} \Omega^0 (Y, \mathfrak{g}_E )_{k-1}
\end{align*}
is surjective. But the operator $(\delta_{(A,\Phi )}^1)^\ast : (\Omega^1 \oplus \Omega^1 )(Y, \mathfrak{g}_E)_{k} \rightarrow  \Omega^0 (Y, \mathfrak{g}_E )_{k-1}$ is surjective (this follows by the same argument as for $(\delta_{(A,\Phi )}^2)^\ast$ from the proof of Proposition \ref{proposition:stableconf}, which uses the fact that $(A,\Phi )$ is stable). Thus, the assertion follows from this, the vanishing of $(\delta^{1}_{(A, \Phi )})^\ast \delta_{(A, \Phi )}^2$ (using $d_{A}^\ast \Phi = 0$) and the decomposition (\ref{hodge1}) with $i = 2$.
%\textcolor{red}{Thus, the assertion follows from this and the decomposition (\ref{hodge1}) with $i = 1$.}

(2) follows now from the decomposition (\ref{K2}), noting that the tangent spaces at $(A, \Phi )$ to $\mathscr{C}_{E,k}^\ast$, to the $\mathscr{G}_{E,k+1}$--orbit of $(A, \Phi )$ and to $\mathcal{S}_{(A, \Phi ),k}$ are given by $\mathscr{K}_{(A, \Phi ) , k}^2$, $\mathscr{T}_{(A, \Phi ),k}^1$ and $\mathscr{K}_{(A, \Phi ) , k}^1 \cap \mathscr{K}_{(A, \Phi ) , k}^2$, respectively. 

We now show (3). In fact, we will see that for an arbitrary pair $(A^\prime, \Phi^\prime) \in \mathscr{A}_{E,k}^c$, not ncessarily in $\mathscr{C}_{E,k}$, there exists a unitary gauge transformation $u$ such that 
\begin{align}
(\delta^{1}_{(A, \Phi )})^\ast ( u\cdot (A^\prime , \Phi^\prime ) - (A, \Phi ) ) = 0 \label{gaugefix}.
\end{align}

Writing $(A^\prime, \Phi^\prime ) = (A, \Phi ) + (a, \phi )  $ we seek, for sufficiently small $(a, \phi)$, a unitary gauge transformation of the form $u = \mathrm{exp}\xi$ where $\xi \in \Omega^0 (Y, \mathfrak{g}_E )_{k+1}$, such that (\ref{gaugefix}) holds. Note that 
\[u \cdot (A^\prime , \Phi^\prime ) - (A, \Phi ) = (-d_A \xi + (\mathrm{exp} \xi)a (\mathrm{exp}\xi)^{-1} ,(\mathrm{exp} \xi)\Phi (\mathrm{exp}\xi)^{-1} +(\mathrm{exp} \xi)\phi (\mathrm{exp}\xi)^{-1} )
\]
By the Implicit Function Theorem for Hilbert spaces, it suffices to show that the linearisation of the equation (\ref{gaugefix}) in the $\xi$ direction is surjective at $(\xi , a , b ) = (0,0,0)$, and this is the operator given by
\begin{align}
\dot{\xi} \in \Omega^0 (Y, \mathfrak{g}_E )_{k+1} \mapsto (\delta_{(A,\Phi)}^1 )^\ast \delta_{(A,\Phi)}^1 \dot{\xi} \in \Omega^0 (Y, \mathfrak{g}_E )_{k-1}. \label{laplacedelta1}
\end{align}
By (\ref{hodge0}-\ref{hodge1}) %for the suitable Sobolev norms with $i = 1$, 
the operator (\ref{laplacedelta1}) surjects onto the complement of $\mathrm{Ker}\delta_{(A,\Phi )}^1$. But $\mathrm{Ker}\delta_{(A,\Phi )}^1 = 0$ by stability of $(A, \Phi )$ (this follows by the same argument we used for $\delta_{(A,\Phi)}^2$ in the proof of Proposition \ref{proposition:stableconf}), and hence the assertion follows.
\end{proof}

By Proposition \ref{proposition:slices}, the map $\mathcal{S}_{(A, \Phi )} \rightarrow \mathscr{B}_{E,k}^\ast = \mathscr{C}_{E,k}^\ast/\mathscr{G}_{E,k+1}$ is a local homeomorphism at $(A, \Phi )$, and the Hilbert manifold structure on the neighborhood of $(A, \Phi ) \in \mathcal{S}_{(A, \Phi ),k}$ provides Hilbert manifold charts for the stable orbit space $\mathscr{B}_{E,k}^\ast$. The tangent spaces of $\mathscr{B}_{E,k}^\ast$ are thus given by
\begin{align}
T_{[(A, \Phi )]} \mathscr{B}_{E,k}^\ast & = T_{(A, \Phi )}\mathcal{S}_{(A, \Phi ),k} =\mathscr{K}_{(A, \Phi ) , k}^1  \cap \mathscr{K}_{(A, \Phi ) , k}^2\nonumber \\
& = \Big\{(\dot{A}, \dot{\Phi}) \in (\Omega^1 \oplus \Omega^1 )(Y, \mathfrak{g}_E)_{k} \, | \, -d_{A}^\ast \dot{A} - \ast [\dot{\Phi},\ast \Phi] = \ast[\dot{A}, \ast \Phi] - d_{A}^\ast \dot{\Phi} = 0  \Big\}.\label{tangentB}
\end{align}

The projection $\mathscr{C}_{E,k}^\ast \rightarrow \mathscr{B}_{E,k}^\ast$ becomes then a submersion of Hilbert manifolds. Noting that the $\mathscr{G}_{E,k+1}$--action on $\mathscr{A}_{E,k}^c$ is by isometries of the flat $L^2$ metric, there is an induced \textit{$L^2$ metric} on $\mathscr{B}_{E,k}^\ast$ making this projection a Riemannian submersion. By Proposition \ref{proposition:slices}(2), in the description (\ref{tangentB}) of $T_{(A, \Phi ) }\mathscr{B}_{E,k}^\ast$ the $L^2$ metric coincides with the one induced from the standard $L^2$ inner product.

Furthermore, recall that $\mathscr{G}_{E,k+1}$ acts on the Hilbert Kähler manifold $(\mathscr{A}_{E,k}^c , \omega_J , J ) $ preserving the Kähler structure and in a Hamiltonian fashion. By the infinite-dimensional version of the Kähler quotient Theorem, the complex structure $J$ and the symplectic form $\omega_J$ from (\ref{J}-\ref{omegaJ}) thus descend to give an (integrable) complex structure and symplectic form on the stable orbit space $\mathscr{B}_{E,k}^\ast$, compatible with the $L^2$ metric on $\mathscr{B}_{E,k}^\ast$. % (and we will still denote them by $J$ and $\omega_J$). 
In the description (\ref{tangentB}) of $T_{[(A,\Phi)]}\mathscr{B}_{E,k}^\ast$, these two structures are given by the same expressions as (\ref{J}-\ref{omegaJ})
%Note also that the complex structure $J$ on $\mathscr{A}_{E}^c$ given by (\ref{J}) is $\mathscr{G}_E$--invariant, but the tangent spaces $T_{(A, \Phi )}\mathscr{C}_{E}^\ast$ to the stable configuration space are not preserved by $J$. However, the $L^2$--orthogonal complement in each $T_{(A, \Phi )}\mathscr{C}_{E}^\ast$ of the Lie algebra of the $\mathscr{G}_E$--orbit at $(A, \Phi )$ --that is, the subspace $T_{(A, \Phi )}\mathcal{S}_{(A, \Phi )}$-- is easily seen to be invariant under $J$. The complex structure $J$ thus descends to $\mathscr{B}_{E}^\ast$.
Summing up, we have:

\begin{corollary}\label{corollary:Bhilbert}
The stable orbit space $\mathscr{B}_{E,k}^\ast$ is a Hilbert Kähler manifold when given the $L^2$ Riemannian metric, the complex structure $J$ and the symplectic form $\omega_J$. \end{corollary}

\subsubsection{The variational viewpoint}

We have obtained the stable orbit space $\mathscr{B}_{E}^\ast$ as the Káhler quotient of $(\mathscr{A}_{E}^c , \omega_J , J )$ by the $\mathscr{G}_E$--action. There is an alternative viewpoint in which the `polystability' condition $d_{A}^\ast \Phi = 0$ arises from a minimization problem. We only mention this briefly, referring to \cite{corlette-harmonic} for the details. 

Namely, consider a complex orbit $\mathscr{G}_{E}^c \cdot (A_0 , \Phi_0 )\subset \mathscr{A}_{E}^c$ and the restriction of the \textit{squared $L^2$--norm functional} $\| \Phi \|_{L^2}^2$ onto the orbit, i.e. the $\mathscr{G}_E$--invariant functional
\begin{align}
\mathscr{G}_{E}^c \cdot (A_0 , \Phi_0 ) \ni (A,\Phi ) \mapsto \| \Phi \|_{L^2}^2 . \label{L2normorbit}
\end{align}
Then $(A, \Phi) \in \mathscr{G}_{E}^c \cdot (A_0 , \Phi_0 )$ satisfies $d_{A}^\ast \Phi = 0$ iff $(A, \Phi )$ is a critical point of (\ref{L2normorbit}). Furthermore, when the $\mathscr{G}_E$--stabilizer of $(A,\Phi )$ is smallest possible, then the function (\ref{L2normorbit}) induces a \textit{convex} function on $\mathscr{G}_{E}^c / \mathscr{G}_E$ and hence if a critical point on the orbit exists then it is unique and it attains the minimum of $\| \Phi \|_{L^2}^2$ on that orbit.

\subsection{Fredholm theory of complex flat connections}\label{subsection:CS}\label{othergroups}

We now use our framework from \S \ref{subsection:setup} to discuss deformation-theoretic aspects of complex flat connections on $3$-manifolds.

\subsubsection{The complex Chern--Simons functional}

Throughout, we make a choice of a `basepoint' in $ \mathscr{A}_{E , k }^c$, i.e. fix a $GL(2,\mathbb{C} )$ connection on $E$ with fixed determinant $(\Lambda, \lambda )$, which we denote $\mathbb{A}_0$.

\begin{definition}\label{definition:complexCS}
The \textit{(complex) Chern--Simons functional} is the function $CS : \mathscr{A}_{E , k }^c \rightarrow \mathbb{C}$ defined at a connection $\mathbb{A} = \mathbb{A}_0 + \alpha$ by
\begin{align*}
CS (\mathbb{A}  ) = \int_Y \mathrm{Tr}\Big( \alpha \wedge \big( (F_{\mathbb{A}} + F_{\mathbb{A}_0}) - \frac{1}{6}[\alpha , \alpha ]\big) \Big).
\end{align*}
\end{definition}

The Chern--Simons functional has the following property. Suppose $M$ is an oriented and compact $4$-manifold with boundary $Y$ equipped with an extension of $E$ over to $M$ (as a bundle with fixed determinant) and extensions of the connection $\mathbb{A}$ and the basepoint connection $\mathbb{A}_0$ (as connections with fixed determinant). Then, using Stokes' Theorem, one can show:
\begin{align}
CS( \mathbb{A} ) = \int_M \mathrm{Tr}(F_{\mathbb{A}}\wedge F_{\mathbb{A}}) - \mathrm{Tr}(F_{\mathbb{A}_0} \wedge F_{\mathbb{A}_0}) . \label{CS-M}
\end{align}

We can use (\ref{CS-M}) to understand the periods of $CS(\mathbb{A})$ under the action of complex gauge transformations with fixed determinant. Namely, for a given $u \in \mathscr{G}_{E , k+1}^c$ choose any path of connections $\mathbb{A}(t) \in \mathscr{A}_{E,k}^c$ from $\mathbb{A}(0) = \mathbb{A}$ to $\mathbb{A}(1) = u \cdot \mathbb{A}$. This path can be regarded as a connection $\underline{\mathbb{A}}$ on the bundle over $[0,1]\times Y$ obtained as the trivial extension of $E \rightarrow Y$. We extend the basepoint connection $\mathbb{A}_0$ trivially over to $[0,1]\times Y$. Denote by $E_u \rightarrow S^1 \times Y$ the bundle obtained by clutching with $u$. By (\ref{CS-M}) and Chern--Weil theory we have

\begin{align}
CS (u \cdot \mathbb{A} ) - CS (\mathbb{A}) & = \int_{[0,1]\times Y} \mathrm{Tr}(F_{\underline{\mathbb{A}}}^2) = 8 \pi^2 ( c_2 (E_u ) - \frac{1}{2}c_1 (E_u )^2 ) \cdot [S^1 \times Y] \nonumber\\
& = 8 \pi^2 c_2 (E_u )\cdot [S^1 \times Y] \quad \in \quad  8 \pi^2 \mathbb{Z} \label{CS-periods}
\end{align}
where we note that $c_{1}(E_u )^2 = 0$ because $u$ acts as the identity on $\Lambda = \Lambda^2 E$. 

In particular, $\mathrm{Re}CS$ has periods in the lattice $8 \pi^2 \mathbb{Z} \subset \mathbb{R}$. In turn $\mathrm{Im}CS$ has no periods; furthermore it is independent of the chosen basepoint connection $\mathbb{A}_0$ and is given by the following formula: in terms of the decomposition $\mathbb{A} = A + i \Phi$ 
\begin{align}
\mathrm{Im}CS (A + i \Phi ) = 2 \int_{Y} \mathrm{Tr}\Big( \Phi \wedge \big( F_A - \frac{1}{6}[\Phi , \Phi] \big)  \Big) . \label{ImCS}
\end{align}

\subsubsection{Stable flat connections}

The variation of the Chern--Simons functional $CS : \mathscr{A}_{E, k}^c \rightarrow \mathbb{C}$ is given by
\begin{align}
CS (\mathbb{A} + \dot{\mathbb{A}} )  =  CS(\mathbb{A} ) = 2 \int_Y \mathrm{Tr}\big( \dot{\mathbb{A}} \wedge F_{\mathbb{A}}\big) + \int_Y \mathrm{Tr}\big( \dot{\mathbb{A}} \wedge d_{\mathbb{A}}\dot{\mathbb{A}}\big) + \frac{1}{3}\int_{Y}\mathrm{Tr}\big( \dot{\mathbb{A}}\wedge [\dot{\mathbb{A}}, \dot{\mathbb{A}}]\big). \label{variation}
\end{align}
From (\ref{variation}) one sees that $CS$ is a smooth (in fact, holomorphic) function on the $L^{2}_k$ connections (as long as $k \geq 1$), and its critical points are the connections $\mathbb{A}$ whose curvature has vanishing trace-less part (since the `Killing form' $(X,Y) \mapsto - \mathrm{Tr}(XY)$ on the Lie algebra $\mathfrak{gl}(2, \mathbb{C} )$ has kernel $Z( \mathfrak{gl}(2,\mathbb{C})) = \mathbb{C}\cdot I $): 
\[
F_{\mathbb{A}} - \frac{I}{2} F_{\lambda} = 0.
\]
Said differently, these are the $GL(2,\mathbb{C})$ connections $\mathbb{A}$ on $E$ with fixed determinant $(\Lambda , \lambda )$ which are \textit{projectively flat}, in the sense that the connection that $\mathbb{A}$ induces on the $PSL(2,\mathbb{C})$ bundle $P^c = E \times_{U(2)} PSL(2,\mathbb{C} )$ is flat. 
%(the latter is the bundle associated to the homomorphism $U(2) \rightarrow PSU(2) \hookrightarrow PSL(2,\mathbb{C})$). 
Informally, we might often refer to such connections plainly as `flat'. In terms of the decomposition $\mathbb{A} = A + i \Phi$, this equation reads
\begin{align}
F_A - \frac{I}{2} F_{\lambda} -  \frac{1}{2} [\Phi , \Phi ]  & = 0 \label{eq1}\\
d_{A}\Phi &= 0 \label{eq2}.
\end{align}

\begin{definition}
A \textit{polystable} (projectively) flat $GL(2,\mathbb{C})$ connection with fixed determinant is a pair $(A, \Phi ) \in \mathscr{A}_{E,k}^c$ which solves the equations (\ref{eq1}-\ref{eq2}) and
\begin{align}
d_{A}^\ast \Phi & =  0 .\label{eq3}
\end{align}
That is, $A + i \Phi$ is flat and $(A, \Phi ) $ lies in the polystable locus $\mathscr{C}_{E,k}$. The group $\mathscr{G}_{E,k+1}$ of unitary gauge transformations fixing the determinant preserves the solutions to (\ref{eq1}-\ref{eq3}). If, in addition, $(A, \Phi )$ has stabilizer $\pm I$ under the $\mathscr{G}_E$--action then $(A, \Phi )$ is \textit{stable}. 

We will denote by $\mathscr{M}_{E,k} \subset \mathscr{B}_{E,k}$ (resp. $\mathscr{M}_{E,k}^\ast \subset \mathscr{B}_{E,k}^\ast$) the topological spaces of $\mathscr{G}_{E,k+1}$--orbits of solutions of (\ref{eq1}-\ref{eq3}) (resp. those which are stable). It follows from Remark \ref{remark:ellipticity} below that every such $\mathscr{G}_{E,k+1}$--orbit contains a \textit{smooth} solution and the topological spaces $\mathscr{M}_{E,k} , \mathscr{M}_{E,k}^\ast$ do not depend on the chosen $k$. Thus, we will subsequently drop the subscript $k$ from our notation.

\end{definition}

\begin{remark}[Ellipticity] \label{remark:ellipticity} The equations (\ref{eq1}-\ref{eq3}) do not form an elliptic system, even modulo $\mathscr{G}_E $. However, this can be achieved if we `augment' the equations by introducing two additional fields $\Psi_1 , \Psi_2 \in \Omega^0 (Y , \mathfrak{g}_E )_{k}$ and considering the equations
\begin{align*}
F_A - \frac{I}{2} F_\lambda - \frac{1}{2} [\Phi , \Phi] & = \ast d_{A} \Psi_1 + \ast [\Psi_2 , \Phi]\\
d_{A} \Phi & =  \pm  \ast ( [ \Phi , \Psi_1 ] + d_{A}\Psi_2 )\\
d_{A}^\ast \Phi & =  [\Psi_1 , \Psi_2]
\end{align*}
on whose space of solutions $\mathscr{G}_{E, k+1}$ acts (by $u \Psi_i u^{-1}$ on the additional fields $\Psi_i$). Above, we mean that there are two possible systems of equations according to the choice of sign $\pm$. (Both have been studied in \cite{oliveira-nagy}; for the $-$ sign and $\Psi _2 = 0$ the equations are known as the \textit{extended Bogomolny equations} \cite{he-mazzeo1,he-mazzeo2}). Both of these augmented systems are elliptic modulo $\mathscr{G}_E$, in the usual sense. %In particular, elliptic regularity implies that any solution $(A, \Phi ) \in \mathscr{A}_{E, k}^c$ of the equations (\ref{eq1}-\ref{eq3}) can be taken to a \textit{smooth} solution by a gauge transformation in $\mathscr{G}_{E, k+1}$.
\end{remark}

\begin{remark}[Invariance under tensoring by line bundles] (See also \S \ref{U2model}) \label{remark:tensoring} For a given $U(1)$ bundle with connection $(L , A_L )$, we can consider the $U(2)$ bundle as $E \otimes L$ and give its determinant $\Lambda \otimes L^2$ the connection $\lambda \otimes A_{L}^2$. This naturally induces biholomorphisms $\mathscr{A}_{E, k}^c \cong \mathscr{A}_{E \otimes L , k }^c$, $ \mathscr{C}_{E, k}^\ast \cong \mathscr{C}_{E \otimes L , k }$, $\mathscr{B}_{E, k}^\ast \cong \mathscr{B}_{E \otimes L , k }^\ast$ which restrict to homeomorphisms $\mathscr{M}_{E} \cong \mathscr{M}_{E \otimes L  } $ and $\mathscr{M}_{E}^\ast \cong \mathscr{M}_{E \otimes L  }^\ast $. In particular, it follows that, up to homeomorphism, the spaces $\mathscr{M}_E , \mathscr{M}_{E}^\ast$ are independent of the fixed connection $\lambda$ on $\Lambda = \Lambda^2 E$; and furthermore they only depend on the $SO(3)$ bundle $P $ underlying $E$.
\end{remark}

%\begin{remark}[Character variety] By the Theorem of Donaldson \cite{donaldson-harmonic} and Corlette \cite{corlette-harmonic}, the quotient by $\mathscr{G}_E$ of the space of pairs $(A, \Phi )$ satisfying (\ref{eq1}-\ref{eq3}) can be identified with a suitable character variety. In the case of \textit{trivial} determinant $(\Lambda , \lambda )$, it is homeomorphic to the completely reducible locus in the $SL(2,\mathbb{C})$--character variety, which is the space that underlies the affine GIT quotient
%\[
%\mathrm{Hom}(\pi_1 (Y) , SL(2,\mathbb{C} ) ) /\!\!/ SL(2, \mathbb{C} ).
%\]
%(The latter is naturally an affine scheme of finite-type over $\mathbb{C}$ ).
%\end{remark}

\subsubsection{Hessians}

From now on we fix an angle $\theta \in [0 , 2 \pi)$, and consider the functional $S = \mathrm{Re}(e^{- i \theta} CS ) : \mathscr{A}_{E , k }^c \rightarrow \mathbb{R}$. For simplicity, when performing calculations and displaying formulas we shall assume $\theta = \pi/2$ (i.e. use the imaginary part of the Chern--Simons functional (\ref{ImCS}), which has the advantage of having no periods under the action of $\mathscr{G}_{ E , k+1}^c$). However, the discussion that follows goes through for all $\theta$ with little modification.

The formal gradient of $S : \mathscr{A}_{E , k }^c \rightarrow \mathbb{R}$ with respect to the flat $L^2$ metric makes sense as a smooth section, denoted $\mathrm{grad}S$, of the $L^{2}_{k-1}$ completion of the tangent bundle $T \mathscr{A}_{E, k-1}^c \rightarrow \mathscr{A}_{E,k}^c$. It is given by
\begin{align}
\mathrm{grad}S = 2 \cdot \big( \ast d_{A}\Phi , \ast (F_A  - \frac{I}{2} F_{\lambda} - \frac{1}{2}[\Phi , \Phi]) \big) . \label{gradS}
\end{align}

Since $S$ is the imaginary part of the holomorphic function $CS$, the zeros of $\mathrm{grad}S$ agree with the critical points of $CS$, i.e. the flat connections. By the following, the critical points of the restriction $S|_{\mathscr{C}_{E , k }^\ast }$ of $S$ to the stable locus are the stable flat connections:

\begin{proposition}\label{proposition:tangent}
The section $\mathrm{grad}S$ of $T \mathscr{A}_{E, k-1}^c \rightarrow \mathscr{A}_{E,k}^c$ takes values on $T \mathscr{C}_{E , k-1}^\ast \subset T \mathscr{A}_{E , k-1}^c$ along $\mathscr{C}_{E , k}^\ast \subset \mathscr{A}_{E , k}^c$. (That is, $\mathrm{grad}S$ is `tangent' to the stable locus.)
\end{proposition}
\begin{proof}
Using the Bianchi identity $d_A F_A = 0$, we compute that $(\delta_{(A, \Phi )}^2)^\ast (\mathrm{grad}S )$ is given by
\begin{align*}
\ast [\ast d_{A}\Phi , \ast \Phi] - d_{A}^\ast (\ast (F_A - \frac{I}{2}F_{\lambda} - \frac{1}{2}[\Phi , \Phi])) = \ast [d_{A}\Phi , \Phi ] - \ast [d_A \Phi , \Phi] = 0. 
\end{align*} \end{proof}
More generally, the conclusion of Proposition \ref{proposition:tangent} follows by the invariance of $S$ under the complex gauge-group action $\mathscr{G}_{E}^c$. (More precisely, all that is needed is that the derivative $(dS)_{(A, \Phi )}$ vanishes in the directions tangent to the complex orbit $\mathscr{G}_{E}^c \cdot (A, \Phi ) $).

%With this in place, we will only consider from now on the imaginary Chern--Simons functional $S$ restricted to the stable locus, $S : \mathscr{C}_{E,k}^\ast \rightarrow \mathbb{R}$. 
Thus, $\mathrm{grad}S$ gives a section of $T \mathscr{C}_{E , k-1}^\ast \rightarrow \mathscr{C}_{E , k }^\ast$. This section is $\mathscr{G}_{E , k+1}$--equivariant; hence it descends to a well-defined `vector field' on the stable orbit space $\mathscr{B}_{E , k }^\ast$. More precisely, we obtain a smooth section $[\mathrm{grad}S]$ of the Hilbert vector bundle %$[\mathrm{grad}S]$ of the Hilbert vector bundle $T\mathscr{B}_{E , k-1}^\ast \rightarrow \mathscr{B}_{E , k}^\ast$.
\begin{align}
\begin{tikzcd}
T\mathscr{B}_{E , k-1}^\ast \arrow{d}\\
\mathscr{B}_{E , k}^\ast \arrow[bend left=80]{u}{[\mathrm{grad}S]}
\end{tikzcd}\label{[gradS]}
\end{align}
%(The section $[\mathrm{grad}S]$ can be regarded as the formal gradient of $S : \mathscr{B}_{E , k}^\ast \rightarrow \mathbb{R}$).
Equivalently, the section $[\mathrm{grad}]$ can be regarded as the formal gradient of $S : \mathscr{B}_{E, k}^\ast \to \mathbb{R}$ with respect to the $L^2$ metric. In what follows we will show that the section $(\ref{[gradS]})$ is Fredholm. 

We first introduce the notation $\mathscr{K}_{(A, \Phi), j } = \mathscr{K}_{(A, \Phi ) , j }^1 \cap \mathscr{K}_{(A, \Phi ) , j }^2$ for $0 \leq j \leq k$ (cf. Definitions (\ref{T}-\ref{K}) and the $L^2$--orthogonal decompositions (\ref{hodge0}-\ref{K1})). As $(A, \Phi )$ varies in the stable configuration space, these subspaces form $\mathscr{G}$--equivariant Hilbert bundles $\mathscr{K}_{j}\rightarrow \mathscr{C}_{E , k }^\ast$ (and similarly for $\mathscr{K}_{j}^1$ and $\mathscr{K}_{j}^2$). In particular, $\mathscr{K}_{j}$ is the pullback of the tangent bundle of $\mathscr{B}_{E, j }^\ast$ along the projection $\mathscr{C}_{E, j}^\ast \rightarrow \mathscr{B}_{E , j }^\ast$ and $\mathscr{K}_{j}^2$ is the tangent bundle of $\mathscr{C}_{E, j}^\ast$.

Now, the $L^2$ Riemannian metric on $\mathscr{B}_{E, k}^\ast$ provides the Levi--Civita connection on the Hilbert vector bundle (\ref{[gradS]}), and we use this connection to define the vertical derivative of the section $[\mathrm{grad}S]$. This is a vector bundle map
\[
\mathscr{D} := \mathcal{D}[\mathrm{grad}S] : \mathscr{K}_k \rightarrow \mathscr{K}_{k-1},
\]
Formally, $\mathscr{D}_{(A, \Phi )}$ is the covariant Hessian of $S$ at $[(A, \Phi )]$ with respect to the $L^2$ metric. The following gives a more explicit description:

\begin{proposition}\label{proposition:hessian_description}
We have 
\begin{align}
\mathscr{D} = (\Pi_{\mathscr{K}_{k-1}} \circ \mathrm{Hess}S)|_{\mathscr{K}_k} : \mathscr{K}_{k}\rightarrow \mathscr{K}_{ k-1} , \label{D}
\end{align}
where
\begin{itemize}
\item $\Pi_{\mathscr{K}_{j}} : (\Omega^1 \oplus \Omega^1 )(Y, \mathfrak{g}_E )_j \rightarrow \mathscr{K}_{j}$ is the $L^2$--orthogonal projection (for $0\leq j \leq k$). Here we are using the following $L^2$--orthogonal decomposition obtained from (\ref{hodge1}-\ref{K2}) at a stable configuration $(A, \Phi ) \in \mathscr{C}_{(A, \Phi) , k }^\ast$
\begin{align}
(\Omega^1 \oplus \Omega^1)(Y, \mathfrak{g}_E )_j = \mathscr{T}_{(A, \Phi ), j}^1 \oplus \mathscr{T}_{(A, \Phi ) , j}^2 \oplus \mathscr{K}_{(A, \Phi ) , j} .\label{T1T2K}
\end{align}
\item $\mathrm{Hess}S : (\Omega^1 \oplus \Omega^1 )(Y, \mathfrak{g}_E )_k  \rightarrow (\Omega^1 \oplus \Omega^1 )(Y, \mathfrak{g}_E )_{k-1}$ is the Hessian of $S : \mathscr{A}_{E, k }^c \rightarrow \mathbb{R}$ for the flat $L^2$ metric. This is given by
\begin{align}
\mathrm{Hess}S = 2 \cdot \begin{pmatrix} \ast [\Phi , \cdot ] & \ast d_A \\
\ast d_A & - \ast [\Phi , \cdot ]
\end{pmatrix} . \label{flatHess}
\end{align}

\end{itemize}
\end{proposition}
\begin{proof}
%The operator $\mathscr{D}$ is the covariant \textit{Hessian} of $S : \mathscr{B}_{E, k }^\ast \rightarrow \mathbb{R}$ taken with respect to the $L^2$ metric on $\mathscr{B}_{E, k }$. 
By definition, the projection $\mathscr{C}_{E , k }^\ast \rightarrow \mathscr{B}_{E, k}^\ast$ is a Riemannian submersion with respect to the $L^2$ metric on $\mathscr{B}_{E, k }^\ast$ and the restriction to $\mathscr{C}_{E , k}^\ast \subset \mathscr{A}_{E , k }^c$ of the flat $L^2$ metric. As a matter of general principles, the covariant Hessian $\mathscr{D}$ of $S$ on $\mathscr{B}_{E}^\ast$ is then given by
\begin{align}
\mathscr{D} = (\Pi_{\mathscr{K}_{k-1}}^{\mathscr{K}_{k-1}^2} \circ \mathrm{\mathscr{D}}^\prime )|_{\mathscr{K}_k} : \mathscr{K}_{k} \rightarrow \mathscr{K}_{k-1} \label{QQprime}
\end{align}
where 
\begin{itemize}
    \item $\mathscr{D}^\prime : \mathscr{K}^{2}_k \to \mathscr{K}^{2}_{k-1}$ denotes the covariant Hessian of $S|_{\mathscr{C}_{E , k }^\ast}$ with respect to the restriction of the flat $L^2$ metric, and
    \item $\Pi_{\mathscr{K}_{k-1}}^{\mathscr{K}_{k-1}^2} :\mathscr{K}_{k-1}^2 \to \mathscr{K}_{k-1}$ is the $L^2$--orthogonal projection in (\ref{K2})).
\end{itemize}

Again by general principles, because $\mathrm{grad}S$ is tangent to the stable locus $\mathscr{C}_{E , k}^\ast \subset \mathscr{A}_{E , k}^c$ (Proposition \ref{proposition:tangent}) then the operator $\mathscr{D}^\prime$ is related to the flat Hessian by 
\begin{align}
    \mathscr{D}^\prime = (\Pi_{\mathscr{K}_{k-1}^2} \circ \mathrm{Hess}S )|_{\mathscr{K}_{k}^2} : \mathscr{K}_{k}^2 \rightarrow \mathscr{K}_{k-1}^2 \label{QprimeHessS} \end{align}
where $\Pi_{\mathscr{K}_{k-1}^2} :(\Omega^1 \oplus \Omega^1 )(Y, \mathfrak{g}_E )_{k-1} \rightarrow \mathscr{K}_{k-1}^2$ is the $L^2$--orthogonal projection in (\ref{hodge1}). Putting together (\ref{QQprime}-\ref{QprimeHessS}) gives the required result. (Recall that $(\delta^1 )^\ast \circ \delta^2 = 0$ at stable flat connections, which implies that $\Pi_{\mathscr{K}_{k-1}}$ agrees with $\Pi_{\mathscr{K}_{k-1}}^{\mathscr{K}^{2}_{k-1}} \circ \Pi_{\mathscr{K}^{2}_{k-1}}$).
\end{proof}

\begin{theorem}\label{theorem:fredholm}
The smooth section $[\mathrm{grad}S]$ from (\ref{[gradS]}) is Fredholm, i.e. at a stable connection $(A, \Phi )$ the Hessian $\mathscr{D}_{(A, \Phi )} : \mathscr{K}_{(A, \Phi ), k}\rightarrow \mathscr{K}_{(A, \Phi ) , k-1}$ is Fredholm. In addition:
\begin{enumerate}
\item $\mathscr{D}_{(A, \Phi )}$ is formally $L^2$-self-adjoint and hence of index zero. Furthermore, there exists a complete $L^2$--orthonormal system of eigenvectors $\mathscr{D}_{(A, \Phi )}$ contained in $L^{2}_{k+1}$, and its eigenvalues form a discrete and doubly-infinite subset of $\mathbb{R}$.
%, the eigenspaces are finite-dimensional, and there are only finitely-many (real) eigenvalues in any bounded interval. 
\item At a stable (projectively) flat connection $(A, \Phi )$, $\mathscr{D}_{(A, \Phi )}$ is given plainly by (\ref{flatHess}).
\end{enumerate}
\end{theorem}

\begin{proof}
We use an argument similar to \cite[Proposition 12.3.1]{KM} to establish the Fredholm property of $\mathscr{D}$. For a configuration $(A, \Phi )$ we consider the \textit{extended Hessian} operator 

\begin{align*}
( \widehat{\mathrm{Hess}}S )_{(A, \Phi ) , k } : (\Omega^1 \oplus \Omega^1 \oplus \Omega^0 \oplus \Omega^0 )(Y, \mathfrak{g}_E )_{k} \rightarrow (\Omega^1 \oplus \Omega^1 \oplus \Omega^0 \oplus \Omega^0 )(Y, \mathfrak{g}_E )_{k-1} \\
\end{align*}

\noindent given in block form by 
\begin{align}
( \widehat{\mathrm{Hess}}S )_{(A, \Phi ) , k }  =   \begin{pmatrix} 2 \ast [\Phi , \cdot ] & 2 \ast d_A &  -d_{A} & - [\cdot , \Phi] \\
2 \ast d_A & - 2 \ast [\Phi , \cdot ] & [\cdot , \Phi ]  & -d_{A} \\
-d_{A}^\ast  & - \ast [ \cdot , \ast \Phi] & 0 &  0\\
\ast [ \cdot , \ast \Phi] & -d_{A}^\ast & 0 & 0
\end{pmatrix} . \label{extendedHess}
\end{align}
\newline

The extended Hessian is a Fredholm operator of index zero, since it is a first-order self-adjoint elliptic operator. The $L^2$--orthogonal decomposition (\ref{T1T2K}) induces a further decomposition when $(A, \Phi ) \in \mathscr{C}_{E,k}^\ast$ is stable: for $0\leq j \leq k$
\begin{align}
(\Omega^1 \oplus \Omega^1 \oplus \Omega^0 \oplus \Omega^0 )(Y, \mathfrak{g}_E )_{j} = \mathscr{T}_{j}^1 \oplus \mathscr{T}_{j}^2 \oplus \mathscr{K}_{j} \oplus \Omega^0 (Y, \mathfrak{g}_E )_j \oplus \Omega^0 (Y, \mathfrak{g}_E )_j .
\end{align}
With respect to the latter decomposition the extended Hessian has the block form
\begin{align*}
    ( \widehat{\mathrm{Hess}}S )_{(A, \Phi ) , k } = \begin{pmatrix}
        0 & 0 & x_1 & \delta^1 & 0 \\
        0 & 0 & x_2 & 0 & \delta^2 \\
        x_{1}^\ast & x_{2}^\ast & \mathscr{D} & 0 & 0 \\
        (\delta^1)^\ast & 0 & 0 & 0 & 0 \\
        0 & (\delta^2)^\ast & 0 & 0 & 0 
    \end{pmatrix} 
\end{align*}
where
\begin{align*}
x_i = ((1 - \Pi_{\mathscr{K}_{k-1}^i} )\circ \mathrm{Hess}S )|_{\mathscr{K}_k} .
\end{align*}
We claim that $x_i$ is a compact operator (and hence its adjoint $x_{i}^\ast$ too). Indeed, the operator $ (1 - \Pi_{\mathscr{K}_{k-1}^i} )$ (equivalently, the projection to $\mathscr{T}_{k-1}^i$ in (\ref{hodge1})) is given by 
\[
(1 - \Pi_{\mathscr{K}_{k-1}^i} ) = \delta^i \circ  ((\delta^i)^\ast \delta^i )^{-1} \circ (\delta^i )^{\ast}
\]
and then observe that:
\begin{itemize}
\item $\delta^i \circ  ((\delta^i)^\ast \delta^i )^{-1}$ is a pseudo-differential operator of order $-1$ (i.e. a smoothing operator of order $1$)
\item $\delta_{i}^\ast \circ \mathrm{Hess}S $ is a \textit{first} order differential operator; in particular, it defines a bounded linear operator $L^{2}_k \rightarrow L^{2}_{k-1}$
\end{itemize}
The second item is verified by direct calculation, e.g. as a differential operator $(\Omega^1 \oplus \Omega^1 )(Y, \mathfrak{g}_E ) \rightarrow \Omega^0 ( Y , \mathfrak{g}_E )$ we have
\[
(\delta^{2})^\ast \mathrm{Hess}S = 2 \cdot \big(  [\ast (F_A - \frac{I}{2}F_\lambda - \frac{1}{2}[\Phi , \Phi]), \cdot ] , - \ast [d_{A}\Phi , \cdot ] \big) .
\]

At this point, we have that the operator obtained from $\widehat{\mathrm{Hess}}S$ by dropping the entries $x_{i} , x_{i}^\ast$ is also Fredholm. Furthermore, this operator contains $\mathscr{D}$ as a summand, hence the Fredholm property of $\mathscr{D}$ follows.

Assertion (1) holds for the extended Hessian, which is a self-adjoint elliptic operator of order one. Hence it also does for its compact perturbation obtained by dropping the entries $x_{i}, x_{i}^\ast$. Since the latter operator contains $\mathscr{D}$ as a summand, then first (1) follows. Assertion (2) follows from the fact that $x_i$ vanishes at a critical point.
\end{proof}

%The Fredholm operator $\mathscr{D}_{(A, \Phi )} : \mathscr{K}_{k, (A, \Phi )} \to \mathscr{K}_{k-1 , (A, \Phi )}$ is formally $L^2$--self-adjoint, hence has index zero. We

%\begin{corollary}\label{corollary:fredholm}
%The operator $\mathscr{D}_{(A, \Phi )} : \mathscr{K}_{(A, \Phi ), k}\rightarrow \mathscr{K}_{(A, \Phi ) , k-1}$ admits a complete orthonormal system of eigenvectors (which lie in $L^{2}_{k+1}$), and the eigenvalues form a discrete and doubly-infinite subset of $\mathbb{R}$.
%, the eigenspaces are finite-dimensional, and there are only finitely-many (real) eigenvalues in any bounded interval. 
%At a critical point $(A, \Phi )$ of $S|_{\mathscr{C}_{(A, \Phi )}^\ast}$, $\mathscr{D}_{(A, \Phi )}$ is given plainly by (\ref{flatHess}).
%\end{corollary}
%\begin{proof}
%\end{proof}

The fact that the complex Chern--Simons functional $CS$ is holomorphic manifests itself in the symmetries of the Hessian of $S = \mathrm{Im}CS $. To give some context, we recall that on a finite-dimensional complex manifold $(M,J)$ the real part $f = \mathrm{Re}F$ of a holomorphic function $F : M \to \mathbb{C}$ is \textit{pluriharmonic}: $d d^c f = 0$. If $(M, J)$ admits a Kähler metric $g$, then we obtain from $g$ the covariant Hessian $\mathrm{Hess}f : TM \to TM$, and pluriharmonicity of $f$ becomes equivalent to the vanishing of the `complex average' of its Hessian: for all $x \in M$, $v,w \in T_x M$
 \[
\frac{1}{2} ( \langle (\mathrm{Hess} f )_x v , w \rangle + \langle (\mathrm{Hess}f)_x Jc , Jw\rangle ) = 0 .
\]
\begin{proposition}\label{proposition:pluriharmonic}
%The function $S$ is pluriharmonic on $\mathscr{B}_{E,k}^\ast$, in the sense that the Hessian $\mathscr{D} = \mathscr{D}_{(A, \Phi )} :\mathscr{K}_{(A, \Phi ), k}\rightarrow \mathscr{K}_{(A, \Phi ) , k-1}$ at 
For any stable connection $(A, \Phi )$ (not necessarily flat), the Hessian $\mathscr{D}_{(A, \Phi ) }: \mathscr{K}_{(A, \Phi ) , k }\to \mathscr{K}_{(A, \Phi ) , k-1}$ satisfies the following identity: for all $v\in \mathscr{K}_{(A, \Phi ) , k }$ and $w \in \mathscr{K}_{(A, \Phi ) , k-1}$ 
\[
 \langle \mathscr{D} v , w \rangle_{L^2} + \langle \mathscr{D}(J v) , J w \rangle_{L^2} = 0 .
\]
\end{proposition}
\begin{proof}
%Consider the complex structure on $(\Omega^1 \oplus \Omega^1 \oplus \Omega^0 \oplus \Omega^0 )(Y, \mathfrak{g}_E )$ given by 
%\[
%\widehat{J} = \begin{pmatrix} 0 & -1 & 0 & 0 \\
%1 & 0 & 0 & 0 \\
%0 & 0 & 0 & 1 \\
%0 & 0 -1 & 0 \end{pmatrix} .
%\]
%Then observe that for this complex structure, t
The required identity holds for the flat Hessian defined in (\ref{flatHess}):
\[
 \langle (\mathrm{Hess}S )_{(A, \Phi )}  v , w \rangle_{L^2}  + \langle (\mathrm{Hess}S )_{(A, \Phi )} J v , Jw \rangle_{L^2}  = 0 . 
 \]
The claim now follows since $\mathscr{K}_{k}$ is a $J$-linear subspace of $(\Omega^1 \oplus \Omega^1 )(Y, \mathfrak{g}_E )$.
%The claim now follows from the fact that the inclusion
%\[
%\mathscr{K}_{k} \hookrightarrow (\Omega^1 \oplus \Omega^1 )(Y, \mathfrak{g}_E )
%\]
%and the $L^2$ projection \[
%\Pi_{\mathscr{K}_{k-1}} : (\Omega^1 \oplus \Omega^1  )(Y, \mathfrak{g}_E ) \rightarrow \mathscr{K}_{k-1}\]
%are complex-linear and Proposition \ref{proposition:hessian_description}.
\end{proof}

In particular, we have that the kernel of $\mathscr{D}_{(A, \Phi )}$ is a complex-linear subspace of $\mathscr{K}_{(A, \Phi ) , k}$.

\subsubsection{Kuranishi models}

An important point that will be used later on is the existence of finite-dimensional models for $S$, which readily follows from Theorem \ref{theorem:fredholm}.

\begin{definition}
At a stable flat connection $[(A, \Phi )]\in \mathscr{M}_{E}^\ast$, the \textit{Zariski tangent space} to $\mathscr{M}_{E}^\ast$ is the finite-dimensional vector space $\mathscr{H}_{(A, \Phi)}^1 = \mathrm{Ker}\mathscr{D}_{(A, \Phi )}$.
\end{definition}

Let $(A, \Phi )$ be a stable flat connection. Recall that the Hilbert space $\mathscr{K}_{(A, \Phi ) , k }$ is identified as the tangent space to the stable configuration space $\mathscr{B}_{E, k}^\ast$. Let $U \subset \mathscr{K}_{(A, \Phi ) , k}$ (resp. $U^\prime \subset \mathscr{K}_{(A, \Phi ), k-1}$) be the closed subspace of $\mathscr{K}_{(A, \Phi ) , k}$ (resp. $\mathscr{K}_{(A, \Phi ) , k-1}$) given as the $L^2$--orthogonal complement of $\mathscr{H}_{(A, \Phi)}^1 $. Then the Hessian $\mathscr{D}_{(A, \Phi )}$ restricts to a bounded linear isomorphism from $U$ to $U^\prime$, since $\mathscr{D}_{(A, \Phi )}$ is a self-adjoint operator.

\begin{corollary}\label{corollary:kuranishi}
Let $(A, \Phi )$ be a stable flat connection. Then there exists 
\begin{itemize}
    \item a diffeomorphism $G$ from an open neighborhood of $[(A, \Phi )] \in \mathscr{B}_{E, k }^\ast$ to an open neighborhood of $(0,0)$ in $U \oplus \mathscr{H}_{A, \Phi}^1$, such that $G ([(A, \Phi )]) = (0,0)$, and
    \item a smooth function $S^{\prime} $ on a neighbourhood of $0 $ in $  \mathscr{H}_{(A, \Phi )}^1$ with $(dS^{\prime})_0 = 0$ and $(\mathrm{Hess}S^{\prime})_0 = 0$,
\end{itemize} 
such that on the codomain of $G$ we have
\[
S(G^{-1}(u,h)) = \frac{1}{2}\langle \mathscr{D}_{(A, \Phi )} u , u \rangle_{L^2} + S^{\prime}(h).
\]
In particular, a neighborhood of $[(A, \Phi)]$ in the moduli space of stable flat connections $\mathscr{M}_{E}^\ast$ is homeomorphic to a neighborhood of $0$ in the critical locus of $S^\prime : \mathscr{H}^{1}_{(A, \Phi ) } \rightarrow \mathbb{R}$.
\end{corollary}
\begin{proof}
The argument closely follows \cite[Proposition 2.5]{donaldson-floer}. At a connection $(A+a , \Phi + \phi )$ the gradient $\mathrm{grad}S$ vector field is tangent to the subspace $\mathscr{K}_{(A+a , \Phi +\phi ),k-1}$. The $L^2$-orthogonal projection $\Pi_{\mathscr{K}_{(A, \Phi ) , k-1}}$ from $(\Omega^1 \oplus \Omega^1 )(Y, \mathfrak{g}_E )_{k-1}$ onto $\mathscr{K}_{(A, \Phi ) , k-1}$ thus restricts to an isomorphism from $\mathscr{K}_{(A+a , \Phi +\phi ),k-1}$ to $\mathscr{K}_{(A, \Phi ) , k-1}$ for $(a,\phi )$ in a small neighborhood of $(0,0)$. Consider the following smooth map defined along the stable Coulomb slice (cf. Proposition \ref{proposition:slices})

\[
f : (A+ a, \Phi + \phi ) \in \mathcal{S}_{(A, \Phi ) , k} \mapsto  \Pi_{\mathscr{K}_{(A, \Phi ) , k-1}}( (\mathrm{grad}S)_{(A+a , \Phi + \phi )}) \in \mathscr{K}_{(A, \Phi ) , k-1} .
\]

Now, the function $\Pi_{U^\prime} \circ f : \mathcal{S}_{(A, \Phi ),k} \to U^\prime$ has the property that its derivative at $(a, \phi ) = (0,0)$ is surjective  when restricted to directions in $U \subset \mathscr{K}_{(A, \Phi ) , k} = T_{(A,\Phi )}\mathcal{S}_{(A, \Phi)}$. Thus, a change of coordinates linearises $\Pi_U \circ f$. Namely, using the Inverse Function Theorem for Hilbert spaces one obtains a diffeomorphism $F$ from an open neighborhood of $(A, \Phi )$ in $ \mathcal{S}_{(A, \Phi ) , k}$ onto an open neighborhood of $(0,0)$ in $U \oplus \mathscr{H}^{1}_{(A, \Phi )}$, sending $(A, \Phi )$ to $(0,0)$, together with a smooth function $f^\prime $ from an open neighborhood of $(0,0)$ in $ U \oplus \mathscr{H}^{1}_{(A, \Phi )}$ to an open neighborhood of $0$ in  $\mathscr{H}^{1}_{(A, \Phi )}$, $f^\prime (0,0) = 0$, and such that 
\[
f( F^{-1}(u,h) ) = \big(  \mathscr{D}_{(A, \Phi )} u \, , \, f^\prime (u,h) \big) .
\]

Consider now the family of smooth functions $S_h : U \rightarrow \mathbb{R}$ obtained by restricting $S\circ F^{-1}$ onto $U \times \{ h \}$. By the above, it follows that each of these functions has a critical point at $0 \in U$ and this is a non-degenerate critical point for $h = 0$. Thus for all $h$ sufficiently close to $0$ the function $f_h$ will have a non-degenerate critical point at $0$. Then by the parametric version of the Morse--Palais Lemma \cite{lang}, a further coordinate change ensures that each $f_h$ is given by a quadratic function on $U$ plus a constant. A further change of coordinates by a family of linear isomorphisms of $U$ ensures that the quadratic functions are independent of $h$, and the constants then give the required function $S^\prime (h)$.
\end{proof}

For future reference, we include here the following:

\begin{definition}\label{definition:morse-bott}
    The moduli space $\mathscr{M}_{E}^\ast$ is \textit{Morse--Bott} if it is a smooth submanifold of $\mathscr{B}_{E}^\ast $ (hence a complex submanifold under $J$) and its dimension at a point $[(A, \Phi)] \in \mathscr{M}_{E}^\ast$ is given by $\mathrm{dim}\mathrm{Ker}\mathscr{D}_{(A, \Phi )}$. By Corollary \ref{corollary:kuranishi}, this is equivalent to being able to choose the function $S^\prime : \mathscr{H}^{1}_{(A, \Phi )} \to \mathbb{R}$ constant.
\end{definition}

\subsection{Other structure groups}

Throughout the discussion we have stuck to the structure group $G = SU(2)$ and its complexification $G^c = SL(2,\mathbb{C})$ for the sake of concreteness, and because this is the main case of interest. However, the discussion from this section can be easily adapted to the case when $G^c$ is a (real or complex) semi-simple Lie group. 

For this we fix a global Cartan involution $\Theta $ of $G^c$, whose fixed subgroup gives a maximal compact subgroup $G \subset G^c$. Let $\mathfrak{g}^c = \mathfrak{g} \oplus \mathfrak{m}$ be the induced `Cartan decomposition' at the Lie algebra level (in terms of the eigenspace decomposition of the involution $\theta$ of $\mathfrak{g}^c$ induced from $\Theta$). We obtain a $G$--invariant inner product on $\mathfrak{g}_c$ by using minus the Killing form on $\mathfrak{g}_c$ with $\theta$ inserted in one entry (for $\mathfrak{g}^c = \mathfrak{sl}(2, \mathbb{C})$ this gave $\mathrm{Tr}(AB^\ast )$).

%We obtain a positive-definite inner product on $\mathfrak{g}^c$ as minus the Killing form (in the case $\mathfrak{g} = \mathfrak{su}(2)$ this is the form $- \mathrm{Tr}(AB)$), and likewise on the Lie algebra of $\mathfrak{g}^c$ by using minus the Killing form with $\theta$ inserted in one entry.

Our basic setup is then modified as follows. We fix a principal $G$--bundle $E$. Similarly to the case we discussed, a $G^c$--connection on $E \times_G G^c$ is decomposed uniquely as $A  + \Phi $ where $A$ is a $G$-connection and $\Phi \in \Omega^1 (Y, E \times_G \mathfrak{m} )$. 

When $G^c$ is real, then the Kähler quotient viewpoint on $\mathscr{B}_{E}^\ast = \{ (A, \Phi ) \, | \, d_{A}^\ast \Phi = 0 \}/\mathscr{G}_E$ doesn't make sense anymore; however the same definition and the variational viewpoint still do. Much as in the case we presented, one can show that the stable configuration space thus obtained will again be a Hilbert Riemannian manifold equipped with the natural $L^2$ metric (albeit not Kähler, when $G^c$ is real).

Of course, the Chern--Simons functional (which uses the Killing form of $\mathfrak{g}^c$ in place of $\mathrm{Tr}$) is just real-valued when $G^c$ is real. All the results discussed above, such as the Fredholm property of the Hessian, carry through in the more general setup.

%We carried out the discussion above for the imaginary part of the complex Chern--Simons functional, but this was not essential (any 

\section{Seifert-fibered $3$-manifolds and dimensional reduction}\label{section:dimensionalreduction}

In this section we study the moduli spaces of stable (projectively) flat $GL(2,\mathbb{C})$-connections when $(Y,g)$ is a closed, oriented, \textit{Seifert-fibered} $3$-manifold equipped with a Riemannian metric $g$ belonging in a suitable class, by establishing dimensional reduction results.  Dimensional reduction for Seifert-fibered $3$-manifolds is a familiar theme in gauge theory: for flat $SU(2)$-connections, this was studied by Furuta--Steer \cite{furutasteer} and for the $U(1)$ Seiberg--Witten equations by Mrowka--Ozsváth--Yu \cite{MOY} (see also \cite{MST,munoz-wang,doan}).

We recall that a closed, oriented $3$-manifold $Y$ is \textit{Seifert-fibered} if it admits a \textit{Seifert-fibration} over a closed, oriented, orbifold surface $C$: there exists an orbifold complex line bundle $N \rightarrow C$, an (orbifold) hermitian metric on $N$ such that the unit sphere $S(N)$ is a smooth $3$-manifold (i.e. a $3$-orbifold with empty orbifold locus), and a diffeomorphism of $Y$ with $S(N)$. The projection $\pi : Y \rightarrow C$ makes $Y$ the total space of an orbifold principal $U(1)$-bundle. Throughout this section $Y$ is a Seifert-fibered $3$-manifold equipped with a fixed Seifert-fibration $\pi: Y \to C$.

%\subsection{Recollections on orbifolds}

\subsection{Seifert metrics and adiabatic connection}

%We make the assumption that the orbifold degree of $N$ is \textit{non-positive} ($\mathrm{deg}N \leq 0$), which holds after possibly reversing the orientation of $C$.

We now introduce the class of Riemannian metrics on $Y$ that we consider. Let $g_C$ be an orbifold Riemannian metric on the base orbifold $C$, with associated volume form denoted $\omega_C$. In particular, $C$ becomes an orbifold Riemann surface. Let $\mu$ %, which for simplicity we will normalize to have $\mathrm{Vol}(Y,g_C) = \pi$.  
be a unitary connection on $N \rightarrow C$ with \textit{harmonic curvature}: i.e. the curvature of $\mu$ is a constant multiple of the volume form,
\begin{align}
F_{\mu} = 2 i \xi \cdot \omega_C \quad , \quad \xi := - \frac{\pi \mathrm{deg}N}{\mathrm{Vol}(C,g_C)}. %(\geq 0) 
\label{harmoniccurvature}
\end{align}
The connection $\mu$ can be regarded as an imaginary $1$-form on $Y$, and when using this viewpoint we will denote it as $\mu = i \eta $ where $\eta \in \Omega^1 (Y,  \mathbb{R} )$. (The harmonic curvature condition then says $d\eta = 2 \xi \cdot \pi^\ast \omega_C$). We obtain a splitting of $TY$
\begin{align}
TY = \mathbb{R} \oplus \pi^\ast (TC) \label{splittingTY}
\end{align}
where $\mathbb{R} = \mathrm{Ker}(d\pi)$ is the vertical subbundle, and $\pi^\ast (TC) = \mathrm{Ker}(\eta)$ is the horizontal subbundles associated to the connection $\mu$. 

\begin{definition}\label{metric}
A \textit{Seifert metric} on $Y$ is a Riemannian metric obtained from the data $(g_C, \mu = i \eta )$ by the formula
\begin{align*}
g = \eta\otimes \eta + \pi^\ast (g_C). 
\end{align*}
\end{definition}
%\textcolor{red}{Equivalently, a Seifert metric is a Riemannian metric on $Y$ which makes $\pi : Y \rightarrow C$ into an (orbifold) Riemannian submersion whose smooth fibers are geodesic and have length $2\pi$, for some orbifold metric $g_C$ on $C$.} 

We fix throughout this section such a Seifert metric $g$ on $Y$. Denote by $\zeta$ the vector field on $Y$ that generates the Seifert $S^1$--action on $Y$. Then $\zeta$ has unit length and is metric dual to the $1$-form $ \eta$.

A Seifert metric $g$ makes the splitting (\ref{splittingTY}) orthogonal. However, the Levi-Civita connection of $g$, which we denote $\widehat{\nabla}$, is \textit{not} compatible with this splitting in general. Instead, we consider a different orthogonal connection on $TY$, namely the \textit{adiabatic connection} $\nabla$ defined in terms of the splitting (\ref{splittingTY}) and the Levi-Civita connection $\nabla_C$ of $C$ by
\begin{align*}
\nabla = d \oplus \pi^\ast (\nabla_C ) .
\end{align*}
This is the natural connection to consider for the purpose dimensional reduction. (E.g. a $1$-form $\alpha$ on $Y$ is the pullback of a $1$-form on $C$ precisely when $\nabla_{\zeta}\alpha = 0$). We will need to compare the Levi-Civita and adiabatic connections: the following calculation can be found in \cite[\S 1]{nicolaescu}
\begin{lemma}\label{lemma:nicolaescu}
Let $\zeta_0 , \zeta_1 , \zeta_2 $ be a local orthonormal frame of $(TY,g)$ with $\zeta_0 = \zeta$, and $\zeta_2 , \zeta_3$ given by the horizontal lifts of a local orthonormal frame of $(TC, g_C )$. Denote the dual local orthonormal frame of $T^\ast Y$ by $\eta^0 (= \eta) , \eta^1 , \eta^2$. Then the difference between the Levi-Civita and the adiabatic connections on $T^\ast Y$, regarded as an endomorphism-valued $1$-form $\widehat{\nabla} - \nabla \in \Omega^1 (Y, \mathfrak{so}(T^\ast Y) )$, is given by
\[
\widehat{\nabla} - \nabla = \begin{pmatrix}  0 & \xi \eta^2 & -\xi \eta^1 \\
-\xi \eta^2 & 0 & -\xi \eta \\
\xi \eta^1 & \xi \eta & 0 
\end{pmatrix}  .
\]
\end{lemma}
(Lemma \ref{lemma:nicolaescu} shows, in particular, that the adiabatic connection $\nabla$ is obtained as the $C^\infty$ limit as $t \to \infty$ of the Levi--Civita connection associated to the metric $g_t = \eta \otimes \eta + t \pi^\ast g_C $.)

\subsection{Vanishing results}

We now discuss the vanishing results that will enable the dimensional reduction. We fix throughout a $U(2)$--bundle $E \rightarrow Y$ with fixed determinant $(\Lambda = \Lambda^2 E , \lambda )$. We fix throughout a $GL(2,\mathbb{C})$--connection $\mathbb{A} = A + i \Phi$ on $E$ with fixed determinant.

\subsubsection{Covariant derivatives}

We first describe how the exterior covariant derivatives $d^A : \Omega^n (Y, \mathfrak{g}_E ) \to \Omega^{n+1}(Y, \mathfrak{g}_E )$, $n\geq 0$, interact with the Seifert fibration $\pi: Y \rightarrow C$.

By pullback, we have the subbundle $\pi^\ast (T^\ast C)  \subset T^\ast Y$. (This is the annihilator of the vertical subbundle $\mathbb{R} = \mathrm{Ker}(d\pi )  \subset TY$). Denote its $n^{th}$ exterior power by $\Lambda^{n}_H := \Lambda^n (\pi^\ast (T^\ast C) ) $, and the space of sections of $\Lambda^{n}_H$ by $\Omega_{H}^n (Y, \mathfrak{g}_E )$. Using the connection $\mu = i \eta $ on $N$, we obtain the splitting $T^\ast Y = \mathbb{R} \oplus \pi^\ast (T^\ast C)$ dual to (\ref{splittingTY}), which induces a splitting for all $n\geq 0$
\begin{align}
\Omega^{n}(Y, \mathfrak{g}_E ) =& \, \, \Omega_{H}^{n-1}(Y, \mathfrak{g}_E ) \oplus \Omega^{n}_H (Y , \mathfrak{g}_E ) \label{splittingOmega}\\
\eta \cdot \alpha + \beta \mapsfrom & (\alpha , \beta ) . \nonumber
\end{align}

We will consider the following differential operators:
\begin{itemize}
\item The adiabatic covariant derivative $\nabla_{\zeta}^A : \Omega^n (Y, \mathfrak{g}_E ) \to \Omega^n (Y, \mathfrak{g}_E ) $ along the \textit{vertical} vector field $\zeta$ that generates the Seifert $S^1$--action,
\item The \textit{horizontal} exterior covariant derivative $d_{A}^H : \Omega_{H}^n (Y , \mathfrak{g}_E ) \rightarrow \Omega_{H}^{n + 1} (Y, \mathfrak{g}_E )$ defined by $(d_{A}^H \gamma )(V_1 , \ldots , V_{n+1}) = (d_A \gamma )(V_{1}^{H} , \ldots , V_{n +1}^{H})$, where $V^H$ stands for the projection of a vector field $V$ on $Y$ onto the horizontal subspace of (\ref{splittingTY}). 

\item The formal $L^2$--adjoint of $d_{A}^H$, given by  $(d_{A}^H)^\ast = - \ast_H d_{A}^H \ast_H  : \Omega^n (Y, \mathfrak{g}_E ) \to \Omega^{n-1}(Y, \mathfrak{g}_E )$, where $\ast_H : \Lambda^{\ast}_H \xrightarrow{\cong} \Lambda^{2-\ast}_H$ is the bundle isomorphism induced by the Hodge star operator $\ast_C $ associated to the orbifold surface $(C,g_C)$
\end{itemize}

We will need the following result:

\begin{lemma}\label{lemma:d_B}
With respect to the splitting (\ref{splittingOmega}), the exterior covariant derivative $d_{A} : \Omega^{n}(Y, \mathfrak{g}_E ) \rightarrow \Omega^{n+1}(Y, \mathfrak{g}_E )$ is given in block form by
\renewcommand\arraystretch{1.4}
\begin{equation*}
d_A = \begin{pmatrix}
- d_{A}^H & \nabla_{\zeta}^A \\
2 \xi (\pi^\ast \omega_C ) \wedge \cdot & d_{A}^H 
\end{pmatrix}
\end{equation*}
%where 
%\begin{itemize}
%\item $\nabla_{v}^B$ denotes the adiabatic covariant derivative along $v$ coupled with the connection on $\mathfrak{g}_E$ induced by $B$, and
%\item $d_{H}^B : \Omega_{H}^\ast (Y , \mathfrak{g}_E ) \rightarrow \Omega_{H}^{\ast + 1} (Y, \mathfrak{g}_E )$ is defined by $(d_{H}^B \gamma )(V_1 , \ldots , V_{\ast+1}) = (d_B \gamma )(V_{1}^{H} , \ldots , V_{\ast +1}^{H})$, where $V^H$ stands for the projection of a vector field $V$ on $Y$ onto the horizontal subspace of (\ref{splittingTY}). With respect to the local orthonormal frame and coframe of Lemma \ref{lemma:nicolaescu} we have $d_{H}^B \beta = \eta^1 \wedge \nabla_{\zeta_1}^B \beta + \eta^2 \wedge \nabla_{\zeta_2}^B \beta $.
%\end{itemize}
and its formal $L^2$--adjoint $d_{A}^\ast : \Omega^n (Y, \mathfrak{g}_E ) \rightarrow \Omega^{n -1} (Y, \mathfrak{g}_E )$ by 
\begin{equation*}
d_{A}^\ast = \begin{pmatrix}- (d_{A}^H)^\ast &  2\xi (\pi^\ast \omega_C \wedge \cdot)^\ast \\
- \nabla_{\zeta}^A & (d_{A}^H)^\ast .
\end{pmatrix}
\end{equation*}
%Here $(d_{H}^B)^\ast$ is the formal $L^2$-adjoint of $d_{H}^B$, which is given by
%\[
%%(d_{H}^B)^\ast = - \ast_H d_{B}^H \ast_H 
%\]
%where $\ast_H : \Lambda^{\ast}_H \xrightarrow{\cong} \Lambda^{2-\ast}_H$ is the bundle isomorphism induced by the Hodge star operator $\ast_C $ associated to the orbifold surface $(C,g_C)$.
\end{lemma}
\begin{proof}
Fix a local orthonormal frame and coframe as in Lemma \ref{lemma:nicolaescu}. We recall the standard formula expressing the exterior derivative $d_A$ in terms of the Levi-Civita connection
\[
d_A  = \eta \wedge \widehat{\nabla}_{\zeta}^A  + \eta^1 \wedge \widehat{\nabla}_{\zeta_1}^A  + \eta^2 \wedge \widehat{\nabla}_{\zeta_2}^A  .
\]

Let $\Psi \in \Omega^{n}_H (Y, \mathfrak{g}_E )$. Then the horizontal exterior covariant derivative is given by
\begin{align}
d_{A}^H \Psi = \eta^1 \wedge \nabla_{\zeta_1}^A \Psi  + \eta^2 \wedge \nabla_{\zeta_2}^A \Psi \label{dH}.
\end{align}

By Lemma \ref{lemma:nicolaescu} when $n = 1$, and hence by the Leibniz rule for general $n$, we have
\\
\begin{align*}
\eta \wedge \widehat{\nabla}_{\zeta}^B \beta  + \eta^1 \wedge \widehat{\nabla}_{\zeta_1}^B \Psi + \eta^2 \wedge \widehat{\nabla}_{\zeta_2}^B \Psi = \eta \wedge \nabla_{\zeta}^B \Psi + \eta^1 \wedge \nabla_{\zeta_1}^B \Psi + \eta^2 \wedge \nabla_{\zeta_2}^B \Psi
\end{align*}
\\
Putting everything together we have $d_A \Psi = \eta \wedge \nabla_{\zeta}^A \Psi + d_{A}^H \Psi $, which gives the second column in the block matrix for $d_A$. By this formula and the Leibniz rule we also have
\begin{align*}
d_A (\eta \wedge \Psi )  = d\eta \wedge \Psi - \eta \wedge d_{A}\Psi  = 2 \xi \pi^\ast (\omega_C ) \wedge \Psi - \eta \wedge d_{A}^H \Psi 
\end{align*}
which gives the second column of the block matrix for $d_A$. The formula for $d_{A}^\ast$ follows by taking the adjoint of the block matrix for $d_A$.
\end{proof}

\subsubsection{A Bochner--Weitzenböck formula}

We decompose the imaginary part $\Phi \in \Omega^1 (Y, \mathfrak{g}_E )$ of our $GL(2, \mathbb{C})$ connection according to (\ref{splittingOmega}) as $\phi = \eta \cdot \alpha + \Psi$, where
\[
\alpha \in \Omega^0 (Y, \mathfrak{g}_E ) \quad , \quad \Psi \in \Omega_{H}^1 (Y, \mathfrak{g}_E ) .
\]

\begin{proposition}\label{proposition:bochner}
Let $A + i \Phi$ be a $GL(2,\mathbb{C})$-connection on $E$ with fixed determinant. Then the following pointwise identity holds 
\begin{align*}
\langle ( d_{A}d_{A}^\ast + d_{A}^\ast d_{A})\Phi \, , \, \Phi \rangle & = \langle -(\nabla_{\zeta}^A )^2 \alpha \, , \, \alpha \rangle  + \langle (\ast_H \nabla_{\zeta}^B)^2 \Psi \, , \, \Psi \rangle \\
& + \langle (d_{A}^H)^\ast d_{A}^H \alpha \, , \, \alpha \rangle + \langle d_{A}^H(d_{A}^H)^\ast \Psi \, , \, \Psi \rangle + \langle [\iota_{\zeta}(F_A ) , \alpha] \, , \, \Psi \rangle \\
& +  \langle (d_{A}^H)^\ast d_{A}^H \Psi \, , \, \Psi  \rangle +  4\xi \langle \ast_H d_{A}^H \Psi \, , \, \alpha   \rangle + 4\xi^2 \cdot | \alpha|^2 
\end{align*}
(Here $\iota_\zeta F_A = ( F_A )(\zeta , \cdot )$ denotes contraction of $F_A$ with the vertical vector field $\zeta$.)
\end{proposition}

\begin{proof}
Consider the self-adjoint differential operator $D_A : (\Omega^0 \oplus \Omega^1)(Y, \mathfrak{g}_E ) \rightarrow (\Omega^0 \oplus \Omega^1)(Y, \mathfrak{g}_E )$ given by the \textit{odd signature operator} coupled with the connection $A$ on $\mathfrak{g}_E$: 
\begin{align*}
D_A & = \begin{pmatrix} 0 &  - d_{A}^\ast \\ - d_A & \ast d_A \end{pmatrix} .
\end{align*}
This operator squares to 
%If $\Delta_B $ stands for the Hodge Laplacian $d_B d_{B}^\ast + d_{B}^\ast d_B$ acting on $\Omega^\ast (Y, \mathfrak{g}_E )$, then we have
\begin{align}
(D_A)^2 &= \begin{pmatrix} d_A d_{A}^\ast + d_{A}^\ast d_A & \ast [F_A , \cdot ] \\ -\ast [F_A , \cdot ] & d_A d_{A}^\ast + d_{A}^\ast d_A \end{pmatrix}. \label{squareD}
\end{align}

We further decompose $(\Omega^0 \oplus \Omega^1 )(Y, \mathfrak{g}_E ) = (\Omega^0 \oplus \Omega^0 \oplus \Omega_{H}^1 )(Y , \mathfrak{g}_E )$ using the splitting (\ref{splittingOmega}) for $n = 1$. By Lemma \ref{lemma:d_B} the operator $D_A$ is given by $D_A = D_{\zeta}^A + D_{H}^A$ where

\begin{align*}
%D_B & = D_{\zeta}^B + D_{H}^B \\
D_{\zeta}^A = \begin{pmatrix} 0 & \nabla_{\zeta}^A & 0 \\ 
- \nabla_{\zeta}^A & 2\xi & 0 \\
0 & 0 & \ast_H \nabla_{\zeta}^A \end{pmatrix} \quad &, \quad D_{H}^A = \begin{pmatrix} 0 & 0 & - (d_{A}^H)^\ast \\ 0 & 0 & \ast_H d_{A}^H \\
- d_{A}^H & - \ast_H d_{A}^H   & 0 \end{pmatrix}.
\end{align*}

%Here $\ast_H : \pi^\ast \Lambda^{j}(T^\ast C) \rightarrow \pi^\ast \Lambda^{2-j}(T^\ast C)$ stands for the horizontal lift of the Hodge star operator on $C$. 

We have the identity
\begin{align}
(D_A)^2 = (D_{\zeta}^A)^2 + (D_{H}^A)^2 + (D_{\zeta}^{A}D_{H}^A + D_{H}^A D_{\zeta}^A) \label{squareD2}
\end{align}
and in what follows now we compute the three operators on the right-hand side of (\ref{squareD2}) one by one.

\begin{align}
(D_{\zeta}^A)^2 & = \begin{pmatrix} - (\nabla_{\zeta}^A)^2 & 2\xi \nabla_{\zeta}^A & 0 \\
-2\xi \nabla_{\zeta}^A & - (\nabla_{\zeta}^A)^2 + 4\xi^2 & 0 \\
0 & 0 & (\ast_H \nabla_{\zeta}^{A} )^2\end{pmatrix} \label{squareDzeta}\\
(D_{H}^A)^2 & = \begin{pmatrix}  (d_{A}^H)^\ast d_{A}^H & \ast_H (d_{A}^H)^2 & 0 \\
 -\ast_H (d_{A}^H)^2 & (d_{A}^H)^\ast d_{A}^H & 0   \\
  0 & 0  & d_{A}^H (d_{A}^H)^\ast + (d_{A}^H)^\ast d_{A}^H
\end{pmatrix} \label{squareDH}
\end{align}
and
$( D_{\zeta}^{B}D_{H}^B + D_{H}^B D_{\zeta}^B) $ is given by
\\
\begin{align}
\begin{pmatrix} 0 & 0 &  \nabla^{A}_{\zeta} \ast_H d_{A}^H - \ast_H d_{A}^H \nabla_{\zeta}^A\\
0 & 0 & \nabla_{\zeta}^A (d_{A}^H)^\ast - (d_{A}^H)^\ast \nabla_{\zeta}^A + 2 \xi  \ast_H d_{A}^H \\
- \ast_H \nabla_{\zeta}^A d_{A}^H + \ast_H d_{A}^H \nabla_{\zeta}^A & - \ast_H \nabla_{\zeta}^A \ast_H d_{A}^H - d_{A}^H \nabla_{\zeta}^A - 2 \xi \ast_H d_{A}^H & 0 \end{pmatrix} . \nonumber
\end{align}
\\

The entries in the last matrix can be described more succinctly. For this, we fix a local orthonormal frame $\zeta_0 (= \zeta ) , \zeta_2 , \zeta_2$ and dual coframe $\eta^0 (= \eta ), \eta^1 , \eta^2$ as in Lemma \ref{lemma:nicolaescu}. By definition of the adiabatic connection, we have $\nabla_{\zeta}\eta^1 = \nabla_{\zeta}\eta^2 = 0$. We then compute as follows, for an auxiliary $s \in \Omega^0 (Y, \mathfrak{g}_E )$:
\begin{align*}
(- \ast_H \nabla_{\zeta}^A \ast_H d_{A}^H - d_{A}^H \nabla_{\zeta}^A )s  & = - \ast_H \nabla_{\zeta}^A  (\eta^2 \wedge \nabla_{\zeta_1}^A s - \eta^1 \wedge \nabla_{\zeta_2}^A s ) - ( \eta^1 \wedge \nabla_{\zeta_1}^A \nabla_{\zeta}^A s + \eta^2 \wedge \nabla_{\zeta_2}^A \nabla_{\zeta}^A s )\\
& = - \ast_H ( \eta^2 \wedge \nabla_{\zeta}^A \nabla_{\zeta_1}^A s - \eta^1 \wedge \nabla_{\zeta}^A \nabla_{\zeta_2}^A s ) - ( \eta^1 \wedge \nabla_{\zeta_1}^A \nabla_{\zeta}^A s + \eta^2 \wedge \nabla_{\zeta_2}^A \nabla_{\zeta}^A s )\\
& = e_1 [\nabla_{\zeta}^A , \nabla_{e_1}^A ] + e_2 [\nabla_{\zeta}^A , \nabla_{e_2}^A] )s  = \frac{1}{2} [ \iota_{\zeta} F_B ,s ] 
\end{align*}

\begin{align*}
(- \ast_H \nabla_{\zeta}^A d_{A}^H + \ast_H d_{A}^H \nabla_{\zeta}^A )s & = -\ast_H \nabla_{\zeta}^A (\eta^1 \wedge \nabla_{\zeta_1}^A s  + \eta^2 \wedge \nabla_{\zeta_2}^A s  ) + \ast_H d_{A}^H \nabla_{\zeta}^A s \\
& =  -\ast_H ( \eta^1 \wedge \nabla_{\zeta}^A \nabla_{\zeta_1}^A s + \eta_2 \wedge \nabla_{\zeta}^A  \nabla_{\zeta_2}^A s ) + \ast_H (  \eta_1 \wedge \nabla_{\zeta_1}^A \nabla_{\zeta}^A s + \eta_2 \wedge \nabla_{\zeta_2}^A \nabla_{\zeta}^A s  ) \\
%& = -e_2  \nabla_{\zeta}^B \nabla_{e_1}^Bs + e_1 \nabla_{\zeta}^B \nabla_{e_2}^B s   + e_2  \nabla_{e_1}^B \nabla_{\zeta}^Bs - e_1 \nabla_{e_2}^B \nabla_{\zeta}^Bs  \\
& = -\ast_H ( \eta_1 \wedge [\nabla_{\zeta}^B , \nabla_{\zeta_1}^B] + \eta_2 \wedge [\nabla_{\zeta}^B , \nabla_{\zeta_2}^B])s  = -\ast_H \frac{1}{2} [\iota_\zeta F_B ,s ] .
\end{align*}
By the self-adjointness of $(D_{\zeta}^{A}D_{H}^A + D_{H}^A D_{\zeta}^A)$, the remaining entries in the matrix can be determined from the previous two. Putting everything together we obtain
\\
\begin{align}
(D_{\zeta}^{A}D_{H}^A + D_{H}^A D_{\zeta}^A) & = \begin{pmatrix} 0 & 0 &  -\ast_H \frac{1}{2} [\iota_\zeta F_A , \cdot ]\\
0 & 0 & \ast_H \frac{1}{2} [\iota_\zeta F_A , \ast_H \cdot ] + 2 \xi \ast_H d_{A}^H \\
-\ast_H \frac{1}{2} [\iota_\zeta F_A , \cdot ] & \frac{1}{2} [ \iota_{\zeta} F_A , \cdot ] - 2\xi \ast_H d_{A}^H & 0 \end{pmatrix} . \label{commutator}
\end{align}
\\

Finally, the required identity follows by applying both sides of (\ref{squareD2}) to $(0 , \alpha , \Psi )$, inserting the expressions (\ref{squareD}),(\ref{squareDzeta}), (\ref{squareDH}), (\ref{commutator}), and then taking the pointwise inner product against $(0 , \alpha , \Psi )$.
\end{proof}

We have now arrived at the crucial vanishing result:

\begin{corollary}\label{corollary:vanishing}
Let $A + i \Phi$ be a $GL(2,\mathbb{C})$-connection on $E$ with fixed determinant. %$(\Lambda = \Lambda^2 E , \lambda )$. Suppose that $\lambda$ satisfies $\iota_\zeta F_\lambda = 0$. 
Suppose that $A+i\Phi$ is projectively flat (cf. (\ref{eq1}-\ref{eq2})) and polystable (cf. (\ref{eq3})). Then the following quantities vanish identically:
\begin{align*}
\nabla_{\zeta}^A \alpha = 0 \quad , \quad   \nabla_{\zeta}^A \Psi = 0\quad , \quad 
 d_{A}^H \alpha = 0 \quad , \quad [\alpha , \Psi] = 0
\end{align*}
and we have
\begin{align*}
F_A - \frac{I}{2} F_\lambda - \frac{1}{2}[\Psi , \Psi] & = 0\\
d_{A}^H \Psi & = 0\\
(d_{A}^H )^\ast \Psi & = 0.
\end{align*}
%Then the following quantities vanish identically:
%\begin{align*}
% \nabla_{\zeta}^A \alpha \quad , \quad   \nabla_{\zeta}^A \Psi \quad , \quad 
% d_{A}^H \alpha \quad , \quad 
% d_{A}^H \Psi \quad , \quad   (d_{A}^H )^\ast \Psi \quad \text{and} \quad [\alpha , \Psi] .
%\end{align*}
Furthermore, if either the orbifold degree $\mathrm{deg}N$ is non-zero or $A + i \Phi$ is stable, then $\alpha = 0$ also.
\end{corollary}
\begin{proof}
By (\ref{eq1}) we have $\iota_\zeta F_A = [\alpha , \Psi ] + \frac{I}{2} \iota_\zeta F_\lambda$, and thus
\[
\langle [\iota_\zeta F_A , \alpha ] , \Psi \rangle = \langle [[\alpha , \Psi],\alpha],\Psi \rangle = | [\alpha , \Psi]|^2 .
\]
By (\ref{eq2}-\ref{eq3}) the left-hand side of the identity in Proposition \ref{proposition:bochner} vanishes identically. Integrating that identity over $Y$ and integrating by parts yields the following identity:
\begin{align*}
0   = & \| \nabla_{\zeta}^A \alpha \|_{L^2 (Y)}^2  + \| \nabla_{\zeta}^A \Psi \|_{L^2(Y)}^2   + \| d_{A}^H \alpha \|_{L^2 (Y)}^2 + \| (d_{A}^H )^\ast \Psi \|_{L^2 (Y)}^2+ \| [\alpha , \Psi ]\|_{L^2 (Y)}^2\\
& + \|\ast_H d_{A}^H \Psi  + 2 \xi \alpha \|_{L^2 (Y)}^2 .
%& + \| d_{A}^H \Psi \|_{L^2 (Y)}^2     + 4\xi^2 \cdot \| \alpha\|_{L^2 (Y)}^2 .
\end{align*}
Thus
\begin{align*}
\nabla_{\zeta}^A \alpha = 0 \quad , \quad \nabla_{\zeta}^A \Psi = 0 \quad , \quad d_{A}^H \alpha = 0\\
 (d_{A}^H )^\ast \Psi = 0 \quad , \quad [\alpha , \Psi ] = 0 \quad , \quad \ast_H d_{A}^H \Psi  + 2 \xi \alpha = 0 .
 \end{align*}
Applying $d_{A}^H$ to the last identity above and using $d_{A}^H \alpha = 0$ gives $d_{A}^H \ast_H d_{A}^H \Psi  = 0$. Integration by parts then gives $d_{A}^H \Psi = 0$ and hence also $\xi \alpha = 0$. 

Thus if $\xi \neq0$ (equivalently, if $\mathrm{deg}N\neq 0$) then $\alpha = 0$. Finally, suppose that $\xi = 0$ but $A + i \Phi$ is stable. Then from $\nabla_{\zeta}^A \alpha = 0$, $d_{A}^H \alpha = 0$ and Lemma \ref{lemma:d_B} we have $d_A \alpha =0$. This, combined with $[\alpha , \Phi ]= [\alpha , \Psi ] = 0$, implies that $\alpha$ vanishes by stability. With this, the proof is complete.
\end{proof}
%We also consider the `vertical' covariant derivative operator defined by combining the connection on $\Lambda_{H}^1 \subset T^\ast Y$ induced by the adiabatic connection on $TY$, and the connection on $\mathfrak{g}_E$ induced from the connection $B$ on $E$,
%\[
%\nabla_{\zeta}^B : \Gamma ( Y , \Lambda^{\ast}_H \otimes \mathfrak{g}_E ) \rightarrow \Gamma ( Y , \Lambda^{\ast}_H \otimes \mathfrak{g}_E ) .
%\]

%We consider the following differential operators
%\begin{align*}
%d_B : \Omega^\ast (Y , \mathfrak{g}_E ) \rightarrow \Omega^{\ast + 1}(Y, \mathfrak{g}_E )\\
%d^{B}_\zeta : \Omega^\ast (Y, \mathfrak{g}_E )
%\end{align*}

We will also need an `infinitesimal' analogue of Corollary \ref{corollary:vanishing}. Let $A + i \Phi$ be a stable projectively flat $GL(2,\mathbb{C})$ connection with fixed determinant. Then the Zariski tangent space $\mathscr{H}^{1}_{(A, \Phi )} = \mathrm{Ker}\mathscr{D}_{(A, \Phi )}$ to the moduli space $\mathscr{M}_{E}^\ast$ at $[(A, \Phi )]$ is given by the space of $(\dot{A}, \dot{\Phi} ) \in (\Omega^1 \oplus \Omega^1 )(Y, \mathfrak{g}_E )$ such that
\begin{align}
- d_{A}^\ast \dot{A} - \ast [\dot{\Phi}, \ast \Phi ] & = 0 \label{Zeq}\\
\ast [\dot{A} , \ast \Phi] - d_{A}^\ast \Phi & = 0 \nonumber\\
\ast [\dot{A} , \Phi ] + \ast d_{A}\dot{\Phi} & = 0 \nonumber\\
\ast d_{A}\dot{A} - \ast [\dot{\Phi} , \Phi] & = 0 \nonumber.
\end{align}
Decompose $\dot{A} = \eta \cdot \dot{a} + \dot{B}$ and $\dot{\Phi} = \eta \cdot \dot{\alpha} + \dot{\Psi}$ according to the splitting \ref{splittingOmega}. 

\begin{proposition}\label{prop:vanishing_inf}
Suppose that $A + i \Phi$ is a polystable projectively flat $GL(2, \mathbb{C} )$ connection on $E$ with fixed determinant. Suppose that $(\dot{A} , \dot{\Phi} ) \in (\Omega^1 \oplus \Omega^1 )(Y, \mathfrak{g}_E )$ solves the equations (\ref{Zeq}). Then the following quantities vanish identically:
\begin{align*}
%\dot{a} = 0\quad , \quad \dot{\alpha} = 0 \quad , \quad 
\nabla_{\zeta} \dot{a} = 0 \, , \, \nabla_{\zeta}^A \dot{B} = 0  \, , \, [\dot{a}, \alpha]= 0 \, , \, [\dot{a}, \Psi] = 0 \, , \, [\dot{B},\alpha] = 0
\\
\nabla_{\zeta}\dot{\alpha} =0  \, , \, \, , \, \nabla_{\zeta}^A \dot{\Psi} = 0 \, , \, [\dot{\alpha},\alpha] = 0 \, , \, [\dot{\alpha} , \Psi] = 0 \, , \, [\dot{\Psi},\alpha] = 0
\end{align*}
and we have
\begin{align*}
- (d_{A}^H )^\ast \dot{B} - \ast_H [ \dot{\Psi} , \ast_H \Psi ] & = 0\\
 \ast_H [ \dot{B}, \ast_H \Psi ] - (d_{A}^H )^\ast \dot{\Psi}& = 0\\
\ast_H [ \dot{B}, \Psi ] + \ast_H d_{A}^H \dot{\Psi} & = 0\\
\ast_H d_{A}^H \dot{B} - \ast_H [\dot{\Psi}, \Psi] &= 0.
\end{align*}
Furthermore, if either the orbifold degree $\mathrm{deg}N$ is non-zero or $A+ i \Phi$ is stable, then $\dot{a} = 0$ and $\dot{\alpha} = 0$ also.
\end{proposition}
\begin{proof}
The equations (\ref{Zeq}) can be decomposed into vertical and horizontal parts using Lemma \ref{lemma:d_B}: 
\begin{align}
- (d_{A}^H )^\ast \dot{B} - \ast_H [ \dot{\Psi} , \ast_H \Psi ] & = -\nabla_{\zeta}^A \dot{a} + [\dot{\alpha} , \alpha ] \label{ZeqVH}\\
 \ast_H [ \dot{B}, \ast_H \Psi ] - (d_{A}^H )^\ast \dot{\Psi}& = -\nabla_{\zeta}^A \dot{\alpha} - [\dot{a} , \alpha ]\nonumber\\
\ast_H [ \dot{B}, \Psi ] + \ast_H d_{A}^H \dot{\Psi} & = -2\xi \dot{\alpha} \nonumber\\
- \ast_H d_{A}^H \dot{\alpha} + \ast_H \nabla_{\zeta}^A \dot{\Psi} + \ast_H [\dot{a}, \Psi] - \ast_H [\dot{B},\alpha] & = 0  \nonumber\\
\ast_H d_{A}^H \dot{B} - \ast_H [\dot{\Psi}, \Psi] &= -2\xi \dot{a}\nonumber\\
-\ast_H d_{A}^H \dot{a} + \ast_H \nabla_{\zeta}^A \dot{B} - \ast_H [\dot{\alpha}, \Psi] + \ast_H [\dot{\Psi},\alpha] & = 0 \nonumber.
\end{align}

We first establish the vanishing of the quantities $\nabla_{\zeta}\dot{\alpha}, d_{A}^H \dot{\alpha},[\dot{\alpha},\alpha] , [\dot{\alpha} , \Psi] $. For this, we express the second and third equations in (\ref{Zeq}) jointly in terms of the operator $D_A$ introduced in the proof of Proposition \ref{proposition:bochner}, as
\begin{align*}
D_A \begin{pmatrix} 0 \\ \dot{\Phi}\end{pmatrix} + \begin{pmatrix} \ast [\dot{A}, \ast \Phi ] \\ \ast [\dot{A} , \Phi] \end{pmatrix} = 0 .
\end{align*}
We then apply $D_A$ to this identity and take the inner product against $(0 , \eta \dot{\alpha} )$: 
\begin{align}
    \langle D_{A}^2 \begin{pmatrix} 0 \\ \dot{\Phi}\end{pmatrix} , \begin{pmatrix} 0 \\ \eta \dot{\alpha} \end{pmatrix}\rangle + \langle D_A \begin{pmatrix} \ast [\dot{A}, \ast \Phi ] \\ \ast [\dot{A} , \Phi] \end{pmatrix}, \begin{pmatrix} 0 \\ \eta \dot{\alpha }\end{pmatrix}\rangle  = 0.\label{D2Z}
\end{align}

We now compute each term in (\ref{D2Z}). For the first term, by (\ref{squareD}), (\ref{squareDzeta}), (\ref{squareDH}), and the vanishing of the terms involving $\iota_\zeta F_A $ in (\ref{commutator}) (since $[\alpha , \Psi] = 0$ by Corollary \ref{corollary:vanishing}) we have
\begin{align}
\langle D_{A}^2 \begin{pmatrix} 0 \\ \dot{\Phi}\end{pmatrix} , \begin{pmatrix} 0 \\ \eta \dot{\alpha} \end{pmatrix}\rangle & = \langle - \nabla_{\zeta}^A \dot{\alpha} , \dot{\alpha}\rangle + 4\xi^2 |\dot{\alpha}|^2 + \langle (d_{A}^H )^\ast d_{A}^H \dot{\alpha}, \dot{\alpha}\rangle + 2 \xi \ast \langle \ast_H d_{A}^H \dot{\Psi}, \dot{\alpha}\rangle \nonumber \\
& = \langle - \nabla_{\zeta}^A \dot{\alpha} , \dot{\alpha}\rangle  + \langle (d_{A}^H )^\ast d_{A}^H \dot{\alpha}, \dot{\alpha}\rangle  - 2\xi \langle \ast_H [\dot{B}, \Psi], \dot{\alpha}\rangle  \label{term1}
\end{align}
where in the last line we have used the third identity in (\ref{ZeqVH}). For the second term in (\ref{D2Z}), we use Lemma \ref{lemma:d_B} and the vanishing of $\nabla_{\zeta}^A \alpha , \nabla_{\zeta}^A \Psi, d_{A}^H \alpha $ and $(d_{A}^H )^\ast \Psi $ to obtain
\begin{align}
 \langle D_A &\begin{pmatrix} \ast [\dot{A}, \ast \Phi ] \\ \ast [\dot{A} , \Phi] \end{pmatrix}, \begin{pmatrix} 0 \\ \eta \dot{\alpha }\end{pmatrix}\rangle  \nonumber \\
& = \langle \, - [\nabla_{\zeta}^A \dot{a} , \alpha] - \ast_H [\nabla_{\zeta}^A \dot{B}, \ast_H \Psi] + \ast_H [d_{A}^H \dot{a} , \ast_H \Psi] + [ (d_{A}^H )^\ast \dot{B} , \alpha] + 2 \xi \ast_H [\dot{B},\Psi] \, , \, \dot{\alpha} \, \rangle \nonumber \\
& = \langle \, - [\ast_H [\dot{\Psi} , \ast_H \Psi],\alpha] - [[\dot{\alpha},\alpha],\alpha] - \ast_H [[\dot{\alpha},\Psi],\ast_H \Psi] + \ast_H [[\dot{\Psi},\alpha],\ast_H \Psi] + 2\xi \ast_H [\dot{B},\Psi] \, , \, \dot{\alpha} \, \rangle \nonumber \\
& = |[\dot{\alpha},\alpha]|^2 + |[\dot{\alpha}, \Psi]|^2 + 2\xi \langle \ast_H [\dot{B},\Psi],\dot{\alpha}\rangle \label{term2}
\end{align}
where to get the second to last line we used the first and last identities in (\ref{ZeqVH}), and to get the last line we took adjoints and used the vanishing of the terms
\[
 - [\ast_H [\dot{\Psi} , \ast_H \Psi],\alpha]   + \ast_H [[\dot{\Psi},\alpha],\ast_H \Psi]  =  \ast_H [[\Psi , \alpha], \ast_H \dot{\Psi}] = 0
\]
due to $[\Psi , \alpha] = 0$ (by Corollary \ref{corollary:vanishing}).

Altogether, combining (\ref{D2Z}-\ref{term2}), cancelling terms and integrating by parts gives the identity
\begin{align*}
0 = \| \nabla_{\zeta}^A \dot{\alpha}\|_{L^2 (Y)}^2 + \| d_{A}^H \dot{\alpha}\|_{L^2 (Y)}^2 + \| [\dot{\alpha}, \alpha]\|_{L^2 (Y)}^2 + \| [\dot{\alpha}, \Psi]\|_{L^2 (Y)}^2 ,
\end{align*}
and hence the vanishing of $\nabla_{\zeta}\dot{\alpha}, d_{A}^H \dot{\alpha},[\dot{\alpha},\alpha] , [\dot{\alpha} , \Psi]$. 

An analogous argument gives the identity
\begin{align*}
0 = \| \nabla_{\zeta}^A \dot{a}\|_{L^2 (Y)}^2 + \| d_{A}^H \dot{a}\|_{L^2 (Y)}^2 + \| [\dot{a}, \alpha]\|_{L^2 (Y)}^2 + \| [\dot{a}, \Psi]\|_{L^2 (Y)}^2 ,
\end{align*}
and hence the vanishing of $\nabla_{\zeta}\dot{a}, d_{A}^H \dot{a},[\dot{a},\alpha] , [\dot{a} , \Psi]$.

By the vanishing of all these quantities, the equations (\ref{Zeq}) now read:
\begin{align}
- (d_{A}^H )^\ast \dot{B} - \ast_H [ \dot{\Psi} , \ast_H \Psi ] & = 0 \label{ZeqVH2}\\
 \ast_H [ \dot{B}, \ast_H \Psi ] - (d_{A}^H )^\ast \dot{\Psi}& = 0 \nonumber\\
\ast_H [ \dot{B}, \Psi ] + \ast_H d_{A}^H \dot{\Psi} & = -2\xi \dot{\alpha} \nonumber\\
 \ast_H \nabla_{\zeta}^A \dot{\Psi}  - \ast_H [\dot{B},\alpha] & = 0  \nonumber\\
\ast_H d_{A}^H \dot{B} - \ast_H [\dot{\Psi}, \Psi] &= -2\xi \dot{a}\nonumber\\
 \ast_H \nabla_{\zeta}^A \dot{B}  + \ast_H [\dot{\Psi},\alpha] & = 0 \nonumber.
\end{align}

Taking $\nabla_{\zeta}^A$ of the fourth equation of (\ref{ZeqVH2}), using $\nabla_{\zeta}^A \alpha = 0$ (cf. Corollary \ref{corollary:vanishing}), then inserting the last equation of (\ref{ZeqVH2}), and finally integrating by parts yields
\[
\| \nabla_{\zeta}^A \dot{\Psi} \|_{L^2 (Y)}^2 + \| [\dot{\Psi }, \alpha ]\|_{L^2 (Y)}^2 = 0 .
\]
This says the quantities $\nabla_{\zeta}^A \dot{\Psi}, [\dot{\Psi} , \alpha ]$ vanish. An analogous argument gives the vanishing of $\nabla_{\zeta}^A \dot{B} , [\dot{B} , \alpha]$. Finally, observe that $\xi \dot{\alpha} = 0$, which follows easily by taking the inner product of the third equation in (\ref{ZeqVH2}) with $\dot{\alpha}$, integrating by parts and using the vanishing of $[\Psi , \dot{\alpha}] , d_{A}^H \dot{\alpha}$. Likewise, $\xi \dot{a} = 0$. 

At this point, we have established all but the final assertion in Proposition \ref{prop:vanishing_inf}. The required conclusion for $\dot{\alpha}$ readily follows from the vanishing of $\xi \dot{\alpha} , [\Psi , \dot{\alpha}] , [\alpha , \dot{\alpha}], d_{A}\dot{\alpha}$, and similarly for $\dot{a}$. This concludes the proof.
\end{proof}

\subsection{Dimensional reduction}

We now use the vanishing results from the previous section (Corollary \ref{corollary:vanishing} and Proposition \ref{prop:vanishing_inf}) to describe the moduli spaces of stable flat connections on $Y$ in terms of stable Higgs pairs on the orbifold base $C$.

\subsubsection{Stable flat connections and Higgs pairs on $C$}

Let $E_0 \to C$ be an orbifold $U(2)$-bundle with fixed determinant $(\Lambda_0 = \Lambda^2 E_0 , \lambda_0 )$ over the closed oriented Riemannian $2$-orbifold $(C, g_C )$, where $\lambda_0$ is a fixed $U(1)$ connection on $\Lambda_0$. We now gather some basic results about the moduli space of stable projectively flat connections  in the bundle $E_0$.

The discussion from \S \ref{subsection:orbitspaces} applies also to the Riemannian $2$-orbifold $(C, g_C )$, and we obtain the spaces $\mathscr{B}_{E_0 , k }(C), \mathscr{B}_{E_0 , k }^\ast (C)$ of polystable or stable orbits of $GL(2,\mathbb{C} )$ connections $B + i \Psi$ with fixed determinant $\lambda_0$ and regularity $L^{2}_{k}$ under the action of the group $\mathscr{G}_{E_0 , k+1}(C)$ of $L^{2}_{k+1}$ gauge transformations. As before, `polystable' here means $d_{B}^\ast \Psi = 0$, and `stable' means that, in addition, the stabilizer of $B + i \Psi$ under the $\mathscr{G}_{E_0 , k+1}$ action is $\pm I$. Choosing any $k\geq 2$ will suffice for our purposes, and $\mathscr{B}_{E_0 , k} (C)$ then becomes a Hilbert Kähler manifold when equipped with the complex structure $J$ and symplectic form $\omega_J$ given in (\ref{J}-\ref{omegaJ}).

The moduli space $\mathscr{M}_{E_0}^\ast (C,g_C )$ of polystable projectively flat $GL(2,\mathbb{C})$ orbifold connections on $E_0 \to C$ with fixed determinant $(\Lambda_0 , \lambda_0 )$ is given by the $\mathscr{G}_{E_0}$--orbits of configurations $\mathbb{B} = B + i \Phi $ solving 
\begin{align}
 F_B - \frac{I}{2}F_{\lambda_0} -  \frac{1}{2} [\Psi , \Psi ] &= 0 \label{eq1C}\\
 d_{B} \Psi & = 0 \label{eq2C}\\ 
d_{B}^\ast \Psi & =  0 \label{eq3C}.
\end{align}
The moduli space $\mathscr{M}_{E_0}^\ast (C) \subset \mathscr{B}_{E_0}^\ast (C) $ arises as zero locus of a holomorphic Fredholm section, denoted $s$, of a holomorphic Hilbert vector bundle, denoted $\mathscr{V}_{k-1} \rightarrow \mathscr{B}_{E_0}^\ast (C)$. Namely, $\mathscr{V}_{k-1}$ is the bundle induced from the $\mathscr{G}_{E_0 , k+1}$--equivariant vector bundle over the stable configuration space $\mathscr{C}_{E_0 , k}^\ast (C)$ with fiber $\Omega^2 (C, \mathfrak{g}_E ) \oplus \Omega^2 (C , \mathfrak{g}_E )$ at every point, and $s$ is induced from the equivariant section given the left-hand side of the equations (\ref{eq1C}) and (\ref{eq2C}).

As a matter of general principle, then the moduli space $\mathscr{M}_{E_0}^\ast$ is naturally a complex-analytic space with complex structure $J$, with local charts provided by finite-dimensional Kuranishi models. In fact, more is true in this case:

\begin{proposition}\label{proposition:transversalityC}
The moduli space $\mathscr{M}_{E_0 }^\ast (C) = s^{-1}(0)$ of stable projectively flat $GL(2, \mathbb{C} )$ connections on a closed oriented Riemannian $2$-orbifold $(C, g_C )$ is transversely cut-out; hence is naturally a complex manifold with the complex structure $J$.
\end{proposition}

For the formula for the dimension of $\mathscr{M}_{E_0}^\ast (C)$ in terms of the isotropy data of $E_0$ we refer to \cite{nasatyr-steer,boden-yokogawa}. The proof of Proposition \ref{proposition:transversalityC} appears in \cite[\S 5]{hitchin} in a slightly different guise. We will briefly outline a version of this argument, for completeness. 

\begin{proof}[Proof of Proposition \ref{proposition:transversalityC}]
At a stable flat connection $(B, \Psi )$ the vertical derivative of $s$ is the operator
\begin{align*}
& (\mathcal{D} s)_{(B, \Psi )} :  \mathscr{K}_{(B, \Psi )} = \mathrm{Ker}(\delta^{1}_{(B, \Psi )})^\ast \cap \mathrm{Ker}(\delta^{2}_{(B, \Psi )})^\ast \subset (\Omega^1 \oplus \Omega^1 )(C, \mathfrak{g}_E ) \to (\Omega^2 \oplus \Omega^2 )(C, \mathfrak{g}_E ) \\
& (\mathcal{D} s)_{(B, \Psi )} (\dot{B} , \dot{\Psi} )= ( d_{B}\dot{B} - [\dot{\Psi},\Psi] , [\dot{B}, \Psi] + d_{B}\dot{\Psi})
\end{align*}
where $\delta^1$ and $\delta^2$ are the corresponding versions on $(C,g_C )$ of the operators defined in \S\ref{subsubsection:moment}. Thus, an element $(\dot{F} , \dot{G} )$ which is $L^2$--orthogonal to the image of $\mathcal{D}s$ satisfies the following identity coming from taking adjoints: for any $(\cdot{B}, \cdot{\Psi}) \in \mathscr{K}$ 
\begin{align}
0 = \langle \, ( d_{B}^\ast \dot{F} - \ast[\Psi , \ast \dot{G}] , \ast [\Psi , \ast \dot{F}] + d_{B}^\ast \dot{G} ) \, , \, (\dot{B} , \dot{\Psi} ) \, \rangle_{L^2 (C)} \label{L2perp}
\end{align}
One can then show that the element $( d_{B}^\ast \dot{F} - \ast[\Psi , \ast \dot{G}] , \ast [\Psi , \ast \dot{F}] + d_{B}^\ast \dot{G})$ lies in the kernel of $(\delta^1)^\ast$ and $(\delta^2 )^\ast$ when $(B, \Psi )$ is flat, and hence (\ref{hodge1}) and (\ref{L2perp}) imply that 
\[
d_{B}^\ast \dot{F} - \ast[\Psi , \ast \dot{G}] = 0 \quad , \quad \ast [\Psi , \ast \dot{F}] + d_{B}^\ast \dot{G}  = 0 .
\]
Taking $d_B$ on both of the two identities, then integrating the first against $\dot{F}$ and the second against $\dot{G}$, and finally adding everything up leads to $0 = \| d_{B}^\ast \dot{F} \|_{L^2 (C) } + \| d_{B}^\ast \dot{G} \|_{L^2 (C) }$, and hence $\dot{F} = \dot{G} = 0$ as required.
\end{proof}

It is well-known that when $\mathrm{dim}C = 2$ then $\mathscr{M}^{\ast}_{E_0}(C,g_C )$ is naturally a \textit{hyperKähler} manifold and admits a holomophic description as a moduli space of \textit{stable Higgs pairs}. We briefly recall this now (see \cite{hitchin,nasatyr-steer,boden-yokogawa} for details). 

The Hodge star operator of the metric $g_C$ takes $1$-forms to $1$-forms, hence gives an additional complex structure, denoted $I$, on $\mathscr{A}_{E_0 , k}^c $ by $I (\dot{B} , \dot{\Psi } ) = (\ast_C \dot{B} , \ast_C \dot{\Psi} )$. The complex structures $I , J$ anti-commute, and hence we have a third complex structure $K = IJ$ (altogether, $I,J,K$ satisfy the quaternion relations). All three complex structures are compatible with the flat Riemannian metric $G$ on $\mathscr{A}_{E_0 , k}^c$ obtained by the $L^2$ inner product, and we obtain three corresponding Kähler (symplectic) forms $\omega_I , \omega_J , \omega_K$. For example, $\omega_J$ was given in (\ref{omegaJ}) and $\omega_I$ is given by
\begin{align}
\omega_I \big( (\dot{B}_1 , \dot{\Psi}_1 ) , (\dot{B}_2 , \dot{\Psi}_2 ) \big) = G \big( I (\dot{B}_1 , \dot{\Psi}_1 ) , (\dot{B}_2 , \dot{\Psi}_2 ) \big)  =  - \int_C \mathrm{Tr}\big(  \dot{B}_1 \wedge \dot{B}_2 + \dot{\Psi}_1 \wedge  \dot{\Psi}_2 \big) . \label{omegaI}
\end{align}
Now, the $\mathscr{G}_{E_0 , k+1}$ action is Hamiltonian for each of the three symplectic forms. With respect to $\omega_J$, the moment map is given by the left-hand side of (\ref{eq3C}), as we saw also in the $3$-dimensional case (this holds in any dimension). The moment maps for $\omega_I$ and $\omega_K$ can be identified with the left-hand sides in (\ref{eq1C}) and (\ref{eq2C}), respectively. Thus, $\mathscr{M}_{E_0}^\ast (C)$ arises as an infinite-dimensional \textit{hyperKähler reduction} of $(\mathscr{A}_{E_0 , k}^c , G , I , J , K )$ by the $\mathscr{G}_{E_0 , k +1}$--action.

In turn, the complex gauge group $\mathscr{G}_{E_0 , k+1}^c$ preserves the structures $I$ and $J$. By infinite-dimensional versions of the Kempf--Ness Theorem due to Hitchin \cite{hitchin} and Donaldson \cite{donaldson-harmonic} (see \cite{nasatyr-steer,boden-yokogawa} for the orbifold case) the moduli space $\mathscr{M}_{E_0}^{\ast}(C,g_C )$ then admits a holomorphic interpretation in the complex structures $I$ and $J$, namely:
\begin{itemize}

\item With the complex structure $I$, $\mathscr{M}_{E_0}^{\ast}(C,g_C ) $ is identified with the moduli space of \textit{stable Higgs pairs} on the orbifold bundle $E_0 \to C$ with fixed determinant \cite{hitchin,nasatyr-steer,boden-yokogawa}:
 \[
 ( \mathscr{M}_{E_0}^{\ast}(C,g_C ) , I ) \cong \mathscr{M}_{E_0}^{\mathrm{Higgs}, \ast } (C).
 \] 
We recall the definition of $ \mathscr{M}_{E_0}^{\mathrm{Higgs}, \ast } (C)$. For this, we use the orbifold complex structure on $C$ defined by $g_C$, and the orbifold holomorphic structure on $\Lambda_0$ induced by $\lambda_0$. Then $ \mathscr{M}_{E_0}^{\mathrm{Higgs}, \ast } (C)$ is given by the isomorphism classes of pairs $(\mathscr{E} , \theta )$ where $\mathscr{E}$ is an orbifold holomorphic structure $\mathscr{E}$ on $E_0$ with fixed determinant and $\theta \in H^0 ( C , T^\ast C \otimes \mathfrak{sl}(\mathscr{E} ))$, subject to the requirement that $(\mathscr{E} , \theta )$ is \textit{stable}: if $\mathscr{L} \subset \mathscr{E}$ is a $\theta$--invariant proper non-trivial holomorphic subbundle, then 
\[
\mathrm{deg}\mathscr{L} < \frac{1}{2}\mathrm{deg}\mathscr{E}.
\] 
The moduli space $ \mathscr{M}_{E_0}^{\mathrm{Higgs}, \ast } (C)$ can be equivalently described as a moduli space of \textit{stable parabolic Higgs bundles} on the desingularised Riemann surface $|C|$ (i.e. the underlying coarse space of $C$, which is smooth) with parabolic weights given by the isotropy data of $E_0$ (this is the more common terminology in the literature, see \cite{nasatyr-steer,boden-yokogawa}). The correspondence between stable (projectively) flat connections and stable Higgs pairs is given by:
\[
\mathbb{B} = B + i  \Psi  \mapsto (\mathscr{E} , \theta ) = ( \overline{\partial}_B , - i \Psi^{1,0} ) .
\]

\item (For simplicity, assuming $(\Lambda_0 , \lambda_0 )$ is trivial) With the complex structure $J$, $\mathscr{M}_{E_0}^{\ast}(C,g_C ) $ is identified with the smooth quasi-affine variety consisting of conjugacy classes of irreducible representations of the orbifold fundamental group of $C$ in $SL(2, \mathbb{C})$, suitably compatible with the isotropy data of $E_0$, \cite{donaldson-harmonic,nasatyr-steer,boden-yokogawa}:
\[
( \mathscr{M}_{E_0}^{\ast}(C,g_C ) , J ) \cong \mathrm{Hom}_{E_0}^\ast ( \pi_{1}^{\mathrm{orb}} (C) , SL(2, \mathbb{C} ) )/SL(2, \mathbb{C} ) .
\] 
\end{itemize}

%\[
%\omega_I =  G (I \cdot , \cdot )  \quad, \quad  \omega_I =  G (I \cdot , \cdot ) \quad , \quad \omega_I =  G (I \cdot , \cdot ) .
%\]
%The section $s$ of $\mathscr{V}_{k-1} \rightarrow \mathscr{B}_{E_0 , k}^\ast $ is indeed Fredholm, since the following complex of differential operators is elliptic when $\mathrm{dim}C = 2$:
%\begin{align*}
%\Omega^0 (C , \mathfrak{g}_{E_0} ) \xrightarrow{\delta^{1}_{(B , \Psi )}} (\Omega^1 \oplus \Omega^1 ) (C , \mathfrak{g}_{E_0} ) \xrightarrow{(ds)_{(B, \Psi )}+ (\delta^{2}_{(B, \Psi )})^\ast} (\Omega^2 \oplus \Omega^2 \oplus \Omega^0 )(C , \mathfrak{g}_{E_0} )
%\end{align*}

%Thus, as a matter of general principle, the moduli space $\mathscr{M}_{E_0}^\ast$ is naturally a complex-analytic space, with local charts provided by finite-dimensional Kuranishi models. In fact, we have much more. Hitchin proved\cite{hitchin} (when $C$ is a Riemann surface, but with identical proof in the orbifold case) that the cohomology at the third step in the complex (\ref{deformationcomplex_C}) vanishes when $A+i\Phi$ is polystable and projectively flat (and the cohomology vanishes at the . Hence $\mathscr{M}_{E_0 }^\ast$
%\begin{proposition}[Hitchin] 
%\end{proposition}

\subsubsection{From stable flat connections on $Y$ to stable flat connections on $C$}

Consider a $U(2)$ bundle $E $ over a Seifert-fibered $3$-manifold $ Y$ with fixed determinant $\Lambda$. We now describe the moduli space $\mathscr{M}_{E}^\ast (Y)$ in terms of the moduli spaces on $C$.
%according to the following two cases:  
%\begin{enumerate}
%\item[\textit{Case 1:} ] $E$ pulls back from $C$, up to tensoring by line bundles. Equivalently, $\Lambda$ pulls back from $C$, up to tensoring by squared line bundles.
%\item[\textit{Case 2:} ] $E$ does not pull back from $C$, up to tensoring by line bundles. Equivalently, $\Lambda$ does not pull back from $C$, up to tensoring by squared line bundles.
%\end{enumerate}

If $E$ satisfies any of the three equivalent conditions from the next Lemma, we will say that $E$ `pulls back from $C$ up to tensoring':
\begin{lemma}\label{lemma:pullbacksquares}
Let $E \to Y$ be a $U(2)$ bundle over a Seifert-fibered $3$-manifold $Y = S(N) \to C$, with determinant $\Lambda = \Lambda^2 E$. The following are equivalent:
\begin{enumerate}
\item $E$ is the pullback of an orbifold $U(2)$ bundle on $C$, up to tensoring by some $U(1)$ bundle $L$ on $Y$
\item $\Lambda$ is the pullback of an orbifold $U(1)$ bundle on $C$, up to tensoring by some $U(1)$ bundle squared $L^{\otimes 2}$ on $Y$
\item The associated $SO(3)$ bundle $P = E \times_{U(2)}SO(3)$ is the pullback of an orbifold $SO(3)$ bundle on $C$.
\end{enumerate}
\end{lemma}
\begin{proof}
Since $Y$ is $3$-dimensional, then $E \cong \mathbb{C} \oplus \Lambda$. From this the equivalence of (1) and (2) follows. Clearly, (1) implies (3). To show that (3) implies (1), suppose $P$ is the pullback of an orbifold $SO(3)$ bundle $P_0$. One can then find an orbifold $U(2)$ bundle $E_0$ on $C$ whose associated orbifold $SO(3)$ bundle is $P_0 \cong E_0 \times_{U(2)} SO(3)$ (see \cite[\S 1]{furutasteer}). It follows that $P = \pi^\ast P_0$ is the associated bundle of both $E$ and $\pi^\ast E_0$, so they differ by tensoring by a line bundle.
\end{proof}

%\begin{lemma} \label{lemma:pullback}Let $Y = S(N) \to C$ be a Seifert-fibered $3$-manifold. Let $E \to Y$ be a principal bundle with structure group $U(1)$ or $U(2)$. Then $E$ is the pullback of an orbifold principal bundle $E_0 \to C$ if and only if $c_1 (E)$ is a torsion class in $H^2 (Y, \mathbb{Z} )$.
%\end{lemma}
%\begin{proof}
%For structure group $U(1)$ this follows from \cite[Corollary 1.6 and Theorem 2.3]{furuta-steer}. When $E$ has structure group $U(2)$ we have $E \cong \mathbb{C} \oplus \Lambda$ where $\Lambda = \Lambda^2 E$, hence the result follows from the $U(1)$ case. %We consider the case of structure group $U(1)$ first, and denote $P = L$. Let $\mathrm{Pic}^t (C)$ stand for the `Picard group' of topological isomorphism classes of orbifold $U(1)$ bundles on $C$. By \cite[Theorem 2.3]{furuta-steer} there is an exact Gysin sequence:
%\[
%0 \to \mathbb{Z} \xrightarrow{1 \mapsto N} \mathrm{Pic}^t (C)  \to H^{2} (Y , \mathbb{Z} ) \to H^{1} (|C| , \mathbb{Z} )\cong \mathbb{Z}^g  \to 0 
%\]
%where $|C|$ is the `desingularised' genus $g$ surface underlying the $2$-orbifold $C$. By \cite[Corollary 1.6]{furuta-steer} the 
%A $U(2)$ bundle $E \to Y$ is the pullback of an orbifold bundle $E_0 \to C$ precisely when its determinant $\Lambda = \Lambda^2 E$ is the pullback of an orbifold line bundle 
%\end{proof}

We refer to \cite[Corollary 1.6, Proposition 1.8]{furutasteer} for the classification of orbifold $U(1)$ and $U(2)$ bundles over $C$, and to \cite[Theorem 2.3]{furutasteer} for the classification of the $U(1)$ and $U(2)$ bundles on $Y$ which pull back from $C$. Given an orbifold $U(2)$ bundle $E_0 \to C$ with fixed determinant, and setting $E = \pi^\ast E_0  $, we have a well-defined holomorphic map on stable orbit spaces given by pulling back configurations from $C$ onto $Y$: 
\[
\pi^\ast : \mathscr{B}_{E_0 , k}^\ast (C) \to \mathscr{B}_{E , k}^\ast (Y) .
\]
Furthermore, one easily sees that
\[
\pi^\ast \big( \mathscr{M}_{E_0 }^\ast (C) \big) \subset \mathscr{M}_{E }^\ast (Y) .
\]
%The following result shows that every stable flat connection over $Y$ is obtained like this:

\begin{theorem}[Dimensional Reduction for Seifert-fibered spaces]\label{theorem:dimensionalreduction}
Let $Y \rightarrow C$ be a closed, oriented Seifert-fibered $3$-manifold equipped with a Seifert metric $g = \eta^2 + g_C$ (Definition \ref{metric}), and let $E$ be a $U(2)$-bundle $E \rightarrow Y$ with fixed determinant $\Lambda = \Lambda^2 E$. Equip $C$ with the metric $g_C$.

 If $E$ does not pull back from $C$ up to tensoring, then the moduli space $\mathscr{M}_{E}^\ast (Y,g)$ is empty. If it does, then after tensoring $E$ by a $U(1)$ bundle (which leaves $\mathscr{M}_{E}^\ast (Y,g)$ unchanged, cf. Remark \ref{remark:tensoring}) we may fix an isomorphism $\Lambda \cong \pi^\ast \Lambda_0$ for some orbifold line bundle $\Lambda_0$ on $C$. Then $\mathscr{M}_{E}^\ast (Y,g)$ is Morse--Bott (cf. Definition \ref{definition:morse-bott}), and the pullback map $\pi^\ast$ induces a biholomorphism
 
\[
\Big( \bigsqcup_{E_0 \, , \, \Lambda^2 E_0 \cong \Lambda_0}  \mathscr{M}_{E_0}^\ast (C,g_C) \Big)  \sqcup \Big(  \bigsqcup_{E_0 \, , \, \Lambda^2 E_0 \cong \Lambda_0 \otimes N}\mathscr{M}_{E_0}^\ast (C  , g_C) \Big) \cong \mathscr{M}_{E}^\ast (Y,g) 
\]
where $E_0$ ranges over all isomorphism classes of orbifold $U(2)$ bundles on $C$ with fixed determinant $\Lambda_0$ and $\Lambda_0 \otimes N$, respectively. 

\end{theorem}

To establish this result, we combine the vanishing results established above together with the following standard fact, which essentially appears in \cite[Proof of Theorem 4.1]{furutasteer}:

\begin{lemma}\label{lemma:pullback}
Let $P \rightarrow Y$ be a principal $G$-bundle with connection $A$, where $G$ is a Lie group with Lie algebra $\mathfrak{g}$. Let $X$ denote the vector field on the total space $P$ given by the horizontal lift with respect to $A$ of the vertical vector field $\zeta$ on $Y$ that generates the Seifert $S^1$-action on $Y$. Then the pair $(P , A )$ is the pullback of a corresponding pair $(P_0 , B )$ on $C$ if and only if the holonomy of $A$ is trivial along the fibers of $Y \rightarrow C$ and the contraction with $X$ of the curvature $2$-form $F_A \in \Omega^2 (P , \mathfrak{g})$ vanishes: $\iota_{X} F_A = 0$.
\end{lemma}

We recall the proof, which is quite simple:

\begin{proof}
Only the `if' statement is non-trivial. The vector field $X$ on $P$ generates an $S^1$-action on $P$ because the holonomy of $A$ is trivial on the fibers. This action makes $P$ into the total space of an orbifold principal $S^1$-bundle $P \rightarrow P_0 := Q/S^1$. Since the $S^1$-action on $P$ commutes with the $G$-action then $P_0 = P/S^1$ inherits the structure of an orbifold principal $G$-bundle over $C$, and the projection $P \rightarrow P_0$ is $G$-equivariant.

To argue that $A$ descends to a connection on $P_0$, regard $A$ as a $1$-form (i.e. as an element in $\Omega^1 (P , \mathfrak{g})$). Then $\iota_X A = 0$ since $X$ is horizontal. Since $\iota_X F_A = 0$ then
\begin{align*}
\mathscr{L}_X A = d(\iota_X A ) + \iota_X (d A ) = \iota_X (F_B ) - \iota_X (\frac{1}{2}[ A ,  A] ) = 0.
\end{align*}
The conditions $\iota_X A = 0$ and $\mathscr{L}_X A = 0$ combined say that $A$ descends to a connection $B$ on the orbifold principal $G$-bundle $P_0 \rightarrow C$.
\end{proof}

\begin{proof}[Proof of Theorem \ref{theorem:dimensionalreduction}]
Let $(A,\Phi )$ be a projectively flat stable $GL(2,\mathbb{C})$-connection on $E \rightarrow Y$ with fixed determinant, and consider the induced flat stable $GL(2,\mathbb{C})$ connection $(A^\prime , \Phi^\prime )$ on the associated $SO(3)$ bundle $P \rightarrow Y$. By the vanishing of the $\alpha$ component of $\Phi$ (Corollary \ref{corollary:vanishing}), it follows that the holonomy of $A^\prime$ along a fiber of $\pi : Y \rightarrow C$ is the same as the holonomy of the flat connection $A^\prime + i \Phi^\prime$. Since the fibers represent central elements in the fundamental group of $Y$ \cite[Theorem 2.1]{furutasteer} and $A^\prime + i \Phi^\prime$ is flat, then it follows that the holonomy of $A^\prime$ lies in the center of $SO(3)$, i.e. it must be trivial. Also, because $[\alpha , \Psi ] = 0$ by Corollary \ref{corollary:vanishing}, then $\iota_\zeta F_{A^\prime} = 0$. Thus Lemma \ref{lemma:pullback} implies that the pair $(P , A^\prime )$ pulls back from a corresponding pair $(P_0 , B^\prime )$. 

By Lemma \ref{lemma:pullbacksquares} it follows that $\Lambda$ pulls back from $C$, up to tensoring $E$ by a line bundle. Without loss of generality, we may assume there is an orbifold line bundle $\Lambda_0$ on $C$, which is fixed once and for all, together with an identification $\Lambda = \pi^\ast \Lambda_0$. We also fix a connection $\lambda$ on $\Lambda = \pi^\ast \lambda_0$ of the form $\lambda = \pi^\ast \lambda_0$, for a connection $\lambda_0$ on $\Lambda_0$. Now, by running again the argument from the first paragraph but applied to $(A, \Phi )$ rather than $(A^\prime , \Phi^\prime )$ we now obtain that the holonomy of $A$ along the non-singular fibers lies in $\pm I$, the center of $SL(2, \mathbb{C} )$. Let $k \in \{ 0 ,1 \}$ be such that the holonomy of a non-singular fiber is $(-1)^k I$.
Consider the $U(2)$ bundle with connection given by $(E, A_k ) = (E, A) \otimes (\mathbb{C} , A_{trivial} + \frac{1}{2} i \eta )^{\otimes k}$ (this is identified with $E$ as a bundle). Then $(E , A_k )$ has trivial holonomy along a non-singular fibers and its induced connection on $\Lambda = \Lambda^2 E $ is $\lambda + i \eta$ instead now. For the same reason as in the first paragraph, we also have $\iota_\zeta F_{A_k} = 0$.
%\textcolor{red}{Note also that the connection on $\mathfrak{g}_E$ induced by $B_k$ is the same as that induced by $B$, so the same equations hold...} Regarding $E$ as a principal bundle, and the curvature of $(E, B_k )$ as a two-form on $E$ valued in $\mathfrak{g}$, we have that 
%\[
%\iota_X F_B = \iota_X (\frac{I}{2}\pi^\ast F_{B_{D_0}} + \frac{I}{2}F_{N} + \frac{1}{2} [\phi_H , \phi_H]) = 0 .
%\]
Thus, by Lemma \ref{lemma:pullback} we have that $(E,A_k )$ descends to a pair $(E_0 , B )$ on $C$ whose determinant is $(\Lambda_0 , \lambda_0 ) \otimes (N , i \eta )^{\otimes k}$. It follows from Corollary \ref{corollary:vanishing} that then $\Phi$ descends to a section $\Psi$ of $\Omega^1 (C , \mathfrak{su}(E_0 ) )$ and that $(B, \Psi ) $ is a polystable projectively flat $GL(2,\mathbb{C})$ connection on $E_0 \to C$, with fixed determinant $(\Lambda_0 , \lambda_0 ) \otimes (N , i \eta )^{\otimes k }$. Furthermore, since $(A, \Phi )$ is stable then it follows that so is $(B , \Psi )$.

Altogether, the construction thus described provides a well-defined map
\[
 \mathscr{M}_{E}^\ast (Y,g) \to  \Big( \bigsqcup_{E_0 \, , \, \Lambda^2 E_0 \cong \Lambda_0}  \mathscr{M}_{E_0}^\ast (C,g_C) \Big)  \sqcup \Big(  \bigsqcup_{E_0 \, , \, \Lambda^2 E_0 \cong \Lambda_0 \otimes N}\mathscr{M}_{E_0}^\ast (C  , g_C) \Big)  =:\mathscr{M}_{E}(C,g_C ) \]
which gives a topological inverse of the pullback map $\pi^\ast$.

Now, the map 
\begin{align}
\pi^\ast : \mathscr{B}_{E, k}^\ast (C , g_C ) := \Big( \bigsqcup_{E_0 \, , \, \Lambda^2 E_0 \cong \Lambda_0}  \mathscr{B}_{E_0 , k }^\ast (C,g_C) \Big)  \sqcup \Big(  \bigsqcup_{E_0 \, , \, \Lambda^2 E_0 \cong \Lambda_0 \otimes N}\mathscr{B}_{E_0 , k }^\ast (C  , g_C) \Big)   \to \mathscr{B}_{E, k}^\ast (Y, g)
\end{align}
is a holomorphic map of complex Hilbert manifolds, mapping the complex submanifold $\mathscr{M}^\ast_E (C , g_C)$ bijectively onto the subspace $\mathscr{M}_E (Y, g ) \subset \mathscr{B}_{E, k }^\ast (Y, g) $. By Proposition \ref{prop:vanishing_inf}, $\pi^\ast$ also induces an isomorphism from the tangent space of $\mathscr{M}_{E}^\ast (C , g_C )$ onto the Zariski tangent space of $\mathscr{M}_{E}^\ast (Y , g )$. This implies that the finite-dimensional Kuranishi models of $S$ from Corollary \ref{corollary:kuranishi} are given by constant functions--hence $\mathscr{M}^\ast_{E} (Y,g)$ is a smooth (in fact, complex) submanifold of $\mathscr{B}_{E}^\ast (Y,g)$, and $\pi^\ast$ gives a complex-analytic diffeomorphism from $\mathscr{M}_{E}^\ast (C , g_C )$ onto $\mathscr{M}_{E}^\ast (Y, g )$. 
\end{proof}

Theorem \ref{theorem:reductionHiggs} follows immediately from Theorem \ref{theorem:dimensionalreduction} and the holomorphic description of the moduli spaces $\mathscr{M}_{E_0}^\ast (C)$ in terms of stable Higgs pairs.

\begin{example}
The simplest case is to take $Y = S^1 \times C$ for a (smooth) genus $g$ closed surface $C$, and $E = \pi^\ast E_0$ given by the pullback of a rank $2$ bundle $E_0$ of \textit{odd} degree, where $Y$ is equipped with a product metric $g = dt^2 + g_C$. Then polystable projectively flat connections on $E$ and $E_0$ are automatically stable (which follows from equations (\ref{eq1}), (\ref{eq1C}) and the Chern--Weil formula), and in order to find solutions we need $g\geq 2$. The moduli space of polystable connections on $E_0$ is a complex manifold of complex dimension $6(g-1)$ \cite{hitchin}, and Theorem \ref{theorem:dimensionalreduction} gives:
\[
\mathscr{M}_E (Y,g) \cong \mathscr{M}_{E_0} (C , g_C) \sqcup \mathscr{M}_{E_0 } (C , g_C) .
\]

\end{example}

\section{Perturbation theory}\label{section:perturbations}

In this section we study the critical locus of the perturbation $S_\varepsilon$ given in Definition \ref{perturbation:intro} (the angle $\theta$ will be omitted from notation, for convenience). We first discuss a finite-dimensional version of the argument that we shall employ in the infinite-dimensional case, and establish a finite-dimensional analogue of Theorem \ref{theorem:localisationseifert}. We then discuss the $SL(2, \mathbb{C})$ holonomy perturbations that we will use to achieve transversality of our perturbation. Finally, after verifying that the linear analysis results from \S \ref{section:connections} and basic compactness properties carry through in the perturbed case, we establish Theorem \ref{theorem:localisationseifert}.

%of the following sort:
%\begin{align*}
%S(A, \Phi ) + \varepsilon \|\Phi\|_{L^2}^2 + \varepsilon^2 f (\underline{\sigma}(A, \Phi )).
%\end{align*}
We shall see that when $(Y, g )$ is a Seifert-fibered $3$-manifold equipped with a Seifert metric, the effect of the perturbation $S_\varepsilon$ in this manner is the following, roughly speaking:
\begin{itemize}
\item The term $ \varepsilon \|\Phi\|_{L^2}^2$ plays the following role. As $\varepsilon \neq 0$ becomes small, the irreducible critical points of $S(A, \Phi ) + \varepsilon \|\Phi\|_{L^2}^2 $ will either tend to \textit{localise} at a certain \textit{compact} submanifold $\mathscr{Z} \subset \mathscr{M}_{E}^\ast = \mathrm{Crit}S$, or else the function $ \|\Phi\|_{L^\infty}$ becomes unbounded on them.
\item The term $\varepsilon^2 f (\underline{\sigma}(A, \Phi ))$ is then added in to ensure the \textit{non-degeneracy} of the critical points that localise at $\mathscr{Z}$. %Here $\underline{\sigma} : \mathscr{B}_{E}^\ast \rightarrow \mathbb{C}^N$ will be a suitable function constructed using $SL(2,\mathbb{C})$ \textit{holonomy perturbations}, and with the property that it induces a smooth embedding of $\mathscr{Z}$ into $\mathbb{C}^N$. Then $f: \mathbb{C}^N \rightarrow \mathbb{R}$ is a smooth function whose restriction to $\mathscr{Z}$ is suitably generic.
\end{itemize}
\subsection{Pertubations of Morse--Bott functions}

We consider the following setup. Let $B$ be a smooth manifold and consider a sequence of smooth functions $S_j : B \rightarrow \mathbb{R}$ for $j \geq 0 $. We regard $S_0$ as our `original' function, and we consider a \textit{perturbation} of $S_0$ of the form
%We suppose that at each point $x \in B$ there exist a neighborhood $\mathscr{N}_x$ of $x$ in $B$ and a constant $\varepsilon_{x} > 0$ such that the power series of functions
\\
\begin{equation}
S_\varepsilon := S_0 + \varepsilon S_1  + \varepsilon^2 S_2  + \varepsilon^3 S_3  + \cdots  \label{powerseries}
\end{equation}
\\
Summation in (\ref{powerseries}) is understood in a formal sense, without making any convergence assumption. What we mean is simply that we have a smooth function $S : B \times \mathbb{R} \to \mathbb{R}$, which gives $S_\varepsilon = S (\cdot , \varepsilon )$ and the functions $S_j$ by
\[
(\frac{\partial}{\partial \varepsilon })^j  S|_{\varepsilon = 0} = \frac{1}{j!} S_j .
\]
%converges on compact subsets of $\mathscr{N}_x$ with all derivatives, for $| \varepsilon| < \varepsilon_{x}$. We shall refer to such a function on $B \times \mathbb{R}_{\varepsilon}$ as being `analytic in $\varepsilon$'. (However, the discussion that follows does not require this analyticity hypothesis in any essential way, as one can easily observe). 

%We regard $S_\varepsilon$ as a \textit{perturbation} of the function $S_0$, and 

We want to describe the \textit{change in the critical locus} $\mathrm{Crit}(S_\varepsilon )$ as $\varepsilon$ becomes non-zero. More precisely, for each critical point $x \in \mathrm{Crit}S_0$ we will obtain an open neighborhood $\mathscr{U}_x$ of $x$ in $B$, a constant $\varepsilon_x >0$ and a parametrisation of the critical locus of $S_\varepsilon$ on the neighborhood $\mathscr{U}_0$ for all $0 < |\varepsilon| < \varepsilon_x$. For this, we will consider the following restricted assumptions, which suffice for our purposes:
\begin{enumerate}
\item $S_0$ is a Morse--Bott function; and denote its critical locus $\mathscr{Z}_0 := \mathrm{Crit}(S_0 )$
\item The restriction of $S_1$ to $\mathscr{Z}_0$ is a Morse--Bott function; and denote its critical locus $\mathscr{Z}_1 := \mathrm{Crit}(S_{1}|_{\mathscr{Z}_0} )$.
\end{enumerate}

\subsubsection{Local theory}\label{subsubsection:localtheory}

%restricted scenario, which shall suffice for our purposes. We suppose that:
%\\
%\begin{center}
%$S^0$ \textit{is a Morse--Bott function on} %$B$.\\
%~\\
%\end{center}

%Denote the critical locus of $S^0$ by $\mathscr{Z}^0 = \mathrm{Crit}(S^0 )$. Also, denote the critical locus of the restriction $S^{1}|_{\mathscr{Z}^0} : \mathscr{Z}^0 \to \mathbb{R}$ of $S^1$ to $\mathscr{Z}^0$ by $\mathscr{Z}^1 = \mathrm{Crit}(S^{1}|_{\mathscr{Z}^0} )$. 

We begin by noting that the critical points of $S_\varepsilon$ are constrained in the following way:

\begin{lemma}\label{lemma:convergence}
Let $\varepsilon_n \in \mathbb{R}$ be a sequence converging to zero, with $\varepsilon_n \neq 0$ for all $n$. Suppose that $x_n \in B$ is a sequence converging to $x \in B$ such that $b_n \in \mathrm{Crit}(S_{\varepsilon_n})$ for all $n$. (Clearly, then $x \in \mathscr{Z}_0$). Then $x$ is contained in $\mathscr{Z}_1$.
\end{lemma}
\begin{proof}
Since $S_0$ is Morse--Bott, there is a neighborhood of $b$ in $B$ modelled on an open neighborhood of $(0,0)$ in the product $U \times H$ of two inner product vector spaces $U$ and $H$, such that in the coordinates $(u , h ) \in U \times H$ we have
\[
S_0 (u,h) = \frac{1}{2}\langle Lu , u \rangle + \mathrm{constant} 
\]
for a self-adjoint linear isomorphism $L : U \xrightarrow{\cong} U$ (Thus, in this local coordinates the critical locus $\mathscr{Z}_0$ is the linear subspace $u = 0$).

Writing $x_n = (u_n , h_n )$, the equation that $x_n$ is a critical point of $S_{\varepsilon_n}$ says:
\begin{align*}
0 &= Lu_n + \varepsilon_n \frac{\partial S_1}{\partial u} (u_n , h_n ) + \varepsilon_{n}^2   \frac{\partial S_2}{\partial u} (u_n , h_n ) + \cdots  \in U\\
0 &= \varepsilon_n \frac{\partial S_1}{\partial h} (u_n , h_n ) + \varepsilon_{n}^2   \frac{\partial S_2}{\partial h} (u_n , h_n ) + \cdots  \in H .
\end{align*}
Dividing out $\varepsilon_n$ from the second equation, and taking the limit $n \to \infty$ shows that $\partial S_1 / \partial h$ vanishes at $(0,0)$. This says that $x\in B$ is not only contained in $\mathscr{Z}_0$, but also in $\mathscr{Z}_1 \subset \mathscr{Z}_0$.
\end{proof}
(We note that the proof of Lemma \ref{lemma:convergence} only used the Morse--Bott hypothesis on $S_0$). %Thus, for each point $x \in \mathscr{Z}_0 \setminus \mathscr{Z}_1$ and a sufficiently small neighborhood$\mathscr{N}_x$ of $x$ in $B$ and sufficiently small $\varepsilon_x >0$, we have that the critical locus $\mathrm{Crit}(S_\varepsilon )$ on $\mathscr{N}_x$ is empty for all $0 < | \varepsilon | < \varepsilon_x$. 

We must thus describe the behavior of the critical locus of our perturbation near a point $x \in \mathscr{Z}_1$: in particular, determine $x$ appear as limits of a sequence as in Lemma \ref{lemma:convergence}. Here we make use of both the Morse--Bott hypothesis on $S_0$ and $S_{1}|_{\mathscr{Z}_0}$. This ensures that for any point $x\in \mathscr{Z}_1 = \mathrm{Crit}(S_{1}|_{\mathscr{Z}_0})$, there is a neighborhood $\mathscr{U}_x$ of $x$ in $B$ diffeomorphism to a neighborhood $\mathscr{U}_0$ of $(0,0,0)$ in the product $U \times V \times W$ of inner product vector spaces, such that in the $U \times V \times W$ coordinates we have:
\begin{itemize}
\item $S_0 (u,v,w) = \frac{1}{2}\langle L_0 u , u \rangle + \mathrm{constant}$, for a self-adjoint linear isomorphism $L_0 : U \xrightarrow{\cong} U$. In particular, the critical locus $\mathscr{Z}_0 = \mathrm{Crit}(S_0 )$ is given by $0 \times V \times W \subset U \times V \times W$
%\item The derivative $(dS^0)_{(0,0,0)}$ vanishes on the whole of $U \oplus V \oplus W = T_{(0,0,0)} (U \times V \times W )$
%\item The quadratic form on $U \oplus V \oplus W$ given by the Hessian of $S^0$ at $(0,0,0)$ has $\mathrm{Ker} (\mathrm{Hess}S^0 )_{(0,0,0)} = 0 \oplus V \oplus W$
%\item $S^1 (0,v,w) = \frac{1}{2}\langle G v , v \rangle + \mathrm{constant}$, for a self-dual linear isomorphism $G : V \xrightarrow{\cong} V^\ast$ (In particular $\mathscr{Z}^1$ is given by $u = 0$ and $v = 0$.
%\item The derivative of $S^1$ at $(0,0,0)$ vanishes along the direction $0 \oplus V \oplus W \subset T_{(0,0,0)} ( U \times V \times W )$
\item $S_1 (0,v,w) = \frac{1}{2}\langle L_1 v , v \rangle + \mathrm{constant}$, for a self-adjoint linear isomorphism $L_1 : V \xrightarrow{\cong} V$. In particular, the critical locus $\mathscr{Z}_1= \mathrm{Crit}(S_{1}|_{\mathscr{Z}_0})$ is given by $0\times 0 \times W \subset U \times V \times W$.\end{itemize}

Consider the equations for a point $(u,v,w)$ to be a critical point of $S_\varepsilon$:\\
\begin{align}
0 &= L_0 u + \varepsilon \frac{\partial S_1}{\partial u} (u , v , w ) + \varepsilon^2   \frac{\partial S_2}{\partial u} (u , v , w ) + \cdots  \in U \label{Du}\\
0 &= \varepsilon \frac{\partial S_1}{\partial v} (u , v , w ) + \varepsilon^2   \frac{\partial S_2}{\partial v} (u , v , w) + \cdots  \in V \label{Dv}\\
0 & = \varepsilon \frac{\partial S_1}{\partial w} (u , v , w ) + \varepsilon^2   \frac{\partial S_2}{\partial w} (u , v , w) + \cdots  \in W \label{Dw} .
\end{align}
\\

Putting together the right-hand side of (\ref{Du}) with the right-hand side of (\ref{Dv}) divided by $\varepsilon$, we obtain a smooth map $F : U \times V \times W \times \mathbb{R} \rightarrow U \times V$
\\
\begin{align*}
%& F : U \times V \times W \times \mathbb{R} \rightarrow U \times V \\
F(u,v,w,\varepsilon ) = \Big( L_0 u + \varepsilon \frac{\partial S_1}{\partial u} (u , v , w ) + \varepsilon^2   \frac{\partial S_2}{\partial u} (u , v , w ) + \cdots \, , \,  \frac{\partial S_1}{\partial v} (u , v , w ) + \varepsilon   \frac{\partial S_2}{\partial v} (u , v , w) + \cdots   \Big) 
\end{align*}
\\
with the property that its derivative at $(u,v,w,\varepsilon ) = (0,0,0,0)$ induces an isomorphism of $U \oplus V \oplus 0 \oplus 0 \subset T_{(0,0,0,0)}( U \times V \times W \times \mathbb{R})$ with $U \oplus V = T_{(0,0)} (U \times V )$. Indeed, the derivative of $F$ at $(0,0,0,0)$ has the shape
\begin{align*}
(dF)_{(0,0,0,0)} = \begin{pmatrix} L_0 & 0 &  0 & \ast \\
\ast  &  L_1 & \ast & \ast \end{pmatrix}.
\end{align*}

It follows, by the Implicit Function Theorem, that the equation $F(u,v,w,\varepsilon) = 0$ uniquely determines $u,v,w$ as functions of variables $(t, \varepsilon ) \in W \times \mathbb{R}$. More precisely, we find a neighborhood $\mathscr{U}_{0}^\prime$ of $0$ in $W$ and $\varepsilon_0 > 0$ and a unique smooth function 
\[
(t,\varepsilon ) \in \mathscr{U}_{0}^\prime \times (-\varepsilon_0 , \varepsilon_0 ) \subset W \times \mathbb{R} \mapsto (u (t, \varepsilon ) , v(t, \varepsilon ) , t ) \in \mathscr{U}_0 \subset U \times V \times W
\]
%in an open neighborhood of $(0,0)$:
%\[
%u = u(t,\varepsilon ) \quad , \quad v = v(t , \varepsilon ) ,\quad , \quad  w = t ,
%\]
subject to $F(u(t,\varepsilon ) , v(t,\varepsilon ),  t ) = 0$ and $u (0,0) = 0$, $v(0,0) = 0$.  Furthermore, it also follows that for every $t$ we have $u(t, 0 ) = 0$ and $v(t , 0 ) = 0$ (since $\partial S_1 / \partial v$ vanishes when $u = 0$ and $v = 0$). %Additionally, the smooth functions $u(t,\varepsilon)$ and $v(t, \varepsilon )$ are analytic in $\varepsilon$ (because so is $F$).

Thus, for $0 < |\varepsilon|<\varepsilon_0$ the critical locus of $S_\varepsilon$ on the neighborhood $\mathscr{U}_0$ of $(0,0,0)$ in $U \times V \times W$ is parametrised by the set of $t \in \mathscr{U}_{0}^\prime \subset W$ such that the triple $(u,v,w) = (u(t,\varepsilon ) , v(t , \varepsilon ) , t )$ solves the remaining equation (\ref{Dw}). We describe this parametrisation more precisely now: as a critical locus. For this, introduce the smooth function on $\mathscr{U}_{0}^\prime \times (-\varepsilon_0 , \varepsilon_0 ) \setminus 0 \subset W \times \mathbb{R} $ given by
\begin{align*}
S^\prime (t , \varepsilon ) & :=  \sum_{j =0}^\infty S_j (u(t,\varepsilon ) , v (t , \varepsilon ) , t )\varepsilon^{j-2} - \varepsilon^{-2}S_0 (0,0,0) - \varepsilon^{-1} S_1 (0,0,0)\\
& = \varepsilon^{-2} S_\varepsilon (u(t, \varepsilon) , v(t , \varepsilon ) , t ) - \varepsilon^{-2}S_0 (0,0,0) - \varepsilon^{-1} S_1 (0,0,0) .
\end{align*}
%Namely, $T_\varepsilon$ is just the function inside the $\partial/\partial t$ sign in (\ref{eq:ddt}) but divided by $\varepsilon$ and with 
and note that $S^\prime$ extends smoothly over to $\mathscr{U}_{0}^\prime \times (-\varepsilon_0 , \varepsilon_0 )$. Indeed, regarding $S_0 $ and $S_1$ as functions of $(t, \varepsilon )$ and recalling that $u(t , 0 ) = 0 $, $v(t , 0 ) = 0$, then we have
\begin{align*}
 & S_0 (t, 0 ) = S_0 ( u(t , 0 ) , v(t , 0 ) , t) = \mathrm{constant}\\
& S_1 (t, 0 ) = S_1 (u(t, 0 ) , v(t , 0 ) , t ) =  \mathrm{constant}\\
%S^0 (u(t,0) , v(t, 0 ) , t ) = S^0 (0,0,t) = \mathrm{constant} = S^0 (0,0,0)\\
%S^1 ( u(t , 0 ) , v(t , 0 ) , t ) = S^1 (0,0,t) = \mathrm{constant} = S^1 (0,0,0)\\
& \frac{\partial S_0}{\partial \varepsilon} ( t , 0 ) = \frac{\partial S_0}{\partial u} (u(t, 0 )  , v(t, 0 ) , t ) \cdot \frac{\partial u}{\partial \varepsilon} + \frac{\partial S_0}{\partial v}(u(t, 0 ) , v(t, 0 ) , t ) \cdot \frac{\partial v}{\partial \varepsilon} = 0  .
\end{align*}
Denote $S^{\prime}_\varepsilon (t) = S^\prime (t , \varepsilon )$. The next Lemma provides the promised local parametrisation of the critical locus of $S_\varepsilon$:

\begin{lemma}\label{lemma:reductionstep}
For $0 < |\varepsilon| \leq \varepsilon_0$, the smooth embedding
\[
t \in \mathscr{U}^{\prime}_0 \subset W   \mapsto (u(t, \varepsilon) , v(t , \varepsilon ) , t ) \in \mathscr{U}_0 \subset U \times V \times W
\]
takes the critical locus $\mathrm{Crit}(S^{\prime}_\varepsilon )$ on the neighborhood $\mathscr{U}^{\prime}_0 \subset W$ of $t = 0$, to the critical locus $\mathrm{Crit}(S_\varepsilon )$ on the neighborhood $\mathscr{U}_0 \subset U \times V \times W$ of $(u,v,w) = (0,0,0)$.
\end{lemma}

\begin{proof}
The remaining equation (\ref{Dw}), after setting $u = u(t,\varepsilon) , v = v(t , \varepsilon ) , w = t$ and dividing by $\epsilon^2 \neq 0$, is 
\[
\varepsilon^{-2} \frac{\partial S_\varepsilon }{\partial w }(u(t,\varepsilon ) , v (t , \varepsilon ) , t ) = 0 .
\]
But since $\partial S_\varepsilon / \partial u$ and $\partial S_\varepsilon /\partial v $ vanish at $(u,v,w ) = (u(t,\varepsilon ) , v (t , \varepsilon ) , t )$, then the left-hand side can be written as
\begin{align*}
 \varepsilon^{-2} \big( \frac{\partial S_\varepsilon }{\partial u} \frac{\partial u}{\partial t} + \frac{\partial S_\varepsilon }{\partial v}\frac{\partial v}{\partial t} + \frac{\partial S_\varepsilon }{\partial w } \big) = \frac{\partial S_{\varepsilon}^\prime }{\partial t } .
\end{align*}
%\begin{align}
%& \frac{\partial S_1}{\partial w} (u(t,\varepsilon) , v(t , \varepsilon) , t ) + \varepsilon  \frac{\partial S_2}{\partial w} (u(t , \varepsilon ) , v(t , \varepsilon) , t) + \cdots \nonumber\\
%& = \sum_{j = 1 }^{\infty} \frac{\partial S^j}{\partial w} (u(t,\varepsilon), v(t, \varepsilon ) , t ) \varepsilon^{j-1} \nonumber\\
%& = \sum_{j = 1}^{\infty} \Big( \frac{\partial S_j}{\partial t} - \frac{\partial S_j}{\partial u} \cdot \frac{\partial u }{\partial t} - \frac{\partial S_j}{\partial v}\cdot \frac{\partial v}{\partial t} \Big) \varepsilon^{j-1} \nonumber\\
%& = \Big( \sum_{j = 1}^{\infty} \frac{\partial S_j}{\partial t} \varepsilon^{j-1}  \Big) - \Big( \sum_{j = 1}^{\infty} \frac{\partial S_j}{\partial u} \varepsilon^{j-1} \Big) \cdot \frac{\partial u}{\partial t} - \Big(  \sum_{j = 1}^{\infty} \frac{\partial S_j }{\partial v}\varepsilon^{j-1} \Big) \cdot \frac{\partial v}{\partial t} \nonumber\\
%& = \Big( \sum_{j = 1}^{\infty} \frac{\partial S_j}{\partial t} \varepsilon^{j-1}  \Big) + \varepsilon^{-1} \langle Lu , \frac{\partial u}{\partial t}\rangle + 0 \quad \text{by (\ref{Du}-\ref{Dv})} \nonumber\\
%& = \frac{\partial }{\partial t} \Big(  \sum_{j =0}^\infty S_j (u(t,\varepsilon ) , v (t , \varepsilon ) , t )\varepsilon^{j-1} \Big) .\nonumber
%\end{align}
\end{proof} 
Consider the sequence of smooth functions on $\mathscr{U}^{\prime}_0 \subset W$ given by
 \[
 S^{\prime}_j (t) := \frac{1}{j!}\frac{\partial^{(j)} S^\prime}{\partial \varepsilon^{{(j)}}} (t , 0 ) \quad , \quad j \geq 0 .
 \]
We note that the smooth function $S^\prime (t , \varepsilon ) $ need not be analytic in $\varepsilon$ anymore: we we may only `formally' write
\[
S^{\prime}_\varepsilon (t ) = S^{\prime}_0 (t ) + \varepsilon S^{\prime}_1 (t ) + \varepsilon^2 S^{\prime}_2 (t) + \cdots .
\]

At this point, it follows that the points $\mathscr{Z}_1 = \mathrm{Crit}(S_{1}|_{\mathscr{Z}_0} )$ which arise as limits of a sequence as in Lemma \ref{lemma:convergence} are further constrained to be critical points of the \textit{leading term} $S^{\prime}_0$ in this formal expansion. The next Lemma describes this leading term:

\begin{lemma} \label{lemma:leading}
Let $\lambda (t ) \in U$ be the function defined by the requirement
\[
L_0 ( \lambda (t ) )= - \frac{\partial S_1}{\partial u} (0,0,t) \text{ as elements of } U .
\]
Then we have 
\begin{align*}
S_{0}^\prime (t)  &= \frac{1}{2}   \frac{\partial S_1 }{\partial u} (0,0,t) \cdot  \lambda (t)  + S_2 (0 , 0 , t ) \\
\frac{\partial S_{0}^\prime}{\partial t}  &= \frac{\partial^2 S_1 }{\partial u \partial w} (0,0,t) \cdot \lambda (t) + \frac{\partial S_2}{\partial w}(0,0,t) \\
\frac{\partial^2 S_{0}^\prime}{\partial t^2 } & = \frac{\partial^3 S_1 }{\partial u \partial w^2 } (0,0,t) \cdot \lambda(t) + \frac{\partial^2 S_1 }{\partial u \partial w} (0,0,t) \cdot \frac{\partial \lambda }{\partial t } + \frac{\partial^2 S_2 }{\partial w^2 }(0,0,t) .
\end{align*}
\end{lemma}
\begin{proof}
We have
\[
S^{\prime}_0 (t) = S^{\prime }(t, 0) =  \frac{1}{2}\frac{\partial^2 S_0}{\partial \varepsilon^2}(t, 0 ) + \frac{\partial S_1}{\partial \varepsilon}(t, 0 ) + S_2 (t,0)
\]
and we now compute each of the first two terms on the right-hand side separately. (In the following calculations we always evaluate at $\varepsilon = 0$.)

\begin{align*}
\frac{1}{2}\frac{\partial^2 S_0}{\partial \varepsilon^2}  = &  \frac{1}{2} \frac{\partial}{\partial \varepsilon} \Big( \frac{\partial S_0}{\partial u} \frac{\partial u}{\partial \varepsilon} + \frac{\partial S_0}{\partial v} \frac{\partial v}{\partial \varepsilon} \Big)\\
 = &  \frac{1}{2} \frac{\partial^2 S_0 }{\partial u^2} \big(\frac{\partial u}{\partial \varepsilon}\big)^2  + \frac{\partial^2 S_0}{\partial u \partial v} \big( \frac{\partial u}{\partial \varepsilon}\big)\big( \frac{\partial v}{\partial \varepsilon}\big) +  \frac{1}{2} \frac{\partial^2 S_0 }{\partial v^2} \big(\frac{\partial v}{\partial \varepsilon}\big)^2 \\
& + \frac{1}{2}  \frac{\partial S_0 }{\partial u} \frac{\partial^2 u}{\partial \varepsilon^2}
+ \frac{1}{2}  \frac{\partial S_0 }{\partial v} \frac{\partial^2 v}{\partial \varepsilon^2}\\
 = &  \frac{1}{2} \frac{\partial^2 S_0 }{\partial u^2} \big(\frac{\partial u}{\partial \varepsilon}\big)^2\\
\\
\frac{\partial S_1}{\partial \varepsilon}  = &\frac{\partial S_1}{\partial u} \frac{\partial u}{\partial \varepsilon} + \frac{\partial S_1}{\partial v} \frac{\partial v}{\partial \varepsilon}  = \frac{\partial S_1}{\partial u} \frac{\partial u}{\partial \varepsilon} .
\end{align*}
\\
Now, by differentiating equation (\ref{Du}) with respect to $\varepsilon$ and then setting $\varepsilon = 0$ we obtain the following relation in $U^\ast$

\[
\frac{\partial^2 S_0}{\partial u^2}\frac{\partial u}{\partial \varepsilon} = - \frac{\partial S_1}{\partial u}
\]
\\
i.e. $\frac{\partial u}{\partial \varepsilon}(t, 0 ) = \lambda (t)$. Inserting this into the above, we obtain the required formula for $S^{\prime}_0$
\begin{align*}
S^{\prime}_0 (t) & =  \frac{1}{2} \frac{\partial^2 S_0 }{\partial u^2} \big(\frac{\partial u}{\partial \varepsilon}\big)^2 + \frac{\partial S_1}{\partial u} \frac{\partial u}{\partial \varepsilon}  + S_2\\
& =- \frac{1}{2} \frac{\partial S_1}{\partial u} \cdot  \lambda   +   \frac{\partial S_1}{\partial u} \cdot  \lambda + S_2\\
& = \frac{1}{2}\frac{\partial S_1}{\partial u} \cdot  \lambda + S_2  .
\end{align*}

For $\partial S_{0}^\prime / \partial t$, recall that 
\[
\frac{\partial S_{\varepsilon}^\prime }{\partial t} = \varepsilon^{-2} \frac{\partial S_\varepsilon}{\partial w} = \varepsilon^{-1} \frac{\partial S_1}{\partial w} + \frac{\partial S_2}{\partial w} + \mathcal{O}(\varepsilon )
\]
Now, we have 
\begin{align*}
\frac{\partial }{\partial \varepsilon}\frac{\partial S_1}{\partial w} = \frac{ \partial^2 S_1}{\partial u \partial w} \frac{\partial u}{\partial \varepsilon } + \frac{\partial^2 S_1}{\partial v \partial w} \frac{\partial v }{\partial \varepsilon } = \frac{\partial^2 S_1}{\partial u \partial w} (0,0,t) \cdot \lambda + \mathcal{O}(\varepsilon ).
\end{align*}
Inserting this into the previous formula yields 
\begin{align}
\frac{\partial S_{\varepsilon}^\prime }{\partial t} = \frac{\partial^2 S_1}{\partial u \partial w} (0,0,t) \cdot \lambda + \frac{\partial S_2}{\partial w} (0,0,t)+ \mathcal{O}(\varepsilon )  \label{dtleading}
\end{align}
from which the formula for $\partial S_{0}^\prime / \partial t$ follows. 

Finally,
\begin{align*}
\frac{\partial^2 S_{0}^\prime}{\partial t^2 }  = &   \frac{\partial^3 S_1}{\partial u^2  \partial w}(0,0,t) \cdot \frac{\partial u }{\partial t}(0,t)  \cdot \lambda(t) + \frac{\partial^3 S_1}{\partial u \partial v  \partial w} (0,0,t) \cdot \frac{\partial v }{\partial t}(0,t)  \cdot \lambda(t) \\
& + \frac{\partial^3 S_1 }{\partial u \partial w^2} (0,0,t) \cdot  \lambda(t)  + 
\frac{\partial^2 S_1}{\partial u \partial w}(0,0,t) \cdot \frac{\partial \lambda}{\partial t}\\
& + \frac{\partial^2 S_2}{\partial u \partial w}(0,0,t) \cdot \frac{\partial u }{\partial t} (t,0)+ \frac{\partial^2 S_2 }{\partial v \partial w} (0,0,t) \cdot \frac{\partial v}{\partial t}(t,0) \\
 & + \frac{\partial^2 S_2}{\partial w^2 }(0,0,t) + \mathcal{O}(\varepsilon ).
\end{align*}
The required formula then follows by noting that $\partial u /\partial t $ and $\partial v /\partial t $ vanish identically at $(u,v) = (t,0)$ since $u(t,0) $ and $v(t,0)$ vanish identically.
\end{proof}

Of course, there may be further obstructions for a critical point of the leading term $S^{\prime}_0$ to appear as a limit of a sequence as in Lemma \ref{lemma:convergence}. However, in the situation that $t = 0$ is a \textit{non-degenerate} (i.e. Morse) critical point of the leading term $S^{\prime}_0$ then we obtain the following conclusion:
\begin{proposition}\label{prop:localpert}
If $t = 0$ is a non-degenerate critical point of the leading term $S^{\prime}_0 : \mathscr{U}^{\prime}_0 \subset W \rightarrow \mathbb{R}$, then there exist a neighborhood $\mathscr{U}_0 \subset U \times V \times W$ of $x_0  = (0,0,0)$ and $\epsilon_0 > 0$ such that $S_\varepsilon$ has a unique critical point $x_\varepsilon$ in $\mathscr{U}_{0}$ for each $0 < |\varepsilon|\leq \varepsilon_0$, and the function $\varepsilon \mapsto x_\varepsilon$ extends to a smooth function on $[-\varepsilon_0 , \varepsilon_0 ]$. Furthermore, this critical point is non-degenerate and when $\pm \varepsilon >0$ its index is
\[
\mathrm{ind}(S_\varepsilon , x_\varepsilon ) = \mathrm{ind}(S_0 , x_0 ) + \mathrm{ind}(\pm S_{1}|_{\mathscr{Z}_0} , x_0 ) + \mathrm{ind}(S_{0}^\prime , 0 ).
\]
(On the other hand, if $t = 0$ is not a critical point of $S_{0}^\prime : \mathscr{U}_{0}^\prime \subset W \rightarrow \mathbb{R}$, then there exist a neighborhood $\mathscr{U}_0 \subset U \times V \times W$ of $(0,0,0)$ and $\epsilon_0 > 0$ such that $S_\varepsilon$ has no critical points in $\mathscr{U}_0$ for each $0 < |\varepsilon|< \varepsilon_0$.)
\end{proposition}

(We recall that the index $\mathrm{ind}(f,x) \in \mathbb{Z}_{\geq 0}$ of a critical point $x$ of a Morse or Morse--Bott function $f$ is defined as the dimension of the negative eigenspace of the Hessian of $f$ along the normal bundle to the critical locus of $f$ at $x$. This only depends on the connected component $C$ of the critical locus of $f$ that contains $x$, so we may also use the notation $\mathrm{ind}(f,C)$ instead).

\begin{proof}
Since $t = 0$ is a non-degenerate critical point of $S_{0}^\prime$ then it persists in the perturbation $S_{\varepsilon}^\prime$ to give a unique non-degenerate critical point $t(\varepsilon )$ of $S_{\varepsilon}^\prime$ for $\varepsilon$ sufficiently close to $\varepsilon = 0$. By Lemma \ref{lemma:reductionstep} we obtain a unique critical point $x_\varepsilon$ of the original perturbation $S_\varepsilon$ in a neighborhood of $(0,0,0)$ for sufficiently small $\varepsilon \neq 0$, given in the $U \times V \times W$ coordinates by $x_\varepsilon = (u_\varepsilon , v_\varepsilon , w_\varepsilon ) := ( u(t(\varepsilon) , \varepsilon ) , v(t(\varepsilon),\varepsilon ) , t(\varepsilon ) )$. (Also, clearly the path $x_\varepsilon$ extends to $\varepsilon = 0$ to give a smooth path.)

%Fix inner products on the vector spaces $U , V , W$ (so as to identify each of them with their dual space). 
Consider smooth paths $\dot{x}_\varepsilon = (\dot{u}_\varepsilon , \dot{v}_\varepsilon , \dot{w}_\varepsilon ) \in U \oplus V \oplus W$  and $\mu_\varepsilon \in \mathbb{R}$, defined for $\varepsilon$ sufficiently close to $\varepsilon = 0$, such that $\dot{x}_\varepsilon \neq 0 $ and
\begin{align}
(\mathrm{Hess}S_\varepsilon )_{x_\varepsilon } (\dot{x}_\varepsilon ) = \mu_\epsilon \cdot \dot{x}_\varepsilon . \label{eigenvectors}
\end{align}
Formally, we write $\mu_\varepsilon = \mu_0 + \varepsilon  \mu_1  + \varepsilon^2  \mu_2 + \cdots $(i.e. $\mu_1$ is the first derivative of $\mu_\varepsilon$ at $\varepsilon = 0$, and so forth, but with no assumption of convergence), and likewise for other quantities (e.g. $\dot{u}_\varepsilon = \dot{u}_0 + \varepsilon \dot{u}_1 + \varepsilon^2 \dot{u}_2 + \cdots$). We consider the following scenarios:

\textbf{(I)} Suppose that the path $\mu_\varepsilon$ vanishes at $\varepsilon = 0$, i.e. $\mu_0 = 0$. Thus $(\mathrm{Hess}S_0 )_{x_0 } \dot{x}_0 = 0 $, and hence $\dot{x}_0 \in 0 \oplus V \oplus W$ (i.e. $\dot{u}_0 = 0$). Differentiating (\ref{eigenvectors}) in $\varepsilon $ yields the following identity:
\begin{align}
(\mathrm{Hess} \frac{\partial S_\varepsilon}{\partial \varepsilon } )_{x_\varepsilon } (\dot{x}_\varepsilon ) + (\nabla^3 S_\varepsilon )_{x_\varepsilon }(\frac{\partial x_\varepsilon}{\partial \varepsilon } , \dot{x}_\varepsilon ) + (\mathrm{Hess}S_\varepsilon )_{x_\varepsilon } ( \frac{\partial \dot{x}_\varepsilon }{\partial \varepsilon } ) = \frac{\partial \mu_\varepsilon }{\partial \varepsilon } \dot{x}_\varepsilon + \mu_\varepsilon \frac{\dot{x}_\varepsilon }{\partial \varepsilon } \label{diff1eigen}
\end{align}
and evaluation at $\varepsilon = 0$ gives the relation
\[
(\mathrm{Hess}S_1)_{x_0 } (\dot{x}_0 ) + (\nabla^3 S_0 )_{x_0} ( x_1 , \dot{x}_0 ) + (\mathrm{Hess}S_0 )_{x_0 }( \dot{x}_1 ) = \dot{\mu}_1 \dot{x}_0 .
\]
Noting that in our $U \times V \times W$ coordinates the function $S_0 $ is $\frac{1}{2}\langle L_0 u , u \rangle + \mathrm{constant}$ and hence $\nabla^3 S_0$ vanishes identically, the previous relation says:
\begin{align}
\mathrm{Hess}(S_{1}|_{0 \times V \times W} )_{x_0 } (\dot{x}_0 ) = \mu_1 \cdot \dot{x}_0 \label{diff11}\\
(\frac{\partial^2 S_1 }{\partial u \partial v} )_{x_0 }(\dot{v}_0 ) + (\frac{\partial^2 S_1 }{\partial u \partial w} )_{x_0 }(\dot{w}_0 ) + L_0 \dot{u}_1 = 0 .\label{diff12}
\end{align}

\textbf{(II)} Suppose further that $\mu_0 = \mu_1 = 0$. By (\ref{diff11}) we now have that $\dot{x}_0 \in 0 \times 0 \times W$ (i.e. $\dot{u}_0 = \dot{v}_0 = 0$). Differentiating (\ref{diff1eigen}) at $\varepsilon = 0$ (and noting the vanishing of $\nabla^3 S_0$ and $\nabla^4 S_0$ in the $U \times V \times W$ coordinates) yields the identity:

\begin{align*}
 (\mathrm{Hess}S_2 )_{x_0} (\dot{x}_0 ) + (\nabla^3 S_1 )_{x_0} (x_1 , \dot{x}_0 ) + (\mathrm{Hess}S_1 )_{x_0} (\dot{x}_1 ) + (\mathrm{Hess}S_0 )_{x_0 } (\dot{x}_2 ) = \mu_2 \cdot  \dot{x}_0 .
\end{align*}
Projecting this identity onto the $W$ factor (and using also that $S_1 (0,v,w) = \frac{1}{2}\langle L_1 v , v \rangle + \mathrm{constant}$ in our coordinates, and $\lambda (0) = u_1$) yields:
\begin{align*}
\mu_2 \cdot \dot{w}_0 & =  (\frac{\partial^2 S_2 }{\partial w^2} )_{x_0} (\dot{w}_0 ) + (\frac{\partial^3 S_1}{\partial u \partial w^2 })_{x_0} (\lambda ( 0) , \dot{w}_0 ) + (\frac{\partial^2 S_1}{\partial u \partial w})_{x_0} (\dot{u}_1 )     \\
& = (\frac{\partial^2 S_{0}^\prime}{\partial t^2 } )_{x_0 }(\dot{w}_0 ) - (\frac{\partial^2 S_1 }{\partial u \partial w} )_{x_0} \cdot (  ( \frac{\partial \lambda }{\partial t } ) (0 )  , \dot{w}_0 ) + (\frac{\partial^2 S_1}{\partial u \partial w})_{x_0} (\dot{u}_1 ) \quad \text{   by Lemma (\ref{lemma:leading})}\\
& = (\frac{\partial^2 S_{0}^\prime}{\partial t^2 } )_{x_0 }(\dot{w}_0 ) + \langle L_0 \dot{u}_1 , \frac{\partial \lambda }{\partial t}(0 ) \rangle  + (\frac{\partial^2 S_1}{\partial u \partial w})_{x_0} (\dot{u}_1 ) \quad \text{  by (\ref{diff12})}\\
& = (\frac{\partial^2 S_{0}^\prime}{\partial t^2 } )_{x_0 }(\dot{w}_0 ) + \langle \dot{u}_1 , L_0 \frac{\partial \lambda }{\partial t}(0 ) \rangle  + (\frac{\partial^2 S_1}{\partial u \partial w})_{x_0} (\dot{u}_1 ) \quad \text{  by (\ref{diff12})}.
\end{align*}
Differentiating in $t $ the identity defining $\lambda(t)$ in Lemma \ref{lemma:leading} gives $L_0 ( \partial \lambda / \partial t ) = - (\frac{\partial^2 S_1}{\partial u \partial w} )_{(0,0,t)}$, and inserting this in the previous identity gives the relation:
\begin{align}
(\mathrm{Hess} S_{0}^\prime )_{x_0 }(\dot{w}_0 ) = \mu_2 \cdot \dot{w}_0 \label{diff2}.
\end{align}
Furthermore, we must have $\mu_2 \neq 0$ because otherwise (\ref{diff2}) implies that $\dot{w}_0= 0$ and hence the contradiction $\dot{x}_0 = (0,0,0)$.

From these calculations all the required results follow at once.
\end{proof}

%we obtain the following conclusion: 
%in a sufficiently small neighborhood of $(0,0,0)$ in $U \times V \times W$ there is only critical point of $S_\varepsilon$ for small $\varepsilon\neq 0$, and this is a non-degenerate critical point.

\subsubsection{Invariant formulation}
Our previous construction of parametrisations of the critical locus of $S_\varepsilon$ made use of local coordinates adapted to the functions $S_0$ and $S_1$. We now address the `intrinsic' aspects of the previous discussion. Namely, we will show that the leading term $S_{0}^\prime $ of the function $S^{\prime}_\varepsilon$ is, in fact, intrinsically defined as a smooth function on the whole of $\mathscr{Z}_1 = \mathrm{Crit}(S_{1}|_{\mathscr{Z}_0})$.

For this, we consider the `constrained optimization problem' of extremising the function $S_1$ restricted to the critical locus of $S_0$. The associated \textit{Lagrange multipliers functional} is given by the smooth function on the total space of the tangent bundle of $B$ given by
\begin{align*}
\mathscr{L} : TB \rightarrow \mathbb{R} \quad , \quad \mathscr{L}(x, \lambda ) = S_1 (x) + (dS_0 )_x (\lambda ) . 
\end{align*}
where $x \in B$, and $\lambda \in T_x B$ (the `Lagrange multiplier'). The critical points of $\mathscr{L}$ consist of pairs $(x , \lambda )$ such that 
\\
\begin{align}
x \in \mathscr{Z}_0 = \mathrm{Crit}S_0 \quad  \text{   and   } \quad (\mathrm{Hess}S_0 )_x (\lambda ) = - (dS_1)_x \quad \text{       ( in  } T_{x}^\ast B  \text{ )}\label{critL} .
\end{align}
\\
Here $(\mathrm{Hess}S_0 )_x : T_x B \otimes T_x B \rightarrow \mathbb{R}$ is the quadratic form given by the Hessian of $S_0$ at $x$ (which is intrinsically defined since $b$ is a critical point of $S_0$). Note that, since we are assuming $S_0$ is Morse--Bott, then $(\mathrm{Hess} S_0 )_x$ induces a non-degenerate quadratic form $(T_x B / T_x \mathscr{Z}_0 ) \otimes (T_x B / T_x\mathscr{Z}_0 ) \rightarrow \mathbb{R}$. Thus, the condition that a point $x \in B$ admits a $\lambda \in T_x B$ such that (\ref{critL}) holds says, precisely, that $x$ is simultaneously both a critical point of $S_0$ and of the restriction of $S_1$ to $\mathscr{Z}_0$. Given such $x \in B$, the collection of all $\lambda$ such that (\ref{critL}) holds forms an affine space over $T_x \mathscr{Z}_0 = \mathrm{Ker}(\mathrm{Hess} S_0 )_x$.

\begin{definition}
The \textit{bundle of Lagrange multipliers} associated to $\mathscr{L}$ is the affine bundle $\Lambda \rightarrow \mathscr{Z}_1 = \mathrm{Crit}(S_{1}|_{\mathscr{Z}_0} )$ with fibers
\[
\Lambda_x = \big\{ \lambda \in T_x B \, | \,(\mathrm{Hess}S_0 )_x (\lambda ) = - (dS_1)_x  \big\} \subset T_x B,
\]
whose associated vector bundle is $(T\mathscr{Z}_0)|_{\mathscr{Z}_1 } \rightarrow \mathscr{Z}_1$.
\end{definition}

Observe that the leading term function $S_{0}^{\prime} : \mathscr{U}_{0}^\prime \subset W \rightarrow \mathbb{R}$ from our previous discussion, which we calculated in Lemma \ref{lemma:leading}, agrees with the following intrinsic and globally-defined smooth function on $\mathscr{Z}_1$ given by
\begin{equation}
\mathfrak{f} := \frac{1}{2}(dS_1 )(\lambda ) + (S_2)|_{\mathscr{Z}_1}\label{intrinsic_obs}
\end{equation}
where $\lambda \in \Gamma (\mathscr{Z}^1 , \Lambda )$ is any smooth section of the bundle of Lagrange multipliers. (Note that the expression (\ref{intrinsic_obs}) is indeed independent of the choice of section $\lambda$, since $\Lambda$ is an affine bundle over $(T\mathscr{Z}_0 )|_{\mathscr{Z}_1}$ and $dS_1 $ vanishes on $(T\mathscr{Z}_0 )|_{\mathscr{Z}_1}$).

This observation allows to formulate the results of our coordinate-dependent discussion in an intrinsical manner:

\begin{theorem}[Perturbation Theorem, finite-dimensional version]\label{theorem:perturbation1}
Let $S_\varepsilon = S_0 + \varepsilon S_1 + \varepsilon^2 S_2 + \cdots : B \rightarrow \mathbb{R}$ be a perturbation, such that the following assumptions hold:
\begin{itemize}
%\item $S^1 : B \rightarrow \mathbb{R}$ is an exhaustive function (i.e. proper and bounded below)
\item $S_0$ is Morse--Bott; denote its critical locus $\mathscr{Z}_0 = \mathrm{Crit}(S_0 )$

\item The restriction of $S_1$ to $\mathscr{Z}_0$ is Morse--Bott; denote its critical locus $\mathscr{Z}_1 = \mathrm{Crit}(S_{1}|_{\mathscr{Z}_0 })$

\item $\mathscr{Z}_1$ is compact.
\end{itemize}
Then there is a canonical smooth function $\mathfrak{f} : \mathscr{Z}_1 \to \mathbb{R}$ (namely (\ref{intrinsic_obs})), an open neighborhood $\mathscr{U}_0$ of $\mathscr{Z}_1 \subset B$ and a constant $\varepsilon_0 >0$, such that the following holds when $0 < | \varepsilon| \leq \varepsilon_0$: if $\mathfrak{f}$ is a Morse function and $\pm\varepsilon >0$, then $S_\varepsilon$ is a Morse function in the neighborhood $\mathscr{U}_0$ and there is a bijection %\textcolor{red}{(need smooth in $\varepsilon>0$)}
\[
F_\varepsilon : \mathrm{Crit}\mathfrak{f} \cong \mathrm{Crit}S_\varepsilon \cap \mathscr{U}_0   , 
\]
such that for each critical point $x\in \mathscr{Z}_1$ 
\[
\mathrm{ind}(S_\varepsilon , F(x) ) =  \mathrm{ind}(S_0 , x) + \mathrm{ind}(\pm S_{1}|_{\mathscr{Z}_0} , x )  + \mathrm{ind}(\mathfrak{f} , x)  .
\]
\end{theorem}
(Above, $\mathrm{ind}(S_0 , x)$ denotes the Morse--Bott index of the critical manifold of $S_0$ passing through $x$, and likewise for $\mathrm{ind}(\pm S_{1}|_{\mathscr{Z}_0} , x)$.)\\

The following is the finite-dimensional prototype of the Localisation Theorem that we will establish for the Chern--Simons functional on a Seifert-fibered space:

\begin{theorem}[Localisation Theorem, finite-dimensional version]
\label{theorem:perturbation2}Let $S_\varepsilon = S_0 + \varepsilon S_1 + \varepsilon^2 S_2 + \cdots : B \rightarrow \mathbb{R}$ be a perturbation, such that the following assumptions hold:
\begin{itemize}
%\item $S^1 : B \rightarrow \mathbb{R}$ is an exhaustive function (i.e. proper and bounded below)
\item $S_0$ is Morse--Bott; denote its critical locus $\mathscr{Z}_0 = \mathrm{Crit}(S_0 )$

\item The restriction of $S_1$ to $\mathscr{Z}_0$ is Morse--Bott; denote its critical locus $\mathscr{Z}_1 = \mathrm{Crit}(S_{1}|_{\mathscr{Z}_0 })$

\item $\mathscr{Z}_1$ is compact

\item The function on $\mathscr{Z}_1$ given by $\frac{1}{2}(dS_1 )(\lambda ) + (S_2)|_{\mathscr{Z}_1}$ is Morse; here $\lambda$ stands for any section of the bundle $\Lambda \rightarrow \mathscr{Z}_1$ of Lagrange multipliers associated to $\mathscr{L} := S_1 + (d S_0 )(\lambda ) : TB \rightarrow \mathbb{R} $.
\item The restriction of $S_1$ to $\mathscr{Z}_0$ is a proper and below function.
\end{itemize}
Let $\psi : B \rightarrow \mathbb{R}_{\geq 0}$ be a given proper continuous function on $B$. Then, for any given open set $\mathscr{U}\subset B$ and a constant $C > 0$ such that
\[
\mathscr{Z}_1 \subset \mathscr{U} \subset \{ \psi < C\}
\]
there exists a constant $\varepsilon_0 = \varepsilon_0 (B , S_i , \psi ) > 0$ such that the following hold for $0 < |\varepsilon | \leq \varepsilon_0$:
\begin{enumerate}
\item For any critical point $x \in \mathrm{Crit}S_\varepsilon$ either $x\in \mathscr{U}$ or else $\psi (x) > C$. Furthermore, the critical points of $S_\varepsilon$ contained in $\mathscr{U}$ are non-degenerate and comprise a finite set.

\item Let $\mathscr{Z}_{0 } = \cup \mathscr{C}_i$ denote the decomposition of $\mathscr{Z}_0$ into connected components. For $\pm \varepsilon> 0$ the count of critical points of $S_\varepsilon$ in $\mathscr{U}$ with signs given by their indices agrees with a signed count of the Euler characteristics (resp. compactly-supported Euler characteristics) of the components $\mathscr{C}_i$:
\[
\sum_{x \in \mathrm{Crit}(S_\varepsilon ) \cap \mathscr{U}} (-1)^{\mathrm{ind}(S_\varepsilon  , x )} = \sum_i (-1)^{\mathrm{ind}(S_0 , \mathscr{C}_i  )} \cdot \chi ( \mathscr{C}_i ) \quad \text{   (resp.   }\sum_i (-1)^{\mathrm{ind}(S_0 , \mathscr{C}_i  )} \cdot \chi_c ( \mathscr{C}_i )  \text{  )}.
\]
\end{enumerate}
\end{theorem}
\begin{proof}
If the first assertion in (1) were false, then we can find a sequence $\varepsilon_n \neq 0$ converging to zero, and a sequence $x_n \in B$ such that for all $n$
\begin{align*}
& x_n \in \mathrm{Crit}(S_{\varepsilon_n})\\
& \psi (x_n ) \leq C \\
& x_n \notin \mathscr{U}
\end{align*}
Since $\psi$ is proper, then by the second condition we may assume that $x_n$ converges to a point $x \in B$ after passing to a subsequence. By the first condition and Lemma \ref{lemma:convergence} we have $x \in \mathscr{Z}_1$ and hence $x_n$ lies in $\mathscr{U}$ eventually, which contradicts the third condition.

Thus there is $\varepsilon_0 >0 $ such that the first assertion in (1) holds. For the second assertion in (1) %note that by the first assertion, the compactness of $\mathscr{Z}_1$ and the continuity of $\psi$, we may shrink $\varepsilon_0$ further so that for $0 < | \varepsilon | \leq \varepsilon_0$ we have 
%\[
%\mathrm{Crit}S_\varepsilon \cap \mathscr{N} \subset \{ \psi \leq C -\delta \}
%\]
%for some $\delta> 0$. Since $\psi$ is proper and by Lemma \ref{lemma:convergence} it follows that $\mathrm{Crit}S_\varepsilon \cap \mathscr{N}$ is compact. 
we may shrink $\varepsilon_0 >0$ further so as to ensure that for $0 < |\varepsilon |\leq \varepsilon_0$ the set $\mathrm{Crit}_\varepsilon \cap \mathscr{U}$ is contained in a neighborhood $\mathscr{U}_0 $ of $\mathscr{Z}_1 \subset B$ such that Theorem \ref{theorem:perturbation1} holds for $\mathscr{U}_0$ and $\varepsilon_0$. In particular, all critical points of $S_\varepsilon$ in $\mathscr{U}$ are non-degenerate. Thus, (1) follows. 

For (2), let $\mathscr{Z}_{1 } \cap \mathscr{C}_i = \cup C_{ij}$ denote the decomposition into connected components. By Theorem \ref{theorem:perturbation1}, for $\pm \varepsilon> 0$ the count of critical points of $S_\varepsilon$ in $\mathscr{U}$ with signs given by their index is:
\[
\sum_{x \in \mathrm{Crit}(S_\varepsilon ) \cap \mathscr{U}} (-1)^{\mathrm{ind}(S_\varepsilon  , x )} = \sum_{i,j} (-1)^{\mathrm{ind}(S_0 , \mathscr{C}_i  ) + \mathrm{ind}(\pm S_{1}|_{\mathscr{C}_i } , C_{ij} )} \cdot  \chi (C_{ij} ).
\]
%(In the last formula, $\mathrm{ind}(S_0 , C_i )$ denotes the index of the Morse--Bott function $S_0$ along the connected component of the critical locus $\mathscr{Z}_0$ containing $C_i$.)
But $S_1$ is a proper and bounded below Morse--Bott function on $\mathscr{C}_i$, and hence for each $i$
\[
\sum_j (-1)^{ \mathrm{ind}(\pm S_{1}|_{\mathscr{C}_i} , C_{ij} )}\chi (C_{ij} )
\]
agrees with the Euler characteristic $\chi (\mathscr{C}_i )$ in the $+$ case, and the compactly-supported Euler characteristic $\chi_c (\mathscr{C}_i )$ in the $-$ case. The result now follows.
\end{proof}
%\begin{remark}
%As we shall see, the formula from Theorem \ref{theorem:localisation}(2) isn't quite the finite-dimensional statement ... 
%\end{remark}

\subsection{Holonomy perturbations}

We now discuss the properties of a suitable class of functions on the stable orbit space $\mathscr{B}_{E}^\ast$ given by an $SL(2, \mathbb{C} )$ analogue of the \textit{holonomy perturbations} which have been employed in several gauge-theoretic contexts to obtain transversality results \cite{donaldson-orientation,floer, taubes-casson,donaldson-braam,donaldson-floer,KM,KMinstanton}.

\subsubsection{Holonomy functions}
%The basic construction is as follows. 
We fix a smooth embedding of a solid torus in $Y$, denoted $q :   S^1 \times D^2 \hookrightarrow Y$, which we regard as a family of parametrised embedded loops $q_z = q(- , z) : S^1 \hookrightarrow Y$ parametrised by $z \in D^2$. For the $U(1)$ connection $\lambda$ on the determinant line bundle $\Lambda = \Lambda^2 E$, taking its holonomy along the loop $q_z$ defines an element $\mathrm{Hol}_{\lambda } (q_z )\in U(1)$ for each $z$. We equip $q$ with a choice of square-root for the holonomy around the central loop: i.e. a fixed element $\mathrm{Hol}^{1/2}_{\lambda}(q_0) \in U(1)$ squaring to $\mathrm{Hol}_{\lambda}(0)$. In turn, this determines a unique smooth map $\mathrm{Hol}_{\lambda}^{1/2} : D^2 \rightarrow U(1)$ such that $\mathrm{Hol}^{1/2}_{\lambda}(q_z )$ squares to $\mathrm{Hol}_{\lambda}(q_z )$ for every $z \in D^2$. (From now one, we shall drop the square-root data from the notation.)

Now, given a $GL(2,\mathbb{C} )$ connection $\mathbb{A} $ on $E$ with fixed determinant $(\Lambda , \lambda )$ we can assign to each loop $q_z$ a well-defined `holonomy' in $SL(2, \mathbb{C} )$ (rather than just in $GL(2, \mathbb{C} )$): 
\begin{align}
\widetilde{\mathrm{Hol}_{\mathbb{A}}} (q_z ) := \mathrm{Hol}_{\mathbb{A}}(q_z ) \cdot ( (\mathrm{Hol}^{1/2}_{\lambda}(q_z ))^{-1} \in SL(E_{q(0,z)} ). \label{holtilde}
\end{align}
More precisely, (\ref{holtilde}) can be regarded as a smooth section, denoted $\widetilde{hol}_\mathbb{A} (q)$, of the pullback bundle $q^\ast SL(E)$ which is $\mathbb{A}$-covariant constant along the loops. (Here $SL(E) = E \times_{\mathrm{Ad}} SL(2, \mathbb{C} )$ is the bundle of Lie groups.) 

Finally, choose a $2$-form $\mu \in \Omega^2 (D^2 )$ on the unit disk $D^2$ which is non-negative, vanishes on a neighborhood of the boundary of $D^2$, and integrates to $1$. 

\begin{definition}
The $SL(2, \mathbb{C})$ \textit{holonomy function} associated to the data $q, \mu$ is the $\mathscr{G}_{E}^c$--invariant complex-valued function on the set $\mathscr{A}_{E}^c$ of complex connections given by 
\begin{align}
\sigma_{q , \mu} (\mathbb{A} ) = \int_{D^2} \mathrm{Tr}\Big( \widetilde{\mathrm{Hol}_{\mathbb{A}}}(q_z) \Big)  \mu (z)  \label{sigmaA} \in \mathbb{C}.
\end{align}
\end{definition}

Let $\pi_E : SL(E) \to \mathfrak{sl}(E) = \mathfrak{g}_E \otimes \mathbb{C} $ be the bundle projection induced by the conjugation-invariant map $\pi : SL(2, \mathbb{C} ) \rightarrow \mathfrak{sl}(2, \mathbb{C} )$ which sends a matrix $M$ to its trace-less part $M - \frac{I}{2} \mathrm{Tr}(M)$. We then define the function
\begin{align*}
T_{q} (\mathbb{A})  = \pi_E \circ \widetilde{\mathrm{Hol}}_\mathbb{A} (q ) \in \Omega^0 ( S^1 \times D^2 ,  q^\ast \mathfrak{sl}(E) ) 
\end{align*}
which again is $\mathbb{A}$-covariant constant over the loops. A short calculation (see e.g. \cite{donaldson-braam,donaldson-floer}) shows that the directional derivatives of $\sigma_q $ are given by 
\begin{align}
(d \sigma_{q, \mu} )_{\mathbb{A}}(\dot{\mathbb{A}}) = -\int_Y \mathrm{Tr}\big( \dot{\mathbb{A}} \wedge T_q (\mathbb{A})\mu ) .\label{dsigma}
\end{align}
%Hence the formal $L^2$ gradient of the real part of %$\sigma_q$ is given by
%\begin{align}
%\mathscr{V}_\mathbb{A} = T_\mathbb{A}
%\end{align}
We set
\begin{align}
V_{q , \mu} (\mathbb{A} ) = T_{q}(\mathbb{A}) \cdot (\ast \mu ) \in \Omega^1 (Y , \mathfrak{sl}(E) ). \label{V}
\end{align}

\subsubsection{Bounds on holonomy functions}

We now discuss the differentiability of the holonomy functions $\sigma_{q,\mu}$ in the $L^{2}_k$ Hilbert manifold structure of the stable orbit spaces, together with various bounds needed ensure that the basic linear analysis and compactness properties of the equations (\ref{eq1}-\ref{eq3}) remain valid after perturbation. The main result we need is:

\begin{proposition}\label{proposition:boundsholonomy}
Fix $q, \mu$ as above, and let $k \geq 2$. For each $n \geq 0$ there is a continuous increasing function $h_n : \mathbb{R} \to \mathbb{R}_{>0}$ such that:
\begin{enumerate}
 \item The function $\sigma_{q, \mu}$ extends to a smooth (in fact, holomorphic) function on the $L^{2}_k$ connections:
\[
\sigma_{q, \mu} : \mathscr{A}_{E , k}^c \to \mathbb{C}.
\]
satisfying the following bounds: for every $n\geq 0$ and $\dot{\mathbb{A}}_{i} \in \Omega^1 (Y , \mathfrak{sl}(E) )_k$
\[
| (\nabla^{(n)} \sigma_{q, \mu} )_\mathbb{A}  ( \dot{\mathbb{A}}_1 , \ldots , \dot{\mathbb{A}}_n ) | \leq h_n  (\| \Phi \|_{L^\infty (Y)} ) \cdot \prod_{i = 1}^n \| \dot{\mathbb{A}}_i  \|_{L^{2}_{k}(Y)}.
\]
\item The function $V_{q, \mu}$ extends to a smooth function
\[
V_{q,\mu}: \mathscr{A}_{E, k}^c \to \Omega^1 (Y , \mathfrak{sl}(E) )_k
\]
i.e. a smooth section of the tangent bundle of $\mathscr{A}_{E , k}^c$, satisfying similar bounds for $n \geq 0$:
%\item the derivative $\mathcal{D} T \in C^\infty ( \mathscr{A}_{E , k}^c , \mathrm{Hom}(\mathscr{T}_k , \mathscr{T}_k ) ) $ extends to a smooth section $\mathcal{D} T \in C^\infty ( \mathscr{A}_{E , k}^c , \mathrm{Hom}(\mathscr{T}_j , \mathscr{T}_j ) )$ for all $j \leq k$
\begin{align*}
\| (\nabla^{(n)} T_{\mathbb{A}})( \dot{\mathbb{A}}_1 , \ldots , \dot{\mathbb{A}}_n ) \|_{L^{2}_k (Y)} \leq h_n  (\| \Phi \|_{L^\infty (Y)} ) \cdot \prod_{i = 1}^n \| \dot{\mathbb{A}}_i  \|_{L^{2}_{k}(Y)} .%& \| T_\mathbb{A} \|_{L^\infty}  \leq h ( \| \Phi\|_{L^\infty} )\\ %& \| T_\mathbb{A} \|_{L_{k}^2}   \leq h ( \| \Phi\|_{L^\infty} ).
\end{align*}
\end{enumerate}
\end{proposition}

\begin{remark}
Since $k \geq 2$ then the $L^{2}_k$ space continuously embeds in $C^0$, and thus the quantity $\| \Phi \|_{L^\infty}$ is defined for $\Phi \in L^{2}_k$. Unlike in the $SU(2)$ case (see \cite{taubes-casson}), when $k = 1$ we are no longer able to ensure the differentiability of of the functions $\sigma_{q , \mu}$ and $V_{q , \mu}$ on the $L^{2}_1$ connections.
\end{remark}

We will establish Proposition \ref{proposition:boundsholonomy} in what follows. First, recall the Frobenius norm on $\mathfrak{gl}(2,\mathbb{C} )$ given by $|M| :=  \sqrt{\mathrm{Tr}(M M^\ast )}$. Some properties of this norm that we shall use are:
\begin{itemize}
\item With the chosen normalization, we have $|M \cdot M^\prime |  \leq \frac{1}{\sqrt{2}} |M| \cdot |M^\prime|$ for any two $M, M^\prime \in \mathfrak{gl}(2, \mathbb{C} )$. %(We recall that with the normalization conventions of this article, the norm squared on the bundle $\mathfrak{g}_E \otimes \mathbb{C} = \mathfrak{sl}(E)$ is taken to be $-\mathrm{Tr} (M M^\ast)$).
\item $| \cdot |$ is invariant under right or left multiplication by $U(2)$. In particular, $| \cdot |$ is well-defined on the associated bundles $SL(E)$ and $\mathfrak{sl}(E)$. 
\item The projection $\pi : \mathfrak{gl}(2,\mathbb{C} ) \to \mathfrak{sl}(2,\mathbb{C} )$ to the trace-less part is an orthogonal projection with respect to the Frobenius inner product, so we have
\begin{align}
|M|^2 = |\pi(M) |^2 + \frac{1}{2} |\mathrm{Tr}M|^2. \label{trace-freepart}
\end{align}
\end{itemize}

We have the following estimate: 

\begin{lemma}\label{lemma:boundsholonomy}
Let $E \to S^1 = \mathbb{R}/\mathbb{Z}$ be a an $SU(2)$ bundle with a smooth $SL(2,\mathbb{C})$ connection $\mathbb{A} = A + i \Phi$. Then the Frobenius norm of the holonomy of $\mathbb{A}$ along $S^1$ satisfies the following bound:
 \[
|\mathrm{Hol} (\mathbb{A} ) |^2 \leq 2 \cdot  \mathrm{exp}\int_{0}^1 \sqrt{2}  |\Phi (t )| dt .
\]
Furthermore, for every $n \geq 0$ there is an increasing continuous function $h_n : \mathbb{R} \to \mathbb{R}_{>0}$ such that for any collection $\dot{\mathbb{A}}_{i}$ ($i = 1, \ldots ,n $) of smooth sections of $\mathfrak{sl}(E) \to S^1$ we have the following bounds on directional derivatives:
\[
| (\nabla^{(n)} \mathrm{Hol} )_\mathbb{A}  (\dot{\mathbb{A}}_1 , \ldots , \dot{\mathbb{A}}_n ) | \leq h_n ( \| \Phi \|_{L^2 (S^1 )} ) \cdot \prod_{i = 1}^n \| \dot{\mathbb{A}}_i \|_{L^2 (S^1 )} .
\]
In particular, it follows that the holonomy along $S^1$ extends to a smooth function on the space $\mathscr{A}_0 (S^1)$ of $L^2$-regularity $SL(2,\mathbb{C})$ connections on $E \to S^1$:
\[
\mathrm{Hol} : \mathscr{A}_0 (S^1 ) \to SL(2, \mathbb{C} ).
\]
\end{lemma}
\begin{proof}
 We fix once and for all a trivialization of the pullback of $E$ along $[0,1] \to S^1$ which is $A$-covariant constant. Thus, the holonomy $\mathrm{Hol}(\mathbb{A}) \in SL(2, \mathbb{C} )$ is given as
\[
\mathrm{Hol}(\mathbb{A} ) = \xi (1) \cdot \mathrm{Hol} (A)
\]
where $\xi : [0,1]\to SL(2, \mathbb{C}) \subset \mathfrak{gl}(2,\mathbb{C} )$ is the unique solution to the initial value problem
\begin{align}
\big( \frac{\partial}{\partial t} + i \Phi (t ) \big) \cdot \xi(t) = 0 , \quad \xi(0) = I .\label{IVP}
\end{align}
Now, observe that
\begin{align*}
\frac{\partial}{\partial t} | \xi (t)|^2 & = \mathrm{Tr}\big( \frac{ \partial \xi}{\partial t} \xi^\ast + \xi (\frac{\partial \xi}{\partial t})^\ast \big)\\
& = - 2 \mathrm{Tr}( i \Phi(t ) \cdot \xi \xi^\ast ) \text{  using (\ref{IVP}) and } \Phi^\ast = -\Phi\\
%& = - 2  \mathrm{Tr}\big( i \Phi(\partial_t ) \cdot ( \xi \xi^\ast - \frac{I}{2} \mathrm{Tr}(\xi \xi^\ast )) \big) \text{  using  } \mathrm{Tr}\Phi(\partial_t ) = 0 \\
%&  \leq  2 |\Phi (\partial_t )|\cdot | \xi \xi^\ast - \frac{I}{2} \mathrm{Tr}(\xi \xi^\ast )|
& \leq  2 |\Phi (t )|\cdot  |\xi \xi^\ast |  \leq \sqrt{2}|\Phi (t )|\cdot  |\xi  (t) |^{2} .
\end{align*}
Integrating, we obtain: for all $t \in [0,1]$
\begin{align}
| \xi (t) |^2 \leq 2 \cdot \mathrm{exp} \int_{0}^t \sqrt{2} | \Phi (x) | dx . \label{xi(t)}
\end{align}

The required bounds on $\mathrm{Hol}(\mathbb{A} )$ now follow from (\ref{xi(t)}) with $t = 1$. For the bounds on the higher directional derivatives we proceed as follows. Similarly as above, the holonomy of the connection $\mathbb{A} + \sum_{i =1}^n s_i \dot{\mathbb{A}}_i $ is given as $\xi_s (1) \cdot \mathrm{Hol}(A)$ where, for each $s \in \mathbb{R}^n$, $\xi_s : [0,1] \to SL(2, \mathbb{C} )$ is the unique solution of the initial value problem
\begin{align}
\big( \frac{\partial}{\partial t} + i \Phi (t) + \sum_{i=1}^n s_i \dot{\mathbb{A}}(t) \big) \cdot \xi_s (t) = 0 , \quad \xi_s (0) = I .\label{hols}
\end{align}
Define the following $\mathfrak{gl}(2,\mathbb{C})$-valued functions on $[0,1]$
\[
X(t) : =  ( \frac{\partial}{\partial s_1 }  \cdots \frac{\partial}{\partial s_n } \xi_s (t )  )|_{s = 0}  \quad , \quad Y_i (t) := (\frac{\partial}{\partial s_1 } \cdots \widehat{\frac{\partial}{\partial s_i } } \cdots  \frac{\partial}{\partial s_n } \xi_s (t) )|_{s = 0} .
\]
and note that 
\[
(\nabla^{(n)} \mathrm{Hol} )_\mathbb{A}  (\dot{\mathbb{A}}_1 , \ldots , \dot{\mathbb{A}}_n ) = X(1) \cdot \mathrm{Hol} (A) \quad , \quad (\nabla^{(n-1)} \mathrm{Hol} )_\mathbb{A}  (\dot{\mathbb{A}}_1 , \ldots , \widehat{ i} , \ldots , \dot{\mathbb{A}}_n )  = Y_i (1) \cdot \mathrm{Hol}(A).\]

Differentiating (\ref{hols}) in all the $s_i$ variables and setting $s_i = 0$
yields the identity:
\begin{align}
( \frac{\partial }{\partial t} + i \Phi (t) ) \cdot X(t)  + \sum_{i = 1}^n  \dot{\mathbb{A}}(t) \cdot Y_i (t) = 0 . \label{diffIVP}
\end{align}
$X(t)$ is the unique solution of the initial value problem given by (\ref{diffIVP}) and $X(0) = 0$. A short calculation gives the closed formula
\begin{align}
X(t) = -  \xi (t) \cdot  \sum_{i = 1}^n \int_{0}^t \xi(x)^{-1} \cdot \dot{\mathbb{A}}_i (x) \cdot Y_i (x) dx .\label{X}
\end{align}

Now, it is easy to see that for an invertible matrix $M \in GL(2, \mathbb{C} )$ the Frobenius norm of its inverse satisfies the bound
\begin{align}
| M^{-1}| \leq \sqrt{2} \frac{ |M|}{|\mathrm{det} M | }.\label{norminverse} 
\end{align}
Combining (\ref{xi(t)}), (\ref{norminverse}) and (\ref{X}) yields the following estimate:
\begin{align*}
| X(t) | %& \leq \frac{1}{2} | \xi(t)| \cdot \sum_{i = 1}^n \int_{0}^t | \xi(x)| | \dot{\mathbb{A}}_i (t) | \cdot |Y_i (x) | dx  \\
& \leq \big( \mathrm{exp} \int_{0}^t \sqrt{2} | \Phi (x) | dx \big)  \cdot \sum_{i = 1}^n \int_{0}^t  | \dot{\mathbb{A}}_i (x) | \cdot |Y_i (x) |dx .
\end{align*}
The required bounds on the directional derivatives of the holonomy map now follow inductively from this estimate.
%\[
%|\mathrm{Hol}_\mathbb{A} |^2 = |\xi (1)|^{2} \leq2 \cdot \mathrm{exp} \int_{0}^1 \sqrt{2}|\Phi_{q(t,z)} ( \partial_t )|_F dt .
%\]
\end{proof}

\begin{proof}[Proof of Proposition \ref{proposition:boundsholonomy}]

When $\mathbb{A}$ and the $\dot{\mathbb{A}}_{i}$ are smooth then the bounds from assertion (1) follow from Lemma \ref{lemma:boundsholonomy}, noting the fact that the $L^\infty$ and $L^{2}_k$ norms on $Y$ with $k\geq 2$ control the $L^2$ norm on the restriction to any loop $q_z : S^1 \hookrightarrow Y$. In turn, such bounds imply the smoothness of $\sigma_{q,\mu}$ as a function on the $L^{2}_k$ connections.

%For a smooth connection $\mathbb{A}$ the directional derivatives of $\sigma_{q,\mu}$ along a collection of smooth sections $\dot{\mathbb{A}}_i \in \Omega^1 (Y , \mathfrak{sl}(E))$ are given as follows: identifying $q^\ast E$ over $[0,1]\times D^2$ with the product bundle $[0,1]\times ( q^\ast E )|_{0 \times D^2 }$, we regard the connections $q^\ast (\dot{\mathbb{A}}_i )$ along $[0,1]\times \{ z \}$ as $\mathfrak{sl}(E_{q(0,z)})$-valued paths, and then
%\[
%(\nabla^{(n)} \sigma_{q, \mu } )_{\mathbb{A}} ( \dot{\mathbb{A}}_1 , \ldots , \dot{\mathbb{A}}_n ) = (-1)^n \int_{D^2} \mathrm{Tr}\Big\{ \big( \prod_{i = 1}^n \int_{0}^1 q_{z}^\ast (\dot{\mathbb{A}}_{i} )  \big) \cdot \mathrm{Hol}_{\mathbb{A}} \Big\} \cdot \mu (z) .
%\]

%Thus, from this and Lemma \ref{lemma:boundsholonomy} we have the following estimate: for each $n \geq 0$ there is a continuous increasing function $h_n : \mathbb{R}\to \mathbb{R}_{>0}$ such that for all $k \geq 2$
%\begin{align*}
%|(\nabla^{(n)} \sigma_{q, \mu } )_{\mathbb{A}} ( \dot{\mathbb{A}}_1 , \ldots , \dot{\mathbb{A}}_n )| \leq  h_n ( \|\Phi \|_{L^\infty (Y)}) \cdot \prod_{i =1}^n \| \dot{\mathbb{A}}_i \|_{L^\infty (Y)} \leq h_n ( \|\Phi \|_{L^\infty (Y)}) \cdot \prod_{i =1}^n \| \dot{\mathbb{A}}_i \|_{L^{2}_{k} (Y)}.
%\end{align*}
%This estimate implies the smooth extension of $\sigma_{q,\mu}$ over to $\mathscr{A}_{E , k}^c$.

For the smoothness and bounds on $V_{q , \mu}$ we consider a variant of the argument in \cite[Proof of Proposition 7]{kronheimer-higher}. Namely, let $\mathscr{H}$ stand for the Hilbert bundle over $D^2$ with fibers
\[
\mathscr{H}_z = L^2 (S^1  , ( q^\ast \mathfrak{sl}(E) )|_{S^1 \times z}), \quad z \in D^2.
\]
Fixing a smooth unitary connection $A_0$ on $E$ identifies $\mathscr{H}_z$ as the space of $L^2$ connections along the loop over $z\in D^2$. Furthermore, $A_0$ yields a canonical connection on $\mathscr{H}$, which we use to define Sobolev spaces of sections of $\mathscr{H}$. By Lemma \ref{lemma:boundsholonomy}, we have a (non-linear) smooth bundle map given by taking holonomy:
\[
\underline{\mathrm{Hol}}:  \mathscr{H}  \to  (q^\ast \mathfrak{gl}(E))|_{0 \times D^2} 
\]
This induces a smooth map on $L^{2}_k$ sections
\begin{align}
\mathrm{Hol}: L^{2}_k ( D^2 , \mathscr{H}) \to L^{2}_k (D^2 , (q^\ast \mathfrak{gl}(E))|_{0 \times D^2 } )\label{mapsections}
\end{align}
which by Lemma \ref{lemma:boundsholonomy} satisfies bounds of the form: for an $L^{2}_k$ section $s$ of $\mathscr{H}$, written in components as $s = (s_A , s_\Phi )$, and $k\geq 2$
\[
\| (\nabla^{(n)} \mathrm{Hol})_s  (\dot{s}_1 , \ldots , \dot{s}_n ) \|_{L^{2}_k (D^2 )} \leq h_n ( \| s_\Phi \|_{L^{\infty}(D^2 )} ) \cdot \prod_{i =1}^n \| \dot{s}_{i}\|_{L^{2}_k (D^2 )}. 
\]
%Since $k \geq 2$ then $L^{2}_k (D^2 , \mathscr{H}) \subset C^0 (D^2 , \mathscr{H})$. Thus, by the bounds from Lemma \ref{lemma:smoothholonomy} we can find a continuous increasing function $h : \mathbb{R} \to \mathbb{R}_{>0}$ such that for all $s \in L^{2}_k (D^2 , \mathscr{H})$ we have bounds for $j = 0 , \ldots , k$
%\begin{align}
%\| \mathrm{Hol}(s) \|_{L^{2}_j (D^2)} \leq h (\| s_\Phi \|_{L^{\infty}(D^2 , \mathscr{H})} ) \cdot \| s \|_{L_{j}^{2}(D^2 , \mathscr{H})}\label{boundmapsections}
%\end{align}
%where $s = (s_A , s_\Phi )$ is the decomposition into real and imaginary parts.
An $L^{2}_k$ connection $\mathbb{A}$ on $Y$ gives, by restriction to each loop, an $L^{2}_k$ section of $\mathscr{H}$, from which we obtain $T_\mathbb{A}$ by applying the map $\mathrm{Hol}$, then applying the projection $\pi : \mathfrak{gl}(2,\mathbb{C} ) \to \mathfrak{sl}(2,\mathbb{C})$ and then multiplying with $\ast \mu$. Thus, it follows that $T_\mathbb{A} $ lies in $L^{2}_k$ and defines a smooth function on the $L^{2}_k$ connections satisfying the required bounds.
 %we obtain $V_\mathbb{A} \in L^{2}_k$ by the following steps:
%\begin{enumerate}
%\item Restrict $\mathbb{A}$ along each loop to obtain a section $s$ in $L^{2}_k (D^2 , \mathscr{H} )$. This satisfies the bounds $\| s\|_{L^{2}_j (D^2 , \mathscr{H} )} \leq \| \mathbb{A} - A_0 \|_{L^{2}_j (Y)}$ for $j = 0 , \ldots , k$.
%\item Apply $\mathrm{Hol}$ to $s$. Then $\mathrm{Hol}(s)$ satisfies the bound (\ref{boundmapsections}).
%\item Projecting to the trace-less part and multiplying by $\ast \mu$ gives a bounded linear map $L^{2}_j (D^2 , (q^\ast \mathfrak{gl}(E))|_{0 \times D^2 } ) \to L^{2}_j (Y ,  \mathfrak{sl}(E)) )$ for all $j$. Apply this map to $\mathrm{Hol} (s)$.
%\end{enumerate}
%It follows from the above that $\mathbb{A} \mapsto V_\mathbb{A}$ is a smooth map and satisfies the required bound on the $L^{2}_j$ norm for $j = 0 , \ldots , k $. The $L^\infty$ bound on $V$ follows Lemma \ref{lemma:boundsholonomy}.
\end{proof}

\subsubsection{Separating properties}

The crucial property of the holonomy functions $\sigma_q$ that we shall use to obtain transversality results is the following:

\begin{proposition}\label{proposition:embedding}
    Let $\mathscr{Z} \subset \mathscr{B}_{E}^\ast$ be a compact finite-dimensional submanifold. Then there exists a finite collection of disjoint smooth embeddings $q_1 , \ldots , q_{N}$ (together with corresponding square-roots and $2$-forms), such that the map
    \[
    \underline{\sigma} := (\sigma_{q_1} , \ldots , \sigma_{q_N}) : \mathscr{B}_{E}^\ast \rightarrow \mathbb{C}^N
    \]
    induces a smooth embedding of $\mathscr{Z}$ into $\mathbb{C}^N$, given by the restriction $\underline{\sigma}|_\mathscr{Z} : \mathscr{Z} \hookrightarrow \mathbb{C}^N $.
\end{proposition}

Before establishing this, we now recall various facts about character schemes that we shall need (we refer to \cite{culler-shalen,sikora,free} for further details). The $SL(2, \mathbb{C})$--\textit{character scheme} of a finitely-presented group $\Gamma$ is the affine scheme (over $\mathbb{C}$) given by the affine GIT quotient of the $SL(2,\mathbb{C})$--representation scheme $\mathrm{Hom}(\Gamma , SL(2,\mathbb{C})$ (an affine scheme) by the conjugation action of $SL(2,\mathbb{C})$:
\[
\mathscr{M} (\Gamma , SL(2,\mathbb{C} ) ):= \mathrm{Hom}(\Gamma , SL(2,\mathbb{C} ) ) \sslash SL(2,\mathbb{C}) .
\]
The (closed) points of the character scheme $\mathscr{M}(\Gamma , SL(2, \mathbb{C} )$ correspond one-to-one with the $SL(2,\mathbb{C} )$--orbits of polystable representations.  A representation $\rho : \Gamma \rightarrow SL(2,\mathbb{C} )$ is \textit{polystable} (or completely-reducible) if every $\rho$-invariant line $L \subset \mathbb{C}^2$ admits a $\rho$-invariant complementary line $L^\prime \subset \mathbb{C}^2$, and it is \textit{stable} (or irreducible) if there exist no $\rho$-invariant line $L \subset \mathbb{C}^2$.

When $\Gamma = F_M $ is the \textit{free group} on $M$ letters the character variety is well-understood. In this case, the non-singular locus in $\mathscr{M}(\Gamma , SL(2, \mathbb{C} ) )$ coincides with the locus of stable orbits \cite[\S 5]{free}. The set of stable representations of $F_M$ is explicitly described as the the open set $( SL(2, \mathbb{C})^{\times M})_\ast \subset SL(2, \mathbb{C})^{\times M}$ of $M$-tuples of matrices which leave no line in $\mathbb{C}^2$ simultaneously invariant. Thus
\[
\mathscr{M}(F_M , SL(2, \mathbb{C}) )^{non-sing} = (SL(2, \mathbb{C})^{\times M})_\ast / SL(2, \mathbb{C} ).
\]

For a finitely-presented group $\Gamma$, the character scheme $\mathscr{M}(\Gamma , SL(2,\mathbb{C} )$ can be embedded as an affine subscheme of some affine space $\mathbb{C}^N$ via the embedding which sends a conjugacy class of representation $[\rho]$ to its trace on a suitable collection of $N$ elements of $\Gamma$ \cite[Corollary 1.4.5]{culler-shalen}. When $\Gamma = F_M$ it suffices to use the collection of words given by 
\[
X_{i_1}\cdots X_{i_M}
\]
where $X_1, \dots , X_M$ stand for the $M$ letters generating $F_M$, and the indices $i_1 , \ldots , i_r  $ run over all possible collections of distinct positive integers $\leq M$ (see the proof of \cite[Proposition 1.4.1]{culler-shalen}).

\begin{proof}[Proof of Proposition \ref{proposition:embedding}]
%Let $(SL(2, \mathbb{C})^{\times M})_\ast$ denote the open subset of $SL(2, \mathbb{C})^{\times M}$ consisting of $M$--tuples of matrices in $SL(2,\mathbb{C})$ which leave no line in $\mathbb{C}^2$ simultaneously invariant. This locus is acted on by $SL(2, \mathbb{C} )$ by simultaneous conjugation on each factor, and this action is proper and has stabilizer $\pm I$ at every point. Hence, the quotient $(SL(2, \mathbb{C})^{\times M})_\ast / SL(2, \mathbb{C} )$ is naturally a smooth manifold. 

Recall that a stable connection $\mathbb{A} \in \mathscr{C}_{E}^\ast $ on $E$ leaves no complex line subbundle $L \subset E$ invariant under parallel transport, by Proposition \ref{proposition:invariant}. In particular, one can find a finite collection of disjoint embedded loops $\underline{\gamma} = (\gamma_1 , \ldots , \gamma_M )$ at a given basepoint $x \in Y$ (equipped with square-roots) such that the associated holonomy map
\begin{align}
\widetilde{\mathrm{Hol}}(\underline{\gamma})  : \mathscr{B}_{E,k}^\ast \rightarrow  (SL(2, \mathbb{C})^{\times M}) / SL(2, \mathbb{C} ) \label{map1}
\end{align}
maps the compact subset $\mathscr{Z} \subset \mathscr{B}_{E}^\ast$ into the manifold locus $(SL(2, \mathbb{C})^{\times N})_\ast / SL(2, \mathbb{C} )$. 

As noted in \cite[\S 5.5.1]{donaldson-floer}, it is a basic fact that a connection $\mathbb{A}$ is determined up to gauge-equivalence by its holonomy. In particular:
\begin{itemize}
\item given two connections $\mathbb{A}_0 , \mathbb{A}_1 \in \mathscr{A}_{E }^c$ that are not in the same $\mathscr{G}_{E}^c$--orbit, there exists $\underline{\gamma}$ such that $\widetilde{\mathrm{Hol}}(\underline{\gamma})$ takes on distinct values at $\mathbb{A}_0$ and $\mathbb{A}_1$, and
\item given a conection $\mathbb{A} \in \mathscr{A}_{E }^c$ and a tangent vector $\dot{\mathbb{A}} \in T_\mathbb{A} \mathscr{A}_{E }^c$ not lying in the tangent space to the $\mathscr{G}_{E}^c$--orbit at $\mathbb{A}$, there exists $\underline{\gamma}$ such that the derivative of $\widetilde{\mathrm{Hol}} (\underline{\gamma} )$ along $\dot{\mathbb{A}}$ is non-trivial.
\end{itemize}
Thus, because $\mathscr{Z}$ is a compact submanifold of $\mathscr{B}_{E}^\ast$, then one can choose the loops $\underline{\gamma} $ so that the map (\ref{map1}) yields a \textit{smooth embedding} of $\mathscr{Z}$ into the manifold locus $(SL(2, \mathbb{C})^{\times M})_\ast / SL(2, \mathbb{C} )$. 

Next, fix an embedding of $\mathscr{M}(F_M , SL(2, \mathbb{C} ) )$ as an affine subscheme of some $\mathbb{C}^N$. In particular, this restricts to a smooth embedding of the manifold locus
\begin{align}
(SL(2, \mathbb{C})^{\times M})_\ast / SL(2, \mathbb{C} ) \hookrightarrow \mathbb{C}^N \label{map2} .
\end{align}
The composition of the maps (\ref{map1}) and (\ref{map2}) gives a map $\mathscr{B}_{E}^\ast \rightarrow \mathbb{C}^N$ inducing a smooth embedding of $\mathscr{Z}$ into $\mathbb{C}^N$. We then modify this map as follows:
\begin{itemize}
\item Slightly deform the $N$ loops in $Y$ obtained as the concatenations $\gamma_{i_1} \cdots \gamma_{i_M}$ to obtain $N$ \textit{disjoint} embedded loops 

\item Thicken the $N$ disjoint embedded loops from the previous step to obtain disjoint embeddings $q_i : S^1 \times D^2 \hookrightarrow Y$ of solid tori neighborhoods with sufficiently concentrated $2$-forms $\mu_i$.
\end{itemize}
It is clear that the map $\underline{\sigma} : \mathscr{B}_{E }^\ast \rightarrow \mathbb{C}^N$ resulting from this procedure will still induce a smooth embedding of $\mathscr{Z}$ into $\mathbb{C}^N$, as required. 
\end{proof}

\subsection{Perturbing the Chern--Simons functional}

For a closed oriented $3$-manifold, consider the functional $S = \mathrm{Re}(e^{-\theta} CS)$ and the perturbed functional
\begin{align}
 S_\varepsilon (A , \Phi ) = S(A, \Phi ) + \varepsilon \cdot \| \Phi \|_{L^2}^2 + \varepsilon^2 \cdot f ( \underline{\sigma}(A, \Phi ) ).\label{perturbedS}
\end{align}
Here, the function $\underline{\sigma} = (\sigma_{q_1} , \ldots , \sigma_{q_N } )$ is constructed by using a collection of disjoint embedded solid tori $q_i$, square-root data and bump forms, and $f : \mathbb{C}^N \to \mathbb{R}$ is a smooth function.

Acting on the space $\mathscr{A}_{E}^c$ of complex connections, the function $S_\varepsilon$ is still invariant under the unitary gauge group $\mathscr{G}_{E}$, but no longer under the complex gauge group $\mathscr{G}_{E }^c$. Thus, when $\theta = \pi/2$ so that there are no periods of $S$ under the $\mathscr{G}_E$--action, the perturbation $S_\varepsilon$ descends to a well-defined function on the stable orbit space $\mathscr{B}_{E , k}^\ast = \mathscr{C}_{E , k}^\ast / \mathscr{G}_{E , k+1}$ (again, setting $\theta = \pi/2$ is not required; for other $\theta$ the $1$-form $dS_\varepsilon$ descends to $\mathscr{B}_{E}^\ast$, which is all that we need).

We now establish an infinite-dimensional analogue of Theorem \ref{theorem:perturbation2} for the perturbed functional $S_\varepsilon : \mathscr{B}_{E , k}^\ast \to \mathbb{R}$. The proof will be formally identical to the finite-dimensional case, once we have verified that the basic linear analysis and compactness results carry over to the perturbed setting.
\subsubsection{Perturbed Hessian}

First, Proposition \ref{proposition:tangent} holds with $S$ replaced with the perturbed functional $S_\varepsilon$ from (\ref{perturbedS}). Indeed, the formal $L^2$ gradient of the functional $\| \Phi \|_{L^2}^2$ on the affine space of complex connections $\mathscr{A}_{E}^c$ is $\mathrm{grad}\| \Phi \|_{L^2}^2 = ( 0 , 2 \Phi )$, which is tangent to the stable locus $\mathscr{C}_{E}^\ast \subset \mathscr{A}_{E}^c$. On the other hand, the formal $L^2$ gradient of $f ( \underline{\sigma} (A, \Phi ) )$ is tangent to $\mathscr{C}_{E}^\ast$ because this function is invariant under the \textit{complex} gauge group $\mathscr{G}_{E}^c$ (see the Remark to Proposition \ref{proposition:tangent}).

Proposition \ref{proposition:hessian_description} and Theorem \ref{theorem:fredholm} then carry through in the perturbed setting, with $S_\varepsilon$ in place of $S$. Indeed, denote the formal Hessian of $S_\varepsilon : \mathscr{B}_{E, k }^\ast \to \mathbb{R}$ (taken with respect to the $L^2$ metric) by 
\[
\mathscr{D}_{(A, \Phi )}^\varepsilon : \mathscr{K}_{(A, \Phi ) , k } \to \mathscr{K}_{(A, \Phi ) , k-1} .
\]
Then this is a \textit{compact} perturbation of the unperturbed Hessian $\mathscr{D}_{(A, \Phi )}$. Indeed,
\[
\mathscr{D}_{(A, \Phi )}^\varepsilon = \mathscr{D}_{(A, \Phi )} + \varepsilon \mathcal{D} [\mathrm{grad}\| \Phi\|_{L^2}^2 ] + \varepsilon^2 \mathcal{D}[\mathrm{grad} f(\underline{\sigma} )]
\]
since $[\mathrm{grad}\| \Phi\|_{L^2}^2]$ and $[\mathrm{grad} f(\underline{\sigma} )]$ (the formal $L^2$ gradients of $\| \Phi\|_{L^2}^2$ and $f ( \underline{\sigma} )$ on $\mathscr{B}_{E , k}^\ast$ with respect to the $L^2$ metric) are sections of $\mathscr{K}_{k}$ (the $L^{2}_k$ completion of the tangent bundle of $\mathscr{B}_{E, k}^\ast$) rather than just $\mathscr{K}_{k-1}$ (cf. Proposition \ref{proposition:boundsholonomy}), and hence their vertical derivatives $[\mathrm{grad}\| \Phi\|_{L^2}^2]$ and $\mathcal{D} [\mathrm{grad} f ( \underline{\sigma} )]$ are compact operators from $\mathscr{K}_{k}$ to $\mathscr{K}_{k-1}$.

\begin{remark}Proposition \ref{proposition:pluriharmonic} no longer holds for $S_\varepsilon$. In fact, the function $\| \Phi \|^2_{L^2}$ on $\mathscr{B}_{E ,k}^\ast$ is \textit{strictly plurisubharmonic}: for the flat Hessian $\mathrm{Hess} \| \Phi \|_{L^2}^2 $ and  $v ,w  \in \mathscr{K}_{(A, \Phi ), k}$ we now have a \textit{positive} symmetric form
\[
\langle (\mathrm{Hess} \| \Phi \|_{L^2}^2 ) Jv , ( \mathrm{Hess} \| \Phi \|_{L^2}^2 )Jw \rangle_{L^2} + \langle (\mathrm{Hess} \| \Phi \|_{L^2}^2 )v , (\mathrm{Hess} \| \Phi \|_{L^2}^2 )w \rangle_{L^2}  = 2\langle v , w \rangle_{L^2}  .
\]
The perturbation $f( \underline{\sigma} )$ can be taken strictly plurisubharmonic as well, by requiring $f : \mathbb{C}^N \to \mathbb{R}$ to be strictly plurisubharmonic (since $\underline{\sigma}$ is holomorphic), and in this case the whole functional $S_\varepsilon$ will be strictly plurisubharmonic on $\mathscr{B}_{E, k}^\ast$.
\end{remark}

We now make some comments about the spectral flow of families of Fredholm operators (for the definition we refer to \cite{APS,taubes-casson,KM}). Consider a continuous paths of self-adjoint Fredholm operators on a real Hilbert: $L_t : H \to H$, $t \in [0,1]$. We will only need to consider the following situation:
\begin{itemize}
\item $H$ is equipped with an isometric linear complex structure $J$ (hence $H$ is a complex Hilbert space)
\item For each $i = 0 , 1$, $L_i$ satisfies either $\langle L_i  Jv , L_i Jw \rangle + \langle L_i v , L_i w \rangle = 0$ for any $v,w \in H$ (hence, $\mathrm{Ker}L_0 \subset (H , J)$ is a $\mathbb{C}$-linear finite-dimensional subspace), or else $L_i$ is invertible.
%\item Either $L_1$ is invertible, or else $L_1$ satisfies the same property as $L_0$.
\end{itemize}
In this situation, there is associated to the path $L_t$ a well-defined \textit{mod 2 spectral flow}:
\[
\mathbf{sf}_{2} (L_t ) \in \mathbb{Z}/2 .
\]
The mod $2$ spectral flow has the following basic properties:
\begin{itemize}
\item it only depends on the underlying homotopy class of the path $L_t$
\item it behaves additively with respect to composition of paths
\item it vanishes if the kernel of $L_t$ is a $\mathbb{C}$-linear subspace of $(H , J )$ for every $t \in [0,1]$. 
\end{itemize}

\subsubsection{$L^\infty$ bounds on $\Phi$ and compactness}

It is well-known that on the moduli space of polystable $GL(2,\mathbb{C})$ projectively flat $GL(2, \mathbb{C})$ connections $\mathscr{M}_{E}$ on $(Y, g )$ the function $\| \Phi \|_{L^2}$ is \textit{proper}--thus controls the non-compactness of $\mathscr{M}_{E}$ \cite{taubes-flat}. In particular, since the $L^\infty$ norm is stronger than the $L^2$ norm, then the function $\| \Phi \|_{L^\infty }$ is proper on $\mathscr{M}_{E }$ (note: since $k \geq 2$ then $\Phi$ is continuous; thus its $L^\infty$ norm is defined). 

The corresponding \textit{perturbed} moduli space $\mathscr{M}_{E}^{\varepsilon} \subset \mathscr{B}_{E,k}$ is defined as the $\mathscr{G}_{E,k+1}$--orbits of polystable configurations $(A, \Phi )$ (i.e. $d_{A}^\ast \Phi = 0$) which are critical points of the $\mathscr{G}_{E,k+1}$--invariant function $S_\varepsilon$ acting on $\mathscr{A}_{E, k}^c$. The stable locus $\mathscr{M}_{E}^{\ast , \varepsilon}$ in $\mathscr{M}_{E}^\varepsilon$ is given by the critical points of $S_\varepsilon$ on the Hilbert manifold of stable orbits $\mathscr{B}_{E,k}^\ast = \mathscr{C}_{E, k}^\ast / \mathscr{G}_{E , k+1}$.

%However, in our perturbed setting we will only show that the function $\| \Phi \|_{L^\infty}$ is proper on $\mathscr{M}_{E }^\varepsilon$. (Note that since $k \geq 2$ then $\Phi$ is continuous, so its $L^\infty$ norm is finite).

\begin{proposition}\label{proposition:compactness}
For any constant $\varepsilon_0 >0$, the function $\psi (A, \Phi ):= \| \Phi\|_{L^\infty }$ is a proper function on the union of the perturbed moduli spaces $\mathscr{M}_{E}^\varepsilon \subset \mathscr{B}_{E,k}$ over all $\varepsilon$ with $0 \leq |\varepsilon| \leq \varepsilon_0 $. \end{proposition}
\begin{proof}
The perturbed equations are:
\begin{align}
 \ast ( F_{A} - \frac{1}{2} [\Phi , \Phi ] ) &= \varepsilon \Phi + \varepsilon^2 \mathfrak{q}_2 (A, \Phi ) \label{perteq1}\\
 \ast d_{A}\Phi  &= \varepsilon^2 \mathfrak{q}_1 (A, \Phi ) \label{perteq2}\\
 d_{A}^\ast \Phi &= 0 \label{perteq3}
\end{align}
where $(\mathfrak{q}_1 , \mathfrak{q}_2 )$ are the components of the vector field $\mathrm{grad} f(\underline{\sigma} )$ on $\mathscr{A}_{E}^c$.

For a given configuration $(A, \Phi )$ a short calculation with the Bochner--Weitzenböck formula yields the following estimate (see \cite[Equation 2.4]{taubes-flat}), where $c_0$ is a constant depending on the Ricci tensor of the metric:
\begin{align*}
\| \nabla_A \Phi \|_{L^2}^2 + \frac{1}{4}\| [\Phi , \Phi ]\|_{L^2}^2 + \| F_A \|_{L^2}^2 \leq c_0 \| \Phi \|_{L^2}^2 + \int_{Y}| F_A -\frac{1}{2}[\Phi , \Phi ] |^2 + |d_{A}\Phi |^2 + |d_{A}^\ast \Phi |^2 ;
\end{align*}
Suppose $(A, \Phi )$ solves (\ref{perteq1}-\ref{perteq3}): then we obtain 
\begin{align}
\|\nabla_A \Phi \|_{L^2}^2 + \frac{1}{4}\| [\Phi , \Phi ]\|_{L^2}^2 + \| F_A \|_{L^2}^2 \leq ( c_0 + 2\varepsilon^2 ) \| \Phi \|_{L^2}^2 + \varepsilon^4 \| \mathfrak{q}_{1}(A, \Phi )\|_{L^2}^2 + 2\varepsilon^4 \| \mathfrak{q}_{2}(A, \Phi ) \|_{L^2}^2 . \label{boundscurv}
\end{align}

Suppose now we are given bounds
\begin{align}
\| \Phi \|_{L^\infty} \leq C \quad , \quad |\varepsilon|\leq \varepsilon_0 \label{boundsphieps}
\end{align}
Then Proposition \ref{proposition:boundsholonomy} gives an $L^\infty$ bound on $\mathfrak{q}_1 (A, \Phi )$ and $\mathfrak{q}_2 (A , \Phi )$. Thus (\ref{boundscurv}) gives an $L^2$ curvature bound 
\begin{align*}
\| F_A \|_{L^2}^2 \leq C^\prime \label{curvaturebound} .
\end{align*}
Given this, it is a routine application of Uhlenbeck's Local Gauge-Fixing Theorem \cite{uhlenbeck}, together with standard elliptic bootstrapping and patching arguments (see e.g. in \cite[\S 2,\S 4]{DK}), to show that a sequence $(A_n , \Phi_n )$ of $L^{2}_k$ solutions of (\ref{perteq1}-\ref{perteq3}) with uniform bounds (\ref{boundsphieps}) satisfies uniform $L^{2}_{k+1}$ bounds after applying a sequence of global $L^{2}_{k+1}$ unitary gauge transformations. Thus, $(A_n , \Phi_n )$ converges in $L^{2}_k$ after passing to a subsequence. This completes the proof.
\end{proof}

\subsubsection{The Perturbation and Localisation Theorems for the Complex Chern--Simons functional}

With the above understood, we are now in good shape to discuss the infinite-dimensional version of Theorems \ref{theorem:perturbation1} and \ref{theorem:perturbation2}.

The setup for this results is the following. We let $(Y, g )$ be a closed, oriented, Riemannian $3$-manifold, equipped with a $U(2)$ bundle $E \to Y$ with fixed determinant. Let $q_1 , \ldots , q_N : S^1 \times D^2 \hookrightarrow Y$ be a finite collection of disjoint smooth embeddings (equipped with `square roots
' and $2$-forms), and let $\underline{\sigma} : \mathscr{B}_{E , k}^\ast \rightarrow \mathbb{C}^N$ denote the associated holonomy function (and we choose $k \geq 2$). Let $f : \mathbb{C}^N \rightarrow \mathbb{R}$ be a smooth function. Consider, for each $\varepsilon \in \mathbb{R} \setminus 0$, the smooth functional $S_\varepsilon  : \mathscr{B}_{E , k}^\ast \rightarrow \mathbb{R}$ given by
\begin{equation*}
S_\varepsilon (A , \Phi ) = S(A, \Phi ) + \varepsilon \| \Phi \|_{L^2 (Y)}^2 + \varepsilon^2 f (\underline{\sigma} (A, \Phi ) ).
\end{equation*}

We make the following assumptions:
\begin{enumerate}
%\item $S^1 : B \rightarrow \mathbb{R}$ is an exhaustive function (i.e. proper and bounded below)
\item[(A1)] The moduli space of stable flat connections $\mathscr{M}_{E }^\ast = \mathrm{Crit}(S)$ is Morse--Bott (cf. Definition \ref{definition:morse-bott})

\item[(A2)] The restriction of the function $ \| \Phi \|_{L^{2}  }^2$ onto $\mathscr{M}_{E }^\ast$ is a Morse--Bott function; denote its critical locus by $\mathscr{Z} = \mathrm{Crit}( \| \Phi \|^{2}_{L^{2}}|_{\mathscr{M}_{E}^\ast } )$

\item[(A3)] $\mathscr{Z}$ is compact

\item[(A4)] The function $\underline{\sigma} : \mathscr{B}_{E , k }^\ast \rightarrow \mathbb{C}^N$ restricts to a smooth embedding of $\mathscr{Z}$ (cf. Proposition \ref{proposition:embedding})

\item[(A5)] The function $f$ restricts `generically' onto $\mathscr{Z}$: the smooth function on $\mathscr{Z}$ given by $\mathfrak{f} := \frac{1}{2} (d\| \Phi \|_{L^2}^2 ) ( \lambda ) + (f \circ \underline{\sigma} )|_{\mathscr{Z}}$ is Morse; here $\lambda$ stands for any section of the bundle of Lagrange multipliers $\Lambda \rightarrow \mathscr{Z}$ %associated to the functional $\mathscr{L}  := \| \Phi \|_{L^2 }^2 + (d S )(\lambda )$ on $ T\mathscr{B}_{E , k }^\ast \rightarrow \mathbb{R} $ 
(see below)
\item[(A6)] The bundle $E \to Y$ is \textit{admissible}: either $Y$ is a homology $3$-sphere or $c_1 (E)$ defines an odd class in $H^{2}(Y, \mathbb{Z} )/\mathrm{Torsion}$.
\end{enumerate}

\begin{definition}
The \textit{bundle of Lagrange multipliers} associated to the Lagrange functional $\mathscr{L}: T\mathscr{B}_{E , k}^\ast \to \mathbb{R}$, $\mathscr{L}([\mathbb{A}] , \lambda ) = \| \Phi \|_{L^2}^2 + (dS)_{\mathbb{A}} (\lambda )$, is the affine bundle $\Lambda \rightarrow \mathscr{Z} = \mathrm{Crit}( \| \Phi \|_{L^2}^{2}|_{\mathscr{M}_{E}^\ast} )$ with fibers
\[
\Lambda_{[\mathbb{A}]} = \big\{ \lambda \in \mathscr{K}_{[\mathbb{A}] , k} = T_{[\mathbb{A}]} \mathscr{B}_{E , k}^\ast \, | \,\mathscr{D}_{\mathbb{A}} (\lambda ) = -  2 (0 , \Phi )  \big\} \subset T_{[\mathbb{A}]} \mathscr{B}_{E , k}^\ast ,
\]
with associated vector bundle $T\mathscr{M}_{E}^{\ast}|_{\mathscr{Z} } \rightarrow \mathscr{Z}$.
\end{definition}
\begin{theorem}[Perturbation Theorem]\label{theorem:perturbation1general}
Assume (A1)-(A5) hold. Then there exists an open neighborhood $\mathscr{U}_0$ of $\mathscr{Z} \subset \mathscr{B}_{E , k}^\ast$ and a constant $\varepsilon_0 >0$, such that when $0 < | \varepsilon| \leq \varepsilon_0$ then the critical points of $S_\varepsilon : \mathscr{B}_{E , k}^\ast \to \mathbb{R}$ in the neighborhood $\mathscr{U}_0$ are non-degenerate and there is a bijection
\[
F_\varepsilon: \mathrm{Crit}\mathfrak{f}  \cong \mathrm{Crit}S_\varepsilon \cap \mathscr{U}_0   .
\]
Furthermore, if for a point $[\mathbb{A}_0] \in \mathscr{Z}$ we denote $[\mathbb{A}_\varepsilon  ]:= F_\varepsilon ( [\mathbb{A}_0] )$, then the mod $2$ spectral flow of the path of operators $\mathscr{D}_{\mathbb{A}_\varepsilon}^t$ from $t = 0$ to $t = \varepsilon$ is:  for $\pm \varepsilon >0$
\[
\mathbf{sf}_{2} ( \mathscr{D}_{\mathbb{A}_\varepsilon}^t ) = \mathrm{ind}(\pm \| \Phi \|_{L^2}^{2}|_{\mathscr{M}_{E}^\ast} , [\mathbb{A}_0 ] )  + \mathrm{ind}(\mathfrak{f} , [\mathbb{A}_0 ])  .
\]
\end{theorem}

(Note: the quantity $\mathrm{ind}(\pm \| \Phi \|_{L^2}^{2}|_{\mathscr{M}_{E}^\ast} , [\mathbb{A}_0 ] ) $ taken mod $2$, as above, is independent of the sign $\pm$ since $\mathscr{M}_{E}^\ast$ has even real dimension.)
%\begin{remark}
%\textcolor{red}{$\pm$ not needed since %$\mathscr{M}_{E}^\ast$ is even dimensional...}
%\end{remark}
The proof of Theorem \ref{theorem:perturbation1general} will be given below. We also have the following:
\begin{theorem}[Localisation Theorem]\label{theorem:perturbation2general}
Assume (A1)-(A6) hold. Let $\mathscr{U}\subset \mathscr{B}_{E , k}^\ast$ be an open set, and $C > 0$ a constant such that
%\begin{enumerate}
%\item[(i)] 
\[\mathscr{Z} \subset \mathscr{U} \subset \{  \| \Phi \|_{L^\infty}  < C \} . \]
When $Y$ is a homology $3$-sphere let $\mathscr{U}_{\mathrm{triv}}$ be a given open neighborhood in $\mathscr{B}_{E,k}$ of the trivial connection $[\mathbb{A}_{\mathrm{triv}}]$ which is disjoint from $\mathscr{U}$.

%\item[(ii)] The closure of $\mathscr{N}$ in $ \mathscr{B}_{E, k}$ does not contain the trivial connection.
%\end{enumerate}
Then there exists a constant $\varepsilon_0 = \varepsilon_0 (Y, g, E, q_i , f , \mathscr{U} , C ) > 0$ such that the following hold for $0 < |\varepsilon|\leq \varepsilon_0$:
\begin{enumerate}
\item If $[\mathbb{A} = (A, \Phi )] \in \mathrm{Crit}S_\varepsilon$ is a critical point of $S_\varepsilon : \mathscr{B}_{E,k}^\ast \to \mathbb{R}$ then either $[\mathbb{A}] \in \mathscr{U}$ or $[\mathbb{A}]\in \mathscr{U}_{\mathrm{triv}} \setminus \{[ \mathbb{A}_\mathrm{triv}] \}$ or $\| \Phi \|_{L^\infty} > C$. Furthermore, the critical points of $S_\varepsilon$ contained in $\mathscr{U}$ are non-degenerate and comprise a finite set.
\item 
The count of critical points $[\mathbb{A}]$ of $S_\varepsilon$ in $\mathscr{U}$ with signs given by the mod $2$ spectral flow of the path of Hessians $\mathscr{D}_{\mathbb{A}}^t$ with $t$ going from $t = 0$ to $t = \varepsilon$ computes the Euler characteristic of the moduli space of stable flat connections: 
\[
\sum_{[\mathbb{A}] \in \mathrm{Crit}(S_\varepsilon ) \cap \mathscr{U}} (-1)^{\mathbf{sf}_{2}(\mathscr{D}_{\mathbb{A}}^t )} = \chi ( \mathscr{M}_{E}^\ast ) . \]
%\item Furthermore, the critical points of $S_\varepsilon$ contained in $\mathscr{N}$ are non-degenerate and comprise a finite set whose signed count computes the Euler characteristic of $\mathscr{Z}$:
%\[
%\sum_{[(A, \Phi )] \in \mathrm{Crit}(S_\varepsilon ) \cap \mathscr{N}} (-1)^{\mathrm{ind}(S_\varepsilon  , {(A, \Phi )} )} = \chi (\mathscr{Z} ).
%\]
%where
%\[
%\mathrm{ind}(S_\varepsilon  , {(A, \Phi )} ) = \mathbf{sf}_{2}\big( (\widehat{\mathrm{Hess}}S)_{(A_0, \Phi_0 )} , (\widehat{\mathrm{Hess}}S_\varepsilon )_{(A, \Phi )} \big) \in \mathbb{Z}/2
%\]
%and $(A_0 , \Phi_0 ) $ is ...
\end{enumerate}
\end{theorem}

\begin{proof}[Proof of Theorem \ref{theorem:perturbation1general}]
We can follow the argument from the proof of the finite-dimensional result, Theorem \ref{theorem:perturbation1}, with the following modifications:

\textit{Local perturbation theory.} The arguments from \S\ref{subsubsection:localtheory} carry through with minor modifications to the infinite-dimensional case, crucially making use now of the finite-dimensional Kuranishi models for $S$ from Corollary \ref{corollary:kuranishi}. That is, we have product coordinates $U \times V \times W$ on a neighborhood of a stable flat connection $[\mathbb{A}_0 ] \in \mathscr{M}_{E}^\ast$ in $\mathscr{B}_{E,k}^\ast$, with $[\mathbb{A}_0 ]$ corresponding to $(0,0,0) \in U \times V \times W$. Namely, 
\begin{itemize}
\item $V\oplus W = \mathscr{H}^{1}_{\mathbb{A}_0}$ is the (finite-dimensional) tangent space to $\mathscr{M}_{E}^\ast$
\item $U$ be the $L^2$--orthogonal complement of $\mathscr{H}^1_{\mathbb{A}_0}$ in $\mathscr{K}_{\mathbb{A}_0 , k} = T_{\mathbb{A}_0} \mathscr{B}_{E , k}^\ast$
\item $W$ is the tangent space to $\mathscr{Z}$ (taken to be ${0}$ if $[\mathbb{A}_0 ] \notin \mathscr{Z}$)
\item $V$ is the $L^2$--orthogonal complement of $W$ in $\mathscr{H}^{1}_{\mathbb{A}_0}$. 
\end{itemize}
We also let $U^\prime$ be the $L^2$ orthogonal complement of $\mathscr{H}^{1}_{\mathbb{A}_0}$ in $\mathscr{K}_{\mathbb{A}_0 , k-1}$. We thus have $L^2$ orthogonal decompositions
\[
\mathscr{K}_{\mathbb{A}_0 , k } = U \oplus V \oplus W \quad , \quad \mathscr{K}_{\mathbb{A}_0 , k-1} = U^\prime \oplus V \oplus W .
\]
By the finite-dimensional Kuranishi model for $S$ (Corollary \ref{corollary:kuranishi}) and (A1) we may assume $S$ is in standard form:
\[
S(u,v,w) =\frac{1}{2} \langle \mathscr{D}_{(A,\Phi )} u , u \rangle + \mathrm{constant} ,
\] 
where now $\mathscr{D}_{\mathbb{A}_0}$ gives an isomorphism from $U$ to $U^\prime $ rather than from $U$ to $U$ (this doesn't bring in any additional difficulties). By (A2), a further change of coordinates ensures that $\| \Phi \|_{L^2}^2$ is in standard form on the slice $0 \times V \times W$. The proofs of Lemmas \ref{lemma:convergence}-\ref{lemma:leading} now go through with almost identical notation, using the Implicit Function Theorem for Hilbert spaces.

\textit{Spectral flow calculation.} Proposition \ref{prop:localpert} also goes through using (A3), (A4) and (A5), except the index assertion no longer makes sense (since the spectrum of the Hessians of $S$ and $S_\varepsilon$ is now unbounded in both directions). This is modified as follows. For $0 < |\varepsilon| \leq \varepsilon_0$, the configuration $[\mathbb{A}_\varepsilon ] = F_\varepsilon ( [\mathbb{A}_0 ] )$ is the unique critical point of $S_\varepsilon$ in a small neighborhood $\mathscr{U}_0$ of $\mathbb{A}_0$, and is non-degenerate. Fix a representative $\mathbb{A}_0$ of $[\mathbb{A}_0 ]$ and take the unique representative $\mathbb{A}_\varepsilon$ of $[\mathbb{A}_\varepsilon ]$ contained in the stable Coulomb slice $\mathcal{S}_{\mathbb{A}_0 , k}$ at $\mathbb{A}_0$ (cf. Proposition \ref{proposition:slices}). Consider the path of operators $\mathscr{D}_{\mathbb{A}_t}^t$ for $t$ going from $t = 0$ to $t = \varepsilon$. The argument from the proof of the index assertion in Proposition \ref{prop:localpert}, together with the intrinsic description of the `leading order term' as the function $\mathfrak{f}$, shows that the mod 2 spectral flow of the path $\mathscr{D}_{\mathbb{A}_t}^t$ is given by
\begin{align*}
\mathrm{ind}(\pm \| \Phi \|_{L^2}^{2}|_{\mathscr{M}_{E}^\ast }, [\mathbb{A}_0 ] ) + \mathrm{ind}(\mathfrak{f} , [\mathbb{A}_0 ]).
\end{align*}
%where $S^\prime$ is the `leading term' function, defined on a neighborhood of $(0,0,0)$ in $W$. 
But observe that the path $\mathscr{D}_{\mathbb{A}_t}^t$ is homotopic to the concatenation of the following two paths:
\begin{itemize}
\item The path $\mathscr{D}_{\mathbb{A}_t}$ from $\mathscr{D}_{\mathbb{A}_0}$ to $\mathscr{D}_{\mathbb{A}_\varepsilon }$. All operators in this path have $\mathbb{C}$-linear kernel, hence there is no mod $2$ spectral flow. 
\item The path $\mathscr{D}_{\mathbb{A}_\varepsilon}^t$ from $\mathscr{D}_{\mathbb{A}_\varepsilon}$ to $\mathscr{D}_{\mathbb{A}_\varepsilon}^\varepsilon$.
\end{itemize}
Hence the mod $2$ spectral flow of the path $\mathscr{D}_{\mathbb{A}_\varepsilon}^t$ agrees with the required quantity, and this completes the proof.
\end{proof}

\begin{proof}[Proof of Theorem \ref{theorem:perturbation2general}]

Suppose no $\varepsilon_0 >0$ exists making the first assertion in (1) hold. Then we can find a sequence $[\mathbb{A}_n = (A_n , \Phi_n)] \in \mathscr{B}_{E,k}^\ast$ and $\varepsilon_n \neq 0$ with $\varepsilon_n \to 0$, such that
\begin{align*}
& [\mathbb{A}_n ]\in \mathscr{M}_{E, k}^{\varepsilon_n , \ast }\\
& [\mathbb{A}_n ] \notin \mathscr{U} \\
& \| \Phi \|_{L^\infty } \leq C \\
& [\mathbb{A}_n ] \notin \mathscr{U}_{\mathrm{triv}}.
\end{align*}
 By Proposition \ref{proposition:compactness}, the function $\| \Phi \|_{L^\infty }$ is proper when restricted to the union of the moduli spaces $\mathscr{M}_{E , k}^\varepsilon$ with $\varepsilon$ taking values in any compact interval. Thus, after applying a sequence of unitary gauge transformations and passing to a subsequence, we may assume that $\mathbb{A}_n$ converges in $L^{2}_k$ to a configuration $\mathbb{A} = (A, \Phi )$, and necessarily $[(A, \Phi ) ] \in \mathscr{M}_{E}$.
 
If $\mathbb{A}$ is stable, then by the infinite-dimensional analogue of Lemma \ref{lemma:convergence} (as explained above, this is established following the same reasoning as in the finite-dimensional case, using now the finite-dimensional Kuranishi models for $S$ from Corollary \ref{corollary:kuranishi}) we have that $[\mathbb{A}] \in \mathscr{Z} \subset \mathscr{U}$, which contradicts the fact that $\mathbb{A}_n \notin \mathscr{U}$. If $\mathbb{A}$ is not stable, then there exists a $U(1)$ bundle with connection $(L, B )$ and $\phi \in \Omega^1 (Y, i \mathbb{R} )$ such that 
\begin{align}
& E = L \oplus L^{-1} \Lambda \quad , \quad A = B \oplus B^{-1}\lambda \quad , \quad \Phi = \begin{pmatrix}  \phi & 0 \\ 0 &  -\phi \end{pmatrix} \nonumber \\
& F_B - \frac{1}{2}F_\lambda = 0 \quad , \quad d\phi = 0 \quad , \quad d^\ast \phi = 0 .\label{reducibleeq}
\end{align}
From the first equation in (\ref{reducibleeq}) it follows that $c_1 (E)$ is an even class in $H^{2}(Y, \mathbb{Z} )/\mathrm{Torsion}$. Since $E \to Y$ is an admissible bundle then the only possibility is that $Y$ is a homology $3$-sphere and both $L$ and $\Lambda$ are trivial. %(hence $L$ and $\Lambda$ are also trivial).
Then $\varphi = 0$ by the second and third equation in (\ref{reducibleeq}), and the first equation has a unique solution modulo $U(1)$ gauge transformations: the trivial connection. %(if $\lambda$ was chosen to be the trivial connection, then this solution is the trivial connection on $E$). 
This contradicts the fact that $\mathbb{A}_n \notin \mathscr{U}_\mathrm{triv}$.

Thus, there is $\varepsilon_0 >0$ such that the first assertion in (1) holds. For the second assertion in (1) we only need to shrink $\varepsilon_0 >0$ further so as to ensure that for $0 < |\varepsilon |\leq \varepsilon_0$ the set $\mathrm{Crit}S_\varepsilon \cap \mathscr{U}$ is contained in a neighborhood $\mathscr{U}_0$ of $\mathscr{Z} \subset \mathscr{B}_{E , k}^\ast$ such that Theorem \ref{theorem:perturbation1general} holds for $\mathscr{U}_0$ and $\varepsilon_0$. For (2), by Theorem \ref{theorem:perturbation1general} we then have: if $\mathscr{Z} = \cup \mathscr{Z}_i$ stands for the decomposition into connected components and $\pm \varepsilon >0$ then
\[
\sum_{[\mathbb{A} ]\in \mathrm{Crit}S_\varepsilon \cap \mathscr{U} } (-1)^{ \mathbf{sf}_{2}(\mathscr{D}_{\mathbb{A}}^t )} = \sum_{i} (-1)^{ \mathrm{ind}(\pm \| \Phi \|_{L^2}^{2}|_{\mathscr{M}_{E}^\ast  }  , \mathscr{Z}_i )} \cdot \chi ( \mathscr{Z}_i ) .
\]
Note that $\| \Phi \|_{L^2}^2$ is a proper and bounded below function on $\mathscr{M}_{E}$, and hence also on $\mathscr{M}_{E}^\ast$ since $E$ is an admissible bundle. It follows that the right-hand side of the above equation computes the Euler characteristic $\chi (\mathscr{M}_{E}^\ast )$ when $\varepsilon >0$ and the compactly-supported Euler characteristic $\chi_c  ( \mathscr{M}_{E}^\ast )$ when $\varepsilon <0$. But observe that both Euler characteristics are identical, since $\mathscr{M}_{E}^\ast$ has \textit{even} real dimension. This completes the proof.
\end{proof}

\subsubsection{The exact case}

So far we have been working with the functional $S = \mathrm{Re}(e^{-i \theta} S )$ for an arbitrary $\theta \in \mathbb{R}/2\pi\mathbb{Z}$. The values $\theta = \pm \pi/2 $ modulo $2\pi$ (cf. \ref{ImCS}) are special since then the functional $S$ is exact, i.e. it has no periods under the action of complex or unitary gauge transformations, and thus defines an honest $\mathbb{R}$-valued function on $\mathscr{B}_{E,k}^\ast $. 

In this case the family of operators $\mathscr{D}_p = \mathscr{D}_{[\mathbb{A}]}^\varepsilon + \delta$ parametrised by $p = ([\mathbb{A}] , \varepsilon , \delta ) \in P := \mathscr{B}_{E,k}^\ast \times \mathbb{R}\times \mathbb{R}$ has the following property:
\begin{lemma}\label{lemma:nospectral}
Suppose $\theta = \pi/2 $ or $3\pi/2$. Let $\gamma : [0,1] \to P$ be a path such that the endpoints $\mathscr{D}_{\gamma(0)}$, $\mathscr{D}_{\gamma(1)}$ are invertible. Then the spectral flow $\mathbf{sf}(\mathscr{D}_{\gamma (t)})$ only depends on $\gamma(0)$ and $\gamma(1)$.
\end{lemma}
\begin{proof}
It suffices to show that if $\gamma$ is a generic loop in $P$ based at a point $\gamma (0) = \gamma(1) = p$ such that $\mathscr{D}_p$ is invertible, then the spectral flow along $\gamma$ vanishes:
\[
\mathbf{sf}(\mathscr{D}_{\gamma (t)}) = 0.
\]
This follows by a standard argument identifying the spectral flow along the loop with the index of a suitable Fredholm operator on $S^1 \times Y$ (see e.g. \cite[Lemma 14.4.6]{KM}). By \cite[Proposition 4.1]{mazzeo-witten}, said operator on $S^1 \times Y$ has vanishing index. 
\end{proof}
We can now attempt to attach an \textit{integral} `Morse index' to the critical points of $S_\varepsilon$ using the spectral flow (as opposed to only a $\mathbb{Z}/2$ index). This is explained in the following refinement of Theorem \ref{theorem:perturbation2general}. Below, $\mathcal{P}_T (\mathscr{M}_{E}^\ast )$ (resp. $\mathcal{P}_{T}^c (\mathscr{M}_{E}^\ast )$) stands for the Poincaré polynomial (resp. compactly-supported) Poincaré polynomial of $\mathscr{M}_{E}^\ast$. 

\begin{theorem}\label{theorem:perturbation2_exact}
Suppose that $\theta = \pi/2$ or $3\pi/2$. Assume (A1)-(A6) hold. Let $\mathscr{Z} = \cup_i \mathscr{Z}_i$ denote the decomposition into connected components. For each $i$ let $\mathscr{U}_i \subset \mathscr{B}_{E,k}^\ast$ be an open subset, and $C>0$ a constant, such that
\begin{itemize}
    \item $\mathscr{Z}_i \subset \mathscr{U}_i \subset \{  \| \Phi \|_{L^\infty}  < C \} . $
    \item $\mathscr{U}_i \cap \mathscr{U}_j = \emptyset$ for $i \neq j$.
\end{itemize}
Set $\mathscr{U} = \cup_i \mathscr{U}_i$. When $Y$ is a homology $3$-sphere let $\mathscr{U}_{\mathrm{triv}}$ be a given open neighborhood in $\mathscr{B}_{E,k}$ of the trivial connection $[\mathbb{A}_{\mathrm{triv}}]$ which is disjoint from $\mathscr{U}$.

%\item[(ii)] The closure of $\mathscr{N}$ in $ \mathscr{B}_{E, k}$ does not contain the trivial connection.
%\end{enumerate}
Then there exists a constant $\varepsilon_0 = \varepsilon_0 (Y, g, E, q_i , f , \mathscr{U} , C ) > 0$ such that the following hold for $0 < |\varepsilon|\leq \varepsilon_0$:
\begin{enumerate}
\item If $[\mathbb{A} = (A, \Phi )] \in \mathrm{Crit}S_\varepsilon$ is a critical point of $S_\varepsilon : \mathscr{B}_{E,k}^\ast \to \mathbb{R}$ then either $[\mathbb{A}] \in \mathscr{U}$ or $[\mathbb{A}]\in \mathscr{U}_{\mathrm{triv}} \setminus \{[ \mathbb{A}_\mathrm{triv}] \}$ or $\| \Phi \|_{L^\infty} > C$. Furthermore, the critical points of $S_\varepsilon$ contained in $\mathscr{U}$ are non-degenerate and comprise a finite set.
\item If $[\mathbb{A}] $ is a critical point of $S_\varepsilon$ in $\mathscr{U}_i$ then there exists a path $[\mathbb{A}_t]$ defined for $t$ going from $t = 0$ to $t = \varepsilon$ with $[\mathbb{A}_0 ] \in \mathscr{Z}_i$, $[\mathbb{A}_1] = [\mathbb{A}]$ and $[\mathbb{A}_t]\in \mathscr{U}_i$ for all $t$; fix one such path for each critical point in $\mathscr{U}$. Let $\delta = \delta ( \varepsilon )>0$ be a constant small enough so that the Hessian $\mathscr{D}$ at every point of $\mathscr{Z}$ has spectral gap $> \delta$, and so does the perturbed Hessian $\mathscr{D}^\varepsilon$ at every point in $ \mathrm{Crit}(S_\varepsilon )\cap \mathscr{U}$. For each critical point $[\mathbb{A}]$ of $S_\varepsilon$ in $\mathscr{U}$ we assign an integral `Morse index' given by the spectral flow of the path of operators $\mathscr{D}_{\mathbb{A}_t}^t + \delta$ with $t$ going from $t = 0$ to $t = \varepsilon$. %and with $\delta>0$ sufficiently small (namely, so that the Hessian $\mathscr{D}$ at every point of $\mathscr{Z}$ has spectral gap $> \delta$, and so does the perturbed Hessian $\mathscr{D}_{\mathbb{A}}^\varepsilon$ at every $[\mathbb{A}] \in \mathrm{Crit}(S_\varepsilon )\cap \mathscr{N}$). 

Define the following Laurent polynomial with non-negative integer coefficients:
\[
M_{T} := \sum_{[\mathbb{A}] \in \mathrm{Crit}(S_\varepsilon ) \cap \mathscr{U}} T^{- \mathbf{sf}(\mathscr{D}_{\mathbb{A}_t}^t + \delta )} \in \mathbb{Z}[T , T^{-1}] . 
\]
Then, evaluating $M_T$ at $T = 1$ gives $\chi (\mathscr{M}_{E}^\ast )$, and we have a `Morse inequality':
\[
M_T \geq \begin{cases}
       \mathcal{P}_T ( \mathscr{M}_{E}^\ast ) &\quad\text{if } \varepsilon >0 \\
       \mathcal{P}^{c}_T ( \mathscr{M}_{E}^\ast ) &\quad\text{if } \varepsilon < 0 .\\
     \end{cases}
\]
(Here, $\geq$ is meant as an inequality involving the individual coefficients in the Laurent polynomial).
\end{enumerate}
\end{theorem}

\begin{proof}
We only need to prove (2), as the rest is identical to Theorem \ref{theorem:perturbation2general}. For a critical point $[\mathbb{A}] \in \mathrm{Crit}S_\varepsilon \cap \mathscr{U}_i$, the index calculation of Proposition \ref{prop:localpert} now gives the following identity over $\mathbb{Z}$ (rather than over $\mathbb{Z}/2$, compare with Theorem \ref{theorem:perturbation1general}): for $\pm \varepsilon >0$
\[
- \mathbf{sf}(\mathscr{D}_{\mathbb{A}_t}^t + \delta) = \mathrm{ind}(\pm \| \Phi \|_{L^2}^2, [\mathbb{A}_0] ) + \mathrm{ind}(\mathfrak{f}, [\mathbb{A}_0] )
\]
where $[\mathbb{A}_{t}]$ is the path provided by Theorem \ref{theorem:perturbation1general}, with $t$ going from $t = 0$ to $t = \varepsilon$, and $[\mathbb{A}_0] \in \mathscr{Z}_i$. By Lemma \ref{lemma:nospectral}, we may replace $[\mathbb{A}_t]$ by any other path with the same endpoints. Furthermore, the family of operators $\mathscr{D}_{[\mathbb{A}_0]} + \delta$ has no spectral flow as we vary $[\mathbb{A}_0 ]$ in the component $\mathscr{Z}_i$ since $\mathscr{Z}_i$ belongs in the Morse--Bott critical locus $\mathscr{M}_{E}^\ast$. 

We thus have: if $\pm \varepsilon >0$ then
\[
M_T := \sum_{[\mathbb{A} ]\in \mathrm{Crit}S_\varepsilon \cap \mathscr{U} } T^{- \mathbf{sf}(\mathscr{D}_{\mathbb{A}_t}^t )} = \sum_{i} T^{ \mathrm{ind}(\pm \| \Phi \|_{L^2}^{2}|_{\mathscr{M}_{E}^\ast  }  , \mathscr{Z}_i )} \cdot M_T ( \mathscr{Z}_i, \mathfrak{f} ) .
\]
Here $M_T ( \mathscr{Z}_i , \mathfrak{f} )$ is the Morse polynomial of the Morse function $\mathfrak{f}$ on $ \mathscr{Z}_i $, and the Morse inequalities say $M_{T} ( \mathscr{Z}_i , \mathfrak{f} ) \geq \mathcal{P}_T ( \mathscr{Z}_i )$, thus:
\[
M_T \geq \sum_{i} T^{ \mathrm{ind}(\pm \| \Phi \|_{L^2}^{2}|_{\mathscr{M}_{E}^\ast  }  , \mathscr{Z}_i )} \cdot \mathcal{P}_T ( \mathscr{Z}_i  ).
\]
Now, the Morse--Bott inequalities (see e.g. \cite[\S 1]{atiyah-bott}) for the proper and bounded below Morse--Bott function $\| \Phi \|_{L^2}^2$ on $\mathscr{M}_{E}^\ast$ give:
\[
\sum_{i} T^{ \mathrm{ind}(\pm \| \Phi \|_{L^2}^{2}|_{\mathscr{M}_{E}^\ast  }  , \mathscr{Z}_i )} \cdot P_T ( \mathscr{Z}_i  ) \geq \mathcal{P}_{T} ( \mathscr{M}_{E}^\ast ) \text{  or  } \mathcal{P}_{T}^c ( \mathscr{M}_{E}^\ast )
\]
according to whether $\varepsilon >0$ or $\varepsilon < 0$. Putting everything together, this concludes the proof. %$M_T \geq P_T ( \mathscr{M}_{E}^\ast )$ or $M_T \geq P_{T}^c ( \mathscr{M}_{E}^\ast )$ according to the sign of $\varepsilon$.
\end{proof}

\subsection{The case of Seifert-fibered $3$-manifolds}

Let $Y = S(N) \to C$ be a closed, oriented, Seifert-fibered $3$-manifold over a closed oriented $2$-orbifold $C$, and equip $Y$ with a Seifert metric $g= \eta^2  + g_C $. We equip $C$ with the holomorphic structure induced from $g_C$. 

\subsubsection{The locus $\mathscr{Z}$} We will now establish the Morse--Bott property of the locus $\mathscr{Z} = \mathscr{Z}_{E}(g)$ given as the critical locus of the restriction of the functional $\| \Phi \|_{L^{2}}^2$ to the moduli space $\mathscr{M}_{E}^\ast (Y, g)$ of stable projectively flat $GL(2, \mathbb{C})$ connections on $E$ with fixed determinant $(\Lambda , \lambda ) $. 

We will use the following notation:
\begin{itemize}
\item By Theorem \ref{theorem:dimensionalreduction} and Remark \ref{remark:tensoring} we may suppose that $(\Lambda , \lambda )$ pulls back from $C$. Thus, fix an orbifold $U(1)$ bundle with connection $(\Lambda_0 , \lambda_0 )$ on $C$, with an identification $\Lambda^2 E \cong \pi^\ast \Lambda_0$. 

 \item $\mathscr{N}_{E}^{\ast} (Y)= \mathscr{M}_{E}^\ast (Y, g ) \cap \{ \Phi = 0\}$ denotes the moduli space of irreducible projectively flat $U(2)$ connections on $E$ with fixed determinant $(\Lambda , \lambda )$. A similar notation, namely $\mathscr{N}_{E_0}^\ast (C)$, is used for the corresponding moduli spaces on an orbifold bundle $E_0$ on $C$, which can now be identified with the moduli space of stable holomorphic structures on $E_0$ with fixed determinant.
 \item For an orbifold complex line bundle $L_0 $ on $C$, $\mathcal{D}_{L_0} (C)$ denotes the moduli space of effective orbifold divisors on $L_0$: pairs $( \mathscr{L}_0 , s )$ where $\mathscr{L}_0$ is an orbifold holomorphic structure on $L_0 \to C$ and $s$ is a holomorphic section of $\mathscr{L}_0$, modulo the action of the complex automorphisms of $L_0$. It is a standard fact that this moduli space is smooth (because $C$ has complex dimension one) and is identified as the symmetric product $\mathcal{D}_{L_0}(C) \cong \mathrm{Sym}^d |C|$ of $|C|$, where $d$ is the degree of the ordinary line bundle given by the desingularisation $|L_0| \to |C|$ \cite[Proposition 2.0.14]{MOY}.
\end{itemize}
 \begin{proposition}\label{proposition:Zmorsebott}
The locus $\mathscr{Z} = \mathrm{Crit}( \| \Phi \|_{L^2}^{2}|_{\mathscr{M}_{E}^\ast (Y)} ) \subset \mathscr{M}_{E}^\ast (Y,g)$ is Morse--Bott and can be described as the disjoint union $\mathscr{Z} = $

\begin{gather*}
     \Big( \bigsqcup_{E_0 \, , \, \Lambda^2 E_0 \cong \Lambda_0 } \mathscr{Z}_{E_0}^{0} \Big) \sqcup   \Big( \bigsqcup_{E_0 \, , \, \Lambda^2 E_0 \cong \Lambda_0 \otimes N } \mathscr{Z}_{E_0}^{1} \Big) \\
\\
 \sqcup  \Big( \bigsqcup_{L_0 \, , \, 0 < \mathrm{deg}L_0 - \frac{1}{2} \mathrm{deg}  \Lambda_0  \leq \frac{1}{2} \mathrm{deg}K_C } \mathscr{Z}_{L_0}^{0} \Big)  \sqcup   \Big( \bigsqcup_{0 < \mathrm{deg}L_0 - \frac{1}{2} \mathrm{deg} ( \Lambda_0 \otimes N ) \leq \frac{1}{2} \mathrm{deg}K_C} \mathscr{Z}_{L_0}^{1} \Big).
\end{gather*}
where:
\begin{enumerate}
\item $\mathscr{Z}_{E_0}^{k} = \mathscr{N}_{E_0}^\ast (C )$ is the moduli space of stable holomorphic orbifold structures on $E_0 $ with fixed determinant $\Lambda_0 \otimes N^{k}$. (These are the minima of $\| \Phi \|_{L^2}^2$ on $\mathscr{M}_{E}^\ast (Y)$).
%The locus $\mathscr{Z} \cap \{ \Phi = 0\}$ is given by the $U(2)$ moduli space $\mathscr{N}_{E}^\ast$.  By Theorem \ref{theorem:dimensionalreduction} this is given by:
%\[
%\Big( \bigsqcup_{E_0 \, , \, \Lambda^2 E_0 \cong \Lambda_0 } \mathscr{N}_{E_0}^\ast (C ) \Big) \sqcup \Big( \bigsqcup_{0 < \mathrm{deg}L_0 - \frac{1}{2} \mathrm{deg} \Lambda_0 \leq \frac{1}{2} \mathrm{deg}K_C} \mathscr{Z}_{L_0}^\ast (C ) \Big)  \]
\item $\mathscr{Z}_{L_0}^{k}$ is a suitable $2^{2g}$--fold cover of the moduli space of effective orbifold divisors $\mathcal{D}_{L_{0}^{-2}\Lambda_0 N^{k} K}$ in the orbifold line bundle $L_{0}^{-2}\Lambda_0 N^{k} K$, where $g$ is the genus of $C$. The Morse--Bott index of $\mathscr{Z}_{L_0}^k$ is given by $2 \mathrm{dim}_{\mathbb{C}} H^0 (C , L_{0}^2 \Lambda_{0}^{-1} N^{-k} K )$.
%\[
%\Big( \bigsqcup_{L_0 \, , \, 0 < \mathrm{deg}L_0 - \frac{1}{2} \mathrm{deg} \Lambda_0 \leq \frac{1}{2} \mathrm{deg}K_C } \widetilde{\mathcal{D}}_{L_{0}^{-2} \Lambda_0  K} \Big) \sqcup \Big( \bigsqcup_{L_0 \, , \, 0 < \mathrm{deg}L_0  - \frac{1}{2} \mathrm{deg} ( \Lambda_0 \otimes N )\leq \frac{1}{2} \mathrm{deg}K_C } \widetilde{\mathcal{D}}_{L_{0}^{-2} \Lambda_0  N  K} \Big)
%\]
%where $\widetilde{\mathcal{D}}_{L_{0}^{-2} \Lambda_0  K}$ (and likewise for $\widetilde{\mathcal{D}}_{L_{0}^{-2} \Lambda_0  N  K}$) denotes a suitable $2^{2g}$--fold covering of $\mathcal{D}_{L_{0}^{-2} \Lambda_0 K }$.
\end{enumerate}

\end{proposition}

\begin{proof}
We combine the Dimensional Reduction Theorem \ref{theorem:dimensionalreduction} with an argument of Hitchin \cite[\S 7]{hitchin} (see also \cite{nasatyr-steer,boden-yokogawa}). Hitchin's argument makes use of the correspondence between moduli spaces of stable Higgs bundles and stable (projectively) flat connections on a Riemann surface (i.e. (\ref{eq1C}-\ref{eq3C}). We give a variant of his argument using basic properties of the Kazdan--Warner equation \cite{kazdan-warner} instead.

%A crucial point is that the moduli space $\mathscr{M}_{E}^\ast (Y)$ carries a \textit{natural $S^1$--action}. This is obtained from the Dimensional Reduction Theorem \ref{theorem:dimensionalreduction} and the corresponding identification of $\mathscr{M}^\ast_{E}(Y)$ with a disjoint union of Higgs bundles moduli spaces $\mathscr{M}_{E_0} (C,g_C ) \cong \mathscr{M}_{E_0}^{\mathrm{Higgs} , \ast} (C)$. Then the $S^1$--action on $\mathscr{M}_{E}^\ast (Y)$ is obtained from the rescaling $\mathbb{C}^\ast$--action on the Higgs field $\theta$ in each $\mathscr{M}_{E_0}^{\mathrm{Higgs} , \ast} (C)$:
%\[
%(\mathscr{E} , \theta ) \mapsto ( \mathscr{E} , \lambda \cdot \theta ) \quad , \quad \lambda \in \mathbb{C}^\ast .
%\]
%As observed by Hitchin \cite[\S 7]{hitchin} this $S^1$--action is \textit{Hamiltonian} in the $\omega_I$ symplectic form (cf. (\ref{omegaI})) and has moment map $\| \Phi \|_{L^{2} (C)}^2$. From this it follows that the $S^1$--fixed points agree with the critical points of the restriction of $\| \Phi \|_{L^{2}(Y)}^2 $ to $\mathscr{M}_{E}^{\ast} (Y,g )$. Furthermore, the function $\| \Phi \|_{L^{2}(Y)}^2 $ restricts to a Morse--Bott function on $\mathscr{M}_{E}^{\ast} (Y,g )$, and the Morse--Bott index of a given critical manifold $C$ agrees, by a result of Frankel \cite{frankel}, with the real dimension of the negative-weight subbundle of the normal bundle of $C \subset \mathscr{M}_{E}^\ast (Y )$.

Consider an orbifold $U(2)$ bundle $E_0 \to C$ with fixed determinant given by either $(\Lambda_0 , \lambda_0 )$ or $(\Lambda_1 , \lambda_1 ) = (\Lambda_0 \otimes N , \lambda_0 \otimes i \eta )$. We consider both cases simultaneously by writing $(\Lambda_k , \lambda_k )$, and we regard $\Lambda_k$ as a holomorphic orbifold line bundle with the holomorphic structure induced by the connection $\lambda_k$. %We may assume that $\lambda_0$ (and hence $\lambda_k$) is chosen so that it's curvature is a constant multiple of $\omega_C$, and the constant is then determined by the Chern--Weil formula:
%\[
%F_{\lambda_0 } = \mathrm{const} \cdot  \omega_C \quad , \quad \mathrm{const} = - \frac{ 2\pi i  \mathrm{deg}\Lambda_0 }{ \mathrm{Area}(C, g_C )}. 
%\]
The moduli space $\mathscr{M}_{E_0}^\ast (C , g_C)$ of stable (projectively) flat connections has a \textit{natural $S^1$--action} defined as:\begin{align}
\mathbb{B} = (B, \Psi ) \mapsto (B , \cos \theta \cdot  \Psi + \sin \theta \cdot  \ast_C \Psi ) \quad , \quad \lambda = e^{i \theta} \in S^1 . \label{S1action}
\end{align}
where $\ast_C $ denotes the Hodge star operator of the orbifold metric $g_C$. It is easy to check that (\ref{S1action}) preserves the set of solutions of (\ref{eq1C}-\ref{eq3C}). The $S^1$--action on the moduli spaces $\mathscr{M}^{\ast}_{E_0}(C)$ then induces one on $\mathscr{M}_{E}^\ast (Y)$ by passing through the correspondence of Theorem \ref{theorem:dimensionalreduction}.
%More directly--from the $3$-dimensional viewpoint--the $S^1$--action on the moduli space $\mathscr{M}_{E}^\ast (Y)$ can defined as:
%\begin{align}
%\mathbb{A} = (A, \Phi ) \mapsto (A , \cos \theta \cdot  \Psi + \sin \theta \cdot  \ast_H \Phi ) \quad , \quad \lambda = e^{i \theta} \in S^1 . \label{S1action}
%\end{align}
%where $\ast_H $ denotes the extension to an endomorphism of $T^\ast Y$ of the horizontal Hodge star operator on $\pi^\ast T^\ast C \subset T^\ast Y$ given by the orbifold metric $g_C$, by defining $\ast_H$ to act as zero on the vertical subbundle of $T^\ast Y$. Our vanishing result from Corollary \ref{corollary:vanishing} ensures that indeed (\ref{S1action}) preserves the set of solutions of (\ref{eq1}-\ref{eq3}) and defines an $S^1$--action on $\mathscr{M}_{E}^\ast (Y)$.
The crucial point is that the $S^1$--action on $\mathscr{M}_{E_0}^\ast (C,g_C )$ is Hamiltonian in the $\omega_I$ symplectic form (cf. (\ref{omegaI})), with moment map given by $\mu ( B , \Psi ) = \| \Psi \|_{L^{2}(C)}^2$ (times $-1/2$), as one easily verifies. From this it follows that $\mu$ is a Morse--Bott function on $\mathscr{M}_{E_0}^\ast (C,g_C )$ whose critical loci agree with the $S^1$--fixed points. Then by Theorem \ref{theorem:dimensionalreduction}, it follows that the restriction of $\|\Phi \|_{L^{2}(Y)}^2$ to $\mathscr{M}_{E}^\ast (Y,g)$ is also Morse--Bott.

We now describe the $S^1$--fixed points on $\mathscr{M}_{E_0}^\ast (C, g_C )$. Clearly, the locus $\Psi = 0$ consists of fixed points, and this accounts for the components of the form $\mathscr{Z}_{E_0}^k$. When $\Psi \neq 0$, the fixed points are given by stable configurations $(B, \Psi )$ solving (\ref{eq1C}-\ref{eq3C}) such that for any $\theta \in S^1$ there exists a unitary gauge transformation $u_\theta \in \mathscr{G}_{E_0}$ such that 
\[
u_\theta \cdot B = B \quad , \quad u_\theta \Psi u_{\theta}^{-1} = \cos \theta \cdot \Psi + \sin \theta \cdot \ast_C \Psi .
\]
Since $u_\theta \neq \pm I$ for $\theta \neq 0 , 2\pi , \ldots$, the first equation shows that $B$ reduces to a $U(1)$ connection: $E_0$ has a decomposition $E_0 = L_0 \oplus L_{0}^{-1} \Lambda_k $ and $B = b \oplus b^{-1} \lambda_k$ for a $U(1)$ bundle with connection $(L_0 , b )$ on $C$. The second equation then implies that the diagonal terms in $\Psi$ vanish,
\[
\Psi = \begin{pmatrix} 0 & - \sigma^\ast \\ \sigma & 0 \end{pmatrix}
\]
and $\sigma \in \Omega^1 ( C , L_{0}^{-2} \Lambda_k )$ satisfies $\ast_C \sigma = \pm i \sigma $. After possibly swapping the order of the summands in the decomposition of $E_0$, we may suppose that $\ast_C \sigma = -i \sigma $, i.e. $\sigma $ is a section of $L_{0}^{-2} \Lambda_k  K $ where $K $ is the orbifold canonical bundle of the orbifold Riemann surface $C$. The equations (\ref{eq1C}-\ref{eq3C}) for $(B, \Psi )$ are then easily seen to reduce to the following pair of equations on $(b , \sigma )$:
\begin{align}
 F_{b} - \frac{1}{2} F_{\lambda_k} & = - i |\sigma|^2 \omega_C \label{fix1}\\
 \overline{\partial}_{b} \sigma & = 0 \label{fix2}.
\end{align}

We denote by $\mathscr{Z}_{L_0} ^{k}$ the moduli space of pairs $(b,\sigma )$ with $\sigma \neq 0$ solving (\ref{fix1}-\ref{fix2}), modulo the natural action of the $U(1)$ gauge transformations $u$ of $L_0$ by $(b , \sigma ) \mapsto ( b - u^{-1}du , u \cdot \sigma )$. Already from the fact that the quadratic term in $\sigma$ in the equation (\ref{fix1}) is \textit{non-degenerate}, it follows by a standard argument that the moduli space $\mathscr{Z}_{L_0}^{k}$ is compact (see e.g. \cite[\S 5]{morgan}). If $(b , \sigma )$ lies in this moduli space then $b$ induces a holomorphic structure $\mathscr{L}_0$ on $L_0$ and $\sigma$ defines a non-trivial holomorphic section of $\mathscr{L}_{0}^{-2} \Lambda_k K$. Denote the space of such pairs $(\mathscr{L}_0 , \sigma )$ modulo the complex gauge transformations of $L_0$ by $\widetilde{\mathcal{D}}_{L_{0}^{-2} \Lambda_k K } (C) $. %, the$\Lambda_k$ and the holomorphic structure on $K$ given by that on $C$. 
In particular, when this space is non-empty then the orbifold degree of $L_{0}^{-2} \Lambda_k K$ must be $\geq 0$, which gives:
\begin{align}
\mathrm{deg} L_0 - \frac{1}{2} \mathrm{deg}(\Lambda_k ) \leq \frac{1}{2} \mathrm{deg} K \label{ineq1}
\end{align}
On the other hand, by (\ref{fix1}), $\sigma \neq 0$ and the Chern--Weil formula we obtain:
\begin{align}
0 < 
\mathrm{deg} L_0 - \frac{1}{2}\mathrm{deg}(\Lambda_k ).\label{ineq2}
\end{align}

To describe the topology of $\mathscr{Z}_{L_0}^k$, we show that provided (\ref{ineq1}-\ref{ineq2}) hold then the following map defines a diffeomorphism:
\[
\mathscr{Z}_{L_0}^k \xrightarrow{\cong} \widetilde{\mathcal{D}}_{L_{0}^{-2} \Lambda_k K }(C) \quad , \quad (b , \sigma ) \mapsto (\mathscr{L}_0 = \overline{\partial}_b , \sigma ).
\]
This follows by a standard argument, which we briefly outline. Given a pair $(\mathscr{L}_0 , \sigma )$ such that $\overline{\partial}_{\mathscr{L}_0} \sigma = 0$ we seek a complex gauge transformation $u$ of $L_0$ such that if $b$ is the Chern connection of the hermitian holomorphic bundle $u \cdot \mathscr{L}_0$ then the pair $(b, \sigma )$ solves (\ref{fix1}). Equivalently, one may seek a new hermitian metric $h_\xi = e^{-\xi} \cdot h_0$, where $h_0$ is the given metric on $L_0$ and $\xi$ is a real-valued function, such that if $b$ is the Chern connection of the hermitian holomorphic bundle $(\mathscr{L}_0 , h_\xi )$ then $(b , \sigma )$ solves (\ref{fix1}) for the new metric $h_\xi$. Expressing (\ref{fix1}) as an equation for the function $\xi$ yields the equation:
\[
\Delta \xi  \cdot \omega + 2 e^{2\xi} | \sigma|^2 \cdot \omega  - 2i ( F_{h_0} - \frac{1}{2}F_{\lambda_k}) = 0 
\]
where $\Delta = d^\ast d $ is the non-negative Laplace operator of $g_C$. By the Theorem of Kazdan--Warner \cite{kazdan-warner} (see also \cite[Lemma 3.4]{bryan-wentworth}) this equation admits a unique smooth solution $\xi$ because
\[
\int_C 2i ( F_{h_0} - \frac{1}{2} F_{\lambda_k } )= 4\pi ( \mathrm{deg}L_0 - \frac{1}{2} \mathrm{deg}\Lambda_k ) > 0.
\]

From the diffeomorphism $\mathscr{Z}_{L_0}^k \cong \widetilde{D}_{L_{0}^{-2} \Lambda_k K} (C)$ the promised description of $\mathscr{Z}_{L_0}^k$ readily follows, as the map $\widetilde{\mathcal{D}}_{L_{0}^{-2} \Lambda_k K}(C) \to \mathcal{D}_{L_{0}^{-2} \Lambda_k K}(C)$ sending $(\mathscr{L}_0 , \sigma )$ to $(\mathscr{L}_{0}^{-2} \Lambda_k K , \sigma )$ is easily seen to be a $2^{2g}$--fold cover.

We are left with calculating the Morse--Bott index of the critical manifold $\mathscr{Z}_{L_0}^k$. For this note that the $\omega_I$--Hamiltonian $S^1$--action on $\mathscr{M}_{E}^\ast$ is also $I$-holomorphic. %By a result of Frankel \cite{frankel}, 
The Morse--Bott index of $\mathscr{Z}_{L_0}^k$ then agrees with twice the weight of the $S^1$--action on the normal bundle to $\mathscr{Z}_{L_0}^k \subset \mathscr{M}^\ast$. Using an argument by Hitchin \cite[Proof of Proposition 7.1]{hitchin} (see also \cite{nasatyr-steer,boden-yokogawa}) one can calculate this weight, which gives the Morse--Bott index  $2 \dim_\mathbb{C} H^0 ( C , L_{0}^2 \Lambda_{0}^{-1} N^{-k} K )$. 
\end{proof}
%For this, let $\widetilde{\mathcal{D}}_{L_{0}^{-2} \Lambda_k K }$ denote the moduli space of pairs $(\mathscr{L}_0 , \sigma )$ where $\mathscr{L}_0$ is a holomorphic structure on $L_0$ and $\sigma$ a holomorphic section of $\mathscr{L}_{0}^{-2} \otimes \Lambda_k \otimes K$,

\begin{remark}\label{remark:dimensions}
For future reference, we note here some information we shall use (in the proof of Corollary \ref{corollary:poincaresheaf}) about the dimensions of the connected components of $\mathscr{M}_{E}^\ast (Y)$ containing the components of $\mathscr{Z}$ described above, for simplicity when $C$ has genus $g = 0$ and $\Lambda_0 = \mathbb{C}$. In this case any orbifold line bundle $L_0 \to C$ admits a unique holomorphic structure up to isomorphism, so we shall not specify it. We have:
\begin{enumerate}
\item $\mathscr{Z}_{E_0}^k = \mathscr{N}_{E_0}^\ast (C)$ lies in a component of $\mathscr{M}_{E}^\ast (Y)$ with twice the dimension.
\item $\mathscr{Z}_{L_0}^k$ lies in a component of $\mathscr{M}_{E}^\ast (Y)$ of complex dimension
\[
2( \mathrm{dim}_\mathbb{C} H^0 (C,L_{0}^{-2}N^k K ) + \mathrm{dim}_\mathbb{C} H^0 (C,L_{0}^{2}N^{-k}K )  -1 ).
\]
\end{enumerate}
Both (1) and (2) follow from the fact \cite[p.95]{hitchin} that at each point $(\mathscr{E} , \theta ) \in \mathscr{M}_{E}^\ast (Y)$ (thought of as a stable Higgs pair with fixed determinant $N^{k}$) the tangent space $T_{(\mathscr{E} , \theta )}\mathscr{M}_{E}^\ast (Y)$ contains an $\omega_J$--Lagrangian subspace $W$ isomorphic to
\[
W \cong H^0 ( C , \mathrm{End}_0 \mathscr{E} \otimes K ) / H^0 ( C , \mathrm{End}_0 \mathscr{E} )\cdot \theta
\]
where $H^0 ( C , \mathrm{End}_0 \mathscr{E} )\cdot \theta$ denotes the subspace of the form $[\psi , \theta]$ where $\psi \in H^0 ( C , \mathrm{End}_0 \mathscr{E} )$. In case (1) then $W \cong H^{1} (C , \mathrm{End}_0 \mathscr{E} )$, which is is the tangent space at $\mathscr{E}$ to the moduli space of stable bundles. In case (2) we have a decomposition $\mathscr{E} = L_0 \oplus L_{0}^{-1} N^{k}$ with respect to which $\theta$ is lower triangular, hence $\theta \in H^0 ( C , L_{0}^{-2} N^{k} K)$. Then $\mathrm{End}_0 \mathscr{E} = \mathscr{O}_C \oplus L_{0}^{-2} N^k \oplus L_{0}^{2}N^{-k}$. We have $H^{0} (C , L_{0}^{-2} N^k )\cdot \theta = 0$ since this subspace is given by $[\psi , \theta ]$ with $\psi$ lower-triangular; and $H^{0} (L_{0}^2 N^{-k} ) \cdot \theta \in H^{0} (C , K ) = 0$ since $g = 0$. Thus in this case
\[
W \cong ( H^{0}(C , L_{0}^{-2} N^{k} K )/\mathbb{C}\theta )  \oplus H^{0}(C , L_{0}^{2} N^{-k} K )
\]
which has the required dimension.
\end{remark}

%For this consider the moduli space $\widehat{ \mathscr{Z}_{L_0}^k }$ of pairs $(\widehat{b} , \sigma )$ such that $\widehat{b}$ is a $U(1)$ connection on $L_{0}^{-2} \Lambda_k K$ and $\sigma$ a non-trivial section of the same bundle, subject to the equations:
%\begin{align*}
%F_{\widehat{b}} - F_{\lambda_K} & = 2 i |\sigma |^2 \omega_C .\\
%\overline{\partial}_{\widehat{b}} \sigma & = 0.
%\end{align*}
%Fix $\lambda_K$ to be a $U(1)$ connection on $K$ with constant curvature and inducing the given holomorphic structure on $K$. Consider the map
%\[
%(b , \sigma ) \mapsto ( \widehat{b} = b^{-2} \lambda_k \lambda_K , \sigma )
%\]

\subsubsection{Proofs of main Theorems}

\begin{proof}[Proof of Theorem \ref{theorem:localisationseifert}]
The same result holds under the assumption that $E \to Y$ is an admissible $U(2)$ bundle (which is more general than $Y$ being an integral homology sphere). Then the moduli space $\mathscr{N}^{\ast}_{E}(Y)$ of irreducible projectively flat $U(2)$ connections with fixed determinant is compact. From Theorem \ref{theorem:dimensionalreduction} the moduli space $\mathscr{M}^{\ast}_{E}(Y)$ is Morse--Bott and by Proposition \ref{proposition:Zmorsebott} the critical locus $\mathscr{Z}$ is compact and Morse--Bott. The result now follows from Theorem \ref{theorem:perturbation2general}
\end{proof}

\begin{proof}[Proof of Theorem \ref{theorem:localisationseifert_exact}]
It remains to show that we can achieve equality in the inequality of Theorem \ref{theorem:perturbation2_exact}(2). Since we are assuming that $Y$ is a homology $3$-sphere, then the genus of $C$ is zero and hence the components of $\mathscr{Z} \subset \mathscr{M}^\ast$ are either components of the irreducible $SU(2)$ moduli space $\mathscr{N}^\ast (Y)$ or diffeomorphic to a complex projective space $\mathbb{C}P^e \cong \mathrm{Sym}^e \mathbb{C}P^1$ of suitable dimension $e$. 

We recall that a Morse--Bott function $g$ on a smooth manifold $M$ is called \textit{perfect} if $g$ achieves equality in the Morse--Bott inequalities \cite{atiyah-bott}:
\[
\sum_{i} T^{\mathrm{ind}(g , Z_i )} \cdot \mathcal{P}_T (Z_i ) = \mathcal{P}_T (M).
\] 
where $Z_i$ denote the connected components of the critical locus of $g$.
%the total number of critical points of $g$ of index $k$ equals $\mathrm{rank}_{\mathbb{Z}} H^k (M , \mathbb{Z} )$. 
In this definition, we also allow $M$ to be non-compact, in which case we assume that $g$ is proper, bounded from below and has compact critical locus.

Of course $\mathbb{C}P^e$ admits a perfect Morse function, and for a Seifert-fibered integral homology $3$-sphere then $\mathscr{N}^\ast (Y)$ admits a perfect Morse function as well \cite{kirk-klassen,bauer-okonek,furutasteer}. Thus, there exists a perfect Morse function $g$ on $\mathscr{Z}$. We may then arrange so that the function $\mathfrak{f} : \mathscr{Z} \to \mathbb{R}$ agrees with $g$. Indeed, $\mathfrak{f} = \frac{1}{2} (d \| \Phi \|_{L^{2}}^2 )(\lambda ) + (f \circ \underline{\sigma})|_{\mathscr{Z}}$, so choosing any smooth function $f: \mathbb{C}^N \to \mathbb{R}$ whose restriction to $\mathscr{Z} \subset \mathbb{C}^N$ (embedded via $\underline{\sigma}$) agrees with $g - \frac{1}{2}(d \| \Phi \|_{L^{2}}^2 )(\lambda )$ will do.

On the other hand, the proper and bounded below function $\| \Phi \|_{L^{2}}^2$  on $\mathscr{M}^\ast$ is a perfect Morse--Bott function. This follows from a result of Frankel \cite{frankel} (also used in \cite[\S 7]{hitchin}) since this function is the moment map for a Hamiltonian $S^1$ action on the Kähler manifold $(\mathscr{M}^\ast , I , \omega_I )$ (as was explained in the proof of Proposition \ref{proposition:Zmorsebott}). 

With this understood, following the argument in the proof of Theorem \ref{theorem:perturbation2_exact} we have: if $\pm \varepsilon >0$ and $|\varepsilon|\leq \varepsilon_0$ then
\begin{align*}
M_T & := \sum_{[\mathbb{A} ]\in \mathrm{Crit}S_\varepsilon \cap \mathscr{U} } T^{- \mathbf{sf}(\mathscr{D}_{\mathbb{A}_t}^t )} \\
& = \sum_{i} T^{ \mathrm{ind}(\pm \| \Phi \|_{L^2}^{2}|_{\mathscr{M}^\ast  }  , \mathscr{Z}_i )} \cdot M_T ( \mathscr{Z}_i, \mathfrak{f} )  \\
& = \sum_{i} T^{ \mathrm{ind}(\pm \| \Phi \|_{L^2}^{2}|_{\mathscr{M}^\ast  }  , \mathscr{Z}_i )} \cdot \mathcal{P}_T ( \mathscr{Z}_i) \\%\text{  because } \mathfrak{f} \text{ is a perfect Morse function, i.e. } M_T ( \mathscr{Z}_i , \mathfrak{f} ) = P_T (\mathscr{Z}_i  )\\
& = \mathcal{P}_{T}(\mathscr{M}^\ast ) \text{ or } \mathcal{P}_{T}^{c} (\mathscr{M}^\ast ) \text{  (according to whether } \varepsilon>0 \text{ or } \varepsilon< 0 \text{)}
\end{align*}
where the third line uses the perfection of the Morse function $\mathfrak{f}$ on each $\mathscr{Z}_i$, i.e. $M_T ( \mathscr{Z}_i , \mathfrak{f} ) = P_T (\mathscr{Z}_i  )$, and the last line uses the perfection of the Morse--Bott function $\| \Phi \|_{L^{2}}^2 $ on $\mathscr{M}^\ast$, together with the fact that it is proper and bounded below.
\end{proof}

\section{Calculations}\label{section:calculations}

Let $Y = S(N) \to C^2$ be a Seifert-fibered $3$-manifold over a $2$-orbifold $C$ with exactly $n$ orbifold points with isotropy orders $\alpha_1 , \ldots , \alpha_n$. Then (see \cite[Theorem 2.3]{furutasteer}):
\begin{itemize}
\item $Y$ is a rational homology $3$-sphere if and only if the genus of $C$ is zero (i.e. $|C| \cong S^2$)
\item $Y$ is an integral homology $3$-sphere if and only if the genus of $C$ is zero, the topological orbifold Picard group of $C$ is infinite cyclic ($\mathrm{Pic}^t (C) \cong \mathbb{Z}$), and $N$ is one of the two generators of $\mathrm{Pic}^t (C)$.
\end{itemize}

Throughout this section we assume that $Y \to C$ is a Seifert-fibered integral homology $3$-sphere over a $2$-orbifold $C = S^2 (\alpha_1 , \ldots , \alpha_n )$. Our goal is to calculate of the Poincaré polynomial of the moduli space of stable $SL(2, \mathbb{C})$ flat connections $\mathscr{M}^\ast (Y)$ (on the trivial $SU(2)$ bundle).

\subsection{The locus $\mathscr{Z}$}

For a given complex structure on the $2$-orbifold $C$, let $\mathcal{D}_{< d}(C)$ be the moduli space of orbifold effective divisors on $C$ of orbifold degree $< d$. 
\begin{corollary}\label{corollary:ZZHS3}
Let $Y = S(N) \to C$ be a Seifert-fibered integral homology $3$-sphere, equipped with a Seifert metric $g$ (and hence also a complex structure on $C$). Then there is a diffeomorphism:
\[
\mathscr{Z}(g) \cong \mathscr{N}^\ast (Y ) \sqcup \mathcal{D}_{<\mathrm{deg}K}(C).
\]
\end{corollary}

\begin{proof}
%We begin by recalling the following result:

%\begin{theorem}[\cite{MOY}]\label{theorem:MOY} Let $Y = S(N) \xrightarrow{\pi} C$ be a Seifert-fibered $3$-manifold equipped with a Seifert metric $g = \eta^2 + \pi^\ast g_C$. Then the moduli space of irreducible Seiberg--Witten on $(Y, g )$, defined using the adiabatic connection $\nabla^{o}$ on $TY$ and ranging over all spin-c structures on $Y$, is Morse--Bott and diffeomorphic to

%\[
%\mathscr{M}_{\mathrm{SW}}^\ast (Y, g) = \Big( \bigsqcup_{L_0 \, , \, 0 \leq \mathrm{deg}L_0 < \frac{1}{2} \mathrm{deg}K} \mathcal{D}_{L_0}(C) \Big) \sqcup \Big( \bigsqcup_{L_0 \, , \, 0 \leq \mathrm{deg}L_0 < \frac{1}{2} \mathrm{deg}K} \mathcal{D}_{K \otimes L_{0}^{-1} }(C) \Big). \]
%\end{theorem}

%\begin{remark}
%In \cite{MOY} the authors assume that the metric $g_C$ has constant curvature, but this assumption is not required in Theorem \ref{theorem:MOY}.
%\end{remark}
%We note that $Y$ is a rational homology $3$-sphere if and only if the genus of $C$ is zero. In this case, and for an $U(2)$ bundle $E \to Y$ with fixed determinant, 
By Proposition \ref{proposition:Zmorsebott} the locus $\mathscr{Z}(g) \cap \{ \Phi \neq 0\} $ is diffeomorphic to:
\begin{align*}
\Big( \bigsqcup_{L_0 \, , \, 0 < \mathrm{deg}L_0  \leq \frac{1}{2} \mathrm{deg}K } \mathcal{D}_{L_{0}^{-2} K}(C) \Big)  \sqcup   \Big( \bigsqcup_{0 < \mathrm{deg}L_0 - \frac{1}{2} \mathrm{deg}  N  \leq \frac{1}{2} \mathrm{deg}K} \mathcal{D}_{L_{0}^{-2}  N K }(C) \Big).
\end{align*}
%\begin{align}
%\Big( \bigsqcup_{L_0 \, , \, 0 < \mathrm{deg}L_0 - \frac{1}{2} \mathrm{deg}  \Lambda_0  \leq \frac{1}{2} \mathrm{deg}K_C } \mathcal{D}_{L_{0}^{-2}\Lambda_0 K}(C) \Big)  \sqcup   \Big( \bigsqcup_{0 < \mathrm{deg}L_0 - \frac{1}{2} \mathrm{deg} ( \Lambda_0 \otimes N ) \leq \frac{1}{2} \mathrm{deg}K_C} \mathcal{D}_{L_{0}^{-2} \Lambda_0 N K }(C) \Big).\label{monopolebranchZ}
%\end{align}
%The Seifert-fibered $3$-manifold $Y = S(N)$ is an integral homology $3$-sphere if and only if the genus of $C$ is zero, the orbifold topological Picard group of $C$ is cyclic ($\mathrm{Pic}^t (C) \cong \mathbb{Z}$) and $N$ is one of its two generators \cite[Theorem 2.3]{furutasteer}. 
%The only bundle to consider here is the trivial $SU(2)$ bundle $E \to Y$. Hence we will take $\Lambda_0 = \mathbb{C}$ in (\ref{monopolebranchZ}). We may assume $\mathrm{deg}K > 0$, otherwise both $\mathscr{Z}(g) \cap \{ \Phi \neq 0\}$ and $\mathcal{D}_{< \mathrm{deg} K}$ are empty, by (\ref{monopolebranchZ}). Up to reversing the orientation of $Y$ we may suppose that $\mathrm{deg}N < 0$. Then $K = N^{-d}$ for some $d \in \mathbb{Z}_{>0}$. 
The orbifold line bundles on $C$ with orbifold degree $< \mathrm{deg}K$ are 
\begin{align}
\mathbb{C} = N^0 , N^{-1} , \ldots , N^{-d +1}.\label{Zbundles}
\end{align}
Consider now the following two cases:

\textit{Non-spin case: $d = 2m +1$, for some $m\in \mathbb{Z}_{\geq 0}$.} 
%The line bundles $L_0$ with $0 \leq \mathrm{deg}L_0 < \frac{1}{2} \mathrm{deg}K$ are:
%\begin{align}
%\mathbb{C} = N^{0} , N^{-1} , \ldots , N^{-m} \label{SW1nonspin}
%\end{align}
%and the line bundles $L_{0}^{-1} K$ with $0 \leq \mathrm{deg}L_0 < \frac{1}{2} \mathrm{deg}K$ are:
%\begin{align}
%N^{-m-1 }, N^{-m-2} , \ldots ,  N^{-d} = K.\label{SW2nonspin}
%\end{align}
The line bundles $L_{0}^{-2} K$ with $0 < \mathrm{deg}L_0 \leq \frac{1}{2} \mathrm{deg}K$ are given by:
\begin{align}
N^{-1} , N^{-3} , \ldots , N^{-d +2}\label{Z1nonspin}
\end{align}
and the line bundles $L_{0}^{-2} N K $ with $0 < \mathrm{deg}L_0 - \frac{1}{2}\mathrm{deg}N \leq \frac{1}{2}\mathrm{deg}K$ are given by 
\begin{align}
\mathbb{C} = N^{0} , N^{-2} , \ldots , N^{-d +1}.\label{Z2nonspin}
\end{align}
The two collections (\ref{Zbundles}) and (\ref{Z1nonspin}-\ref{Z2nonspin}) consist of the same line bundles. %, except that the first set contains $K$ but the second does not. But the space of orbifold effective divisors $\mathcal{D}_K (C)$ in $K$ is empty, because $C$ has genus zero then $|C| \cong \mathbb{P}^1$ and hence $K$ admits no non-trivial holomorphic sections (indeed, holomorphic sections of the orbifold bundle $K$ correspond to those of its desingularisation $|K| = K_{\mathbb{P}^1} \to |C| = \mathbb{P}^1$).
 This establishes the required result in the non-spin case.

\textit{Spin case: $d = 2m$, for some $m \in \mathbb{Z}_{>0}$.} %The line bundles $L_0$ with $0 \leq \mathrm{deg}L_0 < \frac{1}{2} \mathrm{deg}K$ are:
%\begin{align}
%\mathbb{C} = N^{0} , N^{-1} , \ldots , N^{-m +1} \label{SW1spin}
%\end{align}
%and the line bundles $L_{0}^{-1} K$ with $0 \leq \mathrm{deg}L_0 < \frac{1}{2} \mathrm{deg}K$ are:
%\begin{align}
%N^{-m -1 }, N^{-m-2} , \ldots ,  N^{-d} = K.\label{SW2spin}
%\end{align}
The line bundles $L_{0}^{-2} K$ with $0 < \mathrm{deg}L_0 \leq \frac{1}{2} \mathrm{deg}K$ are given by:
\begin{align}
\mathbb{C} = N^{0} , N^{-2} , \ldots , N^{-d +2}\label{Z1spin}
\end{align}
and the line bundles $L_{0}^{-2} N K $ with $0 < \mathrm{deg}L_0 - \frac{1}{2}\mathrm{deg}N \leq \frac{1}{2}\mathrm{deg}K$ are given by: 
\begin{align}
 N^{-1} , N^{-3} , \ldots , N^{-d +1}.\label{Z2spin}
\end{align}
Both collections (\ref{Zbundles}) and (\ref{Z1spin}-\ref{Z2spin}) consist of the same line bundles. %, except that the first set contains $K$ but the second does not and the second set contains $K^{1/2} = N^{-m} $ but the first does not. As before, the space of orbifold effective divisors $\mathcal{D}_K (C)$ in $K$ is empty. Likewise, $K^{1/2}$ doesn't admit non-trivial holomorphic sections because otherwise $K$ would, thus $\mathcal{D}_{K^{1/2}}(C)$ is also empty. 
 This establishes the required result in the spin case.
\end{proof}

The space $\mathcal{D}_{< \mathrm{deg}K}(C)$ has one connected component $\mathcal{D}_L$ for each orbifold line bundle $L$ with $0 \leq \mathrm{deg}L < \mathrm{deg}K$ (namely, $\mathcal{D}_L (C)$ is the moduli space of orbifold effective divisors in the line bundle $L$). Through the isomorphism established in Corollary \ref{corollary:ZZHS3}, the Morse--Bott index of $\| \Phi \|_{L^2}^{2}|_{\mathscr{M}^\ast}$ along the critical submanifold $\mathcal{D}_L (C) \subset \mathscr{Z}$ is given by Proposition \ref{proposition:Zmorsebott} by
%Using an argument by Hitchin \cite[Proof of Proposition 7.1]{hitchin} (see also \cite{nasatyr-steer,boden-yokogawa}) one can calculate this weight and thus the Morse--Bott indices:
\[
\mathrm{ind}( \| \Phi \|_{L^2}^{2}|_{\mathscr{M}^\ast } , \mathcal{D}_L) = 2 \dim_\mathbb{C} H^0 ( C , L^{-1}K^2 )
\]
where $L$ is understood to be equipped with its unique holomorphic structure (since $C$ has genus zero). In what follows we denote the complex dimension of the $k$th sheaf cohomology groups by $h^k$. Using again the fact that $\| \Phi \|_{L^2}^2$ is a perfect Morse--Bott function on $\mathscr{M}^\ast (Y)$ we obtain:
\begin{align}
\mathcal{P}_T ( \mathscr{M}^\ast (Y) ) = \mathcal{P}_T (\mathscr{N}^\ast (Y)) + \sum_{0 \leq \mathrm{deg}L < \mathrm{deg}K} (T^2)^{h^0 (C , L^{-1} K^2) } \cdot \mathcal{P}_T (\mathcal{D}_L (C)) \label{poincareM}
\end{align}

\subsection{Formulae for Poincaré polynomials}

%If the orbifold locus of the obifold Riemann surface $C$ consists of $n$ points with isotropies of order $\alpha_i \leq 2$, $1\leq i \leq n $, then 
Concretely, the connected components of $\mathcal{D}_{<d}(C)$, the moduli space of orbifold effective divisors on $C$ (for some orbifold complex structure on $C$), are indexed by $(n+1)$--tuples of non-negative integers $\underline{e} = (e , \beta_1 , \ldots , \beta_n ) \in (\mathbb{Z}_{\geq 0})^{n+1}$ such that 
\begin{align}
\beta_i < \alpha_i \quad \text{ and } \quad \mathrm{deg}(\underline{e}) := e + \sum_{i=1}^n \frac{\beta_i }{\alpha_i} < d , \label{ineqab}
\end{align}
and the connected component of $\mathcal{D}_{<d} (C)$ corresponding to $\underline{e}$ is diffeomorphic to $\mathrm{Sym}^{e} S^2 \cong \mathbb{C}P^{e}$. 

Namely, an orbifold effective divisor $(\mathscr{L} , s )$ (i.e. an orbifold  holomorphic line bundle with an non-trivial holomorphic section up to $\mathbb{C}^\ast$--rescaling) vanishes on the $i$th orbifold point of $C$ with a fractional order, whose fractional part is $\beta_i / \alpha_i $ with $0 \leq \beta_i < \alpha_i$, and $e$ is the (ordinary) degree of the desingularised holomorphic line bundle $|\mathscr{L}|\to |C|$ (see \cite{MOY,furutasteer} for details).

%, where $|C| \cong S^2$ is the  underlying $C$. (In particular, the smooth manifold underlying $\mathcal{D}_{<d}$ does not depend on the chosen complex structure on $C$.) Also, when $C$ has genus zero then $\mathrm{Sym}^{e} |C| \cong \mathbb{C}P^e$. From the previous result we obtain:

\begin{proof}[Proof of Corollary \ref{corollary:poincare}]
%Recall that the function $\| \Phi \|_{L^2}^2$ on $\mathscr{M}^\ast$ is the moment map for a Hamiltonian $S^1$ action on the Kähler manifold $(\mathscr{M}^\ast , I , \omega_I )$ (as was explained in the proof of Proposition \ref{proposition:Zmorsebott}). Furthermore, the $S^1$--action extends to an $I$--holomorphic $\mathbb{C}^\ast$--action. By a result of Frankel \cite{frankel} then $\| \Phi \|_{L^2}^2$ on $\mathscr{M}^\ast$ is a perfect Morse--Bott function, and the Morse--Bott index of $\| \Phi \|_{L^2}^2$ on a connected component $\mathscr{Z}_i$ of $\mathscr{Z} = \mathrm{Crit}(\| \Phi \|_{L^2}^{2}|_{\mathscr{M}^\ast })$ agrees with twice the weight of the $\mathbb{C}^\ast$--action on the normal bundle to $\mathscr{Z}_i \subset \mathscr{M}^\ast$. 

Let $\mathscr{Z}(\underline{e} )\subset \mathscr{Z}$ be the connected component of $\mathcal{D}_{< \mathrm{deg} K }(C) \subset \mathscr{Z}$ (cf. Corollary \ref{corollary:ZZHS3}) parametrised by the data $\underline{e} = (e , \beta_1 , \ldots , \beta_n ) \in (\mathbb{Z}_{\geq 0})^{n+1}$, with 
\[
\beta_i < \alpha_i \quad , \quad \mathrm{deg}(\underline{e}) < \mathrm{deg}K = - \chi (C ) = -2 + n - \sum_{i =1}^n \frac{1}{\alpha_i } .
\]
Let $L_{\underline{e}}$ denote the unique orbifold line bundle corresponding to the data $\underline{e}$ (equipped with its unique orbifold holomorphic structure up to isomorphism, since $C$ has genus zero). We have $\mathscr{Z}(\underline{e}) = \mathcal{D}_{L_{\underline{e}}} (C) \cong \mathbb{C}P^e$. %By Proposition \ref{proposition:Zmorsebott}, the Morse--Bott index of $\| \Phi \|_{L^2}^{2}|_{\mathscr{M}^\ast}$ along the critical submanifold $\mathscr{Z}(\underline{e} )$ is given by
%Using an argument by Hitchin \cite[Proof of Proposition 7.1]{hitchin} (see also \cite{nasatyr-steer,boden-yokogawa}) one can calculate this weight and thus the Morse--Bott indices:
%\[
%\mathrm{ind}( \| \Phi \|_{L^2}^{2}|_{\mathscr{M}^\ast } , \mathscr{Z}(\underline{e}) ) = 2 \dim_\mathbb{C} H^0 ( C , L_{\underline{e}}^{-1}K^2 )
%\]
%Denote the complex dimension of the $k$th sheaf cohomology groups by $h^k$. Using again the fact that $\| \Phi \|_{L^2}^2$ is a perfect Morse--Bott function on $\mathscr{M}^\ast (Y)$ we obtain

By (\ref{poincareM}) we have
\[
\mathcal{P}_T ( \mathscr{M}^\ast (Y) ) = \mathcal{P}_T (\mathscr{N}^\ast (Y)) + \sum_{\underline{e}} (T^2)^{h^0 (C , L_{\underline{e}}^{-1} K^2) } \cdot \mathcal{P}_T (\mathbb{C}P^e) .
\]
By the orbifold versions of Serre duality and the Riemann--Roch Theorem we have
\[
h^0 ( L_{\underline{e}} K^{-1} ) - h^0 ( L_{\underline{e}}^{-1} K^2 ) = h^0 ( L_{\underline{e}} K^{-1} ) - h^1 (L_{\underline{e}}K^{-1} ) = 1 + \mathrm{deg}| L_{\underline{e}} K^{-1} |.
\]
But $\mathrm{deg} L_{\underline{e}} K^{-1} <0$, hence:
\[
 h^0 ( L_{\underline{e}}^{-1} K^2 ) = -1 - \mathrm{deg}| L_{\underline{e}} K^{-1} |
\]
To compute $\mathrm{deg}| L_{\underline{e}} K^{-1} |$ (the ordinary degree of the desingularisation of $L_{\underline{e}} K^{-1}$), note that the anticanonical bundle $K^{-1}$ has $\mathrm{deg}|K^{-1}| = 2-n$ and isotropy $1$ at each orbifold point of $C$, and $L_{\underline{e}}$ has $\mathrm{deg}|L_{\underline{e}}| = e$ and isotropy $\beta_i$ at the $i$th orbifold point. From this we obtain:
\[
\mathrm{deg}| L_{\underline{e}} K^{-1} | = e + 2-n + \sum_{i = 1}^n \lfloor \frac{\beta_i +1}{\alpha_i}\rfloor .
\]
Thus
\begin{align*}
 h^0 ( L_{\underline{e}}^{-1} K^2 )  = -e -3 + n - \sum_{i = 1}^n \lfloor \frac{\beta_i +1}{\alpha_i}\rfloor  = -\chi(C) - \mathrm{deg} (\underline{e}) -1  +\sum_{i=1}^n   \big\{ \frac{\beta_i +1}{\alpha_i }\big\}.
 \end{align*}
 From this calculation, the promised formula for $\mathcal{P}_T ( \mathscr{M}^\ast (Y))$ now follows.
\end{proof}

%\begin{corollary}
%Let $Y = \pm \Sigma ( \alpha_1, \ldots , \alpha_n )$ be a Seifert-fibered homology $3$-sphere. Then the Euler characteristic of the stable locus $\mathscr{M}^\ast (Y)$ in the $SL(2, \mathbb{C})$ character variety is related to the irreducible locus in the $SU(2)$ character variety $\mathscr{N}^\ast (Y)$ by: 
%\[
%\chi ( \mathscr{M}^\ast (Y) ) = \chi ( \mathscr{N}^\ast (Y) ) + \sum_{\underline{e}} \chi ( \mathbb{C}P^e )
%\]
%where the sum ranges over all $(n+1)$--tuples of non-negative integers $\underline{e} = ( e, \beta_1 , \ldots , \beta_n )$ satisfying (\ref{ineqab}) with $d = \mathrm{deg}K$. \end{corollary}
%\begin{proof}
%Need indices even.
%\end{proof}

\subsection{Relation to the Milnor number}

\begin{proof}[Proof of Corollary \ref{corollary:milnornumber}]
Let $Y = S(N)$ have Seifert data (we use the conventions of \cite[\S 2]{MOY}):
\[
(b , \gamma_1 , \ldots, \gamma_n )
\]
where $b, \gamma_i \in \mathbb{Z}$, $0 < \gamma_i < \alpha_i$ and $\gamma_i$ is coprime to $\alpha_i$. Recall also, that the Seifert data of the canonical bundle $K $ is $(-2, \alpha_1 -1 , \ldots , \alpha_n -1 )$. We orient $Y$ as the link of an isolated singularity, which means that the orbifold degree $\mathrm{deg}N = b + \sum_{i =1}^n \gamma_i /\alpha_i$ is $<0$.

The geometric genus $p_g (X,0)$ of a weighted-homogeneous normal isolated surface singularity $(X, 0)$ can be computed by the following formula due to Pinkham \cite[Theorem 5.7]{pinkham} and Dolgachev \cite{dolgachev} (also attributed to Demazure). We use a formulation due to Némethi (see \cite[Remark 11.14]{nemethiplumbed} or \cite[\S 5.1.26]{nemethibook}, and beware that Némethi's $b_0$ is $-b$ in our convention):
\begin{align*}
p_g (X,0) = \sum_{\ell \geq 0} \max \{ 0 , -N(\ell) -1\} .
\end{align*}
where 
\[
N(\ell) := - \ell b - \sum_{i= 1}^n \lceil \frac{\ell \gamma_i}{\alpha_i}\rceil . 
\]

We observe the following: for $\ell \in \mathbb{Z}_{\geq 0}$
\begin{align*}
\mathrm{deg}| K \otimes N^{\ell}|  & = - 2 + \ell b  + \sum_{i =1}^n \lfloor \frac{\alpha_i -1 + \gamma_i }{\alpha_i}\rfloor\\
& = -2 + \ell b + \sum_{i =1}^n \lceil \frac{\ell \gamma_i}{\alpha_i}\rceil \\
& = -2 - N (\ell )
\end{align*}
and hence we have $-N(\ell ) -1 = \mathrm{deg}|K \otimes N^{\ell}| +1$. By this formula, the integer $-N(\ell) -1$ is positive if and only if $\mathrm{deg}|K \otimes N^{\ell}| \geq0$.

Since $C$ has genus zero, the condition $\mathrm{deg}|K \otimes N^{\ell}| \geq0$ implies that $K \otimes N^{\ell}$ admits non-trivial holomorphic sections (since for an orbifold holomorphic line bundle $L \to C$ we have $H^{0}(C ,L) \cong H^{0} (|C|, |L|)$, see e.g. \cite[Proposition 2.0.14]{MOY}). Hence:
\begin{itemize}
\item $\mathrm{deg}N^{-\ell} \leq \mathrm{deg}K$, i.e. $\ell \leq - \frac{\mathrm{deg}K}{\mathrm{deg}N}$.
\item $h^0 ( C , K \otimes N^{\ell} ) = \mathrm{deg}|K \otimes N^{\ell}| +1$. One way of seeing this is as follows. Consider for $\mathscr{L} := K \otimes N^{\ell}$ the short exact sequence of sheaves
\[
0 \xrightarrow{s} \mathscr{O}_C \to \mathscr{L}\to\mathscr{L}|_{D} \to 0
\]
associated to a non-trivial holomorphic section $s$ of $\mathscr{L}$ (where $\mathscr{L}|_{D}$ is defined as the cokernel sheaf of multiplication with $s$; here $D$ stands for the `zero locus' of $s$). Since $\mathscr{L}|_D$ is supported on a zero-dimensional locus, then $H^{1} (C , \mathscr{L}|_D ) =0$. Since $C$ has genus zero then $H^1 (C , \mathscr{O}_C ) = 0$. Hence by the long exact sequence in cohomology we obtain $H^{1} ( C , \mathscr{L} ) = 0$, i.e. $h^1 (C , \mathscr{L} ) = 0$. Now, by this vanishing and the orbifold version of the Riemann--Roch formula we obtain the required identity (here we also use that $C$ has genus zero):
\begin{align*}
h^{0} (C , \mathscr{L} )  = h^0 (C , \mathscr{L} ) - h^{1} (C , \mathscr{L} ) = 1 + \mathrm{deg}|\mathscr{L}| .
\end{align*}

%since on an orbifold Riemann surface of genus zero the moduli space of effective divisors in an orbifold line bundle $L$ (equipped with its unique holomorphic structure, since $C$ has genus zero) can be described as $\mathrm{Sym}^{\mathrm{deg}|L|}|C| \cong \mathbb{C}P^{\mathrm{deg}|L|}$ but also as $( H^{0}(C, L) \setminus 0 )/\mathbb{C}^\ast$.
\end{itemize}

Putting everything together, and using the fact that $N$ generates the topological Picard group of $C$ (since $Y$ is an integral homology sphere), gives:\\
\begin{align*}
p_g (X, 0)  = \sum_{0 \leq \ell \leq -\frac{\mathrm{deg}K}{\mathrm{deg}N}}h^{0} ( C, K \otimes N^{\ell}) =   \sum_{ 0 \leq \mathrm{deg}L \leq \mathrm{deg}K}h^{0} ( C, L ) = \sum_{ 0 \leq \mathrm{deg}L < \mathrm{deg}K}h^{0} ( C, L )\end{align*}
where the last identity follows from the fact that $H^{0} (C , K ) = 0$ when $C$ has genus zero. The right-hand side in this identity is $\chi ( \mathcal{D}_{< \mathrm{deg}K}(C))$, thus (\ref{poincareM}) gives:
\[
\chi ( \mathscr{M}^\ast (Y) ) = \chi ( \mathscr{N}^\ast (Y) ) + p_g (X, 0 ) .
\]
At this point, what remains of the proof of Corollary \ref{corollary:milnornumber} was already explained in \S \ref{subsection:milnor}. 
%Now, if $M$ denotes the Milnor fiber 
%since $\chi ( \mathscr{N}^\ast (Y) ) = -2 \lambda (Y) = \sigma (M)$ \cite{fintushel-stern,neumann-wahl} where
\end{proof}

\subsection{Comparison with sheaf-theoretic $SL(2, \mathbb{C})$ Floer cohomology}

We conclude by discussing the gauge-theoretic interpretation of the sheaf-theoretic $SL(2, \mathbb{C})$ Floer cohomology of Abouzaid--Manolescu.

For a closed oriented $3$-manifold $Y$, Abouzaid--Manolescu \cite{abouzaid-manolescu} have defined a \textit{sheaf-theoretic $SL(2, \mathbb{C})$ Floer cohomology} $HP^\ast (Y)$ and also cohomology with \textit{compact-supports} $HP^{\ast}_c (Y )$, both of which are finitely-generated $\mathbb{Z}$--graded abelian groups. These invariants arise from certain perverse sheaf $P^\bullet (Y)$ on the stable locus $\mathscr{M}^{\ast}(Y)$ in the $SL(2, \mathbb{C})$ character variety, defined using the methods of \cite{bussi,joyce}, by passing to hypercohomology or hypercohomology with compact-supports: $HP^\ast (Y) = \mathbb{H}^\ast ( \mathscr{M}^\ast (Y) , P^\bullet (Y) )   $, $HP^{\ast}_c (Y) = \mathbb{H}^{\ast}_{c} ( \mathscr{M}^\ast (Y) , P^\bullet (Y) )$. It remains an open problem to describe these invariants using gauge theory, which would likely help to elucidate their geometric content. 

Consider the \textit{sheaf-theoretic $SL(2,\mathbb{C})$ Casson invariant} $\lambda^P (Y) \in \mathbb{Z}$, which is defined as the Euler characteristic of $HP^\ast (Y)$; and more generally the \textit{sheaf-theoretic Poincaré polynomial}
\[
\mathcal{P}_{T}^P (Y) := \sum_{n \in \mathbb{Z}} T^n \cdot \mathrm{rank}_{\mathbb{Z}} HP^n (Y) \in \mathbb{Z}[T , T^{-1}].
\]
One may also consider their compactly-supported analogues, which we denote $\lambda^{P}_c (Y)$ and $\mathcal{P}_{T }^{P,c} (Y)$, but they contain the same information: $\lambda_{c}^P (Y) = \lambda^P (Y)$ and $\mathcal{P}_{T }^{P,c} (Y) = \mathcal{P}_{T^{-1}}^P (Y)$ \cite[Proposition 8.2]{abouzaid-manolescu}. We obtain a gauge-theoretic interpretation of these invariants when $Y$ is a Seifert-fibered homology sphere:
\begin{corollary}\label{corollary:sheaf}
Let $Y$ be a Seifert-fibered homology $3$-sphere. Then:
\begin{enumerate}
\item In the situation of Theorem \ref{theorem:localisationseifert}(2) we have
\begin{align*}
  \sum_{[\mathbb{A}] \in \mathscr{M}_{\theta }^{\ast , \varepsilon} \cap \mathscr{U}} (-1)^{\mathbf{sf}_{2}(\mathscr{D}_{[\mathbb{A}],\theta}^t )} = \chi (\mathscr{M}^\ast (Y) ) =  \lambda^P (Y) 
  \end{align*}
\item In the situation of Theorem \ref{theorem:localisationseifert_exact}(2) we have
\begin{align*}
  \sum_{[\mathbb{A}] \in [\mathbb{A}] \in \mathscr{M}_{\theta }^{\ast , \varepsilon} \cap \mathscr{U}} T^{ \mathbf{sf}(\mathscr{D}_{[\mathbb{A}_t] , \theta}^t + \delta ) - \mathrm{dim}_{\mathbb{C}}\mathrm{Ker}\mathscr{D}_{[\mathbb{A}_0], \theta}} =  \begin{cases}
       \mathcal{P}_{T}^P ( Y ) &\quad\text{if } \varepsilon >0 \\
       \mathcal{P}^{P,c}_T ( Y ) &\quad\text{if } \varepsilon < 0 .\\
     \end{cases} 
  \end{align*}
  \end{enumerate}
\end{corollary}

%The method of calculation that yields Corollary \ref{corollary:poincare} also calculates the invariants $\lambda^P (Y)$ and $HP^\ast (Y)$ for a Seifert-fibered homology sphere. 
The proof will be given below. We also provide formulae for $\lambda^P (Y)$ and $\mathcal{P}_{T}^P (Y)$. To describe these, for a given smooth compact manifold $M$ with connected components $M_j$ we consider the `normalized' Poincaré polynomial $\widehat{\mathcal{P}}_T (M) := \sum_j (T^{-1})^{ \mathrm{dim}M_j}\cdot \mathcal{P}_T (M_j ) \in \mathbb{Z}[T^{-1}]$.
%Corollary \ref{corollary:poincare} also yields a formula relating the sheaf-theoretic $SL(2, \mathbb{C} )$ Casson invariant $\lambda^P (Y)$ \cite{abouzaid-manolescu} and the $SU(2)$ Casson invariant of a Seifert-fibered homology $3$-sphere (orienting $Y$ as the link of an isolated surface singularity):
%\begin{align}
%\lambda^P (Y)  = - 2 \lambda (Y) + \sum_{\underline{e}} \chi ( \mathbb{C}P^e )  .\label{identitylambda}
%\end{align}
%The sheaf-theoretic Poincaré polynomial $\mathcal{P}_{T}^P (Y)$ admits the following . Let  
%: for a Seifert-fibered homology sphere $Y = \Sigma ( \alpha _1 , \ldots , \alpha_n )$ we have
%\[
%\lambda^P (Y) = - 2 \lambda (Y) + \sum_{\underline{e}} \chi ( \mathbb{C}P^e ).
%\]
\begin{corollary}\label{corollary:poincaresheaf} Let $Y$ be a Seifert-fibered homology $3$-sphere. Then:
\begin{enumerate}
\item  $\mathcal{P}_{T}^P (Y)  = \widehat{\mathcal{P}}_T ( \mathscr{N}^\ast (Y)) + \sum_{\underline{e}}\widehat{\mathcal{P}}_T ( \mathbb{C}P^e )$
\item $\lambda^P (Y)  = - 2 \lambda (Y) + \sum_{\underline{e}} \chi ( \mathbb{C}P^e )  $ 
\end{enumerate}
where $\underline{e} = ( e , \beta_1 , \ldots , \beta_n )$ ranges over all vectors of non-negative integers such that $\beta_i < \alpha_i $ and $ \mathrm{deg}(\underline{e}) < -\chi (C)$.
\end{corollary}

\begin{proof}[Proof of Corollary \ref{corollary:sheaf}]

For a complex manifold $M$ consider the following `complex normalized' version of the Poincaré polynomial (resp. compactly-supported Poincaré polynomial): decomposing $M$ into connected components $M_{j}$ we set
\begin{align*}
\widetilde{\mathcal{P}}_T (X  ) := \sum_{j } (T^{-1})^{\mathrm{dim}_{\mathbb{C}} M_j} \cdot \mathcal{P}_T (M_j ) \\  \widetilde{\mathcal{P}}_{T}^{c} (X ) := \sum_{j } (T^{-1})^{\mathrm{dim}_{\mathbb{C}} X_j} \cdot \mathcal{P}_{T}^{c} (X_j).
\end{align*}

For a Seifert-fibered integral homology $3$-sphere, the Poincaré polynomials of the sheaf-theoretic $SL(2,\mathbb{C})$ Floer cohomologies recover the complex normalised Poincaré polynomials of $\mathscr{M}^\ast (Y)$ \cite[Formula 35]{abouzaid-manolescu}:
\begin{align*}
\mathcal{P}_{T}^{P} (Y) =\widetilde{\mathcal{P}}_{T} (\mathscr{M}^\ast (Y) )\quad , \quad \mathcal{P}_{T}^{P,c} (Y) =\widetilde{\mathcal{P}}_{T}^{c} (\mathscr{M}^\ast (Y) ).
\end{align*}
%\sum_{n\in \mathbb{Z}}T^n \cdot \mathrm{rank}_\mathbb{Z} HP^n (Y) = \widehat{\mathcal{P}}_{T} (\mathscr{M}^\ast (Y) )

Since $\mathscr{M}^\ast (Y)$ is the Morse--Bott critical locus of the functional $S_\theta$ by Theorem \ref{theorem:dimensionalreduction}, then its complex dimension along the connected component containing the point $[\mathbb{A}_0] \in \mathscr{M}^\ast (Y)$ is $\mathrm{dim}_\mathbb{C} \mathrm{Ker} \mathscr{D}_{[\mathbb{A}_0 ] , \theta }$. The argument from the proof of Theorem \ref{theorem:localisationseifert_exact} readily gives the required formula for $\mathcal{P}_{T}^{P} (Y)$ and $\mathcal{P}_{T}^{P,c}(Y)$ as a gauge-theoretic count of critical points.
\end{proof}

\begin{proof}[Proof of Corollary \ref{corollary:poincaresheaf}]

Recall from the proof of Corollary \ref{corollary:sheaf} that $\mathcal{P}_{T}^P (Y) = \widetilde{\mathcal{P}}_T (\mathscr{M}^\ast (Y))$. Using again the fact that $\| \Phi \|_{L^2}^2$ is a perfect Morse--Bott function, together with Corollary \ref{corollary:ZZHS3} and Remark \ref{remark:dimensions}(2), we obtain
\begin{align*}
\widetilde{\mathcal{P}}_T (\mathscr{M}^\ast (Y) ) & = \widehat{\mathcal{P}}_T ( \mathscr{N}^\ast (Y)) + \sum_{0 \leq \mathrm{deg}L < \mathrm{deg}K} T^{2 h^0 (L^{-1}K ) - 2( h^0 (L) + h^0 ( L^{-1}K^2) -1 )} \mathcal{P}_T ( \mathcal{D}_L (C) )\\
& = \widehat{\mathcal{P}}_T ( \mathscr{N}^\ast (Y)) + \sum_{0 \leq \mathrm{deg}L < \mathrm{deg}K} \widehat{\mathcal{P}}_T ( \mathcal{D}_L (C) )
\end{align*}
from which the required formula now follows.
\end{proof}

\bibliographystyle{alpha}
\bibliography{main.bib}

\end{document}